\documentclass[a4paper,english]{smfart}

\usepackage[T1]{fontenc} 
\usepackage[utf8]{inputenc}
\usepackage[english]{babel} 
\usepackage{newpxtext} %
\usepackage{eulervm} %

\usepackage{amssymb,textcomp,mathrsfs,braket,commath,relsize}

\usepackage{stmaryrd}  %
	\expandafter\def\csname opt@stmaryrd.sty\endcsname
	{only,shortleftarrow,shortrightarrow}

\usepackage{bm} %
	\SetSymbolFont{stmry}{bold}{U}{stmry}{m}{n} %

\usepackage{etoolbox} 
    \newcommand{\addQEDstyle}[2]{\AtBeginEnvironment{#1}{\pushQED{\qed}\renewcommand{\qedsymbol}{#2}}
    \AtEndEnvironment{#1}{\popQED}} %
    \addQEDstyle{rema}{$\triangle$}
    \addQEDstyle{rema*}{$\triangle$} %
	\addQEDstyle{exem}{$\triangle$}
    \addQEDstyle{exem*}{$\triangle$} %

\usepackage{tikz}
	\usetikzlibrary{cd} %

\usepackage{float} %
\usepackage{extpfeil} %
\usepackage{accents}
\AtBeginDocument{\def\MR#1{}} %
\apptocmd{\sloppy}{\hbadness 10000\relax}{}{} %

\AtBeginDocument{%
}

\usepackage{mathtools}
	\mathtoolsset{showonlyrefs} %

\usepackage{smfthm} 
	\NumberTheoremsAs{subsubsection}\SwapTheoremNumbers

\usepackage{mymacros}

\usepackage[colorlinks,psdextra,pdfencoding=auto]{hyperref}
	\hypersetup{colorlinks,linkcolor={red!50!black},filecolor={green!50!black},urlcolor={blue!80!black},citecolor={blue!50!black}}

\begin{document}
\frontmatter

\title[Wild orbits \& generalised singularity modules]{Wild orbits and generalised singularity modules:
	stratifications and quantisation} %

\author[D.~Calaque]{Damien Calaque}
\address[D.~Calaque]{Institut Montpelliérain Alexander Grothendieck (IMAG),
	University of Montpellier,
	Place Eugène Bataillon 34090 Montpellier (France)}
\email{damien.calaque@umontpellier.fr}
\thanks{D.~C.
	has received funding from the European Research Council (ERC) under the European Union’s Horizon 2020 research and innovation programme (Grant Agreement No.
	768679).}

\author[G.~Felder]{Giovanni Felder}
\address[G.~Felder]{Department of Mathematics,
	ETH Zurich,
	Rämistrasse 101,
	8092 Zürich (Switzerland)}
\email{giovanni.felder@math.ethz.ch}
\thanks{G.~F.
	is partly supported by the National Centre of Competence in Research SwissMAP---%
	The Mathematics of Physics---%
	of the Swiss National Science Foundation.}

\author[G.~Rembado]{Gabriele Rembado}
\address[G.~Rembado]{Institut Montpelliérain Alexander Grothendieck (IMAG),
	University of Montpellier,
	Place Eugène Bataillon 34090 Montpellier (France)}
\email{gabriele.rembado@umontpellier.fr}
\thanks{G.~R.
	was supported by the Deutsche Forschungsgemeinschaft (DFG,
	German Research Foundation) under Germany’s Excellence Strategy - GZ 2047/1,
	Projekt-ID 390685813;
	he is now funded by the European Commission under the grant agreement n.~101108575 (HORIZON-MSCA project~\href{https://cordis.europa.eu/project/id/101108575}{QuantMod}),
	and by the \emph{Ministero de Ciencia} under the \emph{Innovación y Universidades}' grant PID2024-155686NB-I00.}

\author[R.~Wentworth]{Richard Wentworth}
\address[R.~Wentworth]{Department of Mathematics,
	University of Maryland,
	College Park,
	MD 20742 (United States)}
\email{raw@umd.edu}
\thanks{R.~W.'s research is supported by the NSF grants DMS-2204346 and DMS-2506596.}



\begin{abstract}
		We study truncated gauge-orbits through principal parts of irregular-singular connection germs,
	in the untwisted/unramified setting:
	for any connected complex reductive structure group $G$,
	in the general multilevel case.
	
    In particular,
	we compute the stabilisers of the formal normal forms using filtrations of Levi root systems,
    showing that they are connected.
	When the residue is semisimple we then stratify the space of orbits by the conjugacy class of the stabilisers,
    i.e.,
    by quotients of root-valuation strata;
    the dense stratum corresponds to the generic setting of isomonodromic deformations,
    à la Jimbo--Miwa--Ueno.
	
    Then we adapt a result of Alekseev--Lachowska to deformation-quantise nongeneric orbits.
	The $\ast$-product involves affine-Lie-algebra modules,
	extending:
	(i) the parabolic Verma modules (in the case of regular singularities);
	and (ii) the `singularity' modules of F.--R.
	(in the case of generic irregular singularities).
	They contain Whittaker vectors for the Gaiotto--Teschner/Bonelli--Maruyoshi--Tanzini Virasoro pairs in irregular Liouville conformal field theory,
	and they provide all the quotients obtained by leaving the aforementioned dense strata.
	We also construct Shapovalov forms for the corresponding representations of truncated-current Lie algebras,
	which enter into the category $\mathcal O$ of Chaffe--Topley;
	and we state a sharp irreducibility criterion.
	Finally,
	we use these representations to construct vector bundles of genus-zero vacua/covacua,
	equipped with flat connections à la Knizhnik--Zamolodchikov/Reshetikhin.
\end{abstract}

{\let\newpage\relax\maketitle} %

\setcounter{tocdepth}{1}  %
\tableofcontents

\mainmatter

\section{Introduction and main results}
\label{sec:intro}

\subsection{}

This text consists of two parts,
covering a semiclassical and a quantum side.
They are respectively motivated by the theory of isomonodromic deformations of irregular-singular meromorphic connections on Riemann surfaces,
and by irregular versions of the Wess--Zumino--Novikov--Witten model (= WZNW~\cite{wess_zumino_1971_consequences_of_anomalous_ward_identities,novikov_1982_the_hamiltonian_formalism_and_a_multivalued_analogue_of_morse_theory,witten_1983_global_aspects_of_current_algebra,witten_1984_nonabelian_bosonization_in_two_dimensions}),
in 2d conformal field theory (= CFT~\cite{belavin_polyakov_zamolodchikov_1984_infinite_conformal_symmetry_in_two_dimensional_quantum_field_theory,segal_1988_the_definition_of_conformal_field_theory}).

This introduction provides some motivation and historical background (cf.~\S~\ref{sec:background}),
before moving on to a description of the main results (cf.~\S~\ref{sec:main_results_1}--\ref{sec:main_results_2}),
and to a layout of the sections in the body/appendices of this text (cf.~\S~\ref{sec:layout})

\subsection{Some motivation/background}
\label{sec:background}

On the semiclassical side,
the differential equations governing isomonodromic deformations can be regarded as \emph{nonlinear} flat (Ehresmann~\cite{ehresmann_1995_les_connexions_infinitesimales_dans_un_espace_fibre_differentiable}) connections.
They are defined on Poisson/symplectic fibre bundles of moduli spaces arising in 2d gauge theory.
In turn,
the quantisation of certain examples of such isomonodromy connections yields \emph{linear} projectively-flat (Koszul) connections,
defined on vector bundles of conformal blocks and nonabelian theta functions.
This encompasses important archetypes in conformal/topological quantum field theory (cf.~Rmk.~\ref{rmk:quantum_monodromy_2} on the conformal side).

The main point is that the present scope of the semiclassical theory suggests that a lot more quantum connections should exist.

\subsubsection{}

In more details,
recall that Jimbo--Miwa--Ueno~\cite{jimbo_miwa_ueno_1981_monodromy_preserving_deformation_of_linear_ordinary_differential_equations_with_rational_coefficients_i_general_theory_and_tau_function} (cf.~\cite{jimbo_miwa_1981_monodromy_preserving_deformaton_of_linear_ordinary_differential_equations_with_rational_coefficients_ii,jimbo_miwa_1982_monodromy_preserving_deformaton_of_linear_ordinary_differential_equations_with_rational_coefficients_iii}) defined explicit isomonodromy equations for \emph{generic} irregular-singular connections on trivial vector bundles over $\mb C P^1$:
having high-order poles and regular semisimple leading term at each pole.
This vastly extended the logarithmic case of Fuchsian system,
in classical work of Schlesinger~\cite{schlesinger_1905_ueber_die_loesungen_gewisser_linearer_differentialgleichungen_als_funktionen_der_singularen_punkte,schlesinger_1912_ueber_eine_klasse_von_differentialsystemen_beliebiger_ordnung_mit_festen_kritischen_punkten},
R.~Fuchs~\cite{fuchs_1907_ueber_lineare_homogene_differentialgleichungen_zweiter_ordnung_mit_drei_im_endlichen_gelegenen_wesentlich_singulaeren_stellen},
Garnier~\cite{garnier_1912_sur_des_equations_differentielles_du_troisieme_ordre_dont_l_integrale_generale_est_uniforme_et_sur_une_classe_d_equations_nouvelles_d_ordre_superieur_dont_l_integrale_generale_a_ses_points_critiques_fixes},
etc.

The nonlinear partial differential equations of~\cite{jimbo_miwa_ueno_1981_monodromy_preserving_deformation_of_linear_ordinary_differential_equations_with_rational_coefficients_i_general_theory_and_tau_function} were rephrased geometrically by Boalch~\cite{boalch_2001_symplectic_manifolds_and_isomonodromic_deformations},
as a symplectic flat connection on a symplectic fibre bundle.
This was later extended to include general connected reductive structure groups~\cite{boalch_2002_g_bundles_isomonodromy_and_quantum_weyl_groups, boalch_2007_quasi_hamiltonian_geometry_of_meromorphic_connections};
also in the \emph{nongeneric}~\cite{boalch_2014_geometry_and_braiding_of_stokes_data_fission_and_wild_character_varieties,boalch_2012_simply_laced_isomonodromy_systems} and \emph{twisted} setting~\cite{boalch_yamakawa_2015_twisted_wild_character_varieties},
covering the general setup of (reductive) meromorphic 2d gauge theory.
What matters is that the base spaces of these bundles encode the time-variables of isomonodromic deformations in intrinsic fashion,
by parameterising admissible deformations of \emph{wild Riemann surfaces}~\cite[Def.~8.1+~10.1]{boalch_2014_geometry_and_braiding_of_stokes_data_fission_and_wild_character_varieties}:
namely,
they map to moduli spaces of pointed Riemann surfaces by adding on an irregular type/class at each marked point.
The latter control---%
the irregular part of---%
the formal normal forms of the meromorphic connections;
cf.~below,
as well as~\cite{doucot_rembado_tamiozzo_local_wild_mapping_class_groups_and_cabled_braids,doucot_rembado_2025_topology_of_irregular_isomonodromy_times_on_a_fixed_pointed_curve,boalch_doucot_rembado_2025_twisted_local_wild_mapping_class_groups_configuration_spaces_fission_trees_and_complex_braids,doucot_rembado_tamiozzo_moduli_spaces_of_untwisted_wild_riemann_surfaces,doucot_rembado_yamakawa_twisted_g_local_wild_mapping_class_groups} (whose many references point to the important past work of many more people,
besides the authors and their collaborators.)

\subsubsection{}
\label{sec:some_references}

This viewpoint lends itself to \emph{quantisation},
asking whether one can replace such symplectic local systems with (projectively) flat vector bundles over the same base (cf.~\cite[\S~1]{boalch_2001_symplectic_manifolds_and_isomonodromic_deformations}).
Ultimately,
such a theory of \emph{quantum} isomonodromic deformations leads to (projective) representations of generalised mapping class/braid groups,
via the monodromy of flat (projective) connections à la Hitchin/WZNW.

Notably,
this framework relates with:
\begin{itemize}
	\item
	      the geometric quantisation of compact Chern--Simons gauge theory~\cite{hitchin_1990_flat_connections_and_geometric_quantization,axelrod_dellapietra_witten_1991_geometric_quantisation_of_chern_simons_gauge_theory} (cf.~\cite{andersen_malusa_rembado_2022_genus_one_complex_quantum_chern_simons_theory,andersen_malusa_rembado_2024_sp_1_symmetric_hyperkaehler_quantisation} in the complexified case,
	      as well as~\cite{andersen_2012_hitchin_connection_toeplitz_operators_and_symmetry_invariant_deformation_quantisation});

	\item
	      the Knizhnik--Zamolodchikov connection (= KZ~\cite{knizhnik_zamolodchikov_1984_current_algebra_and_wess_zumino_model_in_two_dimensions}),
	      in 2d CFT (cf.~\cite{tsuchiya_kanie_1987_vertex_operators_in_conformal_field_theory_on_cp1_and_monodromy_representations_of_braid_groups,babujian_kitaev_1998_generalized_knizhnik_zamolodchikov_equations_and_isomonodromy_quantization_of_the_equations_integrable_via_the_inverse_scattering_transform_maxwell_bloch_system_with_pumping});

	\item
	      the higher-genus analogues of KZ,
	      i.e.,
	      the connections of Bernard (= KZB~\cite{bernard_1988_on_the_wess_zumino_witten_models_on_the_torus,bernard_1989_on_the_wess_zumino_witten_models_on_riemann_surfaces}) and Tsuchiya--Ueno--Yamada (= TUY~\cite{tsuchiya_ueno_yamada_1989_conformal_field_theory_on_universal_family_of_stable_curves_with_gauge_symmetries});

	\item
	      Reshetikhin's irregular generalisation of KZ~\cite{reshetikhin_1992_the_knizhnik_zamolodchikov_system_as_a_deformation_of_the_isomonodromy_problem} (cf.~\cite{harnad_1996_quantum_isomonodromic_deformations_and_the_knizhnik_zamolodchikov_equations}),
	      which was given a representation-theoretic derivation in~\cite{felder_rembado_2023_singular_modules_for_affine_lie_algebras_and_applications_to_irregular_wznw_conformal_blocks};

	\item
	      the relation between WZNW conformal blocks and nonabelian theta functions~\cite{beauville_laszlo_1994_conformal_blocks_and_generalized_theta_functions,faltings_1994_a_proof_for_the_verlinde_formula},
	      and of the respective flat connections~\cite{laszlo_1998_hitchin_s_and_wzw_connections_are_the_same};

	\item
	      the `parabolic' extension of the previous identification~\cite{biswas_mukhopadhyay_wentworth_2023_a_hitchin_connection_on_nonabelian_theta_functions_for_parabolic_g_bundles, biswas_mukhopadhyay_wentworth_2024_geometrization_of_the_tuy_wzw_kz_connection, biswas_mukhopadhyay_wentworth_2024_a_parabolic_analog_of_a_theorem_of_beilinson_and_schechtman},
	      which passes through a `geometrised' KZ connection (cf.~\cite{egsgaard_2015_hitchin_connection_for_genus_0_quantum_representation,baier_bolognesi_martens_pauly_2023_the_hitchin_connection_in_arbitrary_characteristic});

	\item
	      Harnad's dual Schlesinger system~\cite{harnad_1994_dual_isomonodromic_deformations_and_moment_maps_to_loop_algebras},
	      which combines with Schlesinger's to yield the isomonodromy equations of Jimbo--Miwa--Môri--Sato (= JMMS~\cite{jimbo_miwa_mori_sato_1980_density_matrix_of_an_impenetrable_bose_gas_and_the_fifth_painleve_transcendent});

	\item
	      the quantisation of the dual Schlesinger system~\cite{boalch_2002_g_bundles_isomonodromy_and_quantum_weyl_groups},
	      and more generally of JMMS~\cite{rembado_2019_simply_laced_quantum_connections_generalising_kz},
	      which yield the connections of De Concini/Millson--Toledano Laredo (= DMT~\cite{millson_toledanolaredo_2005_casimir_operators_and_monodromy_representations_of_generalised_braid_groups}) and Felder--Markov--Tarasov--Varchenko (= FMTV~\cite{felder_markov_tarasov_varchenko_2000_differential_equations_compatible_with_kz_equations});

	\item
	      more generally,
	      the deformation quantisation~\cite{rembado_2019_simply_laced_quantum_connections_generalising_kz,yamakawa_2022_quantization_of_simply_laced_isomonodromy_systems_by_the_quantum_spectral_curve_method} of the `simply-laced' isomonodromy systems~\cite{boalch_2012_simply_laced_isomonodromy_systems} (cf.~\cite{rembado_2020_symmetries_of_the_simply_laced_quantum_connections_and_quantisation_of_quiver_varieties});

	\item
	      the confluence view on quantum isomonodromic deformations~\cite{nagoya_sun_2010_confluent_primary_fields_in_the_conformal_theory,nagoya_sun_2011_confluent_kz_equations_for_sl_n_with_poincare_rank_2_at_infinity, gaiur_mazzocco_rubtsov_2023_isomonodromic_deformations_confluence_reduction_and_quantisation};

	\item
	      more literature on irregular conformal blocks~\cite{feigin_frenkel_toledanolaredo_2010_gaudin_models_with_irregular_singularities,fedorov_2010_irregular_wakimoto_modules_and_the_casimir_connection,gaiotto_teschner_2012_irregular_singularities_in_liouville_theory_and_argyres_douglas_type_gauge_theories,gaiotto_lamypoirier_2013_irregular_singularities_in_the_h3plus_wzw_model,bonelli_maruyoshi_tanzini_2012_wild_quiver_gauge_theories, gaiotto_2013_asymptotically_free_n_equal_2_theories_and_irregular_conformal_blocks,its_lisovyy_tykhyy_2015_connection_problem_for_the_sine_gordon_painleve_iii_tau_function_and_irregular_conformal_blocks,nagoya_2015_irregular_conformal_blocks_with_an_application_to_the_fifth_and_fourth_painleve_equations},
	      and wild de Rham moduli spaces~\cite{yamakawa_2019_fundamental_two_forms_for_isomonodromic_deformations,chaffe_rembado_yamakawa_genus_zero_wild_quantum_de_rham_spaces}.
\end{itemize}

One thus expects to be able to generalise some such examples,
as well as their applications in representation theory/low-dimensional topology.

\subsubsection{}

More concretely,
one result of~\cite{felder_rembado_2023_singular_modules_for_affine_lie_algebras_and_applications_to_irregular_wznw_conformal_blocks} is the construction of a flat vector bundle equipped with an irregular-singular version of KZ.
in a quantum version of flat families of moduli spaces of \emph{generic} irregular-singular connections.
Then,
independently,
the articles~\cite{doucot_rembado_tamiozzo_local_wild_mapping_class_groups_and_cabled_braids,doucot_rembado_2025_topology_of_irregular_isomonodromy_times_on_a_fixed_pointed_curve,doucot_rembado_tamiozzo_moduli_spaces_of_untwisted_wild_riemann_surfaces} studied spaces of irregular isomonodromy times in the \emph{nongeneric} case:
the latter suggest a far-reaching generalisation of the setup of~\cite{felder_rembado_2023_singular_modules_for_affine_lie_algebras_and_applications_to_irregular_wznw_conformal_blocks},
which we explore here.

\subsubsection{}

Let thus $G$ be a connected complex reductive Lie group,
and $\Sigma$ a closed Riemann surface.
The semiclassical side involves a symplectic moduli space $\mc M_{\dR}(\Sigma;G)$ of meromorphic connections on holomorphic principal $G$-bundles over $\Sigma$,
with prescribed polar divisor.
Its symplectic form generalises Narasimhan/Atiyah--Bott/Goldman's~\cite{narasimhan_1970_geometry_of_moduli_spaces_of_vector_bundles,atiyah_bott_1983_yang_mills_equations_over_riemann_surfaces,goldman_1984_the_symplectic_nature_of_fundamental_groups_of_surfaces},
and was first constructed in~\cite{boalch_2001_symplectic_manifolds_and_isomonodromic_deformations}.
(We also tacitly view $G$ as an algebraic group;
and much can be phrased in the algebraic category,
cf.~Rmk.~\ref{rmk:affine_setup} +~\ref{rmk:affine_setting}.)
This is the \emph{de Rham} realisation of---%
the meromorphic version of---%
the $G$-valued first cohomology of $\Sigma$,
in the wild nonabelian Hodge correspondence~\cite{sabbah_1999_harmonic_metrics_and_connections_with_irregular_singularities,biquard_boalch_2004_wild_nonabelian_hodge_theory_on_curves,huang_sun_2023_meromorphic_parahoric_higgs_torsors_and_filtered_stokes_g_local_systems_on_curves}:
cf.~\cite{simpson_1990_harmonic_bundles_on_noncompact_curves,boalch_2011_riemann_hilbert_for_tame_complex_parahoric_connections,biquard_garciaprada_mundetiriera_2020_parabolic_higgs_bundles_and_representations_of_the_fundamental_group_of_a_punctured_surface_into_a_real_group} in the tame case,\fn{
	Rmk.~\ref{rmk:tame_wild_terminology} explains the `tame/wild' terminology.
	Also,
	hereafter,
	`$G$-connection' is a shorthand for `connection on a principal $G$-bundle'.}
and~\cite{hitchin_1987_the_self_duality_equations_on_a_riemann_surface,donaldson_1987_twisted_harmonic_maps_and_the_self_duality_equations,corlette_1988_flat_g_bundles_with_canonical_metrics,simpson_1992_higgs_bundles_and_local_systems} in the nonsingular one;
as well as more work of Simpson~\cite{simpson_1994_moduli_of_representations_of_the_fundamental_group_of_a_smooth_projective_variety_i,simpson_1994_moduli_of_representations_of_the_fundamental_group_of_a_smooth_projective_variety_ii},
and the surveys~\cite{boalch_2012_hyperkaehler_manifolds_and_nonbelian_hodge_theory_of_irregular_curves,boalch_2017_wild_character_varieties_meromorphic_hitchin_systems_and_dynkin_graphs}.

Importantly,
as mentioned above,
in the wild case there are new local moduli at each marked point.
Namely,
the \emph{irregular type} of an untwisted connection at a pole $a \in \Sigma$ is an element $Q \in \mf t (\wh{\ms K}_a) \bs \mf t (\wh{\ms O}_a)$,
where $\wh{\ms O}_a \hra \wh{\ms K}_a$ are the completed local ring/field of $\Sigma$ (cf.~\cite[Def.~7.1]{boalch_2014_geometry_and_braiding_of_stokes_data_fission_and_wild_character_varieties}).
In turn,
the \emph{irregular class} is the orbit of $Q$ under the Weyl group $W \ceqq N_G(T) \bs T$,
where $T \sse G$ is the maximal torus integrating $\mf t$.
(This involves choosing/forgetting a compatible frame at $a$ which `diagonalises' the irregular-singular part.)~Neglecting stability and parabolic/parahoric structures,
on top of a pointed Riemann surface with prescribed irregular types/classes sits a complex Poisson de Rham space;
its symplectic leaves are parameterised by (formal) residue orbits,
governing the (formal) monodromy representation around each pole.

\subsubsection{}

Besides \S~\ref{sec:irregular_blocks},
this text deals with the local situation:
we consider \emph{formal germs} of untwisted meromorphic $G$-connections on Riemann surfaces.
Precisely,
we work over $\mb C$ and consider irregular-singular connections on a principal $G$-bundle over the spectrum of a complete discrete valuation ring.
Choosing for convenience a uniformizer $z$,
these can be defined as equivariant maps from trivialisations of the bundle to 1-forms $\wh{\mc A} \in \mf g (\!( z )\!) \dif z$,
with respect to the action of the formal-holomorphic gauge group $G \llb z \rrb \ceqq G \bigl( \mb C \llb z \rrb \bigr)$ (cf.~\S~\ref{sec:wild_orbits};
and~\cite{fernandezherrero_reduction_theory_for_connections_over_the_formal_punctured_disc} over the punctured disc).
This yields orbits for the standard affine $G\llb z \rrb$-action on $\mf g(\!( z )\!) \dif z$.

But the construction of wild de Rham spaces typically involves truncations,
as dictated by the polar divisor.
Notably,
the standard genus-zero open parts $\mc M^*_{\dR} \sse \mc M_{\dR}$ are $G$-Hamiltonian reductions of products of \emph{truncated} $G\llb z \rrb$-orbits:
cf.~\cite{boalch_2001_symplectic_manifolds_and_isomonodromic_deformations,boalch_2002_g_bundles_isomonodromy_and_quantum_weyl_groups,boalch_2007_quasi_hamiltonian_geometry_of_meromorphic_connections} generically,
and~\cite{boalch_2012_simply_laced_isomonodromy_systems,yamakawa_2019_fundamental_two_forms_for_isomonodromic_deformations,chaffe_rembado_yamakawa_genus_zero_wild_quantum_de_rham_spaces} in the general \emph{nonresonant} nongeneric case.
(One new insight is that nonresonance plays a role in the quantum theory,
cf.~Exmp.~\ref{ex:nonresonance}).
Analogously,
and more generally,
on the Betti side of nonabelian Hodge theory certain operations of `fusion/fission' feature in the construction of wild character varieties $\mc M_{\on B}$~\cite{boalch_2007_quasi_hamiltonian_geometry_of_meromorphic_connections,boalch_2009_through_the_analytic_halo_fission_via_irregular_singularities,boalch_2014_geometry_and_braiding_of_stokes_data_fission_and_wild_character_varieties,boalch_yamakawa_2015_twisted_wild_character_varieties},
which are receptacles for Stokes data---%
involving multiplicative analogues of the orbits.
In particular,
fission at an irregular singularity breaks down the structure group $G$ to a connected reductive subgroup thereof:
this will be relevant for us (cf.~Rmk.~\ref{rmk:fission}).

\subsubsection{}

The `wild orbits' in the title are precisely orbits of \emph{principal/polar parts} of formal $G$-connection germs,
under the (linear) action of finite-order jets of $G$-bundles automorphisms.
Upon fixing an invariant pairing on $\mf g$,
the latter is tantamount to the coadjoint action of a \emph{truncated-current Lie group} (= TCLG,
cf.~\eqref{eq:tclg}),
on the dual of its \emph{truncated-current Lie algebra} (= TCLA~\cite{takiff_1971_rings_of_invariant_polynomials_for_a_class_of_lie_algebras,rais_tauvel_1992_indice_et_polynomes_invariants_pour_certaines_algebres_de_lie,geoffriau_1994_sur_le_centre_de_l_algebre_enveloppante_d_une_algebre_de_takiff,geoffriau_1995_homomorphismo_de_harish_chandra_pour_les_algebres_de_takiff_generalisees,chaffe_2023_category_o_for_takiff_lie_algebras,chaffe_topley_2023_category_o_for_truncated_current_lie_algebras},
a.k.a.~`Takiff' Lie algebra;
cf.~\eqref{eq:tcla}).\fn{
	The TCLAs also appear on the Dolbeault side of the nonabelian-Hodge hyperkähler rotation,
	when considering meromorphic $G$-Higgs fields,
	cf.~\cite{donagi_1997_seiberg_witten_integrable_systems,markman_1998_algebraic_geometry_integrable_systems_and_seiberg_witten_theory}.}~Then in this text:
(i) we study and deformation-quantise nongeneric examples of wild orbits,
viewed as local pieces of wild de Rham spaces in meromorphic 2d gauge theory;
and (ii) we relate with local wild mapping class groups,
and with spaces of genus-zero irregular vacua/covacua.

\begin{rema}
	We should stress again that we focus on the untwisted case,
	with a view towards quantisation.
	We thus consider `very good' orbits~\cite[Def.~4]{boalch_2017_wild_character_varieties_meromorphic_hitchin_systems_and_dynkin_graphs},
	and particularly the subcase of \emph{semisimple} residue.

	Nonetheless,
	the natural semiclassical setup allows for residues with nilpotent parts,
	and it goes beyond the untwisted case.
	E.g.,
	there are extensions of untwisted irregular types/classes which might be of independent interest (cf.~\S~\ref{sec:generalised_irregular_type}--\ref{sec:generalised_birkhoff_residue_orbit} and Rmk.~\ref{rmk:generalised_irregular_classes_and_residue_orbits};
	the twisted case was considered in~\cite{boalch_doucot_rembado_2025_twisted_local_wild_mapping_class_groups_configuration_spaces_fission_trees_and_complex_braids,doucot_rembado_yamakawa_twisted_g_local_wild_mapping_class_groups}).
	We take this more general viewpoint in order to give cleaner inductive proofs for the case of interest.
\end{rema}

\subsection{Main results (I)}
\label{sec:main_results_1}

In brief:
(i) we realise the (marked) wild orbits as homogeneous complex manifolds for TCLGs;
(ii) we define a Shapovalov form on suitable parabolically-induced representations of TCLAs,
in order to deform the natural symplectic structure of the orbits;
and (iii) we use the affine version of the modules to construct parabolic generalisations of the irregular KZ connection.
Here we review the first point;
cf.~instead \S~\ref{sec:main_results_2} for the representation-theoretic quantisation,
and the relation with conformal field theory.

\subsubsection{}

First,
one can introduce a `graded' notion of semisimplicity in the \emph{nonreductive} TCLAs,
thereby filtering the space of orbits (cf.~Def.~\ref{def:s_semisimplicity} and~\S~\ref{sec:wild_orbits_whole}).
It is a natural extension of the usual notion in $\mf g \simeq \mf g^{\dual}$,
and indeed we first review the tame case where the irregular type/class vanishes.

Then we proceed by \emph{marking} the orbits,
generalising the tame situation.
Indeed,
the Cartan subalgebra $\mf t \sse \mf g$ is naturally identified with the space of marked semisimple orbits in $\mf g$,
and it is stratified by \emph{Levi} subsystems of the root system $\Phi = \Phi(\mf g,\mf t)$ (cf.~\S~\ref{sec:pure_isomorphism_tame}).
This is the basic example of a Whitney stratification (cf.~\cite{goresky_macpherson_1988_stratified_morse_theory};
the material we use is reviewed in \S~\ref{sec:stratifications}).
What matters is that the viewpoint of stratifications generalises perfectly to the wild case,
and the phenomenon of fission identifies the natural index set for the nonempty root-valuation strata (cf.~\cite[Prop.~3.4.1]{goresky_kottwitz_macpherson_2009_codimensions_of_root_valuation_strata}).

\subsubsection{}

More precisely,
we consider principal parts of (formal) germs of meromorphic $G$-connections,
with a pole of any order $r \geq 1$.
Under a choice of trivialisation/uniformiser (cf.~Rmk.~\ref{rmk:changing_trivialization_and_uniformizer}),
these can be written
\begin{equation}
	\label{eq:principal_part_intro}
	\mc A = \sum_{i = 1}^r \frac{A_i}{z^i} \dif z \in z^{-1} \mf g[ z^{-1}] \dif z \simeq \mf g (\!( z )\!) \dif z \bs \mf g \llb z \rrb \dif z,
	\qquad A_1,\dc,A_r \in \mf g.
\end{equation}
The tame case corresponds to $r = 1$ (provided that $A_r$ is \emph{not} nilpotent,
cf.~Rmk.~\ref{rmk:tame_wild_terminology}).

After choosing an invariant pairing on $\mf g$,
the 1-form~\eqref{eq:principal_part_intro} matches up with an element of the Poisson variety $\mf g_r^{\dual}$,
where
\begin{equation}
	\mf g_r
	\ceqq \mf g \ots_{\mb C} \bigl( \mb C \llb z \rrb \bs z^r \mb C \llb z \rrb \bigr)
	\simeq \mf g \llb z \rrb \bs z^r \mf g \llb z \rrb.
\end{equation}
The latter is the depth-$r$ TCLA associated with $\mf g$,
integrated by a TCLG denoted by
\begin{equation}
	G\llb z \rrb \lthra G_r \ceqq G \bigl( \mb C \llb z \rrb \bs z^r \mb C \llb z \rrb \bigr).
\end{equation}

As a particular case of a more general definition,
the coadjoint $G_r$-orbit $\mc O \ceqq G_r.\mc A \sse \mf g_r^{\dual}$ (through~\eqref{eq:principal_part_intro}) is $r$-\emph{semisimple} if it passes through an element $\mc A' = \sum_{i = 1}^r A'_i z^{-i} \dif z$ such that the new coefficients $A'_1,\dc,A'_r$ are semisimple and commute with each other.
The $G$-orbit of the (formal) normal form $\mc A'$ is then uniquely determined by $\mc O$,
as per the standard classification (à la Turrittin--Levelt/Babbitt--Varadarajan~\cite{turrittin_1955_convergent_solutions_of_ordinary_linear_homogeneous_differential_equations_in_the_neighborhood_of_an_irregular_singular_point,wasow_1965_asymptotic_expansion_for_ordinary_differential_equations,levelt_1975_jordan_decomposition_for_a_class_of_singular_differential_operators,babbitt_varadarajan_1983_formal_reduction_theory_of_meromorphic_differential_equations_a_group_theoretic_view,sibuya_1990_linear_differential_equations_in_the_complex_domain_problems_of_analytic_continuation,malgrange_1991_equations_differentielles_a_coefficients_polynomiaux};
cf.~the `UTS' orbits of~\cite{chaffe_rembado_yamakawa_genus_zero_wild_quantum_de_rham_spaces}).
Moreover,
since $G$ is connected/reductive,
up to acting diagonally by $G$ one may assume that the tuple $(A'_1,\dc,A'_r)$ actually lies within $\mf t^r \sse \mf g^r$.
In turn,
one can \emph{mark} the orbit by the corresponding point $\mc A' \in \mc O'$:
there are finitely many such choices,
and the bijection $(\mc O,\mc A') \mt (A'_1,\dc,A'_r)$ identifies the space $\mf t^r$ with the set of marked $r$-semisimple orbits.
Then one can prove that:

\begin{theo}[cf.~Cor.~\ref{cor:very_good_infinitesimal_centraliser} +~\ref{cor:very_good_centraliser} and~\S~\ref{sec:about_main_theo_1}]
	\label{thm:main_theo_1}

	The $G_r$-stabilizer of $\mc A'$ is \emph{connected},
	and it is uniquely determined by the `fission' sequence of nested centralisers of $A'_r,\dc,A'_1$.
\end{theo}

\subsubsection{}

Therefore,
the complex geometry of $r$-semisimple orbits is controlled by filtrations of Levi factors of nested parabolic subgroups of $G$ which contains the given maximal torus $T \sse G$.
The latter match up with filtrations of \emph{Levi subsystems of} $\Phi$ (cf.~Def.~\ref{def:levi_filtration} +~Rmk.~\ref{rmk:fission}),
and as mentioned above there is a finite stratification of $\mf t^r$ indexed by---%
depth-bounded---%
Levi filtrations of $\Phi$ (cf.~Lem.~\ref{lem:wild_stratification} +~Cor.~\ref{cor:truncated_wild_stratification}).
It follows that:

\begin{coro}[Thm.~\ref{thm:pure_isomorphism_wild}]
	\label{thm:main_theo_2}

	Two marked $r$-semisimple orbits in the same stratum are \emph{canonically} isomorphic as complex homogeneous $G_r$-manifolds.
\end{coro}

\begin{rema}
	\label{rmk:about_partitions}

	Cor.~\ref{thm:main_theo_2} can be rephrased as follows.
	In view of Thm.~\ref{thm:main_theo_1},
	the identification with the space of marked $r$-semisimple orbits yields a map from $\mf t^r$ to the set of quotients $G_r \bs \bm L$,
	where $\bm L \sse G_r$ is a---%
	connected---%
	complex Lie subgroup:
	this map is \emph{constant} on each root-valuation stratum of $\mf t^r$.
\end{rema}

\subsubsection{}

In particular,
the genericity assumption in the classical theory of isomonodromic deformations (cf.~\cite{birkhoff_1913_the_generalized_riemann_problem_for_linear_differential_equations_and_allied_problems_for_linear_difference_and_q_difference_equations,balser_jurkat_lutz_1979_birkhoff_invariants_and_stokes_multipliers_for_meromorphic_linear_differential_equations,jimbo_miwa_mori_sato_1980_density_matrix_of_an_impenetrable_bose_gas_and_the_fifth_painleve_transcendent,jimbo_miwa_ueno_1981_monodromy_preserving_deformation_of_linear_ordinary_differential_equations_with_rational_coefficients_i_general_theory_and_tau_function,boalch_2001_symplectic_manifolds_and_isomonodromic_deformations};
etc.)~is tantamount to working within the dense stratum.
In this case the exponential factors $q_\alpha \ceqq \alpha (Q) \in z^{-1} \mb C [z^{-1}]$ of the irregular type $Q$ are all nonvanishing ($\alpha \in \Phi$),
and have one and the same pole order.
This pole order is the---%
unique---%
level of $Q$,
and the \emph{nongeneric/multilevel} case is precisely about exploring the deeper strata,
as we do here (cf.~\cite{biquard_boalch_2004_wild_nonabelian_hodge_theory_on_curves,boalch_2009_through_the_analytic_halo_fission_via_irregular_singularities,boalch_2012_simply_laced_isomonodromy_systems,boalch_2014_geometry_and_braiding_of_stokes_data_fission_and_wild_character_varieties,yamakawa_2019_fundamental_two_forms_for_isomonodromic_deformations,cotti_dubrovin_guzzetti_2019_isomonodromy_deformations_at_an_irregular_singularity_with_coalescing_eigenvalues,cotti_2021_on_the_universality_of_integrable_deformations_of_solutions_of_degenerate_riemann_hilbert_birkhoff_problems,cotti_2021_degenerate_riemann_hilbert_birkhoff_problems_semisimplicity_and_convergence_of_wdvv_potentials,sabbah_2021_integrable_deformations_and_degenerations_of_some_irregular_singularities,lopezreyes_2022_isomonodromic_deformations_along_the_caustic_of_a_dubrovin_frobenius_manifold,guzzetti_2022_isomonodromic_deformations_along_a_stratum_of_the_coalescence_locus};
etc.).

Note that such strata are considered in~\cite{doucot_rembado_tamiozzo_local_wild_mapping_class_groups_and_cabled_braids} as fine moduli spaces of formal germs of wild Riemann surfaces.
Namely,
op.~cit.~regards them as \emph{universal} admissible deformation spaces of irregular types (cf.~\cite[Def.~10.1]{boalch_2014_geometry_and_braiding_of_stokes_data_fission_and_wild_character_varieties},
as well as the `$B$-equivalence' of~\cite[\S~4]{yamakawa_2019_fundamental_two_forms_for_isomonodromic_deformations}).
The fundamental groups of the strata are thus---%
local---%
wild analogues of \emph{pure} mapping class groups of pointed surfaces,
whose Poisson/symplectic dynamics on wild character varieties correspond to braiding Stokes data (vastly extending Dubrovin's~\cite[App.~F]{dubrovin_1996_geometry_of_2d_topological_field_theories}).
Then in this text we also take into account the formal residue,
thereby shifting focus from the bases of the de Rham/Betti local systems to their fibres.
(Cf.~\S~\ref{sec:irregular_blocks},
as well as~\cite{doucot_rembado_tamiozzo_moduli_spaces_of_untwisted_wild_riemann_surfaces,boalch_doucot_rembado_2025_twisted_local_wild_mapping_class_groups_configuration_spaces_fission_trees_and_complex_braids,doucot_rembado_yamakawa_twisted_g_local_wild_mapping_class_groups},
about the global setting.)

\subsubsection{}

Before moving on to the quantum side,
we also discuss the change of marking.
This relates with the \emph{full/nonpure} versions of local wild mapping class groups (cf.~\cite{doucot_rembado_2025_topology_of_irregular_isomonodromy_times_on_a_fixed_pointed_curve}),
passing from an irregular type to the underlying irregular class.

Precisely,
in the tame case the Weyl group $W$ permutes the Levi strata of $\mf t$,
and gives rise to a \emph{quotient stratification} of the adjoint quotient $\mb A \ceqq \mf g \bs\!\!\bs G \simeq \mf t \bs W$ (cf.~\S~\ref{sec:quotient_stratifications};
this is a `discriminant stratification'~\cite[\S~1]{bessis_2015_finite_complex_reflection_arrangements_are_k_pi_1}).
In turn,
the latter is the natural---%
moduli---%
space of semisimple orbits in $\mf g \simeq \mf g^{\dual}$,
forgetting the choice of marking (cf.~\S~\ref{sec:nonpure_isomorphism_tame}).
Then again,
for an integer $r \geq 1$ one can consider an $r$-semisimple principal part as in~\eqref{eq:principal_part_intro},
and gauge it to a $\mf t$-valued normal form $\mc A'$.
As mentioned above,
the latter is \emph{not} uniquely determined by the wild orbit $\mc O = G_r.\mc A \sse \mf g_r^{\dual}$:
but its class under the diagonal $W$-action \emph{is} a well-defined invariant of $\mc O$,
and conversely the $W$-orbit $W.\mc A' \sse \mc O$ gives all possible markings.

Eventually,
the quotient $\mb A_r \ceqq \mf t^r \bs W$ can be naturally identified with the space of (unmarked) $r$-semisimple orbits,
and one can extend the topological statements from the tame case:

\begin{theo}[cf.~Thm.~\ref{thm:quotient_wild_stratification} +~Prop.~\ref{prop:intrinsic_excellent_orbit_space} +~Thm.~\ref{thm:nonpure_isomorphism_wild}]
	\leavevmode
	\label{thm:main_theo_3}

	\begin{enumerate}
		\item
		      There is a \emph{finite} stratification of $\mb A_r$,
		      indexed by $W$-orbits of (depth-bounded) Levi filtrations of $\Phi$.

		\item
		      And two $r$-semisimple orbits in the same stratum are isomorphic as complex homogeneous $G_r$-manifolds.
	\end{enumerate}
\end{theo}

\begin{rema}
	The set of noncanonical isomorphisms is controlled by Thm.~\ref{thm:nonpure_isomorphism_wild}.

	Moreover,
	the identification with the space of $r$-semisimple orbits yields a map from $\mb A_r$ to the set of complex homogeneous $G_r$-manifolds (cf.~Rmk.~\ref{rmk:about_partitions}).
	The caveat is that the unmarked orbits only determine \emph{conjugacy classes} of---%
	connected---%
	Lie subgroups $\bm L \sse G_r$.\fn{
		The special choice of marking implies that these subgroups are conjugated by the action of the normaliser $N_G(T) \sse G$ of the given maximal torus,
		in the natural embedding $G \sse G_r$ (cf.~Lem.~\ref{lem:down_to_weyl}).}~Nonetheless,
	the induced map to the set of isomorphism classes of such $G_r$-manifolds is constant on each part of the underlying partition of $\mb A_r$.
	(This rephrases Thm.~\ref{thm:main_theo_3}~(2.).)
\end{rema}

\subsection{Main results (II)}
\label{sec:main_results_2}

Everything is ready for the second part,
which deals with a generalisation of the affine-Lie-algebra modules of~\cite{felder_rembado_2023_singular_modules_for_affine_lie_algebras_and_applications_to_irregular_wznw_conformal_blocks},
and of their finite/truncated version---%
for TCLAs.
(We now call them `singularity modules';
cf.~Fn.~1 of op.~cit.)

\subsubsection{}

Recall that the singularity modules are representations of the affinisation $\wh{\mf g} \thra \mf g (\!( z )\!)$,
corresponding to the invariant pairing on $\mf g$,
induced from a Lie subalgebra of an affine Borel subalgebra $\wh{\mf b} \sse \wh{\mf g}$.

In this text we reinterpret the choice of a Borel subalgebra $\mf b \sse \mf g$ in geometric terms,
as a \emph{polarisation} on a generic wild orbit.
Then,
leveraging the results of \S~\ref{sec:main_results_1},
we derive the natural nongeneric extension:
opposite pairs of filtrations of $\Phi$ by \emph{parabolic subsets} of roots yield transverse polarisations on $r$-semisimple wild orbits (cf.~Thm.~\ref{thm:wild_polarisations}).
This brings forward a natural space of parameters for inducing new modules.

Namely,
we define a Lie subalgebra $\wh{\mf S}^{\bm \psi} \sse \wh{\mf g}$ involving \emph{parabolic filtrations} $\bm \psi$ of the root system $\Phi$ (cf.~Def.~\ref{def:affine_singularity_algebra}).
Then we also fix a scalar $\kappa \in \mb C$ for the action of the central element,
and prove that the space of level-$\kappa$ characters of $\wh{\mf S}^{\bm \psi}$ is uniquely determined by the sequence of---%
centres of---%
Levi factors underlying $\bm \psi$ (cf.~Lem.~\ref{lem:bracket_nested_parabolic} + Cor.~\ref{cor:abelianisation}).
This makes it possible to attach new $U\wh{\mf g}$-modules $\wh M$ to \emph{polarised} (marked) $r$-semisimple orbits.
The latter are the \emph{generalised} affine singularity modules,
featuring in the title (cf.~Def.~\ref{def:affine_singularity_module}).
Besides~\cite{felder_rembado_2023_singular_modules_for_affine_lie_algebras_and_applications_to_irregular_wznw_conformal_blocks},
they also generalise the parabolic/generalised affine Verma modules.

\subsubsection{}

In particular,
we compute the annihilator of the natural cyclic vector $w \in \wh M$.
This yields yet a further extension of the Virasoro relations satisfied by irregular Gaiotto--Teschner/Bonelli--Maruyoshi--Tanzini states (= GT/BMT~\cite{gaiotto_teschner_2012_irregular_singularities_in_liouville_theory_and_argyres_douglas_type_gauge_theories,bonelli_maruyoshi_tanzini_2012_wild_quiver_gauge_theories,felinska_jaskolski_kosztolowicz_2012_whittaker_pairs_for_the_virasoro_algebra_and_the_gaiotto_bonelli_maruyoshi_tanzini_states}),
from the generic case (cf.~Rmk.~\ref{rmk:virasoro_symmetries} +~\ref{rmk:virasoro_symmetries_2},
and~\eqref{eq:irregular_state_relations}).
Moreover,
via the Segal--Sugawara representation of the Virasoro algebra on the completion of $U\wh{\mf g}$,
one can prove that:

\begin{theo}[cf.~Def.~\ref{def:covacua_vacua} +~Thm.~\ref{thm:flat_connection}]
	There is a (strongly) flat connection à la KZ/Reshetikhin,
	defined on the vector bundle of genus-zero vacua/covacua for affine generalised singularity modules,
	extending~\cite{felder_rembado_2023_singular_modules_for_affine_lie_algebras_and_applications_to_irregular_wznw_conformal_blocks} from the generic case.
\end{theo}

\begin{rema}
	The blocks of Feigin--Frenkel--Toledano Laredo~\cite{feigin_frenkel_toledanolaredo_2010_gaudin_models_with_irregular_singularities} are of different flavour.
	We rather work with highest-weight modules at noncritical level,
	and our semiclassical limits recover isomonodromic families of de Rham spaces,
	instead of meromorphic Hitchin systems on the Dolbeault side (cf.~\cite{teschner_2011_quantization_of_the_hitchin_moduli_spaces_liouville_theory_and_the_geometric_langlands_correspondence_i} in Liouville CFT).

	On a related note,
	inducing from truncated-current versions of Cartan subalgebras~\cite{fedorov_2010_irregular_wakimoto_modules_and_the_casimir_connection} does not seem to take into account the natural polarisations on wild orbits;
	nonetheless,
	the fact that op.~cit.~recovers the connections of~\cite{felder_markov_tarasov_varchenko_2000_differential_equations_compatible_with_kz_equations,toledanolaredo_2002_a_kohno_drinfeld_theorem_for_quantum_weyl_groups,millson_toledanolaredo_2005_casimir_operators_and_monodromy_representations_of_generalised_braid_groups} is an important feature which we aim to reproduce in the future.
	Here we just showcase the braiding of marked points (cf.~Rmk.~\ref{rmk:quantum_monodromy_1}--\ref{rmk:quantum_monodromy_2}).
\end{rema}

\subsubsection{}

Importantly,
in view of Thm.~\ref{thm:finite_description},
the spaces of generalised (co)vacua involve the submodule $M \sse \wh M$ generated by the action of holomorphic currents $\mf g \llb z \rrb \sse \mf g(\!( z )\!)$ on the cyclic vector $w \in \wh M$.
This yields new parabolically-induced representations of TCLAs (cf.~Thm.~\ref{thm:finite_modules}).

Thus $M$ is the corresponding \emph{finite} generalised singularity module.
It is more general than a `confluent' Verma module~\cite{jimbo_nagoya_sun_2008_remarks_on_the_confluent_kz_equations_for_sl_2_and_quantum_painleve_equations,nagoya_sun_2011_confluent_kz_equations_for_sl_n_with_poincare_rank_2_at_infinity},
and it also extends the generic TCLA setting of Wilson~\cite{wilson_2011_highest_weight_theory_for_truncated_current_lie_algebras} (cf.~\S~\ref{sec:simple_quotients} +~Rmk.~\ref{rmk:category_O}).
Conversely,
Cor.~7.4 of op.~cit.~provides an irreducibility criterion which---%
in our viewpoint---%
corresponds to choosing a parameter in the dense stratum of the (dual) stratification of $\mf t^r$.
Then descending into lower strata yields nontrivial quotients,
and in turn:

\begin{theo}[cf.~Thm.~\ref{thm:wild_parabolic_to_wild_parabolic}]
	\label{thm:main_theo_4}
	The finite generalised singularity modules exhaust \emph{all} such quotients.
\end{theo}
(Cf.~\S~\ref{sec:affine_wild_parabolic_to_wild_parabolic} in the affine case.)

\subsubsection{}

Thm.~\ref{thm:main_theo_4} is consistent with the tame setting,
where a Borel-induced Verma module for $\mf g$ \emph{cannot} be simple if its highest weight lies on some coroot hyperplanes (cf.~Rmk.~\ref{rmk:minimal_modules}).
However,
inducing from parabolic subalgebras does \emph{not} ensure that the corresponding Verma module be simple:
the standard theory identifies the unique maximal proper submodule as the radical of the Shapovalov form~\cite{shapovalov_1972_a_certain_bilinear_form_on_the_universal_enveloping_algebra_of_a_complex_semisimple_lie_algebra},
involving the usual weight-space decompositions (cf.~\S~\ref{sec:tame_shapovalov}).

Therefore,
we generalise all this material in the wild case,
constructing a Shapovalov form $\mc S$---%
on $M$---%
whose radical coincides with the unique maximal proper submodule $N \sse M$ (cf.~Lem./Def.~\ref{lem:maximal_proper_submodule_wild} +~Thm.~\ref{thm:shapovalov_radical_equals_maximal_proper_submodule_wild}).
In particular,
there is an $\mc S$-orthogonal decomposition of $M$ into weight subspaces for the centre of the Levi factor of a parabolic subalgebra of $\mf g$.

\subsubsection{}

After dealing with additional truncations (cf.~\S~\ref{sec:simple_quotients}),
if follows that the nondegeneracy of the Shapovalov form $\mc S$ controls the unique simple quotient $M \thra L \ceqq M \bs N$ of the new TCLA modules.
Then a recursive usage of the `parabolic induction' functors of Chaffe--Topley~\cite{chaffe_2023_category_o_for_takiff_lie_algebras,chaffe_topley_2023_category_o_for_truncated_current_lie_algebras} leads to a necessary/sufficient irreducibility condition,
extending~\cite{wilson_2011_highest_weight_theory_for_truncated_current_lie_algebras} (cf.~Thm.~\ref{thm:simple_finite_modules}).
Interestingly,
this relates with the \emph{nonresonance} conditions for (formal) germs of untwisted $G$-connections,
viz.,
the requirement that the adjoint action of the (formal) residue has no nonzero integer eigenvalues upon restriction to the centraliser of the irregular type (cf.~\cite[p.~6]{boalch_2008_irregular_connections_and_kac_moody_root_systems},
and~\cite[\S~9.6]{boalch_2012_simply_laced_isomonodromy_systems},
and~\cite{yamakawa_2019_fundamental_two_forms_for_isomonodromic_deformations},
and Exmp.~\ref{ex:nonresonance}).
The consequence for \emph{quantum} de Rham spaces has been worked out in~\cite{chaffe_rembado_yamakawa_genus_zero_wild_quantum_de_rham_spaces} (cf.~Rmk.~\ref{rmk:quantum_de_rham}).

\subsubsection{}

This representation-theoretic setup is reminiscent of the orbit method in geometric quantisation~\cite{kirillov_2004_lectures_on_the_orbit_method},
notably the theorem of Borel--Weil~\cite{tits_1955_sur_certaines_classes_d_espaces_homogens_de_groupes_de_lie,serre_1995_representations_lineaires_et_espaces_homogenes_kaehleriens_des_groupes_de_lie_compacts_d_apres_armand_borel_et_andre_weil}(--Bott~\cite{bott_1957_homogeneous_vector_bundles}).
In brief,
recall that the (dual) simple quotients of standard Verma modules can be geometrically realised via global sections of line bundles over flag varieties;
while the (dual) Vermas correspond to local sections over an open cell,
which need not extend.
This is relevant for the geometrisation of KZ~\cite{egsgaard_2015_hitchin_connection_for_genus_0_quantum_representation,biswas_mukhopadhyay_wentworth_2023_a_hitchin_connection_on_nonabelian_theta_functions_for_parabolic_g_bundles, biswas_mukhopadhyay_wentworth_2024_geometrization_of_the_tuy_wzw_kz_connection},
and indeed our view is towards quantum moduli spaces of (gauge-theoretic data on) principal bundles over Riemann surfaces.

Therefore,
we conclude by proving a precise quantisation statement.
To set it up,
note that Prop.~\ref{prop:nonsingular_characters} characterises \emph{nonsingular} characters in the truncated-current setting.
Under this hypothesis,
a nonsymmetric variant of the wild Shapovalov form $\mc S$ extends the Kirillov--Kostant--Souriau (= KKS~\cite{kirillov_1962_unitary_representations_of_nilpotent_lie_groups,kostant_1970_quantisation_and_unitary_representations_i_prequantization,souriau_1970_structure_des_systemes_dynamiques}) holomorphic-symplectic structure on the corresponding $r$-semisimple coadjoint orbit $\mc O \sse \mf g_r^{\dual}$ (cf.~Thm.~\ref{thm:nondegenerate_character_and_symplectic_form} +~Cor./Def.~\ref{cor:nonsymmetric_shapovalov}).
One thus expects to be able to use the canonical quadratic tensor $F_c \ceqq \mc S_c^{-1}$ of a dilated version of $\mc S$,
provided that the latter is \emph{nondegenerate} for generic values of the dilation parameter $c \in \mb C^{\ts}$,
in order to deformation-quantise the Lie--Poisson(--Berezin) bracket on the symplectic leaf $\mc O$.

This is a viewpoint pursued in~\cite{alekseev_lachowska_2005_invariant_star_product_on_coadjoint_orbits_and_the_shapovalov_pairing} in a quite general setting (cf.~\cite{etingof_schiffmann_lectures_on_the_dynamical_yang_baxter_equations,etingof_varchenko_1999_exchange_dynamical_quantum_groups,etingof_schedler_schiffmann_2000_explicit_quantization_of_dynamical_r_matrices_for_finite_dimensional_semisimple_lie_algebras,enriquez_etingof_2003_quantization_of_alekseev_meinrenken_dynamical_r_matrices,enriquez_etingof_2005_quantization_of_classical_dynamical_r_matrices_with_nonabelian_base}),
essentially matching up with the `triangular' Lie-theoretic setup of~\cite{wilson_2011_highest_weight_theory_for_truncated_current_lie_algebras} .
As a result,
Alekseev--Lachowska's quantisation captures the \emph{monolevel} examples of polarised wild orbits;
but it misses the multilevel case where the Levi/parabolic filtrations are \emph{nonconstant}.
Hence,
we adapt the main construction of~\cite{alekseev_lachowska_2005_invariant_star_product_on_coadjoint_orbits_and_the_shapovalov_pairing} by filling in the absence of $\mb Z$-graded Lie algebras with few technical lemmas,
relying on the $\mc S$-orthogonal weight grading of $M$.
The final result reads as follows:

\begin{theo}[cf.~Thm./Def.~\ref{thm:deformation_quantisation}]
	Let $\mc O \sse \mf g_r^{\dual}$ be an $r$-semisimple orbit,
	equipped with a \emph{balanced} polarisation corresponding to a parabolic filtration $\bm\psi$ (cf.~Rmk.~\ref{rmk:balanced_filtrations});
	choose a marking $\bm\lambda \in \mc O$ in the stratum determined by the Levi factor of $\bm\psi$,
	and let $M$ be the corresponding finite generalised singularity module.
	Then a formal/deformed version of the inverse Shapovalov form of $M$ yields a \emph{local} $G_r$-\emph{invariant deformation quantisation} of $\mc O$.
\end{theo}

(Cf.~\S~\ref{sec:deformation_quantisation_our_case}--\ref{sec:deformation_quantisation_conclusion} for the detailed step-by-step construction of the $\ast$-product.)

\begin{rema}
	Of course it is now quite natural to look into quantisations which use all of the geometric structure at hand.
	E.g.,
	since the standard complex coadjoint orbits carry the hyperkähler structure of Kronheimer--Kovalev--Biquard (= KKB~\cite{kronheimer_1990_a_hyperkaehlerian_structure_on_coadjoint_orbits_of_a_semisimple_complex_group,kovalev_1996_nahm_equations_and_complex_adjoint_orbits,biquard_1996_sur_les_equations_de_nahm_et_la_structure_de_poisson_des_algebres_de_lie_semisimples_complexes}),
	one wonders whether the main construction of~\cite{andersen_malusa_rembado_2024_sp_1_symmetric_hyperkaehler_quantisation} might be adapted.

	On the deformation side,
	it is natural to conjecture that the main construction of~\cite{donin_gurevich_shnider_1996_quantization_of_function_algebras_on_semisimple_orbits_in_g_star} might be generalised to TCLAs---%
	and then related to Thm./Def.~\ref{thm:deformation_quantisation}.
	Moreover,
	we plan to consider the twisted version of this story elsewhere.
\end{rema}

\subsection{Layout of the paper}
\label{sec:layout}

\subsubsection{}

In \S~\ref{sec:setup_1} we gather the data used in the first part of the paper,
involving:
(i) the Levi subsystems of the root system of the Lie algebra of a connected complex reductive Lie group;
(ii) the corresponding strata of semisimple orbits;
(iii) their (complex) symplectic structures;
and (iv) the corresponding actions of (subquotients of) the Weyl group.
In \S~\ref{sec:wild_stratification} we (re)introduce root-valuation stratifications in the wild case,
using filtrations of Levi subsystems.
In \S~\ref{sec:wild_orbits}--\ref{sec:wild_orbits_whole} we recall that wild orbits are attached to isomorphism classes of formal irregular-singular connections,
and classify/mark them in a `truncated' notion of semisimplicity.
In \S~\ref{sec:wild_orbit_strata} we decompose the parameter spaces of marked wild orbits into pieces;
in the untwisted case,
this allows for an explicit description of the orbits as complex homogeneous manifolds.
In \S~\ref{sec:intrinsic_wild_orbit_space} we discuss the change of marking,
getting to intrinsic parameter spaces;
in the untwisted case with semisimple residue,
these are controlled by quotient strata for the Weyl-group action.

This is the end of the first part.

In \S~\ref{sec:setup_2} we collect further data used in the second part of the paper,
involving:
(i) affine Lie algebras;
(ii) (generalised) affine/finite Verma modules;
(iii) Shapovalov forms;
and (iv) parabolic subsets of roots.
In \S~\ref{sec:generic_singularity_modules} we recall the definition of the modules of~\cite{felder_rembado_2023_singular_modules_for_affine_lie_algebras_and_applications_to_irregular_wznw_conformal_blocks}.
In \S~\ref{sec:singularity_subalgebra_and_polarisations} we define affine/finite singularity subalgebras,
involving filtrations of parabolic subsets,
and relate them with polarisations of holomorphic-symplectic wild orbits.
(We also determine the characters of singularity subalgebras,
relating them to the underlying Levi factors.)~In \S~\ref{sec:generalised_singularity_modules} we define the \emph{generalised} affine/finite singularity modules,
compute the relations satisfied by the associated (generalised) irregular state,
and realise them explicitly as quotients of the generic modules.
In \S~\ref{sec:simple_quotients} we make few general remarks about simple quotients of finite generalised singularity modules.
In \S~\ref{sec:nonsingular_characters} we introduce/determine nonsingular characters in the finite wild setting,
with a view towards \S~\ref{sec:deformation_quantisation}.
In \S~\ref{sec:shapovalov} we introduce the Shapovalov form on the new finite modules,
using it to characterise their unique maximal proper ideal,
and we prove a sharp criterion about their simplicity.
In \S~\ref{sec:deformation_quantisation} we adapt the construction of~\cite{alekseev_lachowska_2005_invariant_star_product_on_coadjoint_orbits_and_the_shapovalov_pairing} to deformation-quantise \emph{nongeneric} wild orbits in invariant fashion.
In \S~\ref{sec:irregular_blocks} we extend the construction of flat vector bundles of~\cite{felder_rembado_2023_singular_modules_for_affine_lie_algebras_and_applications_to_irregular_wznw_conformal_blocks}.

\subsubsection{}

In \S~\ref{sec:stratifications} we collect the required material on (topological) stratifications.
In \S~\ref{sec:missing_proofs} we postpone selected proofs,
aiming to ease the reading flow.
In \S~\ref{sec:acronyms} we list the acronyms which appear in this text.

\subsection{Concluding remarks and conventions}

\subsubsection{}

Unless otherwise specified,
in this text:
\begin{enumerate}
	\item all Lie algebras/groups are defined over $\mb C$;

	\item all tensor products are $\mb C$-bilinear;

	\item all modules over noncommutative rings are left.
\end{enumerate}

Also,
the end of remarks/examples is signalled by a `$\triangle$'.

\begin{rema}
	\label{rmk:tame_wild_terminology}

	Let $G \ceqq \GL_m(\mb C)$,
	and identify as usual a principal $G$-bundle with a rank-$m$ vector bundle $E \to \Sigma$,
	where $\Sigma$ is a Riemann surface.
	Choose a point $a \in \Sigma$ and a trivialisation of $E$ around $a$:
	as per Deligne's moderate-growth condition~\cite{deligne_1970_equations_differentielles_a_points_singuliers_reguliers},
	a connection is \emph{tame} (or `regular-singular') at $a$ if any horizontal section on $\Sigma^o \ceqq \Sigma \sm \set{a}$ has at most polynomial growth when approaching the puncture along any direction;
	else it is \emph{wild} (or `irregular-singular').
	E.g.,
	a connection is tame if it is \emph{logarithmic},
	i.e.,
	if it has at most a simple pole at $a \in \Sigma$ in any local trivialisation.
	The converse fails in general,
	but if the leading term is \emph{not} nilpotent then connections with nonsimple poles are wild.

	In view of this,
	hereafter,
	the word `tame' refers to objects associated to simple poles of meromorphic $G$-connections on (germs of) Riemann surfaces,
	i.e.,
	the case $r = 1$ in the above notation.
	Analogously,
	the word `wild' points to high-order poles,
	such as the wild orbits in the title.
\end{rema}

\begin{rema}
	\label{rmk:category_O}

	The results of~\cite{chaffe_2023_category_o_for_takiff_lie_algebras,chaffe_topley_2023_category_o_for_truncated_current_lie_algebras} are quite synergistic with this text.
	In brief,
	there is a \emph{category} $\mc O$ for TCLAs,
	à la Bernstein--Gelfand--Gelfand (= BGG~\cite{bernstein_gelfand_gelfand_1971_structure_of_representations_that_are_generated_by_vectors_of_highest_weight,bernstein_gelfand_gelfand_1975_differential_operators_on_the_base_affine_space_and_a_study_of_g_modules};
	cf.~\cite{humphreys_2008_representations_of_semisimple_lie_algebras_in_the_bgg_category_o}).
	In this terminology,
	the finite generic singularity modules are `regular' Verma modules in category $\mc O$.
	This viewpoint is helpful to establish an irreducibility criterion (cf.~Thm.~\ref{thm:simple_finite_modules}),
	and to deformation-quantise wild de Rham spaces (cf.~Rmk.~\ref{rmk:quantum_de_rham}).

	(Incidentally,
	the \emph{affine} singularity modules might feature in the future in an \emph{affine} BGG category $\mc O$ for TCLAs.)
\end{rema}

\section{Lie-algebraic setup (I)}
\label{sec:setup_1}

\subsection{Basic terminology}

We start by reviewing the objects/notations used in the first part of the paper,
aiming to streamline the exposition:
please refer,
e.g.,
to~\cite{bourbaki_1968_elements_de_mathematiques_fascicule_xxxvii_chapitres_iv_v_vi,humphreys_1972_introduction_to_lie_algebras_and_representation_theory,borel_1991_linear_algebraic_groups,collingwood_mcgovern_1993_nilpotent_orbits_in_semisimple_lie_algebras} for more details.
(Experts might want to skip to \S~\ref{sec:wild_stratification}.)

\subsubsection{}

Let $\mf g = \mf Z_{\mf g} \ops [\mf g,\mf g]$ be a \emph{reductive} Lie algebra.
An element $X \in \mf g$ is semisimple if $\ad_X = [X,\cdot] \in \mf{der}(\mf g)$ is a semisimple endomorphism.
A Cartan subalgebra $\mf t \sse \mf g$ is a maximal abelian Lie subalgebra consisting of semisimple elements.

We consider infinitesimal centralisers of semisimple elements $X \in \mf g$,
i.e.,
Lie subalgebras of the form
\begin{equation}
	\mf l = \mf g^X \ceqq \ker(\ad_X) = \Set{ Y \in \mf g | \ad_Y(X) = 0 },
\end{equation}
so that $\mf t \sse \mf l$ if $X \in \mf t$.
These are \emph{Levi factors} of parabolic subalgebras of $\mf g$,
sometimes abusively referred to as `Levi subalgebras' (cf.~\S~\ref{sec:setup_2}.)

Denote by $\Phi = \Phi(\mf g,\mf t) \sse \mf t^{\dual}$ the root system of $(\mf g,\mf t)$,
which can also be regarded as a (spanning) root system inside $\mf Z_{\mf g}^\perp \sse \mf t^{\dual}$.
There is a root-space decomposition
\begin{equation}
	\label{eq:root_decomposition}
	\mf g = \mf t \ops \bops_\Phi \mf g_\alpha,
	\qquad \mf g_\alpha = \Set{ Y \in \mf g | \ad_X(Y) = \Braket{\alpha,X } Y \text{ for } X \in \mf t } \sse \mf g.
\end{equation}
We use the \emph{dual/inverse} root system $\Phi^{\dual} \sse \mf t$,
containing coroots $\alpha^{\dual} \in \mf t_\alpha \ceqq [\mf g_{-\alpha},\mf g_\alpha]$.
It comes with a canonical bijection $\Phi \lxra{\simeq} \Phi^{\dual}$~\cite[Ch.~6,
	\S~1.1]{bourbaki_1968_elements_de_mathematiques_fascicule_xxxvii_chapitres_iv_v_vi}.
The \emph{rank of} $\mf g$ is the complex dimension of $\mf t$,
while the \emph{semisimple rank of} $\mf g$ is the dimension of $[\mf g,\mf g] \cap \mf t$,
i.e.,
the dimension of $\mf t \slash \mf Z_{\mf g}$;
this is also the \emph{rank of} $\Phi$.

In \S~\ref{sec:tame_strata} we consider a connected reductive Lie group $G$ with Lie algebra $\mf g$,
and the adjoint action $\Ad \cl G \to \Aut(\mf g)$.
(The choice of $G$ does \emph{not} matter,
as we only deal with inner Lie-algebra automorphisms of $\mf g$.)
Then we also consider centralisers of semisimple elements $X \in \mf g$,
i.e.,
connected reductive Lie subgroups of the form
\begin{equation}
	L = G^X \ceqq \Set{ g \in G | \Ad_g(X) = X } ,
\end{equation}
so that $T \sse L$ if $X \in \mf t$,
where $T \sse G$ is the maximal torus integrating $\mf t$.
These are Levi factors of parabolic subgroups of $G$ (a.k.a~`Levi subgroups'),
and the standard correspondence yields $\Lie(L) = \mf l$.

A semisimple adjoint $G$-orbit $\mc O \sse \mf g$ (i.e.,
a `semisimple orbit') is an orbit for the adjoint action through a semisimple element,
whence it consists of semisimple elements.
The orbit through $X \in \mf g$ is denoted by
\begin{equation}
	\mc O_X = G \cdot X \ceqq \Ad_G(X) = \Set{ \Ad_g(X) | g \in G }.
\end{equation}

We also use a nondegenerate $\Ad_G$-invariant symmetric bilinear form on $\mf g$~\cite[Lem.~1.3.4]{collingwood_mcgovern_1993_nilpotent_orbits_in_semisimple_lie_algebras},
first in Rmk.~\ref{rmk:dual_version}.
This is a pairing $(\cdot \mid \cdot) \cl \mf g \ots \mf g \to \mb C$ such that
\begin{equation}
	\bigl( X \mid [Y,Z] \bigr) = \bigl( [X,Y] \mid Z \bigr) ,
	\qquad X,Y,Z \in \mf g,
\end{equation}
making $\bigl( \mf g,(\cdot \mid \cdot) \bigr)$ into a \emph{quadratic} Lie algebra.

In \S~\ref{sec:intrinsic_tame_orbit_space} we consider the Weyl group $W = W(\mf g,\mf t) \simeq N_G(T) \bs T$ of the root system,
acting both on $\mf t$ and $\mf t^{\dual}$ as usual---%
and trivially on $\mf Z_{\mf g}$.
It is generated by reflections $\sigma_\alpha \in \GL_{\mb C}(\mf t^{\dual})$,
for $\alpha \in \Phi$:
a \emph{root subsystem of} $\Phi$ is a subset $\phi \sse \Phi$ which is closed under the Weyl reflections generated by its elements,
i.e.,
$\sigma_\alpha(\phi) \sse \phi$ for $\alpha \in \phi$.

In \S~\ref{sec:relation_meromorphic_connections} we consider a formal variable $z$,
and the corresponding (formal) \emph{loop algebra} $\mf g(\!(z)\!) \ceqq \mf g \ots \mb C(\!(z)\!)$,
invoking the field of formal Laurent series.
Pure tensors are denoted by $Xz^i \ceqq X \ots z^i$,
for $X \in \mf g$ and $i \in \mb Z$.
More generally,
if $V \sse \mf g$ is a vector subspace,
set
\begin{equation}
	\label{eq:graded_piece}
	V \cdot z^i \ceqq \Set{ Xz^i | X \in V } \sse \mf g(\!(z)\!).
\end{equation}
Then multiplying everything by $\dif z$ yields (formal) meromorphic $\mf g$-valued 1-forms $X z^i \dif z \in \mf g(\!(z)\!) \dif z$.

Let finally $r \geq 1$ be an integer.
Starting from \S~\ref{sec:wild_orbits},
we consider the quotient $\mb C$-algebra
\begin{equation}
	\mb C_r \ceqq \mb C \llb z \rrb \bs z^r \mb C \llb z \rrb \simeq \bops_{i = 0}^{r-1} \bigl( \mb C \cdot \varepsilon^i \bigr),
	\qquad \varepsilon \ceqq z + z^r \mb C\llb z \rrb.
\end{equation}
The `deeper' version of $\mf g$ is the TCLA
\begin{equation}
	\label{eq:tcla}
	\mf g_r \ceqq \mf g \ots \mb C_r \simeq \mf g \llb z \rrb \bs z^r \mf g \llb z \rrb,
	\qquad \mf g \llb z \rrb = \mf g \ots \mb C \llb z \rrb,
\end{equation}
and $V \cdot \varepsilon^i \sse \mf g_r$ is defined as in~\eqref{eq:graded_piece},
for $V \sse \mf g$ and $i \in \set{0,\dc,r-1}$.
Analogously,
if $\mf h \sse \mf g$ is any Lie subalgebra,
its `deeper' version is $\mf h_r \ceqq \mf h \ots \mb C_r \sse \mf g_r$.

\subsection{Levi subsystems}

A subset $\phi \sse \Phi$ is a \emph{Levi subsystem} if it is obtained by intersecting $\Phi$ with a vector subspace of $\mf t^{\dual}$,
i.e.,
equivalently,
if one has the equality $\spann_{\mb C}(\phi) \cap \Phi = \phi$.
(One may equivalently take $\mb Q$-linear combinations.)
In particular it is a root subsystem,
but this condition is stronger:
e.g.,
the subsets of short/long roots in type $B_2$ are \emph{not} Levi---%
even though the latter is closed under sum.

We will use the following characterisation,
which relates Levi subsystems with the geometry of semisimple orbits:

\begin{lemm}
	\label{lem:levi_subsystems}

	Let $\phi \sse \Phi$ be a root subsystem.
	Then the following are equivalent:
	\begin{enumerate}
		\item the subset $\phi$ is Levi;

		\item there exists $X \in \mf t$ such that $\phi = \Phi \cap \Set{X}^\perp = \Set{\alpha \in \Phi | \Braket{\alpha,X } = 0 }$;

		\item and the following is a \emph{nonempty} hyperplane complement:
		      \begin{equation}
			      \label{eq:tame_stratum}
			      \mf t_{\phi} \ceqq \ker(\phi) \, \bigsm \, \bigcup_{\Phi \sm \phi} \ker(\alpha),
			      \qquad \ker(\phi) = \bigcap_{\phi} \ker(\alpha) \sse \mf t.
		      \end{equation}
	\end{enumerate}
\end{lemm}

\begin{proof}
	Postponed to~\ref{proof:lem_levi_subsystems}.
\end{proof}

\begin{rema}
	Suppose that $\phi \sse \Phi$ is Levi,
	and let $\phi^{\dual} \sse \Phi^{\dual}$ be the subset of coroots obtained by restricting the canonical bijection to $\phi$.
	Then $\phi^{\dual}$ is a root subsystem of $\Phi^{\dual}$,
	and the restriction of the canonical bijection of $(\Phi,\Phi^{\dual})$ yields the canonical bijection of $(\phi,\phi^{\dual})$~\cite[Ch.~6,
		\S~1.1,
		Prop.~4]{bourbaki_1968_elements_de_mathematiques_fascicule_xxxvii_chapitres_iv_v_vi}.
	Furthermore,
	the subset $\phi^{\dual} \sse \Phi^{\dual}$ is still a Levi subsystem:
	this can be seen by choosing a $W$-invariant nondegenerate pairing $\mf t \ots \mf t \to \mb C$,
	such as the restriction of $(\cdot \mid \cdot) \in \mf g^{\dual} \ots \mf g^{\dual}$---%
	whence $\alpha^{\dual} \in \mf t$ maps to a dilation of $\alpha \in \mf t^{\dual}$.
\end{rema}

\subsection{The Levi poset}
\label{sec:levi_poset}

Let $\mc L_\Phi \sse 2^{\Phi}$ be the set of Levi subsystems of $\Phi$.
We consider it as a poset,
equipped with the \emph{anti-inclusion} order:
$\vn$ is the greatest element,
and $\Phi$ the least one.
Consider then the function
\begin{equation}
	\label{eq:rank_function_levi}
	\rho_\Phi \cl \mc L_\Phi \lra \mb Z_{\geq 0},
	\qquad \rho_\Phi(\phi) \ceqq \dim_{\mb C} \bigl( \ker(\phi) \bigr),
\end{equation}
taking values between the dimension of $\mf Z_{\mf g}$ and the rank of $\mf g$.

\begin{lemm}
	\label{lem:graded_poset}

	The function~\eqref{eq:rank_function_levi} makes $\mc L_\Phi$ into a \emph{graded} poset:
	it is a \emph{rank} function.\fn{
		Note that $\mf g$ is semisimple if and only if $\rho_\Phi(\Phi) = 0$,
		in which case the rank of $\mc L_\Phi$ equals $\rk(\mf g)$.}
\end{lemm}

\begin{proof}
	Postponed to~\ref{proof:lem_graded_poset}.
\end{proof}

\begin{rema}
	\label{rmk:levi_systems_simple_roots}

	This recovers the intersection lattice of flats for the root-hyperplane arrangement of $(\mf g,\mf t)$.

	Note also that the proof~\ref{proof:lem_graded_poset} uses the fact that for every Levi subsystem $\phi \sse \Phi$ there exists a base $\Delta \sse \Phi$ of simple roots such that $\Delta \cap \phi$ is a base for $\phi$,
	i.e.,
	such that~\eqref{eq:tame_stratum} contains a $\Delta$-dominant element.
	While this is a useful fact,
	beware that one \emph{cannot} fix a base $\Delta$ to obtain all Levi subsystems from subsets $\Sigma \sse \Delta$.
	E.g.,
	if $\mf g = \mf{gl}_3(\mb C)$ then $\abs{\mc L_\Phi} = 5 > 4 = 2^{\rk(\Phi)}$ (cf.~\S~\ref{sec:rank_3_example}).
\end{rema}

\subsection{Marked strata of semisimple orbits}
\label{sec:tame_strata}

Here we review standard constructions around semisimple orbits,
putting forward the perspective of stratifications.

\subsubsection{}

Every semisimple orbit $\mc O \sse \mf g$ intersects the chosen Cartan subalgebra in a finite set.
More precisely,
the intersection $\mc O \cap \mf t$ is a $W$-orbit~\cite[\S~2]{collingwood_mcgovern_1993_nilpotent_orbits_in_semisimple_lie_algebras},
and it will be useful to mark the orbits:

\begin{defi}
	\label{def:tame_marking}
	\leavevmode

	\begin{enumerate}
		\item
		      A $\mf t$-\emph{marking of} $\mc O \sse \mf g$ is the choice of a point $X \in \mc O \cap \mf t$.

		\item
		      With this choice,
		      $\mc O = \mc O_X$ is a \emph{marked} semisimple orbit.
	\end{enumerate}
\end{defi}

(Since $\mf t$ is fixed,
we just speak of a `marking'.)

\subsubsection{}

The geometry of a marked semisimple orbit $\mc O_X \sse \mf g$ is governed by the centraliser $G^X \sse G$ of the marking,
in an isomorphism $\mc O_X \simeq G \bs G^X$ of (homogeneous) complex $G$-manifolds.
Recall that $G^X$ is connected~\cite[Thm.~2.3.3]{collingwood_mcgovern_1993_nilpotent_orbits_in_semisimple_lie_algebras},
hence controlled by the corresponding Lie subalgebra $\mf g^X \sse \mf g$.
In turn,
the latter is determined by the Levi subsystem
\begin{equation}
	\label{eq:levi_subsystem_semisimple_element}
	\phi_X \ceqq \Phi \cap \Set{X}^\perp = \Set{ \alpha \in \Phi | \Braket{ \alpha,X } = 0 } \sse \Phi,
\end{equation}
via the root-space decomposition $\mf g^X = \mf t \ops \bops_{\phi_X} \mf g_\alpha$.

Hence,
if two elements of $\mf t$ lie in~\eqref{eq:tame_stratum} for some Levi subsystem $\phi \sse \Phi$,
then the corresponding marked semisimple orbits are canonically isomorphic as complex homogeneous $G$-manifolds;
in the common identification with the quotient $G \bs L_{\phi}$,
where $L_{\phi} \sse G$ is the (connected) Lie subgroup integrating
\begin{equation}
	\label{eq:levi_algebra_from_levi_system}
	\mf l_{\phi} \ceqq \mf t \ops \bops_{\phi} \mf g_\alpha \sse \mf g.
\end{equation}
This yields a biholomorphism-preserving partition of the space of marked semisimple orbits,
identified to $\mf t$ via the bijection $\mc O_X \mt X$,
whose set of parts is indexed by the Levi poset $\mc L_\Phi$ of \S~\ref{sec:levi_poset}.

But there is more topological structure.
Namely,
equip $\mf t$ with the analytic topology,
i.e.,
consider it as a complex manifold;\fn{
	So that $\rho_\Phi(\phi) \geq 0$ coincides with the dimension of~\eqref{eq:tame_stratum}.
	Much works the same in Zariski topology,
	with the usual caveats in the definitions of Galois coverings/fundamental groups.}~then:

\begin{lemm}
	\label{lem:tame_stratification}

	The spaces~\eqref{eq:tame_stratum} define an $\mc L_\Phi$-stratification of $\mf t$,
	with extremal strata
	\begin{equation}
		\mf t_{\reg} \ceqq \mf t_{\vn} = \mf t \, \bigsm \, \bigcup_\Phi \ker(\alpha) ,
		\qquad \mf t_\Phi = \mf Z_{\mf g}.
	\end{equation}
\end{lemm}

\begin{proof}
	Postponed to~\ref{proof:lem_tame_stratification}.
\end{proof}

\subsubsection{}
\label{sec:pure_isomorphism_tame}

In summary,
we aim to generalise the following statement:

\begin{center}
	\emph{If two marked semisimple orbits} $\mc O_{X_1},\mc O_{X_2}$ \emph{lie in one and the same stratum of} $\mf t$,
	\emph{then there is a canonical isomorphism of complex} $G$-\emph{manifolds} $\mc O_{X_1} \simeq \mc O_{X_2}$.
\end{center}

\begin{rema}
	\label{rmk:dual_version}

	There is a $G$-equivariant $\mb C$-linear musical isomorphism $(\cdot \mid \cdot)^\sharp \cl \mf g^{\dual} \lxra{\simeq} \mf g$.
	Under this identification,
	all this material transfers to \emph{coadjoint} $G$-orbits $\mc O \sse \mf g^{\dual}$,
	i.e.,
	orbits for the coadjoint action $\Ad^{\dual} \cl G \to \GL_{\mb C}(\mf g^{\dual})$---%
	cf.~the `lazy' approach of~\cite[\S~1.3]{collingwood_mcgovern_1993_nilpotent_orbits_in_semisimple_lie_algebras}.

	Precisely,
	identify $\mf t^{\dual}$ with the annihilator $\bigl( \bops_\Phi \mf g_\alpha \bigr)^{\! \perp} \sse \mf g^{\dual}$,
	using~\eqref{eq:root_decomposition}.
	Then define a marking of a semisimple coadjoint orbit $\mc O \sse \mf g^{\dual}$ as the choice of an element $\lambda \in \mc O \cap \mf t^{\dual}$:
	this yields a marked semisimple orbit $\mc O_{\lambda} = G \cdot \lambda \ceqq \Ad^{\dual}_G(\lambda)$.
	Introduce then the coadjoint stabiliser
	\begin{equation}
		G^{\lambda} \ceqq \Set{ g \in G | \Ad^{\dual}_g(\lambda) = \lambda } \sse G,
	\end{equation}
	and its Lie algebra
	\begin{equation}
		\label{eq:infinitesimal_coadjoint_stabiliser}
		\mf g^{\lambda} \ceqq \Set{ Y \in \mf g | \ad^{\dual}_Y(\lambda) = 0 } \sse \mf g.
	\end{equation}
	(Recall that $\lambda \in \mf g^{\dual}$ is defined to be semisimple if $\mf g^{\lambda}$ is reductive in $\mf g$,
	without appealing to the identification $\mf g \simeq \mf g^{\dual}$;
	cf.~the `sophisticated' approach of~\cite[\S~1.3]{collingwood_mcgovern_1993_nilpotent_orbits_in_semisimple_lie_algebras}.)

	Again one has $\mc O_{\lambda} \simeq G \bs G^{\lambda}$ as complex homogeneous $G$-manifolds,
	and the corresponding stratification of $\mf t^{\dual}$ by these stabilisers will be introduced in Rmk.~\ref{rmk:minimal_modules} (and generalised in \S~\ref{sec:nonsingular_characters});
	in the tame case,
	it is parameterised by Levi subsystems of the inverse root system $\Phi^{\dual} \sse \mf t$.
	At this stage,
	just note that $G^{\lambda} = G^X$ for $X \ceqq ( \cdot \mid \cdot )^\sharp(\lambda) \in \mf g$.
\end{rema}

\subsubsection{}
\label{sec:dual_kks}

The coadjoint orbits are the symplectic leaves for the linear Poisson bracket of $\mf g^{\dual}$,
and are equipped with the KKS symplectic structure.
This carries over to adjoint ones along the lines of the previous Rmk.~\ref{rmk:dual_version}.

Concretely,
choose $X \in \mf g$.
The alternating $\Ad_{G^X}$-invariant $\mb C$-bilinear form
\begin{equation}
	\label{eq:dual_KKS}
	Y \wdg Z \lmt \bigl( X \mid [Y,Z] \bigr) = \bigl( [X,Y] \mid Z \bigr) \in \mb C,
	\qquad Y,Z \in \mf g,
\end{equation}
induces (after modding out the radical $\mf g^X \sse \mf g$) a \emph{nondegenerate} form $\bigwedge^2 (\mf g \bs \mf g^X) \to \mb C$,
which extends uniquely to a $G$-invariant complex symplectic form $\omega$ on the marked orbit $\mc O_X \sse \mf g$---%
upon suitably trivialising the tangent bundle.

Furthermore,
this (dual) KKS structure is holomorphic with respect to the $G$-invariant complex structure coming from the complex scalar multiplication of $\mf g$ (cf.~\S~\ref{sec:standard_polarisations};
the complex structure does \emph{not} depend on the marking).

\subsection{Unmarked strata of semisimple orbits}
\label{sec:intrinsic_tame_orbit_space}

Act now by the Weyl group $W$.
This forgets the choice of marking and permutes the above strata,
getting to the intrinsic moduli/parameter space $\mb A \ceqq \mf t \bs W$ of semisimple orbits.

More precisely,
the whole stratification descends to the topological quotient,
as follows.
Given $\phi \in \mc L_\Phi$ let $N_W(\mf t_{\phi}) \sse W$ be the setwise Weyl-stabiliser of the stratum~\eqref{eq:tame_stratum}.
Then,
as in \S~\ref{sec:quotient_stratifications},
consider the quotients
\begin{equation}
	\label{eq:tame_quotient_stratum}
	\mf t^W_{\phi} \ceqq \mf t_{\phi} \bs N_W(\mf t_{\phi}) \lhra \mb A.
\end{equation}

\begin{prop}
	\label{prop:quotient_stratification_weyl}

	The spaces~\eqref{eq:tame_quotient_stratum} define a stratification of $\mb A$ indexed by the quotient poset $\mc L_\Phi \bs W$,
	with extremal strata $\mf t_{\reg} \bs W$ and $\mf Z_{\mf g} \simeq \mf Z_{\mf g} \bs W$.
\end{prop}

\begin{proof}
	The Weyl group acts on the set of root subsystems,
	preserving the Levi subsystems and their (anti-)inclusion order.
	Furthermore,
	if $\alpha \in \Phi$ has associated reflection $\sigma_\alpha \in W \sse \GL_{\mb C}(\mf t^{\dual})$,
	then one has $\ps{t}{}\sigma_\alpha^{-1}(\mf t_{\phi}) = \mf t_{\sigma_\alpha(\phi)} \sse \mf t$.
	Hence,
	the stratification of Lem.~\ref{lem:tame_stratification} is $W$-compatible.
	One concludes by Prop.~\ref{prop:quotient_stratification},
	noting that $W = N_W(\mf Z_{\mf g})$ acts trivially on $\mf Z_{\mf g}$.
\end{proof}

\begin{rema}
	\label{rmk:tame_free_weyl_actions}

	The $N_W(\mf t_{\phi})$-action on $\mf t_{\phi}$ is \emph{not} free,
	besides the generic case $\phi = \vn$.
	Nonetheless,
	a Weyl reflection fixes one point of each stratum if and only if it acts trivially thereon:
	the strata consist of points in general positions within the flats of the rational reflection group $(\mf t,W)$.

	Consider then the \emph{pointwise} stabiliser of $\mf t_{\phi} \sse \mf t$,
	i.e.,
	\begin{equation}
		\label{eq:parabolic_subgroup}
		W_{\phi} \ceqq \Set{ w \in W | \eval[1] w_{\mf t_{\phi}} = \Id_{\mf t_{\phi}} } \sse W.
	\end{equation}
	This is the parabolic subgroup of $W$ fixing the vector subspace $\ker(\phi) = \ol{\mf t_{\phi}} \sse \mf t$;
	as well as the Weyl group of $(\mf l_{\phi},\mf t)$,
	in the notation of~\eqref{eq:levi_algebra_from_levi_system}.
	Furthermore,
	$W_{\phi}$ is normalized by $N_W(\mf t_{\phi})$,
	and so the quotient stratification involves normalisers of parabolic subgroups of finite Coxeter groups (cf.~\cite{howlett_1980_normalizers_of_parabolic_subgroups_of_reflection_groups};
	this is expanded upon in~\cite{doucot_rembado_yamakawa_twisted_g_local_wild_mapping_class_groups}).

	Hence,
	on the left-hand side of~\eqref{eq:tame_quotient_stratum} one can replace the denominator with the group of outer automorphisms of the stratum:
	\begin{equation}
		\label{eq:outer_stratum_automorphism}
		\Out(\mf t_{\phi}) \ceqq N_W(\mf t_{\phi}) \bs W_{\phi}.
	\end{equation}
	This is a subquotient of $W$ acting faithfully on $\ker(\phi)$ and freely on $\mf t_{\phi}$.
	It is a Coxeter group,
	a.k.a.~the `relative' Weyl group of $(\mf g,\mf t,\phi)$ (cf.~\cite{bonnafe_2004_actions_of_relative_weyl_groups_i,bonnafe_2005_actions_of_relative_weyl_groups_ii}),
	which need \emph{not} act as a reflection group on $\ker(\phi)$.
	Nonetheless,
	the projection $\mf t_{\phi} \thra \mf t^W_{\phi}$ is a Galois covering with Galois group $\Out(\mf t_{\phi})$---%
	à la Steinberg.
\end{rema}

\subsubsection{}
\label{sec:nonpure_isomorphism_tame}

Finally,
the following clean statement will be generalised in the wild case:

\begin{center}
	\emph{If two semisimple orbits lie in one and the same stratum of} $\mb A$,
	\emph{then they are isomorphic as complex} $G$-\emph{manifolds}.
	\emph{Furthermore,
		if the orbit} $\mc O \sse \mf g$ \emph{intersects} $\mf t_{\phi}$,
	\emph{then the set of identifications} $\mc O \simeq G \bs L_{\phi}$ \emph{obtained from a marking is an} $\Out(\mf t_{\phi})$-\emph{torsor}.
\end{center}

\begin{rema}
	\label{rmk:affine_setup}

	In this text we work in the complex-analytic category,
	viewing semisimple orbits as holomorphic-symplectic manifolds,
	etc.
	However,
	they are also Zariski-closed affine subvarieties of $\mf g^{\dual} = \Spec \Sym(\mf g)$,
	and one could work in the complex-algebraic setting:
	the same works for $r$-semisimple orbits in the wild case (cf.~Rmk.~\ref{rmk:affine_setting}).
\end{rema}

\subsection{Example:
	rank-3 general linear case}
\label{sec:rank_3_example}

We dedicate a subsection to detail some of these constructions in the nonsemisimple case of $\mf g = \mf{gl}_3(\mb C)$.

The standard Cartan subalgebra $\mf t \sse \mf g$ of diagonal matrices is identified with $\mb C^3$,
and the six roots are the functionals $\alpha_{ij} \in \mf t^{\dual}$,
where
\begin{equation}
	\alpha_{ij} \cl (\lambda_1,\lambda_2,\lambda_3) \lmt \lambda_i - \lambda_j,
	\qquad i \neq j \in \Set{1,2,3},
	\quad \lambda_1,\lambda_2,\lambda_3 \in \mb C.
\end{equation}

\subsubsection{}

All root subsystems are Levi,
and there are five of them.
The three proper ones (of rank one) are $\phi_{\Set{i,j}} \ceqq \set{\alpha_{ij},\alpha_{ji}} \sse \Phi$ for $i \neq j \in \Set{1,2,3}$.
The Hasse diagram of the graded poset $\mc L_\Phi$ appears in Fig.~\ref{fig:hasse_strata_a_2}.

\begin{figure}
	\centering
	\begin{tikzpicture}
		\node (top) at (0,0) {$\vn$};
		\node (13) [below of =top] {$\phi_{\Set{1,3}}$};
		\node (12) at (-2,-1) {$\phi_{\Set{1,2}}$};
		\node (23) at (2,-1) {$\phi_{\Set{2,3}}$};
		\node (bottom) [below of=13] {$\Phi$};
		\draw (top) -- (12) -- (bottom);
		\draw (top) -- (13) -- (bottom);
		\draw (top) -- (23) -- (bottom);
	\end{tikzpicture}
	\caption{Example of Levi stratification.}
	\label{fig:hasse_strata_a_2}
\end{figure}

\subsubsection{}

The marked strata are:
\begin{enumerate}
	\item the (rank-3) dense stratum
	      \begin{equation}
		      \mf t_{\vn} = \mf t_{\reg} = \Set{ (\lambda_1,\lambda_2,\lambda_3) \in \mb C^3 | \lambda_1 \neq \lambda_2 \neq \lambda_3 \neq \lambda_1 } \simeq \Conf_3(\mb C),
	      \end{equation}
	      viz.,
	      the space of configurations of ordered triples of \emph{distinct} points in the complex plane;

	\item  the rank-2 strata
	      \begin{equation}
		      \mf t_{\Set{i,j}} \ceqq \mf t_{\phi_{\Set{i,j}}} = \Set{ (\lambda_1,\lambda_2,\lambda_3) \in \mb C^3 | \lambda_i = \lambda_j \neq \lambda_k,
			      \,
			      \Set{i,j,k} = \Set{1,2,3}} \sse \mf t,
	      \end{equation}
	      which are copies of $\Conf_2(\mb C)$;

	\item and the (rank-1) minimal stratum
	      \begin{equation}
		      \mf t_\Phi = \mf Z_{\mf g} = \Set{ (\lambda,\lambda,\lambda) | \lambda \in \mb C } \simeq \mb C = \Conf_1(\mb C).
	      \end{equation}
\end{enumerate}

\subsubsection{}

The Weyl group $W$ can be identified with the symmetric group $\mf S_3$ of the set $\Set{1,2,3}$,
acting on $\mf t \simeq \mb C^3$ with the standard action,
whence on the roots by permuting the indices $i \neq j$.
Thus,
$\mb A$ is the topological space of \emph{unordered} triples in $\mb C$,
and the 3-element quotient poset $\mc L_\Phi \bs W$ consists of the chain $\ol \Phi < \ol \phi < \ol \vn$,
denoting by $\ol \phi$ the $W$-orbit of $\phi_{\Set{i,j}}$---%
for any $i \neq j \in \Set{1,2,3}$.

The whole of $W$ preserves the extremal strata,
with $\Out(\mf t_{\vn}) = W$ and $\Out(\mf t_\Phi) = (1)$.
Finally,
the setwise stabiliser of $\mf t_{\Set{i,j}}$ is identified with the subset of permutations fixing the complementary index $k \in \Set{1,2,3} \sm \Set{i,j}$.
This is also the pointwise stabiliser,
whence $\Out \bigl( \mf t_{\Set{i,j}} \bigr) = (1)$;
the fact is that the two eigenspaces of an element $X \in \mf t_{\Set{i,j}}$ have different dimensions,
by which one can order them.

\subsubsection{}

The quotient strata are:
\begin{enumerate}
	\item the dense stratum
	      \begin{equation}
		      (\mf t \bs W)_{\ol \vn} = \mf t_{\reg} \bs W \simeq \Conf_3(\mb C) \bs \mf S_3 \eqqc \UConf_3(\mb C),
	      \end{equation}
	      which is the space of configurations of \emph{unordered} triples of distinct points in the complex plane;

	\item the `rank-2' stratum
	      \begin{equation}
		      (\mf t \bs W)_{\ol \phi} = \Set{ (\lambda_i,\lambda_k) = (\lambda_j,\lambda_k) \in \mb C^2 | \lambda_i = \lambda_j \neq \lambda_k } \simeq \Conf_2(\mb C),
	      \end{equation}
	      which is still a space of ordered pairs;

	\item and the minimal stratum
	      \begin{equation}
		      (\mf t \bs W)_{\ol \Phi} = \mf Z_{\mf g} \simeq \UConf_1(\mb C).
	      \end{equation}
\end{enumerate}

\section{Wild stratifications}
\label{sec:wild_stratification}

\subsection{Stratifications by Levi filtrations}
\label{sec:levi_filtrations}

Therefore,
the geometry of marked semisimple orbits is controlled by the strata~\eqref{eq:tame_stratum},
and that of (unmarked) semisimple orbits by the quotient strata~\eqref{eq:tame_quotient_stratum}.
Here we extend these stratifications (cf.~again~\cite{goresky_kottwitz_macpherson_2009_codimensions_of_root_valuation_strata}),
and in \S~\ref{sec:wild_orbit_strata}--\ref{sec:intrinsic_wild_orbit_space} we explain the relation with the wild orbits in the title.

\subsubsection{}

Again let $\Phi$ be the root system of $(\mf g,\mf t)$.
We consider nondecreasing exhaustive filtrations of $\Phi$ by Levi subsystems:

\begin{defi}
	\label{def:levi_filtration}

	A \emph{Levi filtration of} $\Phi$ is a sequence $\bm \phi = (\phi_i)_{i \geq 0}$,
	of root subsystems of $\Phi$,
	such that:
	\begin{enumerate}
		\item $\phi_i \sse \phi_{i+1}$ for $i \geq 0$;

		\item $\phi_i \sse \Phi$ is Levi for $i \geq 0$;

		\item and $\Phi = \bigcup_{\mb Z_{\geq 0}} \phi_i$.
	\end{enumerate}
\end{defi}

\subsubsection{}

It follows that $\phi_i$ is a Levi subsystem of $\phi_{i+1}$,
and that $\phi_i = \Phi$ for $i \gg 0$.
The minimal integer such that the latter happens is the \emph{depth} of the Levi filtration.

In \S~\ref{sec:levi_poset} we considered the poset $\mc L_\Phi$ of Levi subsystems:
a Levi filtration $\phi_0 \sse \phi_1 \sse \dc$ is the same as a \emph{nonincreasing} sequence $\phi_0 \geq \phi_1 \geq \dc$ which reaches the least element $\Phi \in \mc L_\Phi$.
Such sequences yield a subset $\mc L_\Phi^{(\infty)} \sse \mc L_\Phi^{\mb Z_{\geq 0}}$ of the set of all $\mc L_\Phi$-valued sequences,
which inherits a partial order;
namely,
declare that $\bm \phi \leq \bm \phi'$ if $\phi_i \leq \phi_i'$ for $i \geq 0$,
i.e.,
if $\phi_i \spse \phi'_i$.

\begin{rema}
	If the inclusion $\phi_i \sse \phi_{i+1}$ is \emph{proper},
	then $\rho_\Phi(\phi_i) > \rho_\Phi(\phi_{i+1})$,
	using the rank function of Lem.~\ref{lem:graded_poset}.
	Hence,
	there are finitely many such jumps,
	and a Levi filtration without repetitions is the same as a path in the Hasse diagram of $\mc L_\Phi$.
	(Whose maximal length equals the \emph{semisimple rank} of $\mf g$.)
\end{rema}

\subsubsection{}

On the Lie-algebraic side,
consider instead the topological subspace $\mf t^{(\infty)} \sse \mf t^{\mb Z_{\geq 0}}$,
of the infinite (topological) Cartesian product $\mf t^{\mb Z_{\geq 0}} \ceqq \prod_{i \geq 0} \mf t$,
consisting of sequences which are eventually central.
I.e.,
let
\begin{equation}
	\mf t^{(\infty)}
	\ceqq \Set{ \bm X = (X_0,X_1,\dc) \in \mf t^{\mb Z_{\geq 0}} | X_i \in \mf Z_{\mf g} \text{ for } i \gg 0 }.
\end{equation}
Equivalently,
we ask that $\Braket{ \alpha,X_i } = 0$ for all roots $\alpha \in \Phi$ and for $i \gg 0$.
(Again we equip $\mf t$ with the analytic/Euclidean topology,
cf.~Lem.~\ref{lem:tame_stratification}.)

Finally,
for $\bm \phi \in \mc L_\Phi^{(\infty)}$ set
\begin{equation}
	\label{eq:wild_strata}
	\mf t^{(\infty)}_{\bm \phi} \ceqq \prod_{\mb Z_{\geq 0}} \Bigl( \ker(\phi_i) \bigsm \bigcup_{\phi_{i+1} \sm \phi_i} \ker(\alpha) \Bigr) \sse \mf t^{\mb Z_{\geq 0}}.
\end{equation}
By construction $\prod_{\mb Z_{\geq 0}} \ker(\phi_i)$ lies in $\mf t^{(\infty)}$,
and a fortiori~\eqref{eq:wild_strata} does too.

\begin{rema}
	Choosing a formal variable $z$,
	there is a natural identification $\mf t^{\mb Z_{\geq 0}} \simeq \mf t\llb z \rrb$ as $\mb C$-vector spaces---%
	taking coefficients.
	However,
	we are \emph{not} using the $z$-adic topology:
	this would induce the discrete topology on the finite-dimensional subspace/truncation considered in \S~\ref{sec:truncated_wild_stratification}.
	(The actual choice of topology does not matter in this text;
	but it does when taking fundamental groups,
	cf.~\S~\ref{sec:about_imds}.)

	Moreover,
	contrary to~\cite{goresky_kottwitz_macpherson_2009_codimensions_of_root_valuation_strata},
	here we do not restrict to the subspace of `regular' elements of $\mf t\llb z \rrb$---%
	which are not annihilated by any root $\alpha \in \Phi$.
	This is required,
	in ordered to treat general multilevel irregular types/classes.
\end{rema}

\begin{lemm}
	\label{lem:wild_stratification}

	The subspaces~\eqref{eq:wild_strata} define a \emph{non-locally-finite} partition of $\mf t^{(\infty)}$ into (nonempty) locally-closed subspaces:
	it is indexed by $\mc L_\Phi^{(\infty)}$,
	and it satisfies the strong frontier condition~\eqref{eq:stratified_closure}.
\end{lemm}

\begin{proof}
	As for local closedness,
	the general properties of the product/subspace topologies yield
	\begin{equation}
		\label{eq:wild_stratum_closure}
		\ol{\mf t^{(\infty)}_{\bm \phi}} = \prod_{\mb Z_{\geq 0}} \ker(\phi_i) \sse \mf t^{(\infty)},
		\qquad \bm\phi \in \mc L^{(\infty)}_\Phi.
	\end{equation}
	Then compute
	\begin{equation}
		\ol{\mf t^{(\infty)}_{\bm \phi}} \sm \mf t^{(\infty)}_{\bm \phi}
		= \bigcup_{i \geq 0} \Bigl( \prod_{j < i} \ker(\phi_j) \ts \bigcup_{\phi_{i+1} \sm \phi_i} \bigl( \ker(\phi_i) \cap \ker(\alpha) \bigr) \ts \prod_{j > i} \ker(\phi_j) \Bigr),
	\end{equation}
	cf.~\eqref{eq:local_closedness}.
	Each term of this union is a product of closed subspaces,
	and so it is closed in $\mf t^{(\infty)}$;
	and the union is \emph{finite},
	as there are only finitely many jump-indices $i \geq 0$ such that $\phi_{i+1} \sm \phi_i \neq \vn$.
	Therefore,
	the subspace~\eqref{eq:wild_strata} is open in its closure.

	Now choose any element $\bm X = (X_0,X_1,\dc) \in \mf t^{(\infty)}$,
	and for an integer $i \geq 0$ set
	\begin{equation}
		\phi_i \ceqq \bigcap_{j \geq i} \phi_{X_j} \sse \Phi.
	\end{equation}
	Then tautologically $\bm X$ lies in the subspace corresponding to this Levi filtration.
	Conversely,
	suppose that $\mf t^{(\infty)}_{\bm \phi} \cap \mf t^{(\infty)}_{\bm \phi'} \neq \vn$,
	and let $i \geq 0$ be an integer such that $\phi_{i+1} = \phi'_{i+1}$.
	Then,
	reasoning as in the proof~\ref{proof:lem_tame_stratification},
	we see that $\phi_i = \phi'_i$;
	by induction,
	the sequences coincide over $\Set{0,\dc,i}$.
	But by construction there exists an integer $i \gg 0$ such that $\phi_{i+1} = \phi'_{i+1} = \Phi$,
	and for this choice they also coincide over $\mb Z_{> i}$.
	So indeed one has a partition.

	Moreover,
	suppose that $\mf t^{(\infty)}_{\bm \phi} \cap \prod_{\mb Z_{\geq 0}} \ker(\phi'_i) \neq \vn$,
	and let $i \geq 0$ be an integer such that $\phi'_{i+1} \sse \phi_{i+1}$.
	Then,
	reasoning as in the proof~\ref{proof:lem_tame_stratification},
	we see that $\phi'_i \sse \phi_i$;
	by induction,
	the same inclusion holds over $\Set{0,\dc,i}$.
	Again,
	by construction,
	there exists an integer $i \gg 0$ such that $\phi'_{i+1} \sse \phi_{i+1} = \Phi$,
	and then the inclusions also hold automatically over $\mb Z_{> i}$.
	In conclusion $\bm \phi \leq \bm \phi'$,
	and the remaining implications are clear.

	Finally,
	to see that the partition is \emph{not} locally-finite,
	observe that the origin $\bm 0 \ceqq (0,0,\dc) \in \mf t^{(\infty)}$ lies in the closure~\eqref{eq:wild_stratum_closure} for all $\bm\phi \in \mc L^{(\infty)}_\Phi$.
\end{proof}

\subsection{Truncation}
\label{sec:truncated_wild_stratification}

The failure of local-finiteness in Lem.~\ref{lem:wild_stratification} reflects the fact that $\mc L^{(\infty)}_{\Phi}$ has no greatest element.
Wanting to relate with meromorphic gauge theory,
we now \emph{truncate} the above partition,
which is akin to bound the pole order of the irregular type/class of an untwisted irregular-singular meromorphic connection germ.
(This actually leads to finitely many strata.)

\subsubsection{}

Precisely,
fix an integer $s \geq 0$ and consider the subposet of Levi filtrations of depth bounded above by $s$:
\begin{equation}
	\mc L_\Phi^{(s)} \ceqq \Set{ \bm \phi = (\phi_0,\phi_1,\dc) \in \mc L^{(\infty)}_\Phi | \phi_s = \Phi } \sse \mc L^{(\infty)}_\Phi.
\end{equation}
In particular,
$s = 0$ yields the constant minimal sequence $\bm\phi_0 \ceqq (\Phi,\Phi,\dc)$.
Of course one may neglect the constant infinite tail $(\phi_s = \Phi,\phi_{s+1} = \Phi,\dc)$,
so that equivalently $\mc L_\Phi^{(s)}$ consists of sequences of finite length.

Analogously,
identify the finite topological Cartesian power $\prod_{i = 0}^{s-1} \mf t\simeq \mf t^s$ with the subspace of $\mf t^{(\infty)}$ consisting of sequences $\bm X = (X_0,X_1,\dc)$ such that $X_i = 0$ for $i \geq s$---%
in Euclidean topology.
The composite
\begin{equation}
	\mf t^s \lhra \mf t^{\mb Z_{\geq 0}} \lthra \mf t^{\mb Z_{\geq 0}} \bs \mf t^{\mb Z_{\geq s}},
	\qquad \mf t^{\mb Z_{\geq s}}
	\ceqq \prod_{i \geq s} \mf t,
\end{equation}
yields a $\mb C$-linear homeomorphism,
and so one can equivalently view this as a projection/truncation.
Define now
\begin{equation}
	\label{eq:wild_truncated_strata}
	\mf t^s_{\bm \phi} \ceqq \prod_{i = 0}^{s-1} \Bigl( \ker(\phi_i) \bigsm \bigcup_{\phi_{i+1} \sm \phi_i} \ker(\alpha) \Bigr) \sse \mf t^s,
	\qquad \bm \phi \in \mc L^{(s)}_\Phi,
\end{equation}
removing the infinite tail of a particular case of~\eqref{eq:wild_strata}.
Then:

\begin{coro}
	\label{cor:truncated_wild_stratification}

	The subspaces~\eqref{eq:wild_truncated_strata} define an $\mc L^{(s)}_\Phi$-\emph{stratification of} $\mf t^s$.
	Its dense stratum is
	\begin{equation}
		\label{eq:wild_truncated_bulk}
		\mf t^s_{\reg} \ceqq \mf t^s_{\pmb \vn_s} \simeq \mf t^{s-1} \ts \mf t_{\reg},
		\qquad \pmb \vn_s = (\underbrace{\vn,\dm,\vn}_{s-1 \text{ terms}},\Phi,\Phi,\dc) \in \mc L^{(s)}_\Phi,
	\end{equation}
	and its minimal stratum is $\mf t^s_{\bm\phi_0} = \mf Z_{\mf g}^s$,
	where again $\bm\phi_0 = (\Phi,\Phi,\dc)$.
\end{coro}

\begin{proof}
	Consider first the topological subspace
	\begin{equation}
		\mf t^{(s)}
		\ceqq \bigcup_{\mc L^{(s)}_\Phi} \mf t^{(\infty)}_{\bm\phi} \sse \mf t^{(\infty)}.
	\end{equation}
	This is a union of parts of $\mf t^{(\infty)}$,
	and by Cor.~\ref{cor:union_of_strata} it inherits an $\mc L_\Phi^{(s)}$-stratification:
	here the point is that the strong frontier property~\eqref{eq:stratified_closure} holds,
	and then local closedness/finiteness are automatic since $\mc L^{(s)}_\Phi$ is \emph{finite} (cf.~Rmk.~\ref{rmk:stratifications_and_closures}).
	Moreover,
	since $\pmb \vn_s$ is the greatest element of $\mc L^{(s)}_\Phi$,
	the corresponding stratum is dense and one has
	\begin{equation}
		\mf t^{(s)}
		= \ol{\mf t^{(s)}_{\pmb \vn_s}}
		= \mf t^s \ts \mf Z_{\mf g}^{\mb Z_{\geq s}}.
	\end{equation}
	Now the canonical projection $\mf t^{(s)} \thra \mf t^s$ simply deletes the infinite central tail,
	and it maps each stratum homeomorphically onto its truncated version~\eqref{eq:wild_truncated_strata}.
\end{proof}

\begin{rema}
	\label{rmk:wild_hyperplane_complement}

	Note that~\eqref{eq:wild_truncated_strata} is a hyperplane complement in a vector subspace of $\mf t^s$,
	generalising~\eqref{eq:tame_stratum} from the tame case $s = 1$.

	Namely,
	given $\bm \phi \in \mc L_\Phi^{(s)}$ and $\alpha \in \Phi$ define the integer $d_\alpha = d_{\alpha,\bm \phi} \in \set{0,\dc,s}$ by
	\begin{equation}
		d_\alpha \ceqq \min \Set{ i \in \mb Z_{\geq 0} | \alpha \in \phi_i }.
	\end{equation}
	Thus,
	one has $\alpha \in \phi_{d_\alpha} \sm \phi_{d_\alpha-1}$,
	with the convention $\phi_{-1} \ceqq \vn$,
	and clearly the strata~\eqref{eq:wild_truncated_strata} are determined by these integers---%
	e.g.,
	the dense stratum corresponds to having $d_\alpha = r$ for all $\alpha \in \Phi$.
	Then consider the hyperplane
	\begin{equation}
		H_\alpha^{(d_\alpha)} \ceqq \mf t^{d_\alpha-1} \ts \ker(\alpha) \ts \mf t^{s-d_\alpha} \sse \mf t^s.
	\end{equation}
	Swapping complements with union/intersections yields the equality
	\begin{equation}
		\label{eq:wild_hyperplane_complement}
		\mf t^s_{\bm \phi} = \ker(\bm \phi) \, \bigsm  \,
		\bigcup_{d_\alpha > 0} H_\alpha^{(d_\alpha)} ,
		\qquad \ker(\bm \phi) \ceqq \prod_{i = 0}^{s-1} \ker(\phi_i) \sse \mf t^s.
		\qedhere
	\end{equation}
\end{rema}

\begin{exem}
	\label{ex:sl_2_wild_truncated_stratification}

	Consider $\mf g = \mf{sl}_2(\mb C)$,
	with root system $\Phi = \set{\pm \alpha}$ determined as usual by $\braket{ \alpha, H } = 2$,
	invoking the standard generator $H = \on{diag}(1,-1)$ of the standard Cartan subalgebra $\mf t = \mb CH$.

	Here $\mc L_\Phi = \Set{ \Phi < \vn }$ as a poset,
	so for any integer $s \geq 1$ there is an order-preserving bijection
	\begin{equation}
		\label{eq:levi_filtration_sl_2}
		\Set{0,\dc,s} \lxra{\simeq} \mc L^{(s)}_\Phi,
		\qquad k \lmt \bm \phi^{(k)} \ceqq (\underbrace{\vn,\dc,\vn,\Phi}_{k \text{ terms}},\dc),
	\end{equation}
	and in turn the truncated wild strata are
	\begin{equation}
		t^s_{\bm \phi^{(k)}} \simeq \mf t^{k-1} \ts \mf t_{\reg} \ts \Set{0}^{s-k} \simeq \mf t^{k-1} \ts \mf t_{\reg},
		\qquad k \in \Set{1,\dc,s}.
		\qedhere
	\end{equation}
\end{exem}

\subsubsection{}

Once more,
choosing bases breaks the Weyl symmetry that we will consider below,
so the definitions are independent of that.
Nonetheless,
such a choice is helpful to bound the number of truncated strata.

Namely,
choose an integer $s \geq 1$ and a depth-bounded Levi filtration $\bm \phi \in \mc L_\Phi^{(s)}$:

\begin{lemm}[Cf.~\cite{doucot_rembado_tamiozzo_local_wild_mapping_class_groups_and_cabled_braids}]
	\label{lem:base_levi_filtrations}

	There exists a base $\Delta \sse \Phi$ of simple roots such that $\Delta \cap \phi_i$ is a base of $\phi_i$,
	for all $i \in \set{0,\dc,s}$.
\end{lemm}

\begin{proof}
	Postponed to~\ref{proof:lem_base_levi_filtrations}.
\end{proof}

\begin{coro}
	One has
	\begin{equation}
		\abs{\mc L_\Phi^{(s)}} \leq \abs W (s+1)^{\rk(\Phi)},
		\qquad s \in \mb Z_{\geq 1}.
	\end{equation}
\end{coro}

\begin{proof}
	By Lem.~\ref{lem:base_levi_filtrations},
	for every $\bm \phi \in \mc L_\Phi^{(s)}$ there exists a base $\Delta \sse \Phi$ such that $\phi_0 \sse \phi_1 \sse \dc$ corresponds to a nondecreasing sequence of subsets $\Sigma_i \ceqq \phi_i \cap \Delta \sse \Delta$---%
	such that $\Sigma_s = \Delta$.
	The number of such sequences is $(s+1)^{\, \abs{\Delta}} = (s+1)^{\rk(\Phi)}$,
	as every simple root has $(s+1)$ independent slots where it can first appear in the sequence,
	including not appearing at all;
	then multiply by the cardinality of the set of bases,
	which is a $W$-torsor.
\end{proof}

\subsubsection{}

In brief,
length-bounded nondecreasing sequences of subsets of $\Delta \sse \Phi$ yield a subposet of $\mc L_\Phi^{(s)}$,
which in turn is a nondisjoint union of such:
what matters are the linearly independent subsets of $\Phi$,
and not just the maximal such (cf.~\cite{oxley_1992_matroid_theory}).
In type $A$,
one can count these strata by means of `fission trees'~\cite{boalch_2025_counting_the_fission_trees_and_nonabelian_hodge_graphs_untwisted_case}.

\subsection{Quotients}

Again one can take out the action of the Weyl group.
Namely,
$W$ acts naturally on (each term of) Levi filtrations,
preserving the partial order;
analogously,
it acts diagonally on $\mf t^s$,
in continuous fashion,
for any integer $s \geq 1$.
Then the proof of Prop.~\ref{prop:quotient_stratification_weyl} generalises to yield the following:

\begin{theo}
	\label{thm:quotient_wild_stratification}

	There is a stratification of $\mb A_s \ceqq \mf t^s \bs W$,
	indexed by $W$-orbits of (depth-bounded) Levi filtrations.
	The stratum corresponding to $\bm \phi \in \mc L^{(s)}_\Phi$ is the subspace of $W$-orbits intersecting~\eqref{eq:wild_truncated_strata};
	the dense stratum is $\mf t^s_{\reg} \bs W$,
	and the minimal stratum is $\mf Z_{\mf g}^s \simeq \mf Z^s_{\mf g} \bs W$.
\end{theo}

\begin{proof}[Proof omitted]
\end{proof}

\begin{rema}
	As in Rmk.~\ref{rmk:tame_free_weyl_actions},
	for all $\bm \phi \in \mc L^{(s)}_\Phi$ one can replace the setwise stabiliser $N_W \bigl( \mf t^s_{\bm \Phi} \bigr) \sse W$ with a suitable quotient thereof,
	which acts freely on the stratum.
	The result is still a piecewise (finite) Galois covering,
	with varying groups of deck transformations,
	akin to a constructible sheaf (cf.~\cite{doucot_rembado_2025_topology_of_irregular_isomonodromy_times_on_a_fixed_pointed_curve,doucot_rembado_yamakawa_twisted_g_local_wild_mapping_class_groups}).
\end{rema}

\subsection{Relation to isomonodromic deformations}
\label{sec:about_imds}

As recalled in the introduction,
the fundamental groups $\pi_1 \bigl( \mf t^s_{\bm \phi} \bigr)$ also play the role of \emph{pure} local wild mapping class groups,
while after taking the quotient one gets to the \emph{full/nonpure} version;
in brief,
they generalise the Brieskorn--Deligne $G$-braid groups~\cite{brieskorn_1971_die_fundamentalgruppe_des_raumes_der_regulaeren_orbits_einer_endlichen_komplexen_spiegelungsgruppe,deligne_1972_les_immeubles_des_groupes_de_tresses_generalises},
and they govern the braiding of Stokes data of wild $G$-connections over a disc (cf.~\cite{boalch_2014_geometry_and_braiding_of_stokes_data_fission_and_wild_character_varieties}).

The generic case of isomonodromic deformations thus corresponds to working within the dense stratum $\mf t^s_{\reg} \sse \mf t^s$.
The nongeneric case,
instead,
is about exploring the deeper strata,
i.e.,
the components of
\begin{equation}
	\mf t^s \sm \mf t^s_{\reg} = \ol{\mf t^s_{\reg}} \sm \mf t^s_{\reg} = \bigcup_{\mc L^{(s)}_\Phi \sm \set{\pmb \vn_s}} \mf t^s_{\bm \phi}.
\end{equation}

Note in particular that the nonzero integers $d_\alpha \in \set{1,\dc,s}$ which occur in the hyperplane complement~\eqref{eq:wild_hyperplane_complement} coincide with the \emph{levels} of an irregular-singular untwisted meromorphic $G$-connection germ:
these are by definition the pole orders of the nonvanishing exponential factors based there,
featuring in the formal fundamental horizontal sections,
and the nongeneric case enters into the study of \emph{multilevel} meromorphic connections.

For the twisted/ramified case,
instead,
we refer to~\cite{boalch_yamakawa_2015_twisted_wild_character_varieties,boalch_doucot_rembado_2025_twisted_local_wild_mapping_class_groups_configuration_spaces_fission_trees_and_complex_braids,doucot_rembado_yamakawa_twisted_g_local_wild_mapping_class_groups}.
This involves \emph{complex} reflection groups,
twisted root-valuation stratifications~\cite[\S~4]{goresky_kottwitz_macpherson_2009_codimensions_of_root_valuation_strata},
Springer/Lehrer--Springer theory~\cite{springer_1974_regular_elements_of_finite_reflection_groups,lehrer_springer_1999_reflection_subquotients_of_unitary_reflection_groups},
and Bessis' lift thereof to braid groups~\cite{bessis_2015_finite_complex_reflection_arrangements_are_k_pi_1}.

\section{Introducing wild orbits}
\label{sec:wild_orbits}

\subsection{}
\label{sec:relation_meromorphic_connections}

Here we recall about the wild generalisations of marked/unmarked semisimple orbits.
Please refer to~\cite{boalch_2001_symplectic_manifolds_and_isomonodromic_deformations,boalch_2007_quasi_hamiltonian_geometry_of_meromorphic_connections,yamakawa_2019_fundamental_two_forms_for_isomonodromic_deformations,chaffe_rembado_yamakawa_genus_zero_wild_quantum_de_rham_spaces} for the link with $(r-1)$-jets of principal $G$-bundle automorphisms and de Rham spaces (this is more leisurely explained in~\cite[Ch.~2]{boalch_1999_symplectic_geometry_and_isomonodromic_deformations},
in the generic vector-bundle case);
cf.~also~\cite{fernandezherrero_reduction_theory_for_connections_over_the_formal_punctured_disc}.

\subsubsection{}
\label{sec:wild_orbits_symplectic_reduction}

Let $z$ be a formal variable,
and $\wh{\mc D} \ceqq \Spec \mb C\llb z \rrb$ the corresponding standard disc.
Introduce the vector-space splitting
\begin{equation}
	\mf g(\!(z)\!) = \mf g^- \ops \mf g^+,
	\qquad \mf g^- = z^{-1} \mf g [z^{-1}] \ceqq \mf g \ots z^{-1} \mb C [z^{-1}],
	\quad \mf g ^+ \ceqq \mf g \llb z \rrb.
\end{equation}
Viewing $G$ as---%
the group of $\mb C$-points of---%
a smooth affine algebraic group over $\mb C$,
recall that any fppf/étale-locally-trivial principal $G$-bundle $\mc E \to \wh{\mc D}$ is globally trivializable (by~\cite[Thm.~11.7]{grothendieck_1968_le_groupe_de_brauer_iii_examples_et_complements} +~\cite[Exp.~XXIV, Prop.~8.1]{demazure_grothendieck_1970_schemas_en_groupes_iii_structure_des_schemas_en_groupes_reductifs}).
Upon choosing a trivialization of $\mc E$,
and basing at the trivial connection,
the affine space $\on{Conn}_G\bigl( \wh{\mc D} \bigr)$ of singular $G$-connections on $\mc E$ can be identified with the model vector space $\mf g(\!(z)\!)\dif z$.
The space of principal/polar parts of such connections is then
\begin{equation}
	\on{Conn}_G^-\bigl( \wh{\mc D} \bigr) \ceqq \bigl( \mf g (\!(z)\!) \dif z \bigr) \bs \mf g^+ \dif z \simeq \mf g^- \dif z.
\end{equation}
These may be viewed as principal parts of (formal) \emph{germs of} meromorphic $G$-connections on a Riemann surface,
but---%
besides \S~\ref{sec:irregular_blocks}---%
in this text we do not attach the disc to any Riemann surface.

Now extend the $\Ad_G$-invariant pairing of $\mf g$ to a nondegenerate $\mb C$-bilinear form $( \cdot \mid \cdot)_z \cl \mf g (\!(z)\!) \dif z \ots \mf g (\!(z)\!) \to \mb C$,
defined on pure tensors by
\begin{equation}
	\label{eq:affine_pairing}
	\bigl( X \ots \alpha \, | \, Y \ots f \bigr)_{\!z} \ceqq ( X \mid Y) \Ress_{z = 0}(f\alpha),
	\qquad X,Y \in \mf g,
	\quad f \in \mb C(\!(z)\!),
	\quad \alpha \in \mb C(\!(z)\!) \dif z.
\end{equation}
This is essentially the 2-cocycle defining the central extension $\wh{\mf g} \thra \mf g(\!(z)\!)$ (cf.~\S~\ref{sec:setup_2};
the left slot should be viewed as the linearised version of $\on{Conn}_G \bigl( \wh{\mc D} \bigr)$).
Moreover,
the restriction of~\eqref{eq:affine_pairing} to the subspace $\mf g^- \dif z \ots \mf g^+ \sse \mf g(\!(z)\!) \dif z \ots \mf g(\!(z)\!)$ remains nondegenerate,
whence a $\mb C$-linear isomorphism
\begin{equation}
	\label{eq:loop_algebra_duality}
	\on{Conn}_G^-\bigl( \wh{\mc D} \bigr) \simeq  \bigl( \mf g^+ \bigr)^{\! \dual}.
\end{equation}

Then consider the group $G \llb z \rrb \ceqq G \bigl( \mb C \llb z \rrb \bigr)$,
of gauge transformations of the trivial principal $G$-bundle over $\wh{\mc D}$.
It acts on $\on{Conn}_G\bigl( \wh{\mc D} \bigr)$ in affine fashion.
Moreover,
the pairing~\eqref{eq:affine_pairing} is $G \llb z \rrb$-invariant,
by letting $G \llb z \rrb$ act on the left slot by the \emph{linearised} gauge action,
and on the right slot by the adjoint action on $\mf g^+ \simeq \Lie \bigl( G \llb z \rrb \bigr)$.
In particular,
the duality~\eqref{eq:loop_algebra_duality} is $G\llb z \rrb$-equivariant.

Finally,
choose an integer $r \geq 1$ and consider the finite-dimensional subspace
\begin{equation}
	\label{eq:bounded_principal_parts}
	\on{Conn}_G^-\bigl( \wh{\mc D} \bigr) \supseteq \on{Conn}_G^{-r}\bigl( \wh{\mc D} \bigr) \ceqq \bigl( z^{-r} \mf g^+ \dif z \bigr) \bs \mf g^+ \dif z \simeq \bops_{i = 1}^r \mf g \cdot z^{-i} \dif z,
\end{equation}
consisting of principal parts of bounded pole order.
Under the duality~\eqref{eq:loop_algebra_duality},
it maps isomorphically onto the annihilator of the Lie ideal $z^r \mf g^+ \sse \mf g^+$,
which can be identified with $\mf g_r^{\dual}$ (cf.~\eqref{eq:tcla}).
Moreover,
the space~\eqref{eq:bounded_principal_parts} carries an action of the finite-dimensional $r$-jet Lie group
\begin{equation}
	\label{eq:tclg}
	G_r
	= \on J_r(G)
	\ceqq G (\mb C_r) \simeq G \llb z \rrb \bs H_r,
\end{equation}
where $\mb C_r = \mb C \llb z \rrb \bs z^r \mb C\llb z \rrb$ and $H_r \ceqq \exp( z^r \mf g^+ ) \sse G \llb z \rrb$,
as the latter normal subgroup fixes it pointwise.
Overall,
in view of the above equivariance,
and in the Lie-algebra isomorphism $\Lie(G_r) \simeq \mf g_r$,
this truncation of the linearised gauge action matches up with the coadjoint action of the TCLG~\eqref{eq:tclg} on $\mf g_r^{\dual}$.
(By Weil-restricting scalars~\cite{weil_1982_adeles_and_algebraic_groups},
the TCLG still underlies a connected complex algebraic group,
now \emph{nonreductive}.)

\subsubsection{}

In the end,
the (formal-holomorphic) gauge action on principal parts corresponds to the coadjoint action on a dual TCLA.
For example,
in the tame case $r = 1$ this matches up a (formal) residue term $\Lambda z^{-1} \dif z \in \mf g \cdot z^{-1}\dif z$ with a covector $\lambda \in \mf g^{\dual}$.

The coadjoint orbits $\mc O \sse \mf g_r^{\dual}$ are the wild orbits in the title of this paper.

\subsection{Semidirect splittings}

There is an exact sequence of Lie algebras
\begin{equation}
	0 \lra \mf{bir}_r^{\mf g} \lra \mf g_r \lra \mf g \lra 0,
\end{equation}
where the surjection is the quotient modulo the nilpotent `Birkhoff' Lie ideal $\mf{bir}_r \ceqq \varepsilon \mf g_r \sse \mf g_r$.
(Recall that $\varepsilon$ is the class of $z$ in $\mb C_r$.)
The sequence splits via the embedding $\mf g \simeq \mf g \cdot \varepsilon^0 \hra \mf g_r$,
so that there is a semidirect factorization $\mf g_r \simeq \mf g \lts \mf{bir}_r^{\mf g}$,
involving the $\mb C_r$-linear extension of the adjoint $\mf g$-action.

There is an analogous split exact sequence of Lie/algebraic groups:
\begin{equation}
	1 \lra \on{Bir}_r^G \lra G_r \lra G \lra 1,
\end{equation}
where the `Birkhoff' subgroup $\on{Bir}_r^G \ceqq \exp \bigl( \mf{bir}_r^{\mf g} \bigr) \sse G_r$ corresponds to gauge transformations which are the identity at $z = 0$.
We usually drop the superscripts from Birkhoff Lie subalgebras/subgroups;
note that they are actually the nilradical/unipotent radical of $\mf g_r$ and $G_r$,
respectively.

\subsection{Dual viewpoint}

There is finally a third nondegenerate pairing,
which we use to equivalently work with \emph{adjoint} $G_r$-orbits $\mc O \sse \mf g_r$.
(This makes use of the function $z$,
which the affine scheme $\wh{\mc D}$ comes equipped with.)

To introduce it,
write an element $\bm X \in \mf g_r$ as a sum
\begin{equation}
	\label{eq:element_deeper_lie_algebra}
	\bm X = \sum_{i = 0}^{r-1} X_i \varepsilon^i,
	\qquad X_0,\dc,X_{r-1} \in \mf g,
	\quad X_i \varepsilon^i \ceqq X_i \ots \varepsilon^i.
\end{equation}

\begin{enonce}{Lemma/Definition}
	\label{lem:deeper_pairing}

	Choose a complex number $c \in \mb C$.
	The symmetric $\mb C$-bilinear form $( \cdot \mid \cdot )_c \cl \mf g_r \ots \mf g_r \to \mb C$ defined by
	\begin{equation}
		\label{eq:deeper_pairing}
		\bigl( X\varepsilon^i \mid Y\varepsilon^j \bigr)_{\!c} \ceqq ( X \mid Y ) \delta_{i+j,c-1},
		\qquad X,Y \in \mf g,
		\quad i,j \in \Set{0,\dc,r-1},
	\end{equation}
	is nondegenerate and $\Ad_{G_r}$-invariant \emph{if and only if} $c = r$.
\end{enonce}

\begin{proof}
	Postponed to~\ref{proof:lem_deeper_pairing}.
\end{proof}

\subsubsection{}

Under the musical isomorphism $(\cdot \mid \cdot)_r^\sharp \cl \mf g^{\dual}_r \lxra{\simeq} \mf g_r$,
the correspondence with an element~\eqref{eq:principal_part_intro} becomes
\begin{equation}
	\label{eq:deeper_duality}
	\sum_{i = 1}^r A_i z^{-i} \dif z = \mc A \lmt \bm X = \sum_{i = 0}^{r-1} X_i \varepsilon^i \in \mf g_r,
	\qquad X_i \ceqq A_{r-i} \in \mf g.
\end{equation}
The residue $A_1$ thus matches up with the `lowest' coefficient $X_{r-1}$,
and conversely $X_0$ is the `leading' term.

\begin{rema}
	\label{rmk:changing_trivialization_and_uniformizer}

	Since we only look at (truncated) gauge orbits,
	the choice of an initial trivialization is w.l.o.g.:
	the group $G\llb z \rrb$ acts on the set of such choices in simply-transitive fashion.

	Somewhat analogously,
	the choice of a formal variable $z$ is convenient and also w.l.o.g.,
	as far the symplectomorphism class of wild orbits is concerned.
	Indeed,
	the group of continuous $\mb C$-algebra isomorphisms of $\wh{\ms O} \ceqq \mb C \llb z \rrb$ acts in simply-transitive fashion on the set of uniformizers of the complete discrete valuation ring $\wh{\ms O}$:
	upon truncation,
	this results in a Poisson action on $\mf g_r^{\dual}$ which permutes the coadjoint orbits in symplectic fashion.
	Moreover,
	the latter corresponds to a natural embedding of the group of $\mb C$-algebra automorphisms of $\mb C_r$ into the group of (outer) Lie-algebra automorphisms of $\mf g_r$ (cf.~\cite[App.~A]{chaffe_rembado_yamakawa_genus_zero_wild_quantum_de_rham_spaces}).
\end{rema}

\section{Wild orbits:
  Birkhoff action}
\label{sec:birkhoff_action}

\subsection{}

In the tame case ($r = 1$) an orbit $\mc O \sse \mf g^{\dual} \simeq \mf g$ is either semisimple,
or it is not.
Furthermore,
when it is semisimple,
it lies in one stratum of $\mb A = \mf t \bs W$,
as controlled by (Weyl orbits of) Levi subsystems $\phi \sse \Phi$.

\subsubsection{}

We will extend such statements to any pole order $r \geq 1$.
Concretely,
following the standard reduction theory for (formal) rational differential operators,
we proceed by first using the Birkhoff action to make an element of $\mf g_r$ `as semisimple as possible',
while in \S~\ref{sec:wild_orbits_whole} we use the residual $G$-action to `diagonalise' it.

\begin{defi}
	\label{def:s_semisimplicity}

	Choose an integer $s \in \Set{0,\dc,r}$:
	\begin{enumerate}
		\item a $\on{Bir}_r$-orbit $\mc O \sse \mf g_r$ is $s$-\emph{semisimple} if it contains an element $\bm X = \sum_{i = 0}^{r-1} X_i \varepsilon^i$ such that $X_0,\dc,X_{s-1} \in \mf g$ are semisimple and commute with each other;

		\item and an element $\bm X \in \mf g_r$ is $s$-\emph{semisimple} if it lies in an $s$-semisimple $\on{Bir}$-orbit.
	\end{enumerate}
\end{defi}

\begin{rema}
	If $r = 1$ then $\on{Bir}_r \sse G_r$ is trivial.
	In that case Def.~\ref{def:s_semisimplicity} just says that an element $X \in \mf g = \mf g_1$ is either semisimple (= 1-semisimple),
	or it is not.
	One can therefore assume that $r \geq 2$ to deal with interesting cases.
\end{rema}

\subsubsection{}

Let $\mc O \sse \mf g_r$ be a $\on{Bir}_r$-orbit.
There is a unique (maximal) integer $s \in  \Set{0,\dc,r}$ such that $\mc O$ is $s$-semisimple,
but \emph{not} $(s+1)$-semisimple:
we then say $\mc O$ that is \emph{strictly} $s$-\emph{semisimple},
and use the same terminology for an element $\bm X \in \mf g_r$.

In particular,
setting $\bm s(\bm X) \ceqq s$ if $\bm X$ is strictly $s$-semisimple defines a $\on{Bir}_r$-invariant function $\bm s \cl \mf g_r \to \Set{0,\dc,r}$.
This yields a filtration
\begin{equation}
	\label{eq:semisimplicity_strata}
	\mf g_r = \bigcup_{s = 0}^r \mf g_r^{(\geq s)},
	\qquad \mf g_r^{(\geq s)} \ceqq \bm s^{-1} \bigl( \Set{s,\dc,r} \bigr) \sse \mf g_r,
\end{equation}
so that $\mf g_r^{(\geq s)}$ is the subspace of $s$-semisimple elements,
while $\mf g_r^{(s)} \ceqq \mf g_r^{(\geq s)} \sm \mf g_r^{(\geq s+1)}$ is the subspace of strictly $s$-semisimple elements.\fn{
The latter \emph{cannot} define a stratification in general,
e.g.,
since $\mf g_1^{(1)} = \mf g_1^{(\geq 1)} \sse \mf g$ is dense but not open.
}

But there are finer invariants attached to a strict $s$-semisimple $\on{Bir}_r$-orbit.
To extract them we repeatedly use the following elementary fact:

\begin{rema}
	\label{rmk:semisimple_trick}

	Let $V$ be a finite-dimensional complex vector space,
	and $f \cl V \to V$ a semisimple/diagonalisable endomorphism.
	Then there is a $\mb C$-linear splitting $V = \ker(f) \ops f(V)$.
\end{rema}

\subsection{Generalised irregular type}
\label{sec:generalised_irregular_type}

Consider the truncation/quotient map
\begin{equation}
	\label{eq:truncation_map}
	\tau_k \cl \mf g_r \lthra \mf g_k \simeq \mf g_r \bs \varepsilon^k \mf g_r,
	\qquad k \in \Set{1,\dc,r},
\end{equation}
i.e.,
$\tau_k(\bm X) \ceqq X_0 + \dm + X_{k-1} \varepsilon^{k-1}$ (cf.~\eqref{eq:element_deeper_lie_algebra}).
Choose then a $\on{Bir}_r$-orbit $\mc O \sse \mf g_r$:

\begin{lemm}
	\label{lem:generalised_irregular_type}

	Suppose that $\bm X,\bm X' \in \mc O$ are such that each of $\tau_k(\bm X),\tau_k(\bm X')$ has semisimple commuting coefficients,
	for some $k \in \set{0,\dc,r}$.
	Then $\tau_k(\bm X) = \tau_k(\bm X')$.
\end{lemm}

\begin{proof}
	Postponed to~\ref{proof:lem_generalised_irregular_type}.
\end{proof}

\subsubsection{}
\label{sec:pure_orbit_invariant_1}

By Lem.~\ref{lem:generalised_irregular_type},
setting $\tau_s(\mc O) \ceqq \tau_s(\bm X) \in \mf g_s$ defines an invariant of strictly $s$-semisimple $\on{Bir}_r$-orbits,
where $\bm X \in \mc O$ is any element featuring the (unique) longest sequence of semisimple commuting coefficients.
When $s = r-1$,
this is essentially the \emph{irregular type} of an untwisted meromorphic $G$-connection germ,
on the other side of~\eqref{eq:deeper_duality}.
This is thus the `very good' case of~\cite[Def.~4]{boalch_2017_wild_character_varieties_meromorphic_hitchin_systems_and_dynkin_graphs},
and for $s < r-1$ it generalises into the twisted/ramified case.

At this stage one may then consider the subspace of elements $\bm X \in \mf g_r$ such that $\tau_s(\bm X) \in \mf g_s$ has semisimple commuting coefficients,
and conclude that any (strictly) $s$-semisimple $\on{Bir}_r$-orbit intersects it.

\subsection{Generalised (Birkhoff) residue orbit}
\label{sec:generalised_birkhoff_residue_orbit}

However,
more is true:
for an integer $k \in \set{0,\dc,r-1}$ consider the complementary truncation map
\begin{equation}
	\tau'_k \cl \mf g_r \lthra \mf g_{r-k},
	\qquad \bm X \lmt \varepsilon^{-k} \bigl( \bm X - \tau_k(\bm X) \bigr),
\end{equation}
i.e.,
$\tau'_k(\bm X) \ceqq X_k + \dm + X_{r-1} \varepsilon^{r-k-1}$ (cf.~\eqref{eq:element_deeper_lie_algebra}).
Choose again a $\on{Bir}_r$-orbit $\mc O \sse \mf g_r$:

\begin{lemm}
	\label{lem:commutation_lower_coefficients}

	Suppose that $\bm X \in \mc O$ is such that $\tau_k(\bm X)$ has semisimple commuting coefficients,
	for some $k \in \set{0,\dc,r}$.
	Then there exists $\bm X' \in \mc O$ such that:
	\begin{enumerate}
		\item $\tau_k(\bm X') = \tau_k(\bm X)$;

		\item and all coefficients of $\tau'_k(\bm X')$ commute with all coefficients of $\tau_k(\bm X')$.
	\end{enumerate}
\end{lemm}

(The latter will be abusively written as $\bigl[ \tau_k(\bm X'),\tau'_k(\bm X') \bigr] = 0$.)

\begin{proof}
	Postponed to~\ref{proof:lem_commutation_lower_coefficients}.
\end{proof}

\begin{rema}
	\label{rmk:diagonalising_algorithm}

	Adapting the proof~\ref{proof:lem_commutation_lower_coefficients} yields a recursive construction of the longest list of leading commuting semisimple coefficients of a $\on{Bir}_r$-orbit $\mc O \sse \mf g_r$,
	starting from any input element $\bm X^{(0)} = \sum_{i = 0}^{r-1} X_i^{(0)} \varepsilon^i \in \mc O$.

	Namely,
	if $X_0 = X_0^{(0)}$ is not semisimple then $\bm X^{(0)}$ and $\mc O$ are both strictly 0-semisimple.
	Otherwise,
	use Rmk.~\ref{rmk:semisimple_trick} to find \emph{unique} new coefficients
	\begin{equation}
		X_i^{(1)} = X^{(0)}_i - \ad_{X_0}(Y_i) \in \mf g^{X_0},
		\qquad Y_i \in \ad_{X_0}(\mf g),
		\quad i \in \set{1,\dc,r-1}.
	\end{equation}
	Now either the subleading coefficient $X_1^{(1)}$ is semisimple in $\mf g^{(1)} \ceqq \mf g^{X_0}$,
	or it is not.
	(This is equivalent to asking that $X_1^{(1)}$ be semisimple in $\mf g$,
	and \emph{independent} of whether $X_1^{(0)}$ was semisimple.)
	If it is not,
	the algorithm terminates;
	else,
	repeat inside $\mf g^{(1)}$,
	constructing a new coefficient $X_2^{(2)} \in \mf g^{(2)} \ceqq \smash{\bigl( \mf g^{(1)} \bigr)}^{\! X_1^{(1)}} = \mf g^{X_0^{(0)}} \cap \mf g^{X_1^{(1)}} \sse \mf g$;
	etc.
	In the end,
	there is a sequence of semisimple commuting coefficients $X_i^{(i)} \in \mf g^{(i)}$,
	stopping at some time $s \in \set{0,\dc,r}$.
	Now the proof~\ref{proof:lem_generalised_irregular_type} implies that $\mc O$ is \emph{strictly} $s$-semisimple.
\end{rema}

\subsubsection{}
\label{sec:pure_orbit_invariant_2}

The proof~\ref{proof:lem_generalised_irregular_type} also yields one last invariant for a strictly $s$-semisimple $\on{Bir}_r$-orbit $\mc O \sse \mf g_r$.
Namely,
if $\bm X \in \mc O$ is such that $\tau_s(\bm X) = \tau_s(\mc O)$,
and if the nonsemisimple coefficient $X_s \in \mf g$ commutes with $X_0,\dc,X_{s-1}$ (which one can assume by Lem.~\ref{lem:commutation_lower_coefficients}),
then this coefficient is also uniquely determined by $\mc O$.
This follows by inspecting the recursive step~\eqref{eq:recursive_step}.

Finally,
the lowest coefficients $X_{s+1},\dc,X_{r-1}$ are \emph{not} determined by $\mc O$,
but by Lem.~\ref{lem:commutation_lower_coefficients} they can be assumed to commute with $X_0,\dc,X_{s-1}$.
In turn,
the $\on{Bir}_{r-s}^L$-orbit of $\tau'_s(\bm X) \in \mf g_{r-s}$ \emph{is} determined by $\mc O$,
where
\begin{equation}
	\label{eq:common_centraliser}
	L
	= L_{X_0,\dc,X_{s-1}}
	\ceqq G^{X_0} \cap \dm \cap G^{X_{s-1}} \sse G,
\end{equation}
which is the common centraliser of $X_0,\dc,X_{s-1}$---%
in $G$.

If $s = r-1$ this simply reduces to fixing the coefficient $X_{r-1} \in \mf g$,
i.e.,
the (formal) residue of a meromorphic $G$-connection germ,
on the other side of~\eqref{eq:deeper_duality}.
Moreover,
if $s = r$ then $X_{r-1}$ itself is semisimple:
this corresponds to the secondary assumption of~\cite{biquard_boalch_2004_wild_nonabelian_hodge_theory_on_curves},
which is relevant to the second part of this text.

\section{Wild orbits:
  whole action}
\label{sec:wild_orbits_whole}

\subsection{}

Let us add on the action of the `constant' subgroup $G \sse G_r$.

\subsubsection{}

The $G$-action preserves the semisimplicity and mutual commutation relations of all elements of $\mf g$,
so that an element $\bm X \in \mf g_r$ is (strictly) $s$-semisimple if and only if its $G$-orbit is made of (strictly) $s$-semisimple elements.
Hence,
the filtration~\eqref{eq:semisimplicity_strata} is actually by $G_r$-invariant subspaces,
i.e.,
each piece is a disjoint union of $G_r$-orbits,
and the terminology of (strict) $s$-semisimple $G_r$-orbits makes sense.

(Hereafter,
in this section,
by `orbit' we mean `$G_r$-orbit'.)

\begin{rema}
	\label{rmk:generalised_irregular_classes_and_residue_orbits}

	Let $\mc O \sse \mf g_r$ be a strictly $s$-semisimple orbit.
	Acting by $\on{Bir}_r \sse G_r$,
	we get to an element $\bm X \in \mc O$ such that $\tau_s(\bm X) \in \mf g_s$ consists of commuting semisimple coefficients.
	These are \emph{not} uniquely determined,
	precisely because we can simultaneously conjugate all of them,
	so that the first invariant of $\mc O$ is the $G$-orbit of $\tau_s(\bm X)$.
	When $s = r-1$,
	this is essentially the ordinary \emph{irregular class} underlying the corresponding irregular type of \S~\ref{sec:pure_orbit_invariant_1} (cf.~\cite[Rmk.~10.6]{boalch_2014_geometry_and_braiding_of_stokes_data_fission_and_wild_character_varieties}).
	When $s < r-1$,
	it generalises in the twisted/ramified case.

	Finally,
	using the $\on{Bir}_r$-action one can also suppose that $\bigl[ \tau_s(\bm X),\tau'_s(\bm X) \bigr] = 0$:
	the complementary invariant of $\mc O$ is the orbit of $\tau'_s(\bm X)$ for the action of $L_{r-s} \simeq L \lts \on{Bir}^L_{r-s}$,
	which enlarges the $\on{Bir}^L_{r-s}$-orbit of \S~\ref{sec:pure_orbit_invariant_2}.
	When $s = r-1$,
	this corresponds to the (formal) residue $G$-orbit,
	under the action of the centraliser of the irregular type.
\end{rema}

\subsection{Wild marking}
\label{sec:wild_marking}

Any set of commuting semisimple elements of $\mf g$ lies in a Cartan subalgebra.
As all Cartan subalgebras are conjugated under the adjoint $G$-action,
we can fix one:
this breaks the symmetry down to the Weyl group,
and it suggests a definition for the marking of wild orbits.

Fix thus a Cartan subalgebra $\mf t \sse \mf g$,
and an integer $s \in \set{0,\dc,r}$.
Consider the topological subspaces
\begin{equation}
	\label{eq:wild_marking}
	\mf t_r^{(\geq s)} \ceqq \Set{ \bm X \in \mf g_r | \tau_s(\bm X) \in \mf t_s \text{ and } \bigl[ \tau_s(\bm X),\tau'_s(\bm X) \bigr] = 0 } \sse \mf g_r^{(\geq s)},
\end{equation}
and
\begin{equation}
	\label{eq:wild_marking_strict}
	\mf t_r^{(s)} \ceqq \Set{ \bm X \in \mf t_r^{(\geq s)} | X_s \text{ not semisimple } } \sse \mf g_r^{(s)}.
\end{equation}
The former can be identified with a closed subspace of $\mf t^s \ts \mf g^{r-s}$ cut out by the homogeneous quadratic equations,
in the usual topological identifications $\mf t_s \simeq \mf t^s$ and $\mf g_{r-s} \simeq \mf g^{r-s}$---%
selecting coefficients.

\begin{defi}[cf.~Def.~\ref{def:tame_marking}]
	\leavevmode

	\begin{enumerate}
		\item
		      A $\mf t$-\emph{marking} of an $s$-semisimple orbit $\mc O \sse \mf g_r$ (resp.,
		      a strict such) is the choice of a point $\bm X \in \mc O \cap \mf t_r^{(\geq s)}$ (resp.,
		      of a point $\bm X \in \mc O \cap \mf t_r^{(s)}$).

		\item
		      With this choice,
		      we say that $\mc O = \mc O_{\bm X}$ is a \emph{marked} $s$-semisimple orbit.
	\end{enumerate}
\end{defi}

(Again,
if $\mf t$ is fixed one just speaks of a `marking'.)

\section{Marked strata of wild orbits and centralisers}
\label{sec:wild_orbit_strata}

\subsection{}

Here we define strata of wild orbits corresponding to a choice of marking,
and then we determine the stabiliser of the marking in the $(r-1)$-semisimple case,
in a perfect generalisation of \S~\ref{sec:tame_strata} (cf.~\cite[\S~3]{yamakawa_2019_fundamental_two_forms_for_isomonodromic_deformations}).

\subsubsection{}

Note that~\eqref{eq:wild_marking}--\eqref{eq:wild_marking_strict} yield $\mf t_r^{(\geq 0)} = \mf g_r$ and $\mf t_r^{(0)} = \mf g_r^{(0)}$,
in which case the choice of marking is nil.
Therefore,
one can safely assume that $s > 0$ in this context.

By construction,
mapping $\mc O_{\bm X} \mt \bm X$ establishes a bijection between the set of marked $s$-semisimple orbits and $\mf t_r^{(\geq s)}$,
and analogously in the strict case.
Now invoke the stratification of Cor.~\ref{cor:truncated_wild_stratification}:
there are Levi subsystems $\phi_{X_i} = \Phi \cap \Set{X_i}^\perp$ for $i \in \Set{0,\dc,s-1}$,
and we consider the intersections
\begin{equation}
	\label{eq:levi_filtration_irregular_type}
	\phi_i = \phi_i(\bm X) \ceqq \bigcap_{j = 0}^{s-1-i} \phi_{X_j} = \Set{ \alpha \in \Phi | \Braket{ \alpha,X_0 } = \dm = \Braket{ \alpha,X_{s-1-i} } = 0 } \in \mc L_\Phi.
\end{equation}
(Indices are swapped,
in view of the duality~\eqref{eq:deeper_duality}.)

They yield a Levi filtration $\bm \phi = (\phi_0,\dc,\phi_{s-1},\Phi,\Phi,\dc)$,
i.e.,
an element of $\mc L^{(s)}_\Phi$,
so that $\mf l_{\phi_i} \sse \mf g$ is the common infinitesimal centraliser of $X_0,\dc,X_{s-1-i} \in \mf t$.
Explicitly
\begin{equation}
	\mf l_{\phi_i} = \mf l_{\phi_{X_0}} \cap \dm \cap\mf l_{\phi_{X_{s-1-i}}} = \bigcap_{i = 0}^{s-1-i} \Biggl( \mf t \ops \bops_{\phi_{X_i}} \mf g_\alpha \Biggr) = \mf t \ops \bops_{\phi_i} \mf g_\alpha,
\end{equation}
using~\eqref{eq:levi_algebra_from_levi_system}.
There is an associated stratum $\mf t^s_{\bm \phi} \sse \mf t^s$,
and a subspace
\begin{equation}
	\mf t_r^{(\geq s)}(\bm \phi) \ceqq \tau_s^{-1} \bigl( \mf t^s_{\bm \phi} \bigr) \sse \mf t_r^{(\geq s)}.
\end{equation}
Varying the Levi filtration then yields a disjoint union
\begin{equation}
	\label{eq:stratification_wild_marking}
	\mf t_r^{(\geq s)} = \bigcup_{\mc L^{(s)}_\Phi} \mf t_r^{(\geq s)}(\bm \phi),
\end{equation}
whose terms come with canonical (topological) splittings
\begin{equation}
	\label{eq:wild_marking_strata}
	\mf t_r^{(\geq s)}(\bm \phi) \lxra{\simeq} \mf t^s_{\bm \phi} \ts (\mf l_{\phi_0})_{r-s},
	\qquad \bm X \lmt \bigl( \tau_s(\bm X),\tau'_s(\bm X) \bigr).
\end{equation}
And analogously in the strictly $s$-semisimple case.

\begin{exem}
	The extremal pieces are
	\begin{equation}
		\mf t_r^{(\geq s)}(\pmb \vn) \simeq \bigl( \mf t_{\reg} \ts \mf t^{s-1} \bigr) \ts \mf t_{r-s} \simeq \mf t^r_{\pmb \vn},
	\end{equation}
	and
	\begin{equation}
		\mf t_r^{(\geq s)}(\bm \Phi) \simeq \Set{0} \ts \mf g_{r-s} \simeq \mf g_{r-s},
	\end{equation}
	in the notation of Cor.~\ref{cor:truncated_wild_stratification}.
\end{exem}

\begin{exem}
	Suppose that $\mf g = \mf{sl}_2(\mb C)$,
	and keep the notation of Exmp.~\ref{ex:sl_2_wild_truncated_stratification}.

	Using the standard 3-element basis of $\mf g$ yields
	\begin{equation}
		\mf t_r^{(\geq s)} \simeq \Set{ \bigl( \bm a,(\bm \alpha,\bm \beta,\bm \gamma) \bigr) \in \mb C^s \ts \bigl( \mb C^{r-s} \bigr)^3 | a_i\alpha_j = a_i\gamma_j = 0 } \sse \mb C^{3r - 2s},
	\end{equation}
	where we impose $2s(r-s)$ equations.
	But the description is greatly simplified over each stratum,
	finding identifications
	\begin{equation}
		\mf t_r^{(\geq s)}(\bm \phi_k) \simeq
		\begin{cases}
			\mb C^{3(r-s)},
			 & k = 0,
			\\
			\mb C^{k-1} \ts \mb C^* \ts \mb C^{r-s},
			 & k \in \Set{1,\dc,s}.
		\end{cases}
		\qedhere
	\end{equation}
\end{exem}

\subsection{Infinitesimal centralisers}

The point of the splitting~\eqref{eq:wild_marking_strata} is that it allows for a computation of the infinitesimal centraliser $\mf g_r^{\bm X} = \ker(\ad_{\bm X}) \sse \mf g_r$ of \emph{all} elements $\bm X \in \mf t_r^{(\geq s)}(\bm \phi)$,
in a uniform way.
Namely,
choose $\bm Y = \sum_{0 = 1}^{r-1} Y_i \varepsilon^i \in \mf g_r$;
then:

\begin{lemm}
	\label{lem:wild_infinitesimal_centraliser}

	If $\bm Y \in \mf g_r^{\bm X}$,
	then:
	\begin{enumerate}
		\item one has $Y_i \in \mf l_{\phi_i}$ for $i \in \set{0,\dc,s-1}$;

		\item and $Y_0 \in \mf g^{X_s}$.
	\end{enumerate}
\end{lemm}

\begin{proof}
	Postponed to~\ref{proof:lem_wild_infinitesimal_centraliser}.
\end{proof}

\subsection{Centralisers}

Analogously,
we consider the centraliser $G_r^{\bm X} \sse G_r$,
for any $\bm X \in \mf \mf t_r^{(\geq s)}$,
which controls the (complex) geometry of the corresponding marked orbit.
Choose thus $g \in G$ and $\bm Y \in \mf{bir}_r$,
and consider the group element $\bm g \ceqq g \lts e^{\bm Y} \in G_r$.
Then:

\begin{theo}[cf.~Lem.~3.3 of~\cite{yamakawa_2019_fundamental_two_forms_for_isomonodromic_deformations}]
	\label{thm:wild_centraliser}

	If $\bm g \in G_r^{\bm X}$,
	then (in the notation of~\eqref{eq:common_centraliser}):
	\begin{enumerate}
		\item one has $g \in L \cap G^{X_s}$;

		\item and $\bm Y \in \mf g_r^{\bm X}$.
	\end{enumerate}
\end{theo}

\begin{proof}
	The adjoint action reads
	\begin{equation}
		\label{eq:deeper_adjoint_action}
		\Ad_{\bm g}(\bm X) = \Ad_g e^{\ad_{\bm Y}} (\bm X) = \sum_{i = 0}^{r-1} \Ad_g(X'_i) \varepsilon^i,
	\end{equation}
	where $\bm X' = \sum_{i = 0}^{r-1} X'_i \varepsilon^i \eqqc e^{\ad_{\bm Y}} (\bm X) \sse \mf g_r$.
	Thus,
	we impose that
	\begin{equation}
		\label{eq:equation_centraliser}
		X_l = \Ad_g(X_l') \in \mf g,
		\qquad l \in \Set{0,\dc,r-1}.
	\end{equation}
	In particular $X'_0 = X_0$,
	since the $\on{Bir}_r$-action fixes the leading coefficient,
	so that $g \in G^{X_0}$.

	Now choose $l \in \Set{0,\dc,s-1}$,
	and suppose inductively that:
	i) $g \in G^{X_0} \cap \dm \cap G^{X_l}$;
	and ii) $[Y_i,X_j] = 0$ if $i+j \leq l$.
	One computes
	\begin{equation}
		X'_{l+1} = X_{l+1} + [Y_1,X_l] + \dm + [Y_{l+1},X_0] \in \mf g,
	\end{equation}
	as all nested Lie brackets vanish by the recursive hypothesis.
	Thus,
	using that $\Ad_g \in \Aut(\mf g)$ fixes $X_0,\dc,X_k$,
	the condition~\eqref{eq:equation_centraliser} becomes
	\begin{equation}
		X_{l+1} = \Ad_g(X_{l+1}) + [Y'_1,X_l] + \dm + [Y'_{l+1},X_0],
		\qquad Y'_i \ceqq \Ad_g(Y_i) \in \mf g.
	\end{equation}
	Now one can argue as in the proof~\ref{proof:lem_wild_infinitesimal_centraliser};
	rewrite the above as
	\begin{equation}
		\ad_{X_0}(Y'_{l+1}) = (\Ad_g - \Id_{\mf g})(X_{l+1}) + [Y'_1,X_l] + \dm + [Y'_l,X_1] \in \mf g,
	\end{equation}
	and note that by hypothesis the right-hand side lies in $\mf g^{X_0} = \mf g^{\Ad_g(X_0)}$.
	(When $l = s-1$ we use that $X_s$ commutes with $X_0,\dc,X_{s-1}$,
	but we do \emph{not} need that $X_s \in \mf t$.)
	Rmk.~\ref{rmk:semisimple_trick} implies that $[Y_{l+1},X_0] = 0$,
	and one can repeat the same argument for the semisimple elements $X_1,\dc,X_l \in \mf t$.
	In the end
	\begin{equation}
		[X_0,Y_{l+1}] = \dm = [X_l,Y_1] = (\Ad_g - \Id_{\mf g})(X_{l+1}) = 0 \in \mf g.
		\qedhere
	\end{equation}
\end{proof}

\subsection{}

Let us spell out all consequences for the marked $(r-1)$-semisimple orbits $\mc O_{\bm X} \sse \mf g_r$.
By definition $\bm X \in \mf t_r^{(\geq r-1)}$,
i.e.,
$X_0,\dc,X_{r-2} \in \mf t$ and $X_{r-1} \in \mf l_{\phi_0}$ (cf.~\eqref{eq:levi_filtration_irregular_type}).

Denote by $\mf g^{\bm X} \ceqq \mf l_{\phi_0} \cap \mf g^{X_{r-1}} = \mf g^{X_0} \cap \dm \cap \mf g^{X_{r-1}}$ the infinitesimal centraliser of (all coefficients of) $\bm X$ in $\mf g$,
and define
\begin{equation}
	\mf l_{\bm \phi} \ceqq \prod_{i = 0}^{r-2} \bigl( \mf l_{\phi_i} \cdot \varepsilon^{i+1} \bigr) \sse \mf{bir}_r.
\end{equation}

\begin{coro}
	\label{cor:very_good_infinitesimal_centraliser}

	One has $\mf g_r^{\bm X} = \mf g^{\bm X} \lts \mf l_{\bm \phi} \sse \mf g_r$.
\end{coro}

\begin{proof}
	This follows from Lem.~\ref{lem:wild_infinitesimal_centraliser},
	noting that $\mf g^{\bm X}$ acts on $\mf{bir}_r$ by preserving the Lie subalgebra $\mf l_{\bm \phi}$---%
	in the $\varepsilon$-graded adjoint action.
\end{proof}

\begin{rema}
	\label{rmk:fission}

	The nondecreasing sequence of reductive Lie subgroups
	\begin{equation}
		\label{eq:fission_groups}
		L_{\phi_0} \sse \dm \sse L_{\phi_{r-2}} \sse G,
	\end{equation}
	integrating the corresponding sequence of Lie subalgebras $\mf l_{\phi_i} \sse \mf g$,
	constitutes the \emph{fission} (of $G$) induced by the associated irregular type,
	on the other side of~\eqref{eq:deeper_duality}:
	cf.~the example in~\cite{boalch_2009_through_the_analytic_halo_fission_via_irregular_singularities},
	as well as~\cite[Eq.~(33)]{boalch_2014_geometry_and_braiding_of_stokes_data_fission_and_wild_character_varieties},
	and the decomposition in~\cite[Eqq.~(2.2)--(2.3)]{biquard_boalch_2004_wild_nonabelian_hodge_theory_on_curves}---%
	and see also~\cite{doucot_rembado_tamiozzo_local_wild_mapping_class_groups_and_cabled_braids,doucot_rembado_2025_topology_of_irregular_isomonodromy_times_on_a_fixed_pointed_curve,doucot_rembado_yamakawa_twisted_g_local_wild_mapping_class_groups}.
	These subgroups are \emph{connected},
	as each term of~\eqref{eq:fission_groups} is the centraliser of a semisimple element of $\mf g$,
	taken inside the connected reductive Lie group which follows it in the sequence.
\end{rema}

\subsubsection{}

Denote by $G^{\bm X} \ceqq \bigcap_{i = 0}^{r-1} G^{X_i} \sse G$ the centraliser of (all coefficients of) $\bm X$ in $G$.
Consider then the subspace $L_{\bm \phi} \ceqq \exp(\mf l_{\bm \phi}) \sse \on{Bir}_r$,
which is a subgroup by the Baker--Campbell--Hausdorff formula.

\begin{coro}[cf.~Cor.~\ref{cor:very_good_infinitesimal_centraliser}]
	\label{cor:very_good_centraliser}

	One has $G_r^{\bm X} = G^{\bm X} \lts L_{\bm \phi} \sse G_r$.
\end{coro}

\begin{proof}
	This follows from Thm.~\ref{thm:wild_centraliser},
	noting that $G^{\bm X}$ acts on $\on{Bir}_r$ by preserving the subgroup $L_{\bm \phi}$,
	in view of the identities
	\begin{equation}
		g e^{\bm Y} g^{-1} = e^{\Ad_g(\bm Y)} \in \on{Bir}_r,
		\qquad g \in G,
		\quad \bm Y \in \mf{bir}_r.
		\qedhere
	\end{equation}
\end{proof}

\subsection{Semisimple residue}

Let us finally restrict to orbits such that the above Levi filtration can be extended down to the last coefficient:
suppose that $X_{r-1} \in \mf t$.
We thus get to the marked (strictly) $r$-semisimple orbits $\mc O_{\bm X} \sse \mf g_r$,
where now $\bm X \in \mf t_r \sse \mf g_r$.
(Again,
this is as in the secondary assumption of~\cite{biquard_boalch_2004_wild_nonabelian_hodge_theory_on_curves},
but beyond matrix Lie groups.)

\begin{rema}
	\label{rmk:complete_incomplete_fission}

	There are two cases to distinguish,
	depending on the irregular type,
	i.e.,
	on the coefficients $X_0,\dc,X_{r-2}$:
	\begin{enumerate}
		\item
		      there is \emph{complete/toral} fission,
		      i.e.,
		      one has $\bigcap_{i = 0}^{r-2} \phi_{X_i} = \vn$;

		\item
		      or conversely fission is incomplete,
		      i.e.,
		      one has $\bigcap_{i = 0}^{r-2} \phi_{X_i} \neq \vn$.\fn{
			      Overall,
			      the classes of untwisted wild orbits that we consider are:
			      \begin{equation}
				      \set{\text{Generic}} \ssne \set{\text{Complete fission}} \ssne \set{r\text{-semisimple}} \ssne \set{(r-1)\text{-semisimple}} = \set{\text{untwisted}}.
			      \end{equation}}
	\end{enumerate}
	In the former case assuming that $X_{r-1} \in \mf t$ is w.l.o.g.~(cf.~Lem.~\ref{lem:commutation_lower_coefficients}).
	In the latter,
	instead,
	one imposes more:
	namely,
	that the coefficient $X^{(r-1)}_{r-1}$ constructed at the end of the algorithm described in Rmk.~\ref{rmk:diagonalising_algorithm} is semisimple.
\end{rema}

\subsubsection{}
\label{sec:about_main_theo_1}

So the space of marked $r$-semisimple orbits is in natural bijection with the topological space $\mf t_r \simeq \mf t^r$,
via $\mc O_{\bm X} \mt \bm X$.
Furthermore,
in this case $\mf g^{\bm X} = \mf l_{\phi_0}$,
where now $\phi_0 = \bigcap_{i = 0}^{r-1} \phi_{X_i} \in \mc L_\Phi$ extends the Levi filtration above,
as a particular case of Cor.~\ref{cor:very_good_infinitesimal_centraliser}.
In turn,
this corresponds to a connected Lie subgroup $L_{\phi_0} = G^{\bm X}$,
as a particular case of Cor.~\ref{cor:very_good_centraliser}.
It follows that $G_r^{\bm X} \sse G_r$ is \emph{connected},
and it is thus determined by its Lie algebra $\mf g_r^{\bm X} \sse \mf g_r$.

Hence,
in brief,
to a marked $r$-semisimple orbit one can associate a (depth-bounded) Levi filtration,
and in turn this controls its complex geometry:

\begin{theo}[cf.~\S~\ref{sec:pure_isomorphism_tame}]
	\label{thm:pure_isomorphism_wild}

	If two marked $r$-semisimple orbits $\mc O_{\bm X_1},\mc O_{\bm X_2}$ lie in one and the same stratum of $\mf t^r$,
	then there is a \emph{canonical} isomorphism $\mc O_{\bm X_1} \simeq \mc O_{\bm X_2}$ of complex homogeneous $G_r$-manifolds.
\end{theo}

\begin{proof}
	Both orbits are isomorphic to a complex homogeneous $G_r$-manifold of the form $G_r \bs L_{\bm \phi}$,
	where $L_{\bm \phi} = L_{\phi_0} \lts \exp (\mf l_{\bm \phi'}) \sse G \lts \on{Bir}_r$,
	and in turn $\bm \phi' = (\phi_i)_{i \geq 1}$ is the subsequence obtained from $\bm \phi$ by removing the one relevant term.
\end{proof}

\begin{rema}
	\label{rmk:affine_setting}

	Again,
	throughout this text we work in the complex-analytic setting.
	However,
	one can prove~\cite{chaffe_rembado_yamakawa_genus_zero_wild_quantum_de_rham_spaces} that the $r$-semisimple orbits $\mc O \sse \mf g_r^{\dual}$ are Zariski-closed affine (algebraic) subvarieties of $\mf g_r^{\dual} = \Spec \Sym(\mf g_r)$,
	generalising the analogous statement for tame semisimple orbits in $\mf g^{\dual}$.
	Therefore,
	one could rephrase these results in the complex-algebraic category.
\end{rema}

\section{Unmarked strata of wild orbits}
\label{sec:intrinsic_wild_orbit_space}

\subsection{}

Finally,
we study the change of marking of (strictly) $s$-semisimple $G_r$-orbits.
This leads to their intrinsic parameter space,
extending \S~\ref{sec:intrinsic_tame_orbit_space}.
(Assume again that $s > 0$ to avoid discussing trivial cases;
and write `orbit' for `$G_r$-orbit'.)

Choose now:
(i) $\bm X \in \mf t_r^{(\geq s)}$;
(ii) $g \in G$;
and (iii) $\bm Y = \sum_{i = 1}^{r-1} Y_i \varepsilon^i \in \mf{bir}_r$.
Moreover,
consider the group element $\bm g = g \lts e^{\bm Y} \in G_r$.
The following statement is key:

\begin{lemm}
	\label{lem:wild_marking_change}

	If $\Ad_{\bm g}(\bm X) \eqqc \sum_{i = 0}^{r-1} X'_i \varepsilon^i$ lies in $\mf t_r^{(\geq s)}$,
	then:
	\begin{enumerate}
		\item one has $\Ad_g(X_i) = X'_i$ for $i \in \set{0,\dc,s}$,
		      in the notation of~\eqref{eq:element_deeper_lie_algebra};

		\item and one has $Y_{i+1} \in \mf l_{\phi_i}$ for $i \in \set{0,\dc,s-1}$,
		      in the notation of~\eqref{eq:levi_filtration_irregular_type}.
	\end{enumerate}
\end{lemm}

\begin{proof}
	Clearly $X_0' = \Ad_g(X_0)$,
	since the $\on{Bir}_r$-action cannot modify the leading coefficient.
	Then we prove inductively that:
	i) $[Y_i,X_j] = 0$ for $i+j \leq s$;
	and ii) $X'_i = \Ad_g(X_i)$,
	for $i \in \Set{0,\dc,s}$.
	Namely,
	if the statement holds for $l \in \Set{0,\dc,s-1}$ then
	\begin{equation}
		\label{eq:induction_wild_marking_change}
		X'_{l+1} - \Ad_g(X_{l+1}) = [Y'_1,X'_l] + \dm + [Y'_{l+1},X'_0],
		\qquad Y'_i \ceqq \Ad_g(Y_i) \in \mf g,
	\end{equation}
	using that $\Ad_g \in \Aut(\mf g)$.
	(Again,
	with no nested Lie brackets.)

	Now the right-hand side lies in the vector subspace
	\begin{equation}
		\sum_{i = 0}^l \ad_{X'_i}(\mf g) = \sum_{i = 0}^l \Biggl(\bops_{\Phi \sm \phi_{X'_i}} \mf g_\alpha \Biggr) = \bops_{\Phi \sm \phi'_{s-1-l}} \mf g_\alpha \sse \mf g,
	\end{equation}
	where as usual $\phi_{X'_i} = \Phi \cap \set{X'_i}^\perp$ (cf.~\eqref{eq:levi_subsystem_semisimple_element}),
	and in turn $\phi'_j = \bigcap_{i = 0}^{s-1-j} \phi_{X'_i} \sse \Phi$---%
	for $i,j \in \Set{0,\dc,s-1}$.
	By hypothesis,
	instead,
	the left-hand side of~\eqref{eq:induction_wild_marking_change} lies in the Lie subalgebra
	\begin{equation}
		\mf l_{\phi'_{s-1-l}} = \mf t \ops \bops_{\phi'_{s-1-l}} \mf g_\alpha,
	\end{equation}
	since both $X'_{l+1}$ and $\Ad_g(X_{l+1})$ commute with $X_0' = \Ad_g(X_0),\dc,X_l' = \Ad_g(X_l) \in \mf t$.
	(The left-hand side even lies in $\mf t$ for $l \leq s-2$;
	when $l = s-1$,
	we do \emph{not} use $X_s,X'_s \in \mf t$.)
	Thus,
	both sides of~\eqref{eq:induction_wild_marking_change} vanish,
	whence $X'_{l+1} = \Ad_g(X_{l+1})$,
	and
	\begin{equation}
		\ad_{X'_0}(Y'_{l+1}) = [Y'_1,X'_l] + \dm + [Y'_l,X'_1] \in \mf g.
	\end{equation}
	Now use Rmk.~\ref{rmk:semisimple_trick},
	noting that by hypothesis the right-hand side lies in $\mf g^{X'_0} \sse \mf g$,
	so that in particular
	\begin{equation}
		[Y'_{l+1},X'_0] = \Ad_g \bigl( [Y_{l+1},X_0] \bigr) = 0,
	\end{equation}
	and therefore $[Y_{l+1},X_0] = 0$.
	Iterating for the semisimple elements $X'_1,\dc,X'_{l-1} \in \mf t$ yields
	\begin{equation}
		\Ad_g \bigl( [Y_1,X_l] \bigr) = \dm = \Ad_g \bigl( [Y_l,X_1] \bigr) = 0 \in \mf g.
		\qedhere
	\end{equation}
\end{proof}

\subsubsection{}

In brief,
Lem.~\ref{lem:wild_marking_change} states that it is only the diagonal $G$-action which modifies the coefficients $X_0,\dc,X_s \in \mf g$,
while the exponentiated action of $\mf{bir}_r$ is trivial there.
But the point of choosing a Cartan subalgebra is that we can further break this down when acting on the semisimple part $\tau_s(\bm X) \in \mf t_s$:

\begin{lemm}
	\label{lem:down_to_weyl}

	If $\bm X \in \mf t_r^{(\geq s)}$,
	then
	\begin{equation}
		\tau_s \bigl( \mc O_{\bm X} \cap \mf t_r^{(\geq s)} \bigr) = \Set{ \sum_{i = 0}^{s-1} \Ad_g(X_i) \varepsilon^i \in \mf g_s | g \in N_G(\mf t) }.
	\end{equation}
\end{lemm}

\begin{proof}
	Postponed to~\ref{proof:lem_down_to_weyl}.
\end{proof}

\subsection{}

Lem.~\ref{lem:down_to_weyl} yields an action of the setwise stabiliser $N_G(\mf t) \sse G$ of the Cartan subalgebra on the commuting semisimple coefficients $X_0,\dc,X_{s-1} \in \mf t$.
This comes with a surjection $N_G(\mf t) \thra W$ onto the Weyl group,
whose kernel is precisely the pointwise stabiliser of $\mf t$,
i.e.,
the elements acting trivially on $\tau_s(\bm X) \in \mf t_s$.

For $s = r-1$ one thus recovers the standard fact that untwisted irregular classes are $W$-orbits of irregular types.

\begin{rema}
	The proof~\ref{proof:lem_down_to_weyl} further shows that the $G$-action on the lowest coefficient can be broken down to
	\begin{equation}
		X_{r-1} \lmt X'_{r-1} = \Ad_{g_1 g_2}(X_{r-1}),
		\qquad  g_2 \in N_G(\mf t),
		\quad g_1 \in L_{X_0',\dc,X'_{r-2}},
	\end{equation}
	in the notation of~\eqref{eq:common_centraliser},
	writing $X'_i \ceqq \Ad_{g_2}(X_i) \in \mf t$ for $i \in \set{0,\dc,r-2}$.

	(Presumably one can build on this to treat nonsemisimple formal residues.)
\end{rema}

\subsection{}

Suppose finally that $s = r$.
Then,
in a perfect generalisation of the tame case:

\begin{prop}
	\label{prop:intrinsic_excellent_orbit_space}

	The map $\mc O \mt \mc O \cap \mf t^{(r)}_r$ establishes a \emph{bijection} of the set of (unmarked) $r$-semisimple orbits onto the topological quotient $\mb A_r \ceqq \mf t^r \bs W$,
	and the pieces~\eqref{eq:stratification_wild_marking} correspond to the strata of~\eqref{eq:wild_truncated_strata}.
\end{prop}

\begin{proof}
	This follows from Lem.~\ref{lem:wild_marking_change} +~\ref{lem:down_to_weyl},
	in the topological identifications
	\begin{equation}
		\mf t_r^{(r)} = \mf t_r^{(\geq r)} = \mf t_r \simeq \mf t^r.
		\qedhere
	\end{equation}
\end{proof}

\subsubsection{}

Hence,
$r$-semisimple orbits intersect the deeper Cartan subalgebra $\mf t_r \sse \mf g_r$ in a $W$-orbit,
and the quotient stratum $\smash{\bigl(\mf t^r_{\bm \phi}\bigr)}^{\! W} \sse \mb A_r$ corresponds to the space of orbits intersecting $\mf t^r_{\bm \phi} \sse \mf t^r$.

Finally,
adapting the arguments of \S~\ref{sec:intrinsic_tame_orbit_space} one gets the following clean statement.
Introduce as above the pointwise/setwise stabiliser $W_{\mf t^r_{\bm \phi}} \sse N_W \bigl( \mf t^r_{\bm \phi} \bigr) \sse W$,
and the quotient $\Out \bigl( \mf t^r_{\bm \phi} \bigr) \ceqq N_W \bigl( \mf t^r_{\bm \phi} \bigr) \bs W_{\mf t^r_{\bm \phi}}$;
then:

\begin{theo}[cf.~\S~\ref{sec:nonpure_isomorphism_tame}]
	\label{thm:nonpure_isomorphism_wild}

	If two $r$-semisimple orbits lie in one and the same stratum of $\mb A_r$,
	then they are \emph{isomorphic} as complex homogeneous $G_r$-manifolds.
	Furthermore,
	if the orbit $\mc O \sse \mf g_r$ intersects $\mf t^r_{\bm \phi}$,
	then the set of identifications $\mc O \simeq G_r \bs L_{\bm \phi}$ obtained from a marking is an $\Out(\mf t^r_{\bm \phi})$-torsor.
\end{theo}

\begin{proof}[Proof omitted]
\end{proof}

\section*{Interlude}

As mentioned in \S~\ref{sec:wild_orbits},
all the material of \S~\ref{sec:birkhoff_action}--\ref{sec:intrinsic_wild_orbit_space} (canonically) transfers to \emph{coadjoint} orbits $\mc O \sse \mf g_r^{\dual}$,
in view of the $G_r$-equivariant duality~\eqref{eq:deeper_duality}.
Hence,
there are holomorphic-symplectic marked/unmarked $s$-semisimple coadjoint $G_r$-orbits,
as well as strata thereof,
for all $s \in \set{0,\dc,r}$.
In the second part of the paper we introduce new affine-Lie-algebra modules,
with a view towards their deformation quantisation when $s = r$---%
and the quantisation of wild de Rham spaces (cf.~\cite{chaffe_rembado_yamakawa_genus_zero_wild_quantum_de_rham_spaces}),
and the construction of irregular conformal blocks (cf.~\S~\ref{sec:irregular_blocks});
etc.

\section{Lie-algebraic setup (II)}
\label{sec:setup_2}

\subsection{}

We review the additional objects/notations needed for the second part of the paper:
please refer,
e.g.,
to~\cite{kac_1990_infinite_dimensional_lie_algebras,moody_pianzola_1995_lie_algebras_with_triangular_decomposition,dixmier_1996_algebres_enveloppantes,humphreys_2008_representations_of_semisimple_lie_algebras_in_the_bgg_category_o,etingof_frenkel_kirillov_1998_lectures_on_representation_theory_and_knizhnik_zamolodchikov_equations} for more details.
(Experts might want to skip to \S~\ref{sec:generic_singularity_modules}.)

\subsection{Affine Lie algebras}

Keep the notation of \S~\ref{sec:relation_meromorphic_connections}.
A slight modification of~\eqref{eq:affine_pairing} yields a 2-cocycle $\mf g (\!(z)\!) \wdg \mf g(\!(z)\!) \to \mb C$,
whence a nontrivial central extension
\begin{equation}
	0 \lra \mb C K \lra \wh{\mf g} \lra \mf g (\!(z)\!) \lra 0.
\end{equation}
This is the \emph{affine Lie algebra} associated to the quadratic Lie algebra $\bigl( \mf g,
	(\cdot \mid \cdot) \bigr)$,
with Lie bracket determined by
\begin{equation}
	[ X \ots f,Y \ots f']_{\wh{\mf g}} = [X,Y] ff' + (X \mid Y) \Ress_{z = 0}(f' \dif f) K,
	\qquad X,Y \in \mf g,
	\quad f,f' \in \mb C(\!(z)\!).
\end{equation}

Write as customary $\wh{\mf t} \ceqq \mf t \ops \mb CK \sse \wh{\mf g}$.

\subsection{Parabolic subsets/subalgebras}

Throughout this second part we use \emph{parabolic} subsets $\psi \sse \Phi$,
i.e.,
collections of roots such that:
\begin{enumerate}
	\item $(\psi + \psi) \cap \Phi = \psi$;\fn{
		      I.e.,
		      if $\gamma \ceqq \alpha + \beta$ lies in $\Phi$,
		      for some $\alpha,\beta \in \psi$,
		      then actually $\gamma \in \psi$.}

	\item and $\psi \cup (-\psi) = \Phi$.
\end{enumerate}
If $S \sse \Phi$ is any subset,
write
\begin{equation}
	\Phi_S \ceqq \spann_{\mb Q}(S) \cap \Phi,
	\qquad  \Phi_S^\pm \ceqq \spann_{\pm \mb Q_{\geq 0}}(S)\cap \Phi.
\end{equation}
If $\psi \sse \Phi$ is parabolic,
it follows~\cite[Ch.~VI,
	\S~1.7,
	Prop.~20]{bourbaki_1968_elements_de_mathematiques_fascicule_xxxvii_chapitres_iv_v_vi} that there exists a base $\Delta \sse \Phi$ of simple roots,
and a subset $\Sigma \sse \Delta$,
such that $\psi = \Phi_{\Delta}^+ \cup \Phi_{\Sigma}^-$.
(Here $\Phi_{\Delta}^+$ is the system of positive roots associated with $\Delta$.)
Then there is a \emph{parabolic} subalgebra
\begin{equation}
	\label{eq:parabolic_subalgebra}
	\mf p_{\psi} \ceqq \mf t \ops \bops_{\psi} \mf g_\alpha \sse \mf g,
\end{equation}
containing the \emph{Borel} subalgebra
\begin{equation}
	\label{eq:borel_subalgebra}
	\mf b^+_{\Delta} \ceqq \mf t \ops \bops_{\Phi_{\Delta}^+} \mf g_\alpha.
\end{equation}
All parabolic subalgebras containing $\mf b_{\Delta}^+$ (i.e.,
the \emph{standard} parabolic subalgebras of $(\mf g,\mf b_{\Delta}^+)$) arise in this way.
It also follows~\cite[Ch.~VI,
	\S~1.7,
	Cor.~6]{bourbaki_1968_elements_de_mathematiques_fascicule_xxxvii_chapitres_iv_v_vi} that
\begin{equation}
	\label{eq:levi_factor_roots}
	\phi = \on{Lf}(\psi) \ceqq \psi \cap (-\psi) = \Phi_{\Sigma}^- \cup \Phi_{\Sigma}^+ = \spann_{\mb C}(\Sigma) \cap \Phi,
\end{equation}
which is by definition a Levi subsystem:
it is the \emph{Levi factor} of $\psi$,
and the corresponding Lie subalgebra $\mf l_{\phi} \sse \mf g$ is the \emph{Levi factor} of~\eqref{eq:parabolic_subalgebra}.
Then there is a splitting
\begin{equation}
	\label{eq:parabolic_subalgebra_splitting}
	\mf p_{\psi} = \mf l_{\phi} \ops \mf u_{\psi},
	\qquad \mf u_{\psi} \ceqq \mf{nil}(\mf p_{\psi}) = \bops_{\nu} \mf g_\alpha,
\end{equation}
using~\eqref{eq:levi_algebra_from_levi_system},
and writing $\nu \ceqq \psi \sm \phi$;
this involves the \emph{nilradical} of the parabolic subalgebra.
The opposite parabolic subalgebra is $\mf p^-_{\psi} \ceqq \mf p_{-\psi}$,
so that $\mf l_{\phi} = \mf p^+_{\psi} \cap \mf p^-_{\psi}$ (with $\mf p^+_{\psi} \ceqq \mf p_{\psi}$);
the opposite nilradical is $\mf u^-_{\psi} \ceqq \mf u_{-\psi}$.

Denote now by $\mc P_\Phi \sse 2^\Phi$ the collection of parabolic subsets of $\Phi$,
and equip it with the \emph{anti-inclusion} order.
Hence,
$\Phi \in \mc P_\Phi$ is the least element,
while the maximal elements are
systems of positive roots.
Taking Levi factors (as in~\eqref{eq:levi_factor_roots}) defines a $W$-equivariant order-preserving surjection $\on{Lf} \cl \mc P_\Phi \thra \mc L_\Phi$ onto the Levi poset.
In general,
there is \emph{no} canonical global section $\mc L_\Phi \hra \mc P_\Phi$,
even if a base of simple roots is chosen (cf.~Rmk.~\ref{rmk:levi_systems_simple_roots} +~Exmp.~\ref{ex:rank_3_parabolic_subsets}):
in brief such choice is necessary to pass from the semiclassical first part of the paper to the quantum second part;
geometrically,
it corresponds to polarising the KKS structure on semisimple orbits (cf.~\S~\ref{sec:polarisations}).
In any event,
the setwise stabiliser $W^{\phi} = N_W(\mf t_{\phi}) \sse W$ acts on the fibre $\on{Lf}^{-1}(\phi) \sse \mc P_\Phi$,
for $\phi \in \mc L_\Phi$.
Furthermore,
the negation $\psi \mt -\psi$ commutes with the $W$-action,
and so it induces an involution of the quotient poset $\mc P_\Phi \bs W$ which preserves the fibres of the `reduced' surjection $\mc P_\Phi \bs W \thra \mc L_\Phi \bs W$.

\begin{exem}
	\label{ex:rank_3_parabolic_subsets}

	Consider again $\mf g = \mf{gl}_3(\mb C)$ (cf.~\S~\ref{sec:rank_3_example}).
	If $\psi \in \mc P_\Phi \sm \Set{\Phi}$ then a priori $\abs{\psi} \in \set{3,4,5}$,
	and one can show that $\abs{\psi} = 5$ is impossible.
	There are then six parabolic subsets with $\abs{\psi} = 4$,
	obtained by (suitably) removing one element from two of the three subsets $\set{\alpha_{ij},\alpha_{ji}} \sse \Phi$,
	for $i \neq j \in \set{1,2,3}$.
	The resulting map onto the rank-1 Levi subsystem $\phi_{\set{i,j}} \in \mc L_\Phi$ is 2:1,
	and the action of $W^{\phi_{\set{i,j}}} \simeq \mf S_2$ on the fibre is trivial.
	Finally,
	$\abs{\psi} = 3$ is obtained by (suitably) removing one element from \emph{each} of the three subsets $\set{\alpha_{ij},\alpha_{ji}}$.
	The map onto $\vn \in \mc L_\Phi$ is 6:1,
	and the set of six choices is a torsor for $W \simeq \mf S_3$,
	corresponding to the 6-dimensional Borel subalgebras containing the diagonal matrices:
	each is contained in two distinct 7-dimensional (proper) parabolic subalgebras---%
	and conversely,
	cf.~the Hasse diagram of Fig.~\ref{fig:hasse_parabolic}.

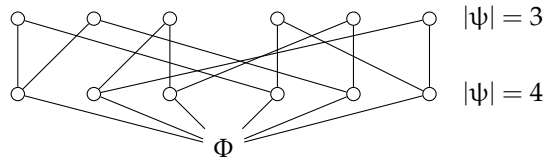
	\begin{figure}
		\centering
		\begin{tikzpicture}

			\node [draw,circle,inner sep=0pt,minimum size=5pt] (322131) [below left of =top] {};
			\node [draw,circle,inner sep=0pt,minimum size=5pt] (211323) [left of =322131] {};
			\node [draw,circle,inner sep=0pt,minimum size=5pt] (122313) [left of =211323] {};
			\node [draw,circle,inner sep=0pt,minimum size=5pt] (133212) [below right of =top] {};
			\node [draw,circle,inner sep=0pt,minimum size=5pt] (233121) [right of =133212] {};
			\node [draw,circle,inner sep=0pt,minimum size=5pt] (311232) [right of =233121] {};

			\node [draw,circle,inner sep=0pt,minimum size=5pt] (12212313) [below of =122313] {};
			\node [draw,circle,inner sep=0pt,minimum size=5pt] (12213231) [below of =211323] {};
			\node [draw,circle,inner sep=0pt,minimum size=5pt] (23322131) [below of =322131] {};
			\node [draw,circle,inner sep=0pt,minimum size=5pt] (23321213) [below of =133212] {};
			\node [draw,circle,inner sep=0pt,minimum size=5pt] (13312321) [below of =233121] {};
			\node [draw,circle,inner sep=0pt,minimum size=5pt] (13313212) [below of =311232] {};

			\node (bottom) [below right of =23322131] {$\Phi$};

			\node () [right of =311232] {$\abs \psi = 3$};
			\node () [right of =13313212] {$\abs \psi = 4$};

			\draw (122313) -- (12212313);
			\draw (122313) -- (23321213);

			\draw (211323) -- (12212313);
			\draw (211323) -- (13312321);

			\draw (322131) -- (12213231);
			\draw (322131) -- (23322131);

			\draw (133212) -- (23321213);
			\draw (133212) -- (13313212);

			\draw (233121) -- (23322131);
			\draw (233121) -- (13312321);

			\draw (311232) -- (12213231);
			\draw (311232) -- (13313212);

			\draw (12212313) -- (bottom);
			\draw (12213231) -- (bottom);
			\draw (23322131) -- (bottom);
			\draw (23321213) -- (bottom);
			\draw (13312321) -- (bottom);
			\draw (13313212) -- (bottom);
		\end{tikzpicture}
		\caption{Example poset of parabolic subsets of roots}
		\label{fig:hasse_parabolic}
	\end{figure}

	On the whole,
	one has $\abs{\mc P_\Phi} = 13$.
	Even if a base $\Delta \sse \Phi$ is given,
	there is no consistent way to choose a parabolic subset lifting each rank-1 Levi subsystem,
	i.e.,
	no consistent choice of 2-dimensional nilradical to add to the corresponding Levi factor.
	As per the Weyl quotient,
	the orbits of the 7-dimensional parabolic subalgebras have three elements,
	so that in the quotient poset there are two (opposite) equivalence classes mapping onto $\ol{\phi} \in \mc L_\Phi \bs W$ (cf.~the Hasse diagram of Fig.~\ref{fig:hasse_parabolic_2}).
	On the whole,
	the quotient $\mc P_\Phi \bs W$ has four elements.
\end{exem}

\begin{figure}
	\centering
	\begin{tikzpicture}

		\node [draw,circle,inner sep=0pt,minimum size=5pt] (Borel) [below of =top] {};

		\node [draw,circle,inner sep=0pt,minimum size=5pt] (12212313) [below left of =Borel] {};
		\node [draw,circle,inner sep=0pt,minimum size=5pt] (12213231) [below right of =Borel] {};

		\node (bottom) [below left of =12213231] {$\Phi$};

		\node [right of =Borel] {$\abs \psi = 3$};
		\node [right of =12213231] {$\abs \psi = 4$};

		\draw (Borel) -- (12212313);
		\draw (Borel) -- (12213231);

		\draw (12212313) -- (bottom);
		\draw (12213231) -- (bottom);
	\end{tikzpicture}
	\caption{Example poset of parabolic subsets of roots,
		up to the Weyl action.}
	\label{fig:hasse_parabolic_2}
\end{figure}

\subsection{Finite Verma modules}
\label{sec:verma_modules}

In~\cite{felder_rembado_2023_singular_modules_for_affine_lie_algebras_and_applications_to_irregular_wznw_conformal_blocks},
the second/third authors looked at certain generalisations of \emph{Verma} modules~\cite{verma_1966_structure_of_certain_induced_representations_of_complex_semisimple_lie_algebras}.
Here we will extend them further,
and we now review the standard setup in the tame case ($r = 1$).

\subsubsection{}

Choose a Borel subalgebra $\mf b \sse \mf g$ containing $\mf t$,
so that $\mf b_{\ab} \ceqq \mf b \bs [\mf b,\mf b] \simeq \mf t$.
Then a \emph{character} $\chi \in \Hom_{\Lie}(\mf b,\mb C) \simeq \Hom_{\mb C}(\mf b_{\ab},\mb C)$ corresponds to a linear functional $\lambda \in \mf t^{\dual}$.
Extending scalars yields a Verma $U\mf g$-module
\begin{equation}
	\label{eq:finite_generic_verma}
	M_{\lambda} = M_{\mf b,\lambda} \ceqq \Ind_{U\mf b}^{U\mf g} \mb C_{\chi} = U\mf g \ots_{U\mf b} \mb C_{\chi},
\end{equation}
where $\mb C_{\chi}$ denotes the 1-dimensional $U\mf b$-module defined by $\chi$,
viewing $U\mf g$ as a right $U\mf b$-module in the (right) regular representation.
It follows that $[\mf b,\mf b] = \mf{nil}(\mf b)$ annihilates the \emph{canonical generator} $v_{\mf b,\lambda} \ceqq 1 \ots_{U\mf b} 1$,
which is thus a highest-weight vector---%
of highest weight $\lambda$.

\subsubsection{}
\label{sec:generalised_verma_modules}

A \emph{generalised} Verma module involves inducing from a parabolic subalgebra.
Namely,
let $\mf p \sse \mf g$ be a parabolic subalgebra containing $\mf t$.
Then $\mf p_{\ab} \simeq \mf Z_{\mf l}$ is identified with the centre of the unique Levi factor $\mf l \sse \mf p$ containing $\mf t$,
and as above one can induce:
\begin{equation}
	\label{eq:tame_parabolic_verma}
	M_{\lambda} = M_{\mf p,\lambda} \ceqq \Ind_{U\mf p}^{U\mf g} \mb C_{\chi},
	\qquad \lambda \in \mf Z_{\mf l}^{\dual} \simeq \Hom_{\Lie}(\mf p,\mb C) \ni \chi.
\end{equation}
Again $[\mf p,\mf p] = [\mf l,\mf l] \ops \mf{nil}(\mf p)$ annihilates the highest-weight cyclic vector $v_{\mf p,\lambda} \ceqq 1 \ots_{U\mf p} 1$.

\subsubsection{}
\label{sec:parabolic_to_parabolic}

Let $\mf p \sse \wt{\mf p} \sse \mf g$ be nested parabolic subalgebras containing $\mf t$,
with nested Levi factors $\mf l \sse \wt{\mf l} \sse \mf g$.
Choose a character $\wt \chi \cl \wt{\mf p} \to \mb C$,
and let $\wt \lambda \ceqq \eval[1]{\wt \chi}_{\mf Z_{\wt{\mf l}}} \in \mf Z_{\wt{\mf l}}^{\dual}$ as above.
The restriction map $\Hom_{\Lie}(\wt{\mf p},\mb C) \to \Hom_{\Lie}(\mf p,\mb C)$ is \emph{injective},
as characters are determined by the values they take on $\mf t$:
consider then the character $\chi \ceqq \eval[1]{\wt \chi}_{\mf p}$,
corresponding to the linear functional $\lambda \ceqq \eval[1]{\chi}_{\mf Z_{\mf l}} \in \mf Z_{\mf l}^{\dual}$---%
which,
conversely,
extends $\wt \lambda$.

Then the Verma module induced from the smaller parabolic subalgebra is \emph{larger}.
Precisely,
the canonical map $U\mf g \ts \mb C \to U\mf g \ots_{U\wt{\mf p}} \mb C_{\wt \chi}$ is $U\mf p$-balanced,
inducing a $U\mf g$-linear surjection $f_{\lambda \mid \wt \lambda} \cl M_{\mf p,\lambda} \thra M_{\wt{\mf p},\wt \lambda}$.
To study the kernel,
suppose that $\mf p \sse \wt{\mf p}$ corresponds to an inclusion $\psi \sse \wt \psi$ of parabolic subsets with Levi factors $\phi \sse \wt \phi$,
respectively.
Denote by $\wt \nu \ceqq \wt \psi \sm \wt \phi$ and $\nu \ceqq \psi \sm \phi$ the subsets of roots defining the nilradicals of $\wt{\mf p} = \mf p_{\wt \psi}$ and $\mf p = \mf p_{\psi}$,
and note that one has $\wt \nu \sse \nu$ (cf.~Lem.~\ref{lem:nested_nilradicals}).
Then:

\begin{lemm}
	\label{lem:parabolic_to_parabolic}

	One has
	\begin{equation}
		\label{eq:parabolic_to_parabolic}
		\ker \bigl( f_{\lambda \mid \wt \lambda} \bigr) = U\mf g \bigl( \mf u^-_{\psi \mid \wt \psi} \cdot v_{\mf p,\lambda} \bigr) \sse M_{\mf p,\lambda},
		\qquad \mf u^-_{\psi \mid \wt \psi} \ceqq \bops_{\nu \sm \wt \nu} \mf g_{-\alpha} \sse \mf g.
	\end{equation}
\end{lemm}

\begin{proof}
	Postponed to~\ref{proof:lem_parabolic_to_parabolic}.
\end{proof}

\begin{rema}
	\label{rmk:minimal_modules}

	Hence,
	there is a well-defined notion of `minimal' (finite) Verma module of given weight.

	Namely,
	consider a Levi subsystem $\phi \sse \Phi^{\dual}$ of the inverse root system.
	The associated \emph{dual stratum of} $\mf t^{\dual}$ is
	\begin{equation}
		\label{eq:dual_tame_stratum}
		\mf t^{\dual}_{\phi} \ceqq \phi^\perp \, \bigsm \,
		\,
		\bigcup_{\Phi^{\dual} \sm \phi} \set{\alpha^{\dual}}^\perp  \sse \mf t^{\dual},
	\end{equation}
	involving annihilators of subsets of $\mf t$---%
	i.e.,
	e.g.,
	$\phi^\perp = \bigcap_{\phi} \Set{ \lambda \in \mf t^{\dual} | \braket{ \lambda,\alpha^{\dual}} = 0 } \sse \mf t^{\dual}$.
	(The proof that these subspaces provide a stratification is analogous to~\ref{proof:lem_tame_stratification}.)

	Choose now $\lambda \in \mf t^{\dual}$,
	and let
	\begin{equation}
		\phi_{\lambda} \ceqq \Phi^{\dual} \cap \ker(\lambda) = \Set{ \alpha^{\dual} \in \Phi^{\dual} | \Braket{ \lambda,\alpha^{\dual} } = 0 }.
	\end{equation}
	It is a Levi subsystem of $\Phi^{\dual}$,
	determining a stratum via~\eqref{eq:dual_tame_stratum}.
	The canonical bijection maps it to a Levi subsystem $\phi \in \mc L_\Phi$,\fn{
		Namely,
		to $\phi_X \sse \Phi$,
		in the notation of~\eqref{eq:levi_subsystem_semisimple_element},
		where $X \ceqq (\cdot \mid \cdot)^\sharp(\lambda) \in \mf t$.}~and we let $\psi \in \mc P_\Phi$ be a parabolic subset with Levi factor $\phi$.
	Then there exists a character $\mf p_{\psi} \to \mb C$ extending $\lambda$,
	and the same holds for \emph{any} parabolic subalgebra $\mf p$ contained in $\mf p_{\psi}$:
	but if $\mf p \ssne \mf p_{\psi}$ then Lem.~\ref{lem:parabolic_to_parabolic} implies that the corresponding generalised Verma module \emph{cannot} be simple.
	Conversely,
	inducing from a parabolic subalgebra $\mf p$ with $\mf p_{\psi} \ssne \mf p$ does \emph{not} make sense,
	as one cannot extend $\lambda$ to a character of $\mf p$.

	Equivalently,
	start from a parabolic subalgebra $\mf p \sse \mf g$,
	and look for a weight $\lambda \in \mf t^{\dual}$ such that the corresponding generalised Verma module is \emph{simple}.
	Then:
	(i) the weight $\lambda$ vanishes on $\mf t_\alpha = [\mf g_{-\alpha},\mf g_\alpha]$ for all $\alpha \in \phi \ceqq \on{Lf}(\psi)$ (to extend to a character of $\mf p$);
	and (ii) it does \emph{not} vanish on $\mf t_\alpha$ for all $\alpha \in \nu = \psi \sm \phi$ (again,
	by Lem.~\ref{lem:parabolic_to_parabolic}).
	This means that $\lambda$ lies in the stratum $\mf t^{\dual}_{\phi^{\dual}}$,
	where $\phi^{\dual} \sse \Phi^{\dual}$ is the Levi subsystem corresponding to $\phi \sse \Phi$ under the canonical bijection.
	(Again such `minimal' Verma modules will not be simple in general.)

	E.g.,
	if $\lambda$ is \emph{regular},
	i.e.,
	if $\phi_{\lambda} = \vn$,
	then one can only induce from Borel subalgebras.
	At the opposite end of the poset,
	if $\lambda \in \mf Z_{\mf g}^{\dual} \simeq [\mf g,\mf g]^\perp \sse \mf t^{\dual}$ then one can induce from \emph{any} parabolic subalgebra,
	but the only simple quotient is $ M_{\mf g,\lambda} \simeq \mb C_{\lambda}$.
	This relates with nonsingular characters of the Levi factor $\mf l_{\phi} \sse \mf p_{\psi}$ (cf.~\S~\ref{sec:tame_nonsingular_characters} +~\ref{sec:nonsingular_characters}).
\end{rema}

\subsection{Affine Verma modules}

\subsubsection{}

Consider the \emph{affine} Borel subalgebra
\begin{equation}
	\label{eq:affine_borel}
	\wh{\mf b} \ceqq \bigl( \mf b \ops z\mf g\llb z \rrb \bigr) \ops \mb CK \sse \wh{\mf g}.
\end{equation}
One has $[\mf b,\mf g] = [\mf g,\mf g] \sse \mf g$,
since the inverse root system spans $\mf t \cap [\mf g,\mf g]$.
Therefore,
one has $\wh{\mf b}_{\ab} \simeq \wh{\mf t} \ops z \mf Z_{\mf g}\llb z \rrb$.
Let us only consider characters $\chi$ which vanish on the central tail $z \mf Z_{\mf g} \llb z \rrb$,
i.e.,
functionals coded by pairs $(\lambda,\kappa) \in \mf t^{\dual} \ops \mb C \simeq  \wh t^{\dual}$---%
in the natural identification $(\mb CK)^{\dual} \simeq \mb C$.
The corresponding \emph{level}-$\kappa$ \emph{affine Verma module} is
\begin{equation}
	\label{eq:affine_generic_verma}
	\wh M_{\lambda,\kappa} = \wh M_{\mf b,\lambda,\kappa} \ceqq \Ind_{U\wh{\mf b}}^{U\wh{\mf g}} \mb C_{\chi},
\end{equation}
with cyclic vector $v_{\mf b,\lambda,\kappa} = v_{\lambda,\kappa} \ceqq 1 \ots_{U\wh{\mf b}} 1$.
By construction,
the latter is annihilated by $\mf{nil}(\mf b) \ops z \mf g \llb z \rrb \spse \bigl[ \wh{\mf b},\wh{\mf b} \bigr]$,
so it is still a highest-weight vector of highest weight $\lambda$.

\subsubsection{}

Finally,
consider the \emph{affine} parabolic subalgebra
\begin{equation}
	\wh{\mf p} \ceqq \bigl( \mf p \ops z \mf g \llb z \rrb \bigr) \ops \mb CK.
\end{equation}
Reasoning as above yields $\wh{\mf p}_{\ab} \simeq \bigl( \mf Z_{\mf l} \ops z \mf Z_{\mf g}\llb z \rrb \bigr) \ops \mb CK$,
and in turn the level-$\kappa$ \emph{generalised} affine Verma module is
\begin{equation}
	\label{eq:affine_generalised_verma}
	\wh M_{\lambda,\kappa} = \wh M_{\mf p,\lambda,\kappa} \ceqq \Ind_{U\wh{\mf p}}^{U\wh{\mf g}} \mb C_{\chi},
	\qquad (\lambda,\kappa) \in (\mf Z_{\mf l} \ops \mb CK)^{\dual} \sse \Hom_{\Lie}(\wh{\mf p},\mb C) \ni \chi.
\end{equation}
In particular,
the canonical generator $v_{\mf p,\lambda,\kappa} = v_{\lambda,\kappa} \ceqq 1 \ots_{U\wh{\mf p}} 1$ is now annihilated by $[\mf l,\mf l] \ops \mf{nil}(\mf p) \ops z \mf g \llb z \rrb \spse \bigl[ \wh{\mf p},\wh{\mf p} \bigr]$.

\begin{rema}
	The finite Verma modules are naturally subspaces of the affine ones,
	obtained by acting with $U\bigl(\mf g \llb z \rrb\bigr) \sse U\wh{\mf g}$ on the canonical generator:
	hereafter,
	such embeddings $M_{\lambda} \hra \wh M_{\lambda,\kappa}$ are tacit.
\end{rema}

\subsubsection{}
\label{sec:affine_parabolic_to_affine_parabolic}

The arguments of \S~\ref{sec:parabolic_to_parabolic} extend essentially verbatim to the affine case.

Namely,
suppose that $\psi \sse \wt \psi$ are parabolic subsets of $\Phi$ with Levi factors $\phi \sse \wt \phi$.
Choose an element $\wt \lambda \in \mf Z_{\wt \phi}^{\dual}$ and a level $\kappa \in \mb C$,
which induces a character $\wh{\mf p}_{\wt \psi} \to \mb C$ as above;
the latter restricts to a character $\wh{\mf p}_{\psi} \to \mb C$,
which in turn corresponds to a pair $(\lambda,\kappa) \in \mf Z_{\phi}^{\dual} \ts \mb C$.
There is a natural $U\wh{\mf g}$-linear surjection $\wh f_{\lambda \mid \wt \lambda} \cl \wh M_{\mf p_{\psi},\lambda,\kappa} \thra \wh M_{\mf p_{\wt \psi},\wt \lambda,\kappa}$ with
\begin{equation}
	\ker \bigl( \wh f_{\lambda \mid \wt \lambda} \bigr) = U\wh{\mf g} \bigl( \mf u^-_{\psi \mid \wt \psi} \cdot \wh v_{\mf p_{\psi},\lambda,\kappa} \bigr) \sse \wh M_{\mf p_{\psi},\lambda,\kappa},
\end{equation}
where the subspace $\mf u^-_{\psi \mid \wt \psi} \sse \mf g$ is defined as in~\eqref{eq:parabolic_to_parabolic}.
The point is that the subalgebra $U \bigl( z^{-1} \mf g [z^{-1}] \bigr) \sse U\wh{\mf g}$ acts freely on both modules,
so that one has,
e.g.,
a $\mb C$-linear isomorphism
\begin{equation}
	\wh M_{\mf p_{\psi},\lambda,\kappa} \simeq U \bigl( \wh{\mf u}^-_{\psi} \bigr),
	\qquad \wh u^-_{\psi} \ceqq z^{-1} \mf g[z^{-1}] \ops \mf u^-_{\psi} \sse \mf g[z^{-1}].
\end{equation}

\subsection{From finite modules to loop-algebra modules (I)}
\label{sec:from_finite_to_loop_algebra}

It is also possible to recover the action of loop algebras from the finite Verma modules,
by inducing representations from the positive part.

Precisely,
abstractly,
let $\wt{\mf h}$ be a Lie algebra and $\wt V$ a $U\wt{\mf h}$-module:

\begin{lemm}
	\label{lem:induction_from_submodule}

	Let $\mf h \sse \wt{\mf h}$ be a Lie subalgebra,
	and $V \sse \wt V$ an $U\mf h$-stable vector subspace---%
	regarded as a $U\mf h$-module.
	Furthermore,
	suppose that there is a vector $v \in V$ such that:
	\begin{enumerate}
		\item $U\wt{\mf h} \cdot v = \wt V$;

		\item and $\Ann_{U\wt{\mf h}} ( v ) \sse U\wt{\mf h} \cdot \Ann_{U\mf h} ( v )$.
	\end{enumerate}
	Then there is a $U\wt{\mf h}$-linear isomorphism $\wt V \simeq \Ind_{U\mf h}^{U\wt{\mf h}} V$.
\end{lemm}

\begin{proof}
	Postponed to~\ref{proof:lem_induction_from_submodule}.
\end{proof}

\subsubsection{}

Again let $\mf p \sse \mf g$ be a parabolic subalgebra,
and consider the unique Levi factor $\mf l \sse \mf p$ containing a fixed Cartan subalgebra $\mf t \sse \mf g$.
Choose a pair $(\lambda,\kappa) \in \mf Z_{\mf l}^{\dual} \ts \mb C$,
and consider the Lie algebras $\mf h \ceqq \mf g \llb z \rrb \sse \mf g(\!(z)\!) \eqqc \wt{\mf h}$.
The finite Verma module $M = M_{\lambda}$ is a subspace of the affine Verma module $\wh M = \wh M_{\lambda,\kappa}$,
and it is $U\mf h$-stable---%
with trivial action of $z \mf g \llb z \rrb \sse \mf h$.
Then Lem.~\ref{lem:induction_from_submodule} yields an isomorphism of $U \bigl( \mf g(\!(z)\!) \bigr)$-modules
\begin{equation}
	\wh M \simeq U \bigl( \mf g(\!(z)\!) \bigr) \ots_{U ( \mf g \llb z \rrb )} M,
\end{equation}
with tacit restriction of the action on the left-hand side.
Indeed,
the canonical generator $v \ceqq v_{\lambda,\kappa}$ satisfies the two properties in the statement of Lem.~\ref{lem:induction_from_submodule},
essentially because the affine Verma modules are originally defined by extension of scalars from a Lie subalgebra of $\mf g \llb z \rrb \ops \mb C K$ (cf.~\S~\ref{sec:from_finite_to_loop_algebra_wild} for the wild analogue).

\begin{rema}
	Explicitly,
	one has
	\begin{equation}
		\Ann_{U ( \mf g(\!(z)\!) )} (v_{\lambda,\kappa}) = U \bigl( \mf g(\!(z)\!) \bigr) \wh K_{\mf p,\lambda},
	\end{equation}
	where
	\begin{equation}
		\label{eq:annihilator_cyclic_vector_affine_generalised_verma}
		\wh K_{\mf p,\lambda} \ceqq [\mf l,\mf l] \ops \mf{nil}(\mf p) \ops \spann_{\mb C} \Set{ X - \Braket{ \lambda,X } | X \in \mf Z_{\mf l}} \ops z \mf g \llb z \rrb \sse U \bigl( \mf g \llb z \rrb \bigr).
	\end{equation}

	The same holds in particular for a Borel subalgebra $\mf b \sse \mf g$,
	in which case
	\begin{equation}
		\wh K_{\mf b,\lambda} = \mf{nil}(\mf b) \ops \spann_{\mb C} \Set{ X - \Braket{ \lambda,X } | X \in \mf t} \ops z \mf g \llb z \rrb.
		\qedhere
	\end{equation}
\end{rema}

\subsection{About nonsingular characters}
\label{sec:tame_nonsingular_characters}

Choose a parabolic subset $\psi \in \mc P_\Phi$ with Levi factor $\phi \in \mc L_\Phi$,
determining the Lie subalgebras $\mf u^\pm_{\psi},
	\mf l_{\phi} \sse \mf g$.
Let $\mf g = \mf u^-_{\psi} \ops \mf l_{\phi} \ops \mf u^+_{\psi}$ be the associated triangular decomposition of $\mf g$,
and denote by $\pi_{\phi} \cl \mf g \thra \mf l_{\phi}$ the canonical projection parallel to $\mf u^-_{\psi} \ops \mf u^+_{\psi} \sse \mf g$.

\begin{defi}
	\label{def:nonsingular_character_tame}

	A character $\chi \in \Hom_{\Lie}(\mf l_{\phi},\mb C)$ is \emph{nonsingular} if the pairing
	\begin{equation}
		\label{eq:pairing_nondegenerate_character}
		B_{\lambda}^{\psi} \cl \mf u^+_{\psi} \ots \mf u^-_{\psi} \lra \mb C,
		\qquad Y \ots Y' \lmt \Braket{ \chi,\pi_{\phi} \bigl( [Y,Y'] \bigr) },
		\
	\end{equation}
	is nondegenerate.
\end{defi}

\subsubsection{}

In Def.~\ref{def:nonsingular_character_tame},
we write again $\lambda \in \mf t^{\dual}$ for the $\mb C$-linear map corresponding to the character;
it vanishes on $\mf t_\alpha = [\mf g_\alpha,\mf g_{-\alpha}]$ for $\alpha \in \phi$.
Nonsingular characters correspond to those linear maps that further do \emph{not} vanish on $\mf t_\beta$ for $\beta \in \Phi \sm \phi$,
i.e.,
elements of the dual stratum~\eqref{eq:dual_tame_stratum}.

This is natural in the symplectic viewpoint.
Indeed,
we can use the root-splitting of $(\mf g,\mf t)$ to (tacitly) identify $\lambda$ with an element of $\mf g^{\dual}$ vanishing on $\bops_\Phi \mf g_\alpha \sse \mf g$,
and consider the coadjoint orbit $\mc O_{\lambda} \ceqq G \cdot \lambda \sse \mf g^{\dual}$ (cf.~\S~\ref{rmk:dual_version}).
Then~\eqref{eq:pairing_nondegenerate_character} can be written
\begin{equation}
	B_{\lambda}^{\psi}(Y,Y') = \Braket{ \lambda,[Y,Y'] } = \set{ Y,Y' }(\lambda),
	\qquad Y \in \mf u^+,
	\quad Y' \in \mf u^-.
\end{equation}
Therefore,
it is a restriction of the Poisson bracket of (linear functions on) $\mf g^{\dual}$,
evaluated at $\lambda$.
But by construction
\begin{equation}
	[\mf l_{\phi},\mf g] \sse  \bops_{\phi} \mf t_\alpha  \ops \bops_\Phi \mf g_\alpha \sse \ker(\lambda),
\end{equation}
so that $\mf l_{\phi} \sse \mf g^{\lambda}$ (cf.~\eqref{eq:infinitesimal_coadjoint_stabiliser}),
and there is a well-induced alternating form on $\bigwedge^2 (\mf g \bs \mf l_{\phi})$.
Furthermore,
one has $\mf g \bs \mf l_{\phi} \simeq \mf u^+_{\psi} \ops \mf u^-_{\psi}$,
and both summands are annihilated by $\lambda$:
hence,
indeed,
the form~\eqref{eq:pairing_nondegenerate_character} is a reduction of the Poisson bracket.
It will coincide with the nondegenerate KKS symplectic form on $\mc O_{\lambda} \simeq G \bs G^{\lambda}$,
at the tangent space $T_{\lambda} \mc O_{\lambda} \simeq \mf g \bs \mf g^{\lambda}$,
if and only if $\mf l_{\phi} = \mf g^{\lambda}$.
Finally,
this happens precisely when $\lambda \in \mf t^{\dual}_{\phi^{\dual}} \sse \mf t^{\dual}$,
and this observation will be generalised to the wild case (in \S~\ref{sec:nonsingular_characters}) to get a more explicit description of the symplectic structure---%
with a view towards quantisation.

\subsection{Shapovalov forms}
\label{sec:tame_shapovalov}

In \S~\ref{sec:shapovalov} we introduce a wild generalisation of the Shapovalov form~\cite{shapovalov_1972_a_certain_bilinear_form_on_the_universal_enveloping_algebra_of_a_complex_semisimple_lie_algebra},
in the finite case.
We review it here in the tame setting ($r = 1$).

\subsubsection{}

After choosing suitable Poincaré--Birkhoff--Witt (= PBW) bases,
the triangular decomposition $\mf g = \mf u^-_{\psi} \ops \mf l_{\phi} \ops \mf u^+_{\psi}$ induces a vector-space splitting of the universal enveloping algebra (= UEA)
\begin{equation}
	\label{eq:universal_enveloping_triangular_decomposition}
	U\mf g = U\mf l_{\phi} \ops \bigl( \mf u^-_{\psi} \cdot U\mf g + U\mf g \cdot \mf u^+_{\psi} \bigr),
\end{equation}
whence a projection $\pi_{\phi} \cl U\mf g \thra U\mf l_{\phi}$ parallel to the rightmost direct summand.
This is the UEA version of the projection $\mf g \thra \mf l_{\phi}$ parallel to $\mf u^-_{\psi} \ops \mf u^+_{\psi}$,
which we abusively write the same.
(As noted in~\cite[\S~2.8]{moody_pianzola_1995_lie_algebras_with_triangular_decomposition},
this is best suited to treat \emph{highest-weight} modules,
rather than lowest-weight,
as the symmetry between the opposite nilradicals is broken by choosing which to put on the left/right.)

Now choose a \emph{pinning of} $(\mf g,\mf t)$,\fn{
	An `épinglage'~\cite[Exp.~XXIII, Déf.~1.1]{demazure_grothendieck_1970_schemas_en_groupes_iii_structure_des_schemas_en_groupes_reductifs}.
	We work over $S = \Spec \mb C$,
	so everything splits;
	and we only need the infinitesimal/Lie-algebra version.}~i.e.:
(i) a base $\Delta \sse \Phi$ of simple roots;
and (ii) generators $E_\theta \in \mf g_\theta \sm (0)$ for the root-lines $\mf g_\theta \sse \mf g$---%
with $\theta \in \Delta$.
The negative-simple root-vectors $F_\theta = E_{-\theta} \in \mf g_{-\theta} \sm (0)$ are canonically determined by imposing that
\begin{equation}
	[E_\theta,F_\theta]
	= H_\theta \ceqq \theta^{\dual} \in [\mf g_\theta,\mf g_{-\theta}] \sse \mf t.
\end{equation}
Then consider the corresponding `transposition' involution $X \mt \ps{t}{} X \cl \mf g \lxra{\simeq} \mf g^{\op}$ of $\mf g$,
i.e.,
the Lie-algebra antiautomorphism defined by---%
uniquely---%
extending
\begin{equation}
	\label{eq:transposition_involution}
	\ps{t}{} X \ceqq X,
	\quad \ps{t}{} E_{\pm \theta} \ceqq E_{\mp \theta},
	\qquad X \in \mf t,
	\quad \theta \in \Delta.
\end{equation}
(If $\mf g = \mf{gl}_m(\mb C)$,
this is literal matrix transposition for the standard pinning.)
By construction,
the map~\eqref{eq:transposition_involution}:
(i) swaps the nilradicals $\mf u^\pm_{\psi} \sse \mf g$;
(ii) restricts to the analogous involution $\mf l_{\phi} \lxra{\simeq} \mf l_{\phi}^{\op}$;
and (iii) extends---%
uniquely---%
to a ring involution $U\mf g \lxra{\simeq} U(\mf g^{\op}) \simeq (U\mf g)^{\op}$,
abusively written the same.
Then the (generalised) \emph{universal Shapovalov pairing} is
\begin{equation}
	\label{eq:universal_shapovalov}
	\mc S^{\bm \psi} \cl U\mf g \ots U\mf g \lra U\mf l_{\phi},
	\qquad X \ots X' \lmt \pi_{\phi} \bigl( \ps{t}{} X \cdot X' \bigr).
\end{equation}
Cf.~\cite[\S~2.8]{moody_pianzola_1995_lie_algebras_with_triangular_decomposition} in the generic case $\phi = \vn$,
and note that (on the other side of the poset) $\phi = \Phi$ simply yields the associative product of $U\mf g$.

By construction,
the pairing~\eqref{eq:universal_shapovalov} is $\mb C$-bilinear and \emph{contragredient} with respect to the regular representation of $U\mf g$ and the transposition,
viz.,
one has
\begin{equation}
	\mc S^{\psi} (X \cdot Y,Z) = \mc S^{\psi} \bigl( Y,\ps{t}{} X \cdot Z \bigr),
	\qquad X,Y,Z \in U\mf g.
\end{equation}

Finally,
let $\mf Z_{\phi} \sse \mf t$ act on $U\mf g$ by the adjoint action.
Then $U\mf g$ is a weight $\mf Z_{\phi}$-module,
with weight-space decomposition
\begin{equation}
	\label{eq:weight_decomposition_universal_enveloping_algebra}
	U\mf g = \bops_{Q_{\phi}} (U\mf g)_\mu,
	\qquad Q_{\phi} \ceqq \spann_{\mb Z}(\Phi \sm \phi) = \spann_{\mb Z}(\nu) \sse \mf t^{\dual}.
\end{equation}
This involves a $\mb Z$-submodule of the \emph{root lattice} $Q \ceqq Q_{\vn} = \spann_{\mb Z}(\Phi)$,
with $\nu = \psi \sm \phi$ as above.
(This follows from the decomposition $\mf g = \mf l_{\phi} \ops \bops_{\Phi \sm \phi} \mf g_\alpha$.)

Now the splitting~\eqref{eq:weight_decomposition_universal_enveloping_algebra} is $\mc S^{\psi}$-\emph{orthogonal},
i.e.,
for $\mu,\mu' \in Q_{\phi}$ with $\mu \neq \mu'$ one has
\begin{equation}
	\label{eq:universal_shapovalov_orthogonality}
	\mc S^{\psi} (X,X') = 0,
	\qquad X \in (U\mf g)_\mu,
	\quad Y \in (U\mf g)_{\mu'}.
\end{equation}
This is because one has $\ps{t}{} X \cdot X' \in (U\mf g)_{\mu' - \mu}$,
and $U\mf l_{\phi} \sse (U\mf g)_0$,
so that
\begin{equation}
	\bops_{Q_{\phi} \sm (0)} (U\mf g)_\mu \sse \mf u^-_{\psi} \cdot U\mf g + U\mf g \cdot U^+_{\psi},
\end{equation}
noting that~\eqref{eq:universal_enveloping_triangular_decomposition} is a direct sum of $\mf Z_{\phi}$-submodules.

\subsubsection{}

Choose a character $\chi \in \Hom_{\Lie}(\mf l_{\phi},\mb C)$,
and extend it tacitly to a ring morphism $U\mf l_{\phi} \to \mb C$.
Then the (generalised) \emph{Shapovalov form} associated with $\chi$ is the composition
\begin{equation}
	\label{eq:shapovalov_form}
	\mc S^{\psi}_{\lambda} \ceqq \chi \circ \mc S^{\psi} \cl U\mf g \ots U\mf g \lra \mb C,
\end{equation}
where again by $\lambda \ceqq \eval[1]\chi_{\mf t} \in \mf t^{\dual}$.
By the above,
this is a transpose-contragredient $\mb C$-bilinear form,
compatible with the weight-space decomposition~\eqref{eq:weight_decomposition_universal_enveloping_algebra}.

\begin{lemm}
	\label{lem:shapovalov_is_symmetric}

	The Shapovalov form~\eqref{eq:shapovalov_form} is \emph{symmetric}.
\end{lemm}

\begin{proof}
	Postponed to~\ref{proof:lem_shapovalov_is_symmetric}.
\end{proof}

\begin{rema}
	In the generic case $\phi = \vn$ one has the stronger statement that the universal Shapovalov pairing~\eqref{eq:universal_shapovalov} is symmetric,
	even before evaluating at a character,
	since~\eqref{eq:transposition_involution} acts as the identity on $\mf l_{\vn} = \mf t$.
\end{rema}

\subsubsection{}

Now one can prove that~\eqref{eq:shapovalov_form} induces a bilinear form on the generalised Verma module $M \ceqq M_{\mf p_{\psi},\lambda}$ of highest weight $\lambda$,
in the notation of \S~\ref{sec:generalised_verma_modules}.
Namely,
if $v \ceqq v_{\mf p_{\psi},\lambda} \in M$ is the canonical generator,
then truncating~\eqref{eq:annihilator_cyclic_vector_affine_generalised_verma} yields
\begin{equation}
	\label{eq:annihilator_cyclic_vector_finite_generalised_verma}
	\Ann_{U\mf g} (v) = U\mf g \cdot K^{\psi}_{\lambda},
	\qquad K^{\psi}_{\lambda} \ceqq [\mf p_{\psi},\mf p_{\psi}] \ops \spann_{\mb C} \Set{ X - \Braket{\lambda,X } | X \in \mf Z_{\phi} } \sse U\mf g.
\end{equation}

\begin{lemm}
	\label{lem:annihilator_cyclic_vector_in_shapovalov_radical}

	One has
	\begin{equation}
		\mc S^{\psi}_{\lambda} (X,X') = 0,
		\qquad X \in U\mf g,
		\quad X' \in K^{\psi}_{\lambda}.
	\end{equation}
\end{lemm}

\begin{proof}
	Postponed to~\ref{proof:lem_annihilator_cyclic_vector_in_shapovalov_radical}.
\end{proof}

\subsubsection{}

By contragrediency,
Lem.~\ref{lem:annihilator_cyclic_vector_in_shapovalov_radical} implies that the left ideal $\Ann_{U\mf g} (v) \sse U\mf g$ lies in the radical $\Rad \bigl( \mc S^{\psi}_{\lambda} \bigr)$ of~\eqref{eq:shapovalov_form}.
Hence,
there is a reduced Shapovalov form $M \ots M \to \mb C$,
which
satisfies the same properties of~\eqref{eq:shapovalov_form}.

The orthogonality~\eqref{eq:universal_shapovalov_orthogonality} of the universal version translates into the following fact.
Consider the submonoid $Q^+_{\phi} \ceqq \spann_{\mb Z_{\geq 0}} (\nu) \sse Q_{\phi}$:
then the $\mf Z_{\phi}$-weight space decomposition of the generalised Verma module,
i.e.,
\begin{equation}
	\label{eq:weight_decomposition_generalised_verma}
	M = \bops_{Q^+_{\phi}} M[\mu],
	\qquad M[\mu] \ceqq \Set{ \wt v \in M | X \wt v = \Braket{ \lambda - \mu,X } \wt v \text{ for } X \in \mf Z_{\phi} },
\end{equation}
is by \emph{mutually-orthogonal} subspaces for the Shapovalov form.

Finally,
recall that $M$ contains a unique maximal proper submodule $N = N^{\psi}_{\lambda} \sse M$,
satisfying $N \sse \bops_{Q^+_{\phi} \sm (0)} M[\mu]$.
Then the contragrediency/orthogonality of the Shapovalov form,
and the identity $\mc S^{\psi}_{\lambda} (v,v) = 1$,
together imply that $N = \Rad \bigl( \mc S^{\psi}_{\lambda} \bigr)$.
(Cf.~\S~\ref{sec:shapovalov} more generally.)

\section{Generic singularity modules}
\label{sec:generic_singularity_modules}

\subsection{}

Here we review few definitions of~\cite{felder_rembado_2023_singular_modules_for_affine_lie_algebras_and_applications_to_irregular_wznw_conformal_blocks},
and extend them from the simple to the reductive case:
cf.~op.~cit.~for more details (and note that we slightly change the terminology therein,
adapting it to the first part of this text.)

\subsection{Affine case}

The first relevant generalisation of~\eqref{eq:affine_borel} goes as follows.
Let $\mf b \sse \mf g$ be a Borel subalgebra containing $\mf t$,
and $r \geq 1$ an integer:

\begin{defi}[cf.~\cite{felder_rembado_2023_singular_modules_for_affine_lie_algebras_and_applications_to_irregular_wznw_conformal_blocks}]
	The \emph{generic affine singularity algebra} of depth $r$ is the Lie subalgebra
	\begin{equation}
		\label{eq:generic_singularity_subalgebra}
		\wh{\mf S}^{(r)} = \wh{\mf S}^{(r)}_{\mf b} \ceqq \bigl( \mf b \llb z \rrb + z^r \mf g \llb z \rrb \bigr) \ops \mb CK \sse \wh{\mf g}.
	\end{equation}
\end{defi}

\subsubsection{}

Note that~\eqref{eq:generic_singularity_subalgebra} lies in the subalgebra $\mf g \llb z \rrb \ops \mb C K \sse \wh{\mf g}$,
where the central extension becomes trivial.
Its abelianisation comes with a $\mb C$-linear identification
\begin{equation}
	\wh{\mf S}^{(r)}_{\ab} \simeq \wh{\mf t}_r \ops z^r \mf Z_{\mf g} \llb z \rrb,
	\qquad \wh{\mf t}_r \ceqq \mf t_r \ops \mb CK,
\end{equation}
invoking the deeper Cartan subalgebra $\mf t_r = \mf t \ots \mb C_r \simeq \mf t \llb z \rrb \bs z^r \mf t \llb z \rrb$.

Now we can again induce representations,
from $U\wh{\mf S}^{(r)}$ to $U\wh{\mf g}$,
choosing a character $\bm \chi \in \Hom_{\Lie} \bigl( \wh{\mf S},\mb C \bigr)$ which vanishes on $z^r \mf Z_{\mf g} \llb z \rrb$.
This is equivalent to giving a pair $(\bm \lambda,\kappa) \in \mf t_r^{\dual} \ops \mb C \simeq \wh{\mf t}_r^{\dual}$,
where $\bm \lambda = (\lambda_0,\dc,\lambda_{r-1})$ is a tuple of elements $\lambda_i \in (\mf t \cdot \varepsilon^i)^{\dual}$.
(Note that at times we identify $(\mf t \cdot \varepsilon^i)^{\dual} \simeq \mf t^{\dual}$,
forgetting the $\varepsilon$-degree.)
One now defines the \emph{generic affine singularity module} of depth $r \in \mb Z_{\geq 1}$,
level $\kappa \in \mb C$,
and formal type (= highest weight) $\bm \lambda \in \mf t_r^{\dual}$,
as the $U\wh{\mf g}$-module:
\begin{equation}
	\label{eq:generic_affine_singularity_module}
	\wh M = \wh M^{(r)}_{\mf b,\bm \lambda,\kappa} \ceqq \Ind_{U\wh{\mf S}^{(r)}}^{U\wh{\mf g}} \mb C_{\bm \chi}.
\end{equation}
It coincides with~\eqref{eq:affine_generic_verma} in the tame case $r = 1$.
The terminology evokes the generic irregular-singularities of---%
formal germs of---%
$G$-connections with a pole of order $r$,
as $\bm \lambda \in \mf t_r^{\dual}$ corresponds to a principal part
\begin{equation}
	\mc A'
	= \sum_{i = 1}^r A_i z^{-i} \dif z \in \bigl( z^{-r}\mf t \llb z \rrb \dif z\bigr) \bs \mf t \llb z \rrb \dif z,
\end{equation}
in the duality~\eqref{eq:loop_algebra_duality} (cf.~\S~\ref{sec:relation_meromorphic_connections};
the latter is precisely a formal normal form).

\subsubsection{}

The module~\eqref{eq:generic_affine_singularity_module} is cyclically generated by the vector
\begin{equation}
	\label{eq:generic_irregular_state}
	w = w^{(r)}_{\mf b,\bm \lambda,\kappa} \ceqq 1 \ots_{U\wh{\mf S}^{(r)}} 1  \in \wh M.
\end{equation}
Choosing a base $\Delta \sse \Phi$ of simple roots,
one checks that this vector satisfies a generalisation of the usual highest-weight relations (at level $\kappa$),
namely:
\begin{align}
	\label{eq:generic_irregular_state_relations}
	\begin{cases}
		\bigl (X z^i - \braket{ \lambda_i,X } \bigr) w              & = 0 ,
		\qquad X \in \mf t,
		\quad i \in \Set{0,\dc,r-1},
		\\
		\bigl( \mf u^+ \llb z \rrb + z^r \mf g \llb z \rrb \bigr) w & = (0),
		\qquad \mf u^+ \ceqq \mf u_{\Phi^+_{\Delta}}.
	\end{cases}
\end{align}

\begin{rema}
	\label{rmk:virasoro_symmetries}

	Recall that the cyclic vector~\eqref{eq:generic_irregular_state} also satisfies the Virasoro relations of an irregular GT/BMT state~\cite{gaiotto_teschner_2012_irregular_singularities_in_liouville_theory_and_argyres_douglas_type_gauge_theories,bonelli_maruyoshi_tanzini_2012_wild_quiver_gauge_theories,felinska_jaskolski_kosztolowicz_2012_whittaker_pairs_for_the_virasoro_algebra_and_the_gaiotto_bonelli_maruyoshi_tanzini_states},
	which in turn enters into the Alday--Gaiotto--Tachikawa correspondence (= AGT~\cite{alday_tachikawa_2010_affine_sl_2_conformal_blocks_from_4d_gauge_theories,alday_gaiotto_tachikawa_2010_liouville_correlation_functions_from_four_dimensional_gauge_theories}) in irregular 2d Liouville CFT.
	Namely,
	the affine singularity modules are smooth,\fn{
		I.e.,
		for all $\wh w \in \wh M$ one has $z^N \mf g \llb z \rrb \wh w = (0)$ if $N \gg 0$.
		(Beware that this is \emph{not} local nilpotency.)}~and they carry a generalisation of the Segal--Sugawara representation of $\mf{Vir}$ on~\eqref{eq:affine_generic_verma}:
	then $w \in \wh M$ is a simultaneous eigenvector for a string of the standard positive generators $L_i \in \mf{Witt} \sse \mf{Vir}$,
	and furthermore it is annihilated by $L_i$ for $i \gg 0$.
	(This happens at noncritical level $\kappa \neq -h^{\dual}$,
	invoking the dual Coxeter number of the quadratic Lie algebra.)

	In brief,
	we naturally have a \emph{Whittaker} vector for the GT/BMT Virasoro pair.
	However,
	contrary to~\cite{felinska_jaskolski_kosztolowicz_2012_whittaker_pairs_for_the_virasoro_algebra_and_the_gaiotto_bonelli_maruyoshi_tanzini_states},
	here in general~\eqref{eq:generic_irregular_state} does \emph{not} cyclically generate the whole of~\eqref{eq:generic_affine_singularity_module} under the $\mf{Vir}$-action.
	Thus,
	the singularity modules are not Whittaker modules,
	and in principle they escape the classification of op.~cit.
	The same holds a fortiori in the generalised case (cf.~Rmk.~\ref{rmk:virasoro_symmetries_2}).
\end{rema}

\subsection{Finite case}

Consider finally the $U \bigl(\mf g\llb z \rrb \bigr)$-submodule $M \sse \wh M$ generated by the vector~\eqref{eq:generic_irregular_state},
and note that~\eqref{eq:generic_irregular_state_relations} implies that it is naturally a $U\mf g_r$-module.
It follows that there is a $U\mf g_r$-linear isomorphism
\begin{equation}
	\label{eq:generic_finite_singularity_module}
	M \simeq M^{(r)}_{\mf b,\bm \lambda} \ceqq \Ind_{U\mf b_r}^{U\mf g_r} \mb C_{\bm \chi},
\end{equation}
where $\mf b_r \sse \mf g_r$ is the deeper Borel subalgebra,
whose abelianisation can be identified with $\mf t_r$---%
so that $\bm \lambda \in \mf t_r^{\dual}$ defines a character $\bm \chi \cl \mf b_r \to \mb C$.
We say that~\eqref{eq:generic_finite_singularity_module} is the generic \emph{finite} singularity module of depth $r \in \mb Z_{\geq 1}$,
and formal type $\bm \lambda \in \mf t_r^{\dual}$.

Here $\mf S^{(r)} = \mf S^{(r)}_{\mf b} \ceqq \mf b_r \sse \mf g_r$ thus plays the role of the generic \emph{finite} singularity subalgebra,
in a finite-dimensional analogue of~\eqref{eq:generic_singularity_subalgebra}.

\section{Singularity subalgebras:
  polarisations and characters}
\label{sec:singularity_subalgebra_and_polarisations}

\subsection{}

Here we generalise the Lie subalgebra~\eqref{eq:generic_singularity_subalgebra},
with a view towards (the quantisation of) moduli spaces of \emph{nongeneric} irregular-singular meromorphic $G$-connections on Riemann surfaces.

\subsection{Parabolic filtrations}

In \S~\ref{sec:levi_filtrations} we considered Levi filtrations,
controlling strata/biholomorphism classes of $r$-semisimple orbits (cf.~Prop.~\ref{thm:pure_isomorphism_wild} +~Thm.~\ref{thm:nonpure_isomorphism_wild}).
We now replace Levi subsystems with parabolic subsets with given Levi factors.

Consider thus nondecreasing exhaustive filtrations of $\Phi$ by parabolic subsets:

\begin{defi}[Cf.~Def.~\ref{def:levi_filtration}]
	\label{def:parabolic_filtration}

	A \emph{parabolic filtration of} $\Phi$ is a sequence $\bm \psi = (\psi_i)_{i \in \mb Z_{\geq 0}}$,
	of subsets of $\Phi$,
	such that:
	\begin{enumerate}
		\item $\psi_i \sse \psi_{i+1}$ for $i \geq 0$;

		\item $\psi_i \sse \Phi$ is parabolic for $i \geq 0$;

		\item and $\Phi = \bigcup_{\mb Z_{\geq 0}} \psi_i$.
	\end{enumerate}
\end{defi}

\subsubsection{}

It follows that $\psi_i = \Phi$ for $i \gg 0$.
As in the case of Levi filtrations,
the minimal integer such that this happens is the \emph{depth of} $\bm \psi$,
denoted by $r = r(\bm \psi) \geq 0$.

In terms of the poset of parabolic subsets,
this is the same as considering \emph{nonincreasing} sequences in $\mc P_\Phi$ reaching the least element.
Their collection is denoted by $\mc P_\Phi^{(\infty)} \sse \mc P_\Phi^{\mb Z_{\geq 0}}$,
and it is a poset with the induced order.
The Levi-factor map~\eqref{eq:levi_factor_roots} yields an order-preserving surjection $\mc P_\Phi^{(\infty)} \thra \mc L_\Phi^{(\infty)}$ onto the collection of Levi filtrations.
Truncating as in \S~\ref{sec:truncated_wild_stratification} then restricts to an analogous surjection $\on{Lf}_r \cl \mc P_\Phi^{(r)} \thra \mc L_\Phi^{(r)}$ onto the collections of depth-bounded filtrations,
which intertwines the diagonal actions of the Weyl group.
Concretely,
the $i$-th term of $\bm \phi = \on{Lf}_r(\bm \psi)$ is $\phi_i \ceqq \on{Lf}(\psi_i) = \psi_i \cap (-\psi_i) \in \mc L_\Phi$,
and we will say that $\bm \phi \in \mc L_\Phi^{(r)}$ is the \emph{Levi factor} of the parabolic filtration $\bm \psi \in \mc P_\Phi^{(r)}$.

\begin{exem}
	\label{ex:sl_2_parabolic_filtrations}

	Consider $\mf g = \mf{sl}_2(\mb C)$,
	and keep the notation of Exmp.~\ref{ex:sl_2_wild_truncated_stratification} (except that $s = r$ now).

	There are two---%
	opposite---%
	proper parabolic subsets of $\Phi = \set{\pm \alpha}$,
	i.e.,
	$\pm \psi = \set{\pm \alpha}$.
	The corresponding parabolic subalgebras are Borel;
	they are $\mf p_{\psi} = \mf t \ops \mb C E$ and $\mf p_{-\psi} = \mf t \ops \mb C F$,
	using the standard root vectors.
	Hence,
	one has $\mc P_\Phi = \set{\psi > \Phi < -\psi}$ as a poset,
	with $\on{Lf}(\pm \psi) = \vn$ and $\on{Lf}(\Phi) = \Phi$.
	There are now two types of nonconstant sequences,
	with at most one jump (down to the minimal element $\Phi$).
	Thus,
	as a set one has $\mc P_\Phi^{(r)} = \bigl\{ \pm \bm \psi^{(1)},\dc,\pm \bm \psi^{(r)} \bigr\} \cup \set{\bm \Phi}$,
	where
	\begin{equation}
		\label{eq:sl_2_parabolic_filtrations}
		\pm \bm \psi^{(k)} \ceqq \bigl(\underbrace{\pm \psi,\dc,\pm \psi,\Phi}_{k \text{ terms}},\Phi,\dc),
	\end{equation}
	which are opposite (and incomparable).
	Their associated Levi factor is $\bm \phi^{(k)} \in \mc L_\Phi^{(r)}$,
	in the notation of~\eqref{eq:levi_filtration_sl_2}.
\end{exem}

\subsection{Generalised singularity subalgebras}

\subsubsection{}

Everything is ready for the main new definitions on the quantum side:
choose a parabolic filtration $\bm \psi \in \mc P_\Phi^{(\infty)}$.

\begin{defi}
	\label{def:affine_singularity_algebra}

	The \emph{affine singularity algebra of polarisation} $\bm \psi$ is the Lie subalgebra
	\begin{equation}
		\label{eq:affine_singularity_algebra}
		\wh{\mf S}^{\bm \psi} = \wh{\mf p}_{\bm \psi} \ceqq \bops_{i \geq 0} \bigl( \mf p_{\psi_i} \cdot z^i \bigr) \ops \mb CK \sse \wh{\mf g}.
	\end{equation}
\end{defi}

(Again,
it actually lies inside $\mf g \llb z \rrb \ops \mb C K$.)

\subsubsection{}

Analogously,
choose an integer $r \geq 1$ and an element $\bm \psi \in \mc P_\Phi^{(r)}$.

\begin{defi}
	\label{def:finite_singularity_algebra}

	The \emph{finite singularity algebra of polarisation} $\bm \psi$ is the Lie subalgebra
	\begin{equation}
		\label{eq:finite_singularity_algebra}
		\mf S^{\bm \psi} = \mf p_{\bm \psi} \ceqq \bops_{i = 0}^{r-1} \bigl( \mf p_{\psi_i} \cdot \varepsilon^i \bigr) \sse \mf g_r.
	\end{equation}
\end{defi}

\begin{exem}
	If $r(\bm \psi) = 1$ we just find an affine/finite parabolic subalgebra.
	Conversely,
	if $r \geq 1$ is arbitrary and $\bm \psi = (\underbrace{\psi,\dc,\psi,\Phi}_{r \text{ terms}},\Phi,\dc)$ for a maximal element $\psi \in \mc P_\Phi$,
	then we are in the generic cases.
\end{exem}

\subsection{Polarisations}
\label{sec:polarisations}

Let us dedicate a section to justify the terminology of Def.~\ref{def:affine_singularity_algebra} +~\ref{def:finite_singularity_algebra},
which relates with the quantisation of wild orbits (cf.~\cite[\S~3]{yamakawa_2019_fundamental_two_forms_for_isomonodromic_deformations}).

\subsubsection{}
\label{sec:standard_polarisations}

In the tame case first:
choose  a Levi subsystem $\phi \in \mc L_\Phi$,
and let as usual $\mc O_X \ceqq G \cdot X \sse \mf g$ be the marked semisimple (adjoint) orbit through an element $X \in \mf t_{\phi} \sse \mf t$ lying in the corresponding stratum.

Consider $\mc O_X$ as a $C^\infty$ manifold with tangent bundle $\on T \! \mc O_X \to \mc O_X$.
The fibre-bundle map $G \thra \mc O_X \cl g \mt \Ad_g(X)$ yields a vector-bundle isomorphism
\begin{equation}
	G \ts^{L_{\phi}} (\mf g \bs \mf l_{\phi}) \lxra{\simeq} \on T\!\mc O_X,
\end{equation}
where $L_{\phi} \sse G$ is the centraliser of each element in the stratum---%
and $\mf l_{\phi} = \Lie(L_{\phi})$ is the infinitesimal version.
Denote then by $J \in \Omega^0 \bigl( \mc O_X,\End(\on T \! \mc O_X) \bigr)$ the $G$-invariant integrable almost-complex structure obtained from the scalar multiplication of $\mf g$,
so that the (dual) KKS form $\omega \in \Omega^{2,0}(\mc O_X,\mb C)$ is $J$-holomorphic (cf.~\S~\ref{sec:dual_kks};
again,
one could restrict to semisimple orbits and work in the complex-algebraic category).

By definition,
a \emph{polarisation} of the holomorphic-symplectic manifold $(\mc O_X,J,\omega)$ is an involutive holomorphic Lagrangian sub-bundle of the \emph{holomorphic} tangent bundle $\on T_{1,0} \!\mc O_X \sse \on T\!\mc O_X \ots_{\mb R} \mb C$.
(Note that this differs from the notions of complex/Kähler polarisations of \emph{real} symplectic manifolds.)

In our situation,
polarisations in particular arise from subspaces $\mf u \sse \mf g$ such that:
\begin{enumerate}
	\item $P_{\mf u} \ceqq \mf l_{\phi} + \mf u \sse \mf g$ is a Lie subalgebra;

	\item and $P_{\mf u} \bs \mf l_{\phi}$ is a Lagrangian subspace for $\omega \cl \bigwedge^2 (\mf g \bs \mf l_{\phi}) \to \mb C$---%
	      evaluating at the marking.
\end{enumerate}
There are no canonical choices;
but suppose further that $\phi \ceqq \on{Lf}(\psi) = \psi \cap (-\psi) \in \mc L_\Phi$ is the Levi factor of a parabolic subset $\psi \in \mc P_\Phi$.
Then:

\begin{lemm}[Cf.~\cite{calaque_naef_2015_a_trace_formula_for_the_quantization_coadjoint_orbits},
		Exmp.~2.2]
	\label{lem:polarisation}

	The nilradical $ \mf u_{\psi} \sse \mf p_{\psi}$ defines a polarisation of $(\mc O_X,J,\omega)$.
\end{lemm}

\begin{proof}
	Postponed to~\ref{proof:lem_polarisation}.
\end{proof}

\subsubsection{}
\label{sec:tame_lagrangian_splitting}

The proof~\ref{proof:lem_polarisation} shows more.
Namely,
if $\phi \neq \Phi$ then there are two transverse polarisations,
i.e.,
a $G$-invariant involutive holomorphic \emph{Lagrangian splitting} of $\on T_{1,0}\!\mc O_X$,
ultimately encoded by the direct sum
\begin{equation}
	\mf g \bs \mf l_{\phi} \simeq \mf u^-_{\psi} \ops \mf u^+_{\psi},
	\qquad \on{Lf}(\pm \psi) = \phi.
\end{equation}

Choosing $\phi \in \mc L_\Phi \sm \set{\Phi}$,\fn{
	This yields a cleaner statement,
	noting that the orbits through $\mf t_\Phi = \mf Z_{\mf g}$ are \emph{points}.
}~it follows that:

\begin{center}
	\emph{Each opposite pair $\pm \psi \in \on{Lf}^{-1}(\phi) \in \mc P_\Phi$ of parabolic subsets with Levi factor $\phi$ yields an involutive $G$-invariant holomorphic Lagrangian splitting of $\on T_{1,0}\!\mc O_X$---%
		for all $X \in \mf t_{\phi} \sse \mf t$.}
\end{center}

\subsubsection{}

Consider now an integer $r \geq 1$,
and a depth-bounded Levi filtration $\bm \phi \in \mc L_\Phi^{(r)}$.
Furthermore,
assume---%
w.l.o.g.---%
that $\bm \phi$ has maximal depth,
i.e.,
that $\phi_{r-1} \ssne \Phi$.
Then:

\begin{theo}[cf.~\S~\ref{sec:tame_lagrangian_splitting}]
	\label{thm:wild_polarisations}

	Each opposite pair $\pm \bm \psi \in \on{Lf}_r^{-1}(\bm \phi)$ of parabolic filtrations with Levi factor $\bm \phi$ yields an involutive $G_r$-invariant holomorphic Lagrangian splitting of $\on T_{1,0}\!\mc O_{\bm X}$---%
	for all $\bm X \in \mf t^r_{\bm \phi} \sse \mf t^r \simeq \mf t_r$.
\end{theo}

\begin{proof}
	By Cor.~\ref{cor:very_good_centraliser},
	the holomorphic tangent spaces of $\mc O_{\bm X} \sse \mf g_r$ can be identified with the complex vector space $\mf g_r \bs \mf g_r^{\bm X} = \mf g_r \bs \mf l_{\bm \phi}$,
	where
	\begin{equation}
		\label{eq:deeper_levi_factor}
		\mf l_{\bm \phi} = \Lie(L_{\bm \phi}) = \bops_{i = 0}^{r-1} \bigl( \mf l_{\phi_i} \cdot \varepsilon^i \bigr) \sse \mf g_r,
	\end{equation}
	invoking the centraliser $L_{\bm \phi} \sse G_r$ of (each element in the stratum of) $\bm X$.
	Moreover,~\eqref{eq:deeper_levi_factor} is the radical of the $L_{\bm \phi}$-invariant alternating bilinear form $\mf g_r \wdg \mf g_r \to \mb C$ defining the (dual) holomorphic-symplectic KKS structure,
	viz.,
	\begin{equation}
		\label{eq:wild_dual_kks}
		\bm Y \wdg \bm Z \lmt \bigl( X \mid [\bm Y,\bm Z] \bigr)_{\!r},
		\qquad \bm Y,\bm Z \in \mf g_r,
	\end{equation}
	using the $G_r$-invariant nondegenerate pairing of Lem./Def.~\ref{lem:deeper_pairing}.

	Choose now a parabolic filtration $\bm \psi \in \mc P_\Phi^{(r)}$ with Levi factor $\bm \phi$.
	The proof~\ref{proof:lem_polarisation} yields (nontrivial) opposite nilpotent radicals $\mf u^\pm_{\psi_i} \sse \mf p^\pm_{\psi_i}$,
	such that $\mf g = \mf u^-_{\psi_i} \ops \mf l_{\phi_i} \ops \mf u^+_{\psi_i}$,
	where as above $\mf p^\pm_{\psi_i} \ceqq \mf p_{\pm \psi_i}$.
	Consider thus the subspaces
	\begin{equation}
		\label{eq:deeper_nilradical}
		\mf u^\pm_{\bm \psi} \ceqq \bops_{i = 0}^{r-1} \bigl( \mf u^\pm_{\psi_i} \cdot \varepsilon^i \bigr) \sse \mf g_r.
	\end{equation}
	Then the sums $P^\pm = P_{\mf u^\pm_{\bm \psi}} \ceqq \mf l_{\bm \phi} + \mf u^\pm_{\bm \psi} \sse \mf g_r$ are direct,
	and they coincide with two (opposite) finite singularity subalgebras from Def.~\ref{def:finite_singularity_algebra}.
	Involutivity follows.

	Choose now a maximal element $\psi \in \mc P_\Phi$ with $\psi \geq \psi_0$,
	which yields two opposite Borel subalgebras $\mf p^\pm_{\psi} \sse \mf g$ such that $\mf p^\pm_{\psi_0} \sse \dm \sse \mf p^\pm_{\psi_{r-1}}$ is a sequence of nested \emph{standard} parabolic subalgebras of $\bigl( \mf g,\mf p^\pm_{\psi} \bigr)$.
	Then $\mf u^\pm_{\psi_i} \sse \mf u^\pm_{\psi}$ for $i \in \set{0,\dc,r-1}$,
	and the latter is the sum of all positive/negative root lines $\mf g_\alpha \sse \mf g$.
	Hence,
	$\eval[1]{(\cdot \mid \cdot)}_{\mf u^\pm_{\bm \psi} \ots \mf u^\pm_{\bm \psi}} = 0$,
	and isotropicity follows from~\eqref{eq:deeper_pairing} +~\eqref{eq:deeper_bracket}---%
	in the identification $P^\pm \bs \mf l_{\bm \phi} \simeq \mf u^\pm_{\bm \psi}$.
	(These are \emph{maximal} isotropic subspaces for dimensional reasons.)
\end{proof}

\begin{rema}
	\label{rmk:balanced_filtrations}

	Contrary to the tame case,
	in general the subspaces~\eqref{eq:deeper_nilradical} are \emph{not} Lie subalgebras.
	Namely,
	they are closed under Lie bracket if and only if
	\begin{equation}
		\bigl[ \mf u^\pm_{\psi_i},\mf u^\pm_{\psi_j} \bigr] \sse \mf u^\pm_{\psi_{i+j}},
		\qquad i,j \in \set{0,\dc,r-1},
		\quad i+j \leq r-1,
	\end{equation}
	in which case we will say the parabolic filtration $\bm \psi \in \mc P^{(r)}_\Phi$ is \emph{balanced}---%
	as the nilradicals shrink in suitable fashion.
	It is equivalent to ask that this holds for either the positive or the negative version,
	and in this case $\mf u^\pm_{\bm \psi} \sse \mf g_r$ are moreover \emph{nilpotent} Lie algebras by Engel's theorem.

	Clearly this holds if $\bm \psi$ is constant,
	thus,
	e.g.,
	in the generic case (cf.~more generally Lem.~\ref{lem:nested_nilradicals} just below).
\end{rema}

\begin{lemm}
	\label{lem:nested_nilradicals}

	Let $\mf p \sse \wt{\mf p} \sse \mf g$ be nested parabolic subalgebras containing $\mf t$.
	Then $\mf{nil}(\wt{\mf p})$ is a Lie ideal of $\mf{nil}(\mf p)$.
\end{lemm}

\begin{proof}
	Postponed to~\ref{proof:lem_nested_radicals}.
\end{proof}

\subsubsection{}

By Lem.~\ref{lem:nested_nilradicals},
one always has
\begin{equation}
	\bigl[ \mf u^\pm_{\psi_i},\mf u^\pm_{\psi_j} \bigr] \sse \mf u^\pm_{\psi_k},
	\qquad i,j \in \set{0,\dc,r-1},
	\quad k \ceqq \max \set{i,j}.
\end{equation}
Hence,
all depth-2 parabolic filtrations are balanced,
while a depth-3 parabolic filtration is balanced if and only if $\bigl[ \mf u^\pm_1,\mf u^\pm_1 \bigr] \sse \mf u^\pm_2$;
etc.

\subsection{Characters of singularity algebras}

Now we want to understand the characters of~\eqref{eq:affine_singularity_algebra}--\eqref{eq:finite_singularity_algebra},
to be able to induce representations later on.
These are revealed by the following:

\begin{lemm}
	\label{lem:bracket_nested_parabolic}
	Let $\mf p \sse \wt{\mf p} \sse \mf g$ be nested parabolic subalgebras containing $\mf t$.
	Then
	\begin{equation}
		\bigl[ \, \mf p,\wt{\mf p} \,
			\bigr] = \bigl[\,
			\wt{\mf p},\wt{\mf p} \,
			\bigr] \sse \mf g.
	\end{equation}
\end{lemm}

\begin{proof}
	Postponed to~\ref{proof:lem_bracket_nested_parabolic}.
\end{proof}

\begin{coro}
	\label{cor:abelianisation}

	Choose a parabolic filtration $\bm \psi \in \mc P_\Phi^{(\infty)}$ of depth $r = r(\bm \psi) \in \mb Z_{\geq 1}$,
	and let $\bm \phi \in \mc L_\Phi^{(r)}$ be its Levi factor.
	Then there are $\mb C$-linear isomorphisms
	\begin{equation}
		\label{eq:general_character_space}
		\wh{\mf S}^{\bm \psi}_{\ab} \simeq \wh{\mf Z}_{\bm \phi} \ops \mb CK,
		\qquad \wh{\mf Z}_{\bm \phi} \ceqq \bops_{i \geq 0} (\mf Z_{\phi_i} \cdot z^i ) \sse \mf t \llb z \rrb,\fn{
			The latter is finite-dimensional if and only if $\mf g$ is semisimple.}
	\end{equation}
	and
	\begin{equation}
		\label{eq:general_character_space_finite}
		\mf S^{\bm \psi}_{\ab} \simeq \mf Z_{\bm \phi} \ceqq \bops_{i = 0}^{r-1} (\mf Z_{\phi_i} \cdot \varepsilon^i) \sse \mf t_r,
	\end{equation}
	invoking the centres $\mf Z_{\phi_i} \ceqq \mf Z(\mf l_{\phi_i}) = \ker(\phi_i) \sse \mf t$ of the Levi factors (i.e.,
	the `Levi centres').
\end{coro}

\begin{proof}
	Using the $z$-grading and Lem.~\ref{lem:bracket_nested_parabolic} yields
	\begin{equation}
		\label{eq:derived_singularity_subalgebra}
		\bigl[ \wh{\mf S}^{\bm \psi},\wh{\mf S}^{\bm \psi} \bigr] = \bops_{i \geq 0} \bigl( \bigl[ \mf p_{\psi_i},\mf p_{\psi_i} \bigr] \cdot z^i \bigr) \sse \mf g \llb z \rrb.
	\end{equation}
	In turn,
	by~\eqref{eq:decomposition_derived_levi},
	one has
	\begin{equation}
		\mf p_{\psi_i} \bs \bigl[ \mf p_{\psi_i},\mf p_{\psi_i} \bigr] \simeq \mf Z_{\mf g} \ops \bops_{\Phi \sm \phi_i} \mf t_\alpha = \ker(\phi_i) = \mf Z_{\phi_i}.
	\end{equation}
	The statement then follows in the affine case,
	and the finite one is analogous.
\end{proof}

\begin{rema}
	\label{rmk:deeper_characters}

	If $r = 1$ one has the simpler statement that the characters of a parabolic subalgebra $\mf p \sse \mf g$ are in canonical bijection with those of a Levi factor $\mf l \sse \mf p$:
	all characters of $\mf p$ vanish on the nilradical,
	and all characters of $\mf l$ extend---%
	by zero---%
	to characters of $\mf p$.

	This does \emph{not} work verbatim in the wild case,
	already in the finite setting.
	Namely,
	let $\bm \psi \in \mc P^{(r)}_\Phi$ be a depth-bounded parabolic filtration with Levi factor $\bm \phi \ceqq \on{Lf}_r(\bm \psi) \in \mc L^{(r)}_\Phi$.
	Then there is an inclusion
	\begin{equation}
		\bigl[ \mf l_{\bm \phi},\mf l_{\bm \phi} \bigr] \sse \bops_{i = 0}^{r-1} \bigl( \bigl[ \mf l_{\phi_i},\mf l_{\phi_i} \bigr] \cdot \varepsilon^i \bigr),
	\end{equation}
	in the notation of~\eqref{eq:deeper_levi_factor},
	which need \emph{not} be an equality in the nongeneric case (i.e.,
	the abelianisation of $\mf l_{\bm \phi}$ might be larger than the direct sum of the Levi centres $\mf Z_{\phi_i}$).
	Thus,
	in general,
	there are characters $\mf l_{\bm \phi} \to \mb C$ which do \emph{not} extend to $\mf S^{\bm \psi}$.
	Those which \emph{do} extend correspond to sequences of linear maps
	\begin{equation}
		\bm \lambda = (\lambda_0,\dc,\lambda_{r-1}) \in \mf Z_{\bm \phi}^{\dual} \simeq \bops_{i \geq 0}^{r-1} \bigl( \mf Z^{\dual}_{\phi_i}\bigr),
	\end{equation}
	as in~\eqref{eq:general_character_space_finite}---%
	forgetting the $\varepsilon$-degree.
	Conversely,
	it is still true a character can only be extended by zero,
	if it extends at all:
	this is because $\mf S^{\bm \psi} = \mf l_{\bm \phi} \ops \mf u^+_{\bm \psi}$ as vector spaces,
	and $\mf u^+_{\bm \psi} \sse \bigl[ \mf S^{\bm \psi},\mf S^{\bm \psi} \bigr]$ by (the truncated analogue of)~\eqref{eq:derived_singularity_subalgebra}.
\end{rema}

(Hereafter,
we also tacitly identify $\mf Z_{\phi_i}^{\dual}$ with the subspace of $\mf t^{\dual}$ consisting of covectors vanishing on $\mf t \cap \bigl[ \mf l_{\phi_i},\mf l_{\phi_i} \bigr]$.)

\section{Generalised singularity modules}
\label{sec:generalised_singularity_modules}

\subsection{}

Here we define \emph{generalised} singularity modules,
building on \S~\ref{sec:singularity_subalgebra_and_polarisations}.

\subsection{Affine case}

By Cor.~\ref{cor:abelianisation},
the space of characters of~\eqref{eq:affine_singularity_algebra}--\eqref{eq:finite_singularity_algebra} comes with a bijection onto the associated sequence of (dual) Levi centres---%
plus possibly a level.
In the affine case we will still truncate at the given depth:
namely,
if $r(\bm \psi) = r \in \mb Z_{\geq 0}$ then we consider characters $\bm \chi \in \Hom_{\Lie} \bigl( \wh{\mf p}_{\bm \psi},\mb C \bigr)$ encoded by pairs $(\bm \lambda,\kappa) \in \mf Z_{\bm \phi}^{\dual} \ops \mb C \simeq \bigl( \mf Z_{\bm \phi} \ops \mb CK \bigr)^{\! \dual}$ such that $\bm \lambda$ vanishes on $z^r \mf g \llb z \rrb \sse \mf g \llb z \rrb$,
i.e.,
with $\bm \lambda \in \mf g_r^{\dual}$.
These now yield a 1-dimensional $U\wh{\mf p}_{\bm \psi}$-module $\mb C_{\bm \chi}$.

One is thus naturally led to the following:

\begin{defi}
	\label{def:affine_singularity_module}

	The \emph{affine generalised singularity module} of level $\kappa \in \mb C$,
	polarisation $\bm \psi$,
	and formal type $\bm \lambda \in \mf Z_{\bm \phi}^{\dual}$,
	is the induced $U\wh{\mf g}$-module
	\begin{equation}
		\label{eq:generalised_affine_singularity_module}
		\wh M = \wh M^{\bm \psi}_{\bm \lambda,\kappa} \ceqq U \wh{\mf g} \ots_{U \wh{\mf p}_{\bm \psi}} \mb C_{\bm \chi}.
	\end{equation}
\end{defi}

\subsubsection{}

As in the generic case,
consider the canonical cyclic vector $w = w^{\bm \psi}_{\bm \lambda,\kappa} = 1 \ots_{U\wh{\mf p}_{\bm \psi}} 1 \in \wh M$.
It satisfies yet a further generalisation of the highest-weight relations~\eqref{eq:generic_irregular_state_relations} (at level $\kappa$).

To spell them out,
choose a base $\Delta \sse \Phi$ of simple roots,
and denote as usual by $\Sigma_0 \sse \Sigma_1 \dc$ the sequence of subsets of $\Delta$ determining the sequences $\bm \psi \mt \bm \phi = \on{Lf}_r(\bm \psi)$ of parabolic subsets and corresponding Levi factors (cf.~\S~\ref{sec:setup_2} +~Def.~\ref{def:parabolic_filtration}).
Moreover,
extend a pinning of $(\mf g,\mf t)$ (as in \S~\ref{sec:tame_shapovalov}) to a choice of Chevalley generators,
involving generators $E_\alpha \in \mf g_\alpha \sm (0)$ for all the root lines ($\alpha \in \Phi$);
and set $H_\alpha \ceqq [E_\alpha,E_{-\alpha}] \in \mf t_\alpha$.
Then,
in addition to $K w = \kappa w$,
one has:
\begin{align}
	\label{eq:irregular_state_relations}
	\begin{cases}
		\bigl( X z^i - \braket{ \lambda_i,X } \bigr) w              & = 0,
		\qquad X \in \mf Z_{\phi_i},
		\quad i \in \Set{ 0,\dc,r-1},
		\\
		E_{-\alpha} z^i w = H_\alpha z^i w                          & = 0,
		\qquad \alpha \in \Phi^+_{\Sigma_i},
		\quad i \in \Set{ 0,\dc,r-1},
		\\
		\bigl( \mf u^+ \llb z \rrb + z^r \mf g \llb z \rrb \bigr) w & = (0),
		\qquad \mf u^+ \ceqq \mf u_{\Phi^+_{\Delta}}.
		\\
	\end{cases}
\end{align}

\begin{rema}
	Comparing with~\eqref{eq:generic_irregular_state_relations},
	there are more elements annihilating the canonical generator,
	as expected:
	particularly in the second row,
	which comes from the derived subalgebras $[\mf l_{\phi_i},\mf l_{\phi_i}] \sse [\mf p_{\psi_i},\mf p_{\psi_i}]$.

	Hence,
	these modules are quotients of the generic ones,
	as the latter are the \emph{universal} modules cyclically generated by an irregular state:
	indeed,
	the explicit relations are in Thm.~\ref{thm:wild_parabolic_to_wild_parabolic} +~\S~\ref{sec:affine_wild_parabolic_to_wild_parabolic}.
\end{rema}

\begin{rema}
	\label{rmk:virasoro_symmetries_2}

	Once more one can prove that the cyclic vector is a Whittaker vector for a GT/BMT Virasoro pair (cf.~Rmk.~\ref{rmk:virasoro_symmetries}).

	Namely,
	for $i \in \mb Z_{\geq 0}$ the last line of~\eqref{eq:irregular_state_relations} yields
	\begin{equation}
		2(\kappa + h^{\dual}) L_i \wh w
		= \sum_{j = 1-r}^{r-1-i} \Biggl( \sum_k \cl X_kz^{-j} \cdot X^k z^{i + j} \cl \Biggr) \wh w \in \wh M,
	\end{equation}
	summing over a basis $(X_k)_k$ of $\mf g$,
	and invoking both the $(\cdot \mid \cdot)$-dual basis $(X^k)_k$ and the normal-ordered product.
	Thus,
	this vanishes if $i > 2(r - 1)$.
	Furthermore,
	if $i \in \set{r-1,\dc,2(r-1)}$ and $j \in \set{1-r,\dc,r-1-i}$ then
	\begin{equation}
		-j,i+j \in \set{1-r+i,\dc,r-i} \sse \set{0,\dc,r-1},
	\end{equation}
	so that the normal-ordered product is void and the first two lines of~\eqref{eq:irregular_state_relations} imply that the overall action is by scalar multiplication.
\end{rema}

\subsection{Finite case}
\label{sec:finite_generalised_singularity_module}

Finally,
consider the $U \bigl( \mf g \llb z \rrb \bigr)$-submodule $M \sse \wh M$ generated by $w \in \wh M$,
noting that it is naturally a $U\mf g_r$-module---%
by the third line of~\eqref{eq:irregular_state_relations}.
We then conclude the following,
for any element $\bm \phi \in \mc L_\Phi^{(r)}$ with a chosen lift $\bm \psi \in \mc P_\Phi^{(r)}$:

\begin{enonce}{Theorem/Definition}
	\label{thm:finite_modules}

	There is an isomorphism $M^{\bm \psi}_{\bm \lambda} \lxra{\simeq} M$ of $U\mf g_r$-modules,
	where the \emph{finite generalised singularity module} $M^{\bm \psi}_{\bm \lambda}$ (of polarisation $\bm \psi$,
	and formal type $\bm \lambda$) is defined by
	\begin{equation}
		\label{eq:generalised_finite_singularity_module}
		M^{\bm \psi}_{\bm \lambda} \ceqq U \mf g_r \ots_{U \mf p_{\bm \psi}} \mb C_{\bm \chi},
	\end{equation}
	identifying $\bm \lambda \in \mf Z_{\bm \phi}^{\dual}$ with a character $\bm \chi \in \Hom_{\Lie}(\mf p_{\bm \psi},\mb C)$---%
	of~\eqref{eq:finite_singularity_algebra}.
\end{enonce}

\begin{proof}
	The isomorphism is obtained by mapping $w^{\bm \psi}_{\bm \lambda} \ceqq 1 \ots_{U\mf p_{\bm \psi}} 1 \mt w \in M \sse \wh M$.
	Indeed,
	the resulting $U\mf g_r$-linear surjection has no kernel since both cyclic vectors satisfy precisely the truncated version of the relations~\eqref{eq:irregular_state_relations},
	viz.,
	\begin{align}
		\begin{cases}
			\bigl( X z^i - \braket{ \lambda_i,X } \bigr) w   & = 0,
			\qquad X \in \mf Z_{\phi_i},
			\\
			E_{-\alpha} z^i w = X_\alpha z^i w = \mf u^+_r w & = 0,
			\qquad \alpha \in \Phi^+_{\Sigma_i},
		\end{cases}
	\end{align}
	for $i \in \Set{ 0,\dc,r-1}$,
	where as usual $\mf u^+_r = \mf u^+ \ots \mb C_r \sse \mf g_r$.
\end{proof}

\subsection{Identification with symmetric algebras}
\label{sec:identification_verma_symmetric_algebra}

\subsubsection{}
\label{sec:pbw_bases_tame}

Let again $\mf p \sse \mf g$ be a parabolic subalgebra containing a chosen Cartan subalgebra $\mf t$,
and $\mf l \sse \mf p$ the Levi factor containing $\mf t$.
If $\lambda \in \mf Z_{\mf l}^{\dual}$ encodes a character of $\mf p$,
recall that there is a $\mb C$-linear identification $U(\mf u^-) \lxra{\simeq} M_{\lambda}$,
where $\mf u^- \sse \mf g$ is the nilradical of the opposite parabolic subalgebra $\mf p^- \sse \mf g$;
it is obtained by acting on the canonical generator $v \in M_{\lambda}$,
and it is further an isomorphism of $U(\mf u^-)$-modules.
This works because of the vector-space splitting $\mf g = \mf u^- \ops \mf p$,
noting that $\mf u^-$ acts freely on $v$,
and that $\mf p$ acts via the given character (which fixes the line through $v$).

Then the PBW isomorphism $\gr U(\mf u^-) \simeq \Sym (\mf u^-)$ can be used to transfer a basis of monomials for the symmetric algebra onto a $\mb C$-basis of $M_{\lambda}$.

\subsubsection{}

In the nongeneric wild case,
the material of \S~\ref{sec:pbw_bases_tame} needs to be slightly adapted (cf.~\cite[\S~5.1.6]{dixmier_1996_algebres_enveloppantes}).

Let $\mf t_r \sse \mf l_{\bm \phi} \sse \mf p_{\bm \psi} \sse \mf g_r$ be the usual sequence of inclusions of Lie subalgebras obtained from a parabolic filtration $\bm \psi \in \mc P^{(r)}_\Phi$ with Levi factor $\bm \phi \ceqq \on{Lf}_r(\bm \psi) \in \mc L^{(r)}_\Phi$.
There is a vector-space splitting $\mf g_r = \mf u^-_{\bm \psi} \ops \mf p_{\bm \psi}$,
where $\mf u^-_{\bm \psi} \sse \mf g_r$ is the subspace obtained from the nonincreasing sequence $\mf u^-_{\psi_0} \supseteq \dm \supseteq \mf u^-_{\psi_{r-1}}$ of nilradicals of the opposite finite singularity subalgebra $\mf p^-_{\bm \psi} \sse \mf g_r$ (cf.~\eqref{eq:deeper_nilradical}).
Given an element $\bm \lambda \in \mf Z_{\bm \phi}^{\dual}$,
i.e.,
a character $\bm \chi \cl \mf p_{\bm \psi} \to \mb C$,
it is thus still true that the subspace $\mf u^-_{\bm \psi} \sse U\mf g_r$ acts freely on the canonical generator $w$ of the finite singularity module $M^{\bm \psi}_{\bm \lambda}$:
but there is a priori \emph{no} UEA attached to $\mf u^-_{\bm \psi}$.

Nonetheless,
denote by $(\bm X_1,\dc,\bm X_s)$ an (ordered) $\mb C$-basis of $\mf u^-_{\bm \psi}$,
where $s \ceqq \dim_{\mb C}(\mf u^-_{\bm \psi})$.
Then the set of monomials
\begin{equation}
	\bm X_I = \bm X_1^{i_1} \dm \bm X_s^{i_s} \in U\mf g_r,
	\qquad I \ceqq (i_1,\dc,i_s) \in \mb Z_{\geq 0}^s,
\end{equation}
provides a basis of $U\mf g_r$ as a (free) right $U\mf p^+_{\bm \psi}$-module,
and the properties of balanced tensor products yield a $\mb C$-linear isomorphism
\begin{equation}
	\label{eq:direct_sum_splitting_verma}
	M^{\bm \psi}_{\bm \lambda} = \Biggl( \,
	\bops_{\mb Z_{\geq 0}^s} X_I \cdot U\mf p^+_{\bm \psi} \Biggr) \ots_{U\mf p^+_{\bm \psi}} \mb C_{\chi} \simeq  \bops_{\mb Z_{\geq 0}^s} \bigl( \mb C \cdot X_I \ots \mb C_{\bm \chi} \bigr).
\end{equation}
Furthermore,
in this identification the summand $\mb C \cdot X_I \ots \mb C_{\chi}$ corresponds to the line through $X_I w \in M^{\bm \psi}_{\bm \lambda}$,
so that in conclusion a $\mb C$-basis of the finite singularity module is provided by ordering $\set{ X_I w | I \in \mb Z_{\geq 0}^s }$.
It is then convenient to consider $\set{ X_I }_I$ as a $\mb C$-basis of the symmetric algebra $\Sym (\mf u^-_{\bm \psi})$,
which we embed into $U\mf g_r$ in the given choice of ordering.
Then acting on the canonical generator yields a $\mb C$-linear isomorphism $\Sym (\mf u^-_{\bm \psi}) \lxra{\simeq} M^{\bm \psi}_{\bm \lambda}$.

While this in principle breaks the $\mf g_r$-module structure,
it is helpful to construct $\mb C$-bases of finite singularity modules and to discuss their natural weight-space decompositions (cf.~just below,
as well as \S~\ref{sec:shapovalov}).

\subsection{Relations between generalised singularity modules}

\subsubsection{}
\label{sec:wild_parabolic_to_wild_parabolic}

Here we extend the argument of \S~\ref{sec:parabolic_to_parabolic} in the wild case.

Fix an integer $r \geq 1$,
and consider parabolic filtrations $\bm \psi \geq \wt{\bm \psi} \in \mc P_\Phi^{(r)}$ with Levi factors $\bm \phi \geq \wt{\bm \phi} \in \mc L_\Phi^{(r)}$---%
respectively.
Choose an element $\wt{\bm \lambda} \in \mf Z_{\wt{\bm \phi}}^{\dual}$,
which induces a character $\mf p_{\wt{\bm \psi}} \to \mb C$,
and restrict the latter to a character $\mf p_{\bm \psi} \to \mb C$;
in turn,
this corresponds to an element $\bm \lambda \in \mf Z_{\bm \phi}^{\dual}$ extending $\wt{\bm \lambda}$ along the embedding $\mf Z_{\wt{\bm \phi}} \hra \mf Z_{\bm \phi} \sse \mf t_r$.

\begin{theo}
	\label{thm:wild_parabolic_to_wild_parabolic}

	There is a \emph{canonical} $U\mf g_r$-linear surjection $f_{\bm \lambda \mid \wt{\bm \lambda}} \cl M^{\bm \psi}_{\bm \lambda} \thra M^{\wt{\bm \psi}}_{\wt{\bm \lambda}}$ with
	\begin{equation}
		\label{eq:wild_parabolic_to_wild_parabolic}
		\ker \bigl( f_{\bm \lambda \mid \wt{\bm \lambda}} \bigr) = U\mf g_r \bigl( \mf u^-_{\bm \psi \mid \wt{\bm \psi}} \cdot w^{\psi}_{\bm \lambda} \bigr) \sse M^{\bm \psi}_{\bm \lambda},
		\qquad \mf u^-_{\bm \psi \mid \wt{\bm \psi}} \ceqq \bops_{i = 0}^{r-1} \bigl( \mf u^-_{\psi_i \mid \wt \psi_i} \cdot \varepsilon^i \bigr) \sse \mf g_r,
	\end{equation}
	where in turn the subspaces $\mf u^-_{\psi_i \mid \wt \psi_i} \sse \mf g$ are defined as in~\eqref{eq:parabolic_to_parabolic}.
\end{theo}

\begin{proof}
	The $U\mf p_{\bm \psi}$-balanced canonical map $U\mf g_r \ts \mb C \to U\mf g_r \ots_{U \mf p_{\wt{\bm \psi}}} \mb C_{\wt{\bm \lambda}}$ induces the surjection in the statement.

	To compute the kernel,
	write $\abs \Phi = 2m$ for some integer $m \geq 0$;
	choose a base $\Delta \sse \Phi$ of simple roots,
	and order the positive roots as $\Phi^+_{\Delta} = \set{\alpha_1,\dc,\alpha_m}$.
	Then there exist integers $0 \leq t_{r-1} \leq \dm \leq t_0 \leq m$ and $0 \leq \wt t_{r-1} \leq \dm \leq \wt t_0 \leq m$ such that
	\begin{equation}
		\phi_i \cap \Phi^+_{\Delta} = \set{\alpha_{t_i},\dc,\alpha_m} \sse \set{\alpha_{\wt t_i},\dc,\alpha_m} = \wt \phi_i \cap \Phi^+_{\Delta}.
	\end{equation}
	(Whence moreover $\wt t_i \leq t_i$ for $i \in \set{0,\dc,r-1}$.)

	Now use the $\mb C$-linear isomorphism $\Sym (\mf u^-_{\bm \psi}) \lxra{\simeq} M^{\bm \psi}_{\bm \lambda}$ of \S~\ref{sec:identification_verma_symmetric_algebra}.
	Namely,
	construct a $\mb C$-basis of $\Sym (\mf u^-_{\bm \psi})$ such that the following products are put on the right in all monomials:
	\begin{equation}
		F_{\alpha_{\wt t_0}}^{n^{(0)}_{\wt t_0}} \dm F_{\alpha_{t'_0}}^{n^{(0)}_{t'_0}} (F_{\alpha_{\wt t_1}} \varepsilon)^{n^{(1)}_{\wt t_1}} \dm (F_{\alpha_{t'_1}} \varepsilon)^{n^{(1)}_{t'_1}} \dm \dm \bigl( F_{\alpha_{\wt t_{r-1}}} \varepsilon^{r-1} \bigr)^{\! n^{(r-1)}_{\wt t_{r-1}}} \dm \bigl( F_{\alpha_{t'_{r-1}}} \varepsilon^{r-1} \bigr)^{\! n^{(r-1)}_{t'_{r-1}}},
	\end{equation}
	where $t'_i \ceqq t_i - 1$,
	for integers $n^{(i)}_j \geq 0$.
	Any basis element is mapped to zero if and only if $n^{(i)}_j > 0$ for some $i \in \set{0,\dc,r-1}$ and $j \in \set{\wt t_i,\dc,t'_i}$.
	The statement follows,
	since by construction
	\begin{equation}
		\set{\alpha_{\wt t_i},\dc,\alpha_{t_i - 1}} = \bigl( \wt \phi_i \cap \Phi^+_{\Delta} \bigr) \bigsm \bigl( \phi_i \cap \Phi^+_{\Delta} \bigr) = \nu_i \sm \wt \nu_i,
		\qquad i \in \set{0,\dc,r-1},
	\end{equation}
	where $\wt \nu_i \ceqq \wt \psi_i \sm \wt \phi_i \sse \psi_i \sm \phi_i \eqqc \nu_i$ (extending \S~\ref{sec:parabolic_to_parabolic} from the case $r = 1$).
\end{proof}

\subsubsection{}
\label{sec:affine_wild_parabolic_to_wild_parabolic}

The arguments of \S~\ref{sec:wild_parabolic_to_wild_parabolic} extend to the affine case,
generalising \S~\ref{sec:affine_parabolic_to_affine_parabolic}.
In brief,
if we also fix a level $\kappa \in \mb C$ then there is a $U\wh{\mf g}$-linear surjection $\wh f_{\bm \lambda \mid \wt{\bm \lambda}} \cl \wh M^{\bm \psi}_{\bm \lambda,\kappa} \thra \wh M^{\wt{\bm \psi}}_{\wt{\bm \lambda},\kappa}$ with
\begin{equation}
	\ker \bigl( \wh f_{\bm \lambda \mid \wt{\bm \lambda}} \bigr) = U\wh{\mf g} \bigl( \mf u^-_{\bm \psi \mid \wt{\bm \psi}} \cdot \wh w^{\bm \psi}_{\bm \lambda,\kappa} \bigr) \sse \wh M^{\bm \psi}_{\bm \lambda,\kappa},
\end{equation}
where the subspace $\mf u^-_{\bm \psi \mid \wt{\bm \psi}} \sse \wh{\mf g}$ is defined as in~\eqref{eq:wild_parabolic_to_wild_parabolic}---%
replacing $\varepsilon$ with `$z$'.
Again,
the point is that the subalgebra $U \bigl( z^{-1}\mf g[z^{-1}] \bigr) \sse U\wh{\mf g}$ acts freely,
and then one uses the $\mb C$-basis in the proof of Thm.~\ref{thm:wild_parabolic_to_wild_parabolic}.

\subsection{From finite modules to loop-algebras modules (II)}
\label{sec:from_finite_to_loop_algebra_wild}

Here we extend \S~\ref{sec:from_finite_to_loop_algebra} to the nongeneric wild case---%
encompassing the generic one.

Let again $\bm \psi \in \mc P_\Phi^{(r)}$ be a depth-bounded parabolic filtration of $\Phi$,
and $\bm \lambda = (\lambda_0,\dc,\lambda_{r-1}) \in \mf Z_{\bm \phi}^{\dual}$ a formal type for the associated Levi factor $\bm \phi \in \mc L_\Phi^{(r)}$;
choose then a level $\kappa \in \mb C$.
There is a finite generalised singularity module $M^{\bm \psi}_{\bm \lambda}$,
sitting as a vector subspace inside the affine module $\wh M^{\bm \psi}_{\bm \lambda,\kappa}$.
The former is naturally a module for $U\bigl(\mf g \llb z \rrb \bigr)$,
with trivial action of $z^r \mf g \llb z \rrb$,
while the latter is acted on by $U \bigl( \mf g (\!(z)\!) \bigr)$.
Then:

\begin{coro}
	\label{cor:from_finite_to_loop_algebra_wild}

	There is an isomorphism of $U \bigl( \mf g (\!(z)\!) \bigr)$-modules
	\begin{equation}
		\wh M^{\bm \psi}_{\bm \lambda,\kappa} \lxra{\simeq} \Ind_{U ( \mf g \llb z \rrb )}^{U ( \mf g(\!(z)\!) )} M^{\bm \psi}_{\bm \lambda}.
	\end{equation}
\end{coro}

\begin{proof}
	It is enough to show that the canonical generator $w \in M^{\bm \psi}_{\bm \lambda} \sse \wh M^{\bm \psi}_{\bm \lambda,\kappa}$ satisfies the two properties in the statement of Lem.~\ref{lem:induction_from_submodule}:
	these follow from the fact that~\eqref{eq:generalised_affine_singularity_module} is constructed by inducing representations from~\eqref{eq:affine_singularity_algebra},
	which is a Lie subalgebra of $\mf g \llb z \rrb \ops \mb CK$.
\end{proof}

\begin{rema}
	Explicitly,
	one has
	\begin{equation}
		\Ann_{U ( \mf g (\!(z)\!) )} (w) = U \bigl( \mf g (\!(z)\!) \bigr) \wh K^{\bm \psi}_{\bm \lambda},
	\end{equation}
	where
	\begin{align}
		\label{eq:annihilator_cyclic_vector_wild_affine_generalised_verma}
		\wh K^{\bm \psi}_{\bm \lambda} \ceqq & \bops_{i \geq 0} \Bigl( [\mf l_{\phi_i},\mf l_{\phi_i}] \ops \mf{nil}(\mf p_{\psi_i}) \Bigr) \cdot z^i \\
		                                     & \ops \spann_{\mb C} \Set{ Xz^i - \Braket{ \lambda_i,X } | i \geq 0,
			\quad X \in \mf Z_{\phi_i} } \ops z^r \mf Z_{\mf g} \llb z \rrb \sse U \bigl( \mf g \llb z \rrb \bigr),
	\end{align}
	with $\bm \psi = (\psi_1,\psi_2,\dm)$ and $\bm \phi = (\bm \phi_1 = \on{Lf}(\psi_1),\phi_2 = \on{Lf}(\psi_2),\dm)$.

	In the generic case,
	where $\mf p_{\psi} = \mf b$ is a Borel subalgebra for $i \in \set{0,\dc,r-1}$,
	and $\mf p_{\psi_i} = \mf g$ for $i \geq r$,
	this simplifies to
	\begin{align}
		\wh K^{\bm \psi}_{\bm \lambda} = & \bops_{i = 0}^{r-1} \bigl( \mf{nil}(\mf b) \cdot z^i \bigr)                      \\
		                                 & \ops \spann_{\mb C} \Set{ Xz^i - \Braket{ \lambda_i,X } | i \in \set{0,\dc,r-1},
			\quad X \in \mf t} \ops z^r \mf Z_{\mf g} \llb z \rrb.
	\end{align}
	Indeed,
	in this case $\mf l_{\phi_i} = \mf Z_{\phi_i} = \mf t$ for $i \in \set{0,\dc,r-1}$,
	while $\mf l_{\phi_i} = \mf g$ for $i \geq r$.

	(Note that~\eqref{eq:annihilator_cyclic_vector_wild_affine_generalised_verma} will be used to define wild Shapovalov forms in \S~\ref{sec:shapovalov}.)
\end{rema}

\section{Simple quotients of finite singularity modules:
  general remarks}
\label{sec:simple_quotients}

\subsection{}

Choose an integer $r \geq 1$,
a parabolic filtration $\bm \psi \in \mc P^{(r)}_\Phi$,
and a formal type $\bm \lambda = (\lambda_0,\dc,\lambda_{r-1}) \in \mf Z_{\bm \phi}^{\dual}$,
where $\bm \phi = \on{Lf}_r(\bm \psi) \in \mc L^{(r)}_\Phi$ is the Levi factor.
By \S~\ref{sec:generalised_singularity_modules},
the finite generalised singularity module $M \ceqq M^{\bm \psi}_{\bm \lambda}$ can be viewed as a parabolically-induced Verma module for the TCLA $\mf g_r$ (cf.~\cite{chaffe_2023_category_o_for_takiff_lie_algebras,chaffe_topley_2023_category_o_for_truncated_current_lie_algebras}).

With a view towards the quantisation of wild orbits/de Rham spaces,
from here until \S~\ref{sec:deformation_quantisation} we study the \emph{irreducibility} of $M$.

\subsection{}

As noted in \S~\ref{sec:main_results_2},
the statement of Thm.~\ref{thm:wild_parabolic_to_wild_parabolic} is coherent with~\cite[Cor.~7.4]{wilson_2011_highest_weight_theory_for_truncated_current_lie_algebras}.
In our terminology,
the latter characterises the simple finite singularity modules in the \emph{generic} case.
Again,
on the side of (principal parts of) meromorphic $G$-connection germs this means that the leading coefficient $A_r$ of $\mc A' = \sum_{i = 1}^r A_i z^{-i}\dif z$ is a regular vector of $\mf t$.
This implies that:
(i) $\mf p_{\bm \psi} = \mf b_r \sse \mf g_r$ for some Borel subalgebra $\mf b \sse \mf g$ containing $\mf t$;
and (ii) there are \emph{no} nontrivial extensions to characters of singularity subalgebras containing $\mf b_r$.
Conversely,
Thm.~\ref{thm:wild_parabolic_to_wild_parabolic} refines the result of~\cite{wilson_2011_highest_weight_theory_for_truncated_current_lie_algebras} by providing explicit nontrivial quotients coming from the nongeneric case.\fn{
	Strictly speaking,
	the triangular setup of op.~cit.~also covers the more general monolevel case,
	where the parabolic filtration is \emph{constant} (cf.~Prop.~\ref{prop:grading_constant_case}).
	Also,
	the `nilpotency index' $N$ of op.~cit.~becomes the integer $r-1$ in this text.}

Extrapolating the results of op.~cit.,
one might think that $M$ will be simple,
at least when $r \geq 2$ and $\bm \psi \in \mc P^{(r)}_\Phi$ has maximal depth,
if $\bm \lambda\in \mf t_r^{\dual}$ lies in the stratum of $\mf t_r^{\dual}$ determined by the dual Levi filtration $\bm \phi^{\dual} \in \mc L^{(r)}_{\Phi^{\dual}}$ (cf.~\eqref{eq:dual_wild_stratum});
i.e.,
if the module is `minimal',
generalising Rmk.~\ref{rmk:minimal_modules}.
However,
the nongeneric situation is more complicated,
with additional conditions on the degree-zero part of the character,
which are best revealed by the `parabolic induction' functors of~\cite{chaffe_2023_category_o_for_takiff_lie_algebras,chaffe_topley_2023_category_o_for_truncated_current_lie_algebras}.
To state them we first classify nonsingular characters (cf.~\S~\ref{sec:nonsingular_characters}),
and then introduce wild Shapovalov forms (cf.~\S~\ref{sec:shapovalov}),
which are anyway needed for deformation quantisation (cf.~\S~\ref{sec:deformation_quantisation_our_case}).

\begin{rema}
	Intuitively,
	this mismatch in the irreducibility criterion is related with the integrability condition for characters of Lie subalgebras of $\mf g$,
	as opposed to the nilradical $\mf{bir}_r = \varepsilon \mf g_r \sse \mf g_r$.
	In the viewpoint of meromorphic gauge theory,
	this reminisces different roles played by formal residues vs.~irregular types/classes.
\end{rema}

\subsection{Truncated quotients}

A first source of quotients is the following.

For any integer $k \in \set{1,\dc,r-1}$ consider the kernel of~\eqref{eq:truncation_map},
i.e.,
the (proper) Lie-ideal $\varepsilon^k \mf g_r \sse \mf g_r$.
The left ideal $\mf I^{(k)}_r \ceqq U\mf g_r \cdot (\varepsilon^k \mf g_r) \sse U\mf g_r$ is bilateral,
and the quotient ring comes with an identification $U\mf g_r \bs \mf I^{(k)}_r \lxra{\simeq} U\mf g_k$.
Hence,
there is a corresponding $U\mf g_r$-submodule $\mf I^{(k)}_r M \sse M$,
and the quotient $M \bs \mf I^{(k)}_r M$ has a natural structure of $U\mf g_k$-module.
However,
generically,
this does \emph{not} give new examples:

\begin{prop}
	\label{prop:trivial_quotients}

	One has $\mf I^{(k)}_r M \ssne M$ if and only if the linear maps $\lambda_k,\dc,\lambda_{r-1}$ vanish on the Cartan subalgebra $\mf t \cap [\mf g,\mf g]$ of the semisimple part $[\mf g,\mf g] \sse \mf g$.
\end{prop}

\begin{proof}
	Note that $\mf I^{(k)}_r M = \mf I^{(k)}_r w \sse M$,
	where $w \ceqq w^{\bm \psi}_{\bm \lambda}$ is the canonical generator,
	since $M = U\mf g_r w$---%
	and so $\mf I^{(k)}_r$ is also a right ideal.
	By linearity,
	and by the defining equalities~\eqref{eq:irregular_state_relations},
	the submodule is thus generated over $U\mf g_r$ by elements
	\begin{equation}
		w^{(i)}_\alpha \ceqq E_{-\alpha} \varepsilon^i w,
		\qquad \alpha \in \Phi^+_{\Delta},
		\quad i \in \set{k,\dc,r-1},
	\end{equation}
	choosing a base $\Delta \sse \Phi$ of simple roots adapted to the parabolic filtration,
	and using root vectors $E_\alpha \in \mf g_\alpha \sm (0)$.
	(E.g.,
	one has $w^{(i)}_\alpha = 0$ if $\alpha \in \phi^+_i \sse \Phi^+_{\Delta}$.)
	Furthermore,
	the submodule is proper if and only if it does \emph{not} contain $w$,
	in view of the weight structure discussed in \S~\ref{sec:shapovalov}.

	Now suppose that $\eval[1]{\lambda_i}_{\mf t'} \neq 0$ for some $i \in \set{k,\dc,r-1}$,
	setting $\mf t' \ceqq \mf t \cap [\mf g,\mf g]$.
	Since $\lambda_i$ vanishes on $\ker(\phi_i) = \bops_{\phi} \mf t_\alpha \sse \mf t'$,
	there exists a root $\alpha \in \nu_i = \psi_i \sm \phi \sse \Phi^+_{\Delta}$ such that $\Braket{ \lambda_i,\alpha^{\dual} } \neq 0$,
	using the associated coroot $\alpha^{\dual} \in \mf t$.
	Thus,
	the element $E_\alpha w^{(i)}_\alpha = \Braket{ \lambda_i,H_\alpha } w \in \mf I^{(k)}_rM$ is a nonzero multiple of the canonical generator,
	setting $H_\alpha \ceqq [E_\alpha,E_{-\alpha}] \in \mf t_\alpha$ (which is parallel to $\alpha^{\dual}$).

	Conversely,
	suppose that $\eval[1]{\lambda_k}_{\mf t'} = \dm = \eval[1]{\lambda_{r-1}}_{\mf t'} = 0$.
	The nonvanishing generators $w^{(i)}_\alpha \in \mf I^{(k)}_r M$ are weight vectors for $\mf Z_{\phi_0}$,
	with weight strictly lower than $\lambda_0 \in \mf Z_{\phi_0}^{\dual}$,
	and by hypothesis
	\begin{equation}
		\mf u^+_{\bm \psi} w^{(i)}_\alpha \sse \spann_{\mb C} \Set{ w^{(j)}_\beta | \beta \in \Phi^+_{\Delta},
		\quad j \in \set{0,\dc,r-1}}.
	\end{equation}
	Hence,
	the vectors of $\mf I^{(k)}_r$ whose weight is maximal are the elements $w^{(i)}_{\theta} \neq 0$,
	where $\theta \in \Delta$ is a simple root.
	But now these are all \emph{singular} vectors,
	i.e.,
	\begin{equation}
		E_\alpha\varepsilon^j w^{(i)}_{\theta} = 0,
		\qquad \alpha \in \Phi^+_{\Delta},
		\quad j \in \set{0,\dc,r-1}.
	\end{equation}
	Hence,
	$\mf I^{(k)}_r M$ does \emph{not} meet the highest-weight line of $M$.
\end{proof}

\subsubsection{}

Thus,
to discuss simplicity,
one may assume that $\lambda_{r-1}$ does \emph{not} vanish on all the coroots,
so that in particular $\psi_{r-1} \ssne \Phi$---%
i.e.,
the parabolic filtration $\bm \psi \in \mc P^{(r)}_\Phi$ has maximal depth.
By Prop.~\ref{prop:trivial_quotients},
this is equivalent to having no nontrivial quotient of $M$ obtained by truncation.

On the side of meromorphic connection germs,
this is the same as asking that the leading semisimple coefficient $A_r$ is \emph{not} central,
so that at least one root does not vanish thereon.

\begin{rema}
	One might want to neglect the centre of $\mf g$,
	so that the above condition just reads $A_r \neq 0$.
	This is done analogously.

	Namely,
	the Lie-centre $\mf Z_{\mf g} \sse \mf g$ has a deeper analogue $\mf Z_{\mf g} \ots \mb C_r$,
	and one can show recursively that this coincides with $\mf Z_{\mf g_r} \sse \mf g_r$---%
	using~\eqref{eq:deeper_bracket}.
	Hence,
	there is an identification of Lie algebras:
	\begin{equation}
		\mf g_r \bs \mf Z_{\mf g_r} \simeq \mf g'_r,
		\qquad \mf g' \ceqq [\mf g,\mf g] \sse \mf g.
	\end{equation}
	Now $U\mf g_r \cdot \mf Z_{\mf g_r} \sse U\mf g_r$ is a bilateral ideal,
	and again we have a corresponding submodule $\wt M \sse M$ such that $M' \ceqq M \bs \wt M$ has a natural structure of module for the quotient ring $U\mf g_r \bs (U\mf g_r \cdot \mf Z_{\mf g_r}) \simeq U \mf g'_r$.
	So this is now controlled by the semisimple part of $\mf g$,
	and reasoning as above one sees that $M'$ is nontrivial if and only if $\eval[1]{\lambda_0}_{\mf Z_{\mf g}},\dc,\eval[1]{\lambda_{r-1}}_{\mf Z_{\mf g}} = 0$:
	indeed,
	in that case the centre acts trivially,
	so that actually $\wt M = (0)$;
	else,
	one has $w \in \wt M$.

	Thus,
	if one only looks at characters $\bm \chi \cl \mf S^{\bm \psi} \to \mb C$ vanishing on the central part,
	then working with semisimple Lie algebras is w.l.o.g.
	Again,
	on the side of meromorphic gauge theory this means that all coefficients $A_1,\dc,A_r$ have no central component;
	e.g.,
	they are traceless when $G = \GL_m(\mb C)$.
\end{rema}

\section{Nonsingular characters and wild KKS structures}
\label{sec:nonsingular_characters}

\subsection{}

To define the singularity modules we only induce from characters of deeper parabolic subalgebras,
i.e.,
of singularity subalgebras of given polarisation (cf.~Rmk.~\ref{rmk:deeper_characters}).
Here we characterise the nonsingular such characters,
generalising \S~\ref{sec:tame_nonsingular_characters}:
this relates with the nondegeneracy of the wild Shapovalov form of \S~\ref{sec:shapovalov},
and (in turn) to the deformation quantisation of \S~\ref{sec:deformation_quantisation_our_case}.
Furthermore,
this is intimately related with the symplectic structure on wild orbits (cf.~\eqref{eq:kks_wild_form}).

\subsubsection{}

Consider again the decomposition $\mf g_r = \mf u^-_{\bm \psi} \ops \mf l_{\bm \phi} \ops \mf u^+_{\bm \psi}$ determined by a depth-bounded parabolic filtration $\bm \psi \in \mc P^{(r)}_\Phi$ and its Levi factor $\bm \phi \in \mc L^{(r)}_\Phi$.
(Recall that in general this is \emph{not} $\mf l_{\bm \phi}$-stable,
i.e.,
one might have $\bigl[ \mf l_{\bm \phi},\mf u^\pm_{\bm \psi} \bigr] \nsse \mf u^\pm_{\bm \psi}$.)

\begin{defi}[cf.~Def.~\ref{def:nonsingular_character_tame}]
	\label{def:nonsingular_character}

	A character $\bm \chi \cl \mf l_{\phi} \to \mb C$ is \emph{nonsingular} if:
	\begin{enumerate}
		\item
		      it extends---%
		      by zero---%
		      to a character of the finite singularity algebra $\mf S^{\bm \psi} \sse \mf g_r$;\fn{
			      This happens if and only if it extends to a character of the opposite singularity algebra $\mf S^{-\bm \psi} \sse \mf g_r$.}

		\item
		      and the pairing $B_{\bm \lambda}^{\bm \psi} \cl \mf u^+_{\bm \psi} \ots \mf u^-_{\bm \psi} \to \mb C$ analogous to~\eqref{eq:pairing_nondegenerate_character} is nondegenerate.
	\end{enumerate}
\end{defi}

\begin{rema}
	Explicitly,
	the pairing of Def.~\ref{def:nonsingular_character} reads
	\begin{equation}
		\label{eq:nondegenerate_pairing_deeper_character}
		B_{\bm \lambda}^{\bm \psi}(\bm Y,\bm Y') = \sum_{k = 0}^{r-1} \Biggl( \sum_{i + j = k} \Braket{ \chi_k,\pi_{\phi_k} \bigl( [Y_i,Y'_j] \bigr) } \Biggr) \in \mb C,
	\end{equation}
	where $\bm Y = \sum_i Y_i \varepsilon^i \in \mf u^+_{\bm \psi}$ (i.e.,
	$Y_i \in \mf u^+_{\psi_i}$),
	and $\bm Y' = \sum_j Y'_j \varepsilon^j \in \mf u^-_{\bm \psi}$ (i.e.,
	$Y'_j \in \mf u^-_{\psi_j}$),
	and $\chi_k \cl \mf l_{\phi_k} \to \mb C$ is the character obtained by extending a suitable linear map $\lambda_k \in \mf Z_{\phi_k}^{\dual}$ (using~\eqref{eq:general_character_space_finite}).
\end{rema}

\subsubsection{}

Now we characterise the nonsingular characters of Def.~\ref{def:nonsingular_character},
introducing strata of $\mf t_r^{\dual} \simeq \smash{\bigl( \mf t^{\dual} \bigr)}^{\!r}$ which generalise~\eqref{eq:dual_tame_stratum} (cf.~\eqref{eq:wild_truncated_strata}).

Let $\bm \phi \in \mc L^{(r)}_{\Phi^{\dual}}$ be a depth-bounded Levi filtration in the inverse root system $\Phi^{\dual} \sse \mf t$.
Then set
\begin{equation}
	\label{eq:dual_wild_stratum}
	\mf t^{\dual,r}_{\bm \phi} \ceqq \prod_{i = 0}^{r-1} \Bigl( \phi_i^\perp \bigsm \bigcup_{\phi_{i+1} \sm \phi_i} \set{\alpha^{\dual}}^\perp \Bigr) \sse \bigl(\mf t^{\dual} \bigr)^r \simeq \mf t_r^{\dual}.
\end{equation}
(The proof that these subspaces provide a stratification is as in Lem.~\ref{lem:wild_stratification} +~Cor.~\ref{cor:truncated_wild_stratification}.)

\begin{prop}
	\label{prop:nonsingular_characters}

	Let $\bm \psi \in \mc P^{(r)}_\Phi$ be a depth-bounded parabolic filtration with Levi factor $\bm \phi \in \mc L_\Phi^{(r)}$,
	and choose a character $\bm \chi \cl \mf p_{\bm \psi} \to \mb C$.
	Then $\eval[1]{\bm \chi}_{\mf l_{\bm \phi}}$ is \emph{nonsingular} if and only if the corresponding covector $\bm \lambda = (\lambda_0,\dc,\lambda_{r-1}) \in \mf Z_{\bm \phi}^{\dual}$ lies in the stratum~\eqref{eq:dual_wild_stratum} associated with the dual Levi filtration $\bm \phi^{\dual} \ceqq (\phi_0^{\dual},\dc,\phi_{r-1}^{\dual},\Phi^{\dual},\dc) \in \mc L^{(r)}_{\Phi^{\dual}}$.
\end{prop}

\begin{proof}[Proof postponed to~\ref{proof:prop_nonsingular_characters}]
\end{proof}

\begin{lemm}
	\label{lem:nested_left_radical_condition}

	Let $\phi \sse \Phi$ be a Levi subsystem,
	and choose $Y' \in \bops_{\phi} \mf g_\alpha \sse \mf l_{\phi}$.
	Then one has
	\begin{equation}
		\Braket{ \chi,\pi_{\phi} \bigl[ Y,Y' ] \bigr) } = 0,
		\qquad Y \in \mf g,
		\quad \chi \in \Hom_{\Lie}(\mf l_{\phi},\mb C).
	\end{equation}
\end{lemm}

\begin{proof}
	By hypothesis,
	one has
	\begin{equation}
		[Y,Y'] \in  \,
		\bops_{\phi} \mf g_\alpha + \bops_{\phi \ts \Phi} [\mf g_\alpha,\mf g_\beta] \sse \mf g.
	\end{equation}

	Now the first summand lies in $\bops_\Phi \mf g_\alpha \sse \mf g$,
	hence its projection to $\mf l_{\phi}$ is annihilated by any character thereof---%
	as these are obtained by extending linear functions $\mf t \to \mb C$ \emph{by zero}.
	The second summand instead gives elements of $\mf t$ when $\alpha + \beta = 0$,
	in which case $\beta = - \alpha \in \phi$ and
	\begin{equation}
		[\mf g_\alpha,\mf g_{-\alpha}] = \mf t_\alpha \sse [\mf l_{\phi},\mf l_{\phi}] \sse \ker(\chi).
		\qedhere
	\end{equation}
\end{proof}

\begin{exem}
	\label{ex:rank_3_deeper_character}

	Consider $\mf g = \mf{gl}_3(\mb C)$,
	keeping the notation from \S~\ref{sec:rank_3_example} +~Exmp.~\ref{ex:rank_3_parabolic_subsets}.
	The root system is $\Phi = \Set{\alpha_{ij} | i \neq j \in \set{1,2,3}}$,
	and we consider the depth-2 nongeneric parabolic filtration $\bm \psi \in \mc P^{(2)}_\Phi$ given by the inclusions $\psi \sse \wt \psi \sse \Phi$,
	where
	\begin{equation}
		\psi \ceqq \set{\alpha_{12},\alpha_{13},\alpha_{23}},
		\qquad \wt \psi \ceqq \set{\alpha_{12},\alpha_{21},\alpha_{13},\alpha_{23}}.
	\end{equation}
	The corresponding (nongeneric) Levi filtration $\bm \phi \in \mc L^{(2)}_{\phi}$ is given by the sequence $\phi \sse \wt \phi \sse \Phi$,
	with $\phi = \vn$ and $\wt \phi \ceqq \phi_{\set{1,2}} = \set{\alpha_{12},\alpha_{21}}$;
	hence,
	one has $\nu = \psi$ and $\wt \nu = \set{\alpha_{13},\alpha_{23}}$.
	Thus,
	$\mf p^+_{\psi} \sse \mf g$ is the standard positive Borel subalgebra,
	while
	\begin{equation}
		\mf p^+_{\wt \psi} = \Set{
			\begin{pmatrix}
				a & b & c \\
				d & e & f \\
				0 & 0 & g
			\end{pmatrix} | a,b,c,d,e,f,g \in \mb C }.
	\end{equation}
	Now $\mf u^+_{\psi} \sse \mf p^+_{\psi}$ is the Lie subalgebra of strictly upper triangular matrices,
	while
	\begin{equation}
		\label{eq:nilradical_sl_3}
		\mf u^+_{\wt \psi} = \Set{
			\begin{pmatrix}
				0 & 0 & c \\
				0 & 0 & f \\
				0 & 0 & 0
			\end{pmatrix} | c,f \in \mb C} \sse \mf p_{\wt \psi}.
	\end{equation}
	The Levi factors are $\mf l_{\phi} = \mf t$ and
	\begin{equation}
		\mf l_{\wt \phi} = \Set{
			\begin{pmatrix}
				a & b & 0 \\
				d & e & 0 \\
				0 & 0 & g
			\end{pmatrix} | a,b,d,e,g \in \mb C} \sse \mf p_{\wt \psi}.
	\end{equation}

	Finally,
	we must choose characters $\chi \cl \mf t \to \mb C$ and $\wt \chi \cl \mf l_{\wt \psi} \to \mb C$.
	The former corresponds to a triple $\lambda = \bigl( \lambda_1,\lambda_2,\lambda_3 \bigr) \in \mb C^3$,
	via $E_{ii} \mt \lambda_i$ for $i \in \set{1,2,3}$,
	using the canonical $\mb C$-basis of the Cartan subalgebra;
	the latter to a pair $\wt \lambda = \bigl( \wt \lambda_1,\wt \lambda_2 \bigr) \in \mb C^2$,
	via $\alpha E_{11} + \beta E_{22} \mt \wt\lambda_1(\alpha + \beta)$ (for $\alpha,\beta \in \mb C$) and $E_{33} \mt \wt\lambda_2$.
	More intrinsically,
	$\wt \lambda \in \mf t^{\dual}$ is a linear map vanishing on the (opposite) coroots $\alpha_{12}^{\dual},
		\alpha_{21}^{\dual} \in \mf t$,
	i.e.,
	on the dual Levi subsystem $\wt \phi^{\dual} \sse \Phi^{\dual}$---%
	while $\phi^{\dual} = \vn$,
	so there are no conditions on $\lambda$.

	With this notation,
	the stratum corresponding to the dual Levi filtration $\bm \phi^{\dual} \in \mc L^{(2)}_\Phi$ (given by $\phi^{\dual} \sse \wt \phi^{\dual} \sse \Phi^{\dual}$) becomes
	\begin{equation}
		\mf t^{2,\dual}_{\bm \phi^{\dual}} \simeq \Set{ \bm \lambda = (\lambda,\wt \lambda) \in \mb C^5 | \wt \lambda_1 \neq \wt \lambda_2 \text{ and } \lambda_1 \neq \lambda_2}.
	\end{equation}

	Let us show that such inequalities ensure that the deeper character $\bm \chi \cl \mf l_{\bm \phi} \to \mb C$ determined by $\bm \lambda \in \mf t_2^{\dual} \simeq \bigl( \mf t^{\dual} \bigr)^2$ is nonsingular.
	Choose thus an element $\bm Y = Y_0 + Y_1 \varepsilon \in \mf u^+_{\bm \psi}$,
	i.e.,
	$Y_0 \in \mf u^+_{\psi}$ (strictly upper triangular) and $Y_1 \in \mf u^+_{\wt \psi}$ (an element of~\eqref{eq:nilradical_sl_3}).
	If $E_{ij} \ceqq E_{\alpha_{ij}} \in \mf g$ are the root vectors,
	for $i \neq j \in \set{1,2,3}$,
	we can write
	\begin{equation}
		Y_0 = a_0 E_{12} + b_0 E_{13} + c_0 E_{23},
		\quad Y_1 = b_1 E_{13} + c_1 E_{23},
		\qquad a_0,b_0,c_0,b_1,c_1 \in \mb C.
	\end{equation}
	Analogously,
	introduce an element $\bm Y' = Y'_0 + Y'_1 \varepsilon \in \mf u^-_{\bm \psi}$,
	involving coefficients $a'_0,b'_0,c'_0,b'_1,c'_1 \in \mb C$---%
	corresponding to the root vectors with indices $i > j$.
	Then we compute/impose
	\begin{align}
		\label{eq:deeper_bracket_gl_3}
		B_{\bm \lambda}^{\bm \psi}(\bm Y,\bm Y') & = \lambda_1(a_0a'_0 + b_0b'_0) + \lambda_2(c_0c'_0 - a_0a'_0) - \lambda_3(b_0b'_0 + c_0c'_0) \\
		                                         & + (\wt \lambda_1 - \wt \lambda_2) (b_0b'_1 + c_0c'_1 + b_1b'_0 + c_1c'_0) \overset ! = 0.
	\end{align}

	Now assume that $\wt \lambda_1 \neq \wt \lambda_2$.
	Then taking $a'_0 = b'_0 = c'_0 = b'_1 = 0 \neq c'_1$ in~\eqref{eq:deeper_bracket_gl_3} yields $c_0 = 0$;
	and analogously $b_0 = 0$ is forced if the roles of $b'_1$ and $c'_1$ are swapped.
	Hence,
	$Y_0 = a_0 E_{12} \in \ker \bigl( \pi^+_{\wt \psi} \bigr)$,
	or equivalently $Y_0 \in \mf u^+_{\psi | \wt \psi} = \mf g_{12} \ceqq \mf g_{\alpha_{12}}$ (As $\nu \sm \wt \nu = \set{ \alpha_{12}}$,
	cf.~the base of the first recursion in the proof~\ref{proof:prop_nonsingular_characters}.)

	Therefore,
	the expression~\eqref{eq:deeper_bracket_gl_3} simplifies to
	\begin{equation}
		B_{\bm \lambda}^{\bm \psi}(\bm Y,\bm Y') = (\lambda_1 - \lambda_2)a_0a'_0 + (\wt \lambda_1 - \wt \lambda_2) (b_1b'_0 + c_1c'_0) \overset ! = 0.
	\end{equation}
	Choose then $a'_0 = b'_0 = 0 \neq c'_0$,
	finding $b_1 = 0$;
	while $a'_0 = c'_0 = 0 \neq b'_0$ yields $c_1 = 0$.
	(Cf.~the first recursive step in~\ref{proof:prop_nonsingular_characters}.)

	Thus,
	we are left with $B_{\bm \lambda}^{\bm \psi}(\bm Y,\bm Y') = (\lambda_1 - \lambda_2)a_0a'_0 \overset ! = 0$,
	and finally we assume that $\lambda_1 \neq \lambda_2$ to conclude that $a_0 = 0$ as well---%
	taking $a'_0 \neq 0$.
	Conversely,
	it is possible to find a nonvanishing vector in the left radical if either $\wt \lambda_1 = \wt \lambda_2$ or $\lambda_1 = \lambda_2$.
\end{exem}

\subsection{Relation with Poisson/symplectic structures}
\label{sec:deeper_character_and_symplectic_form}

Let us reinterpret Prop.~\ref{prop:nonsingular_characters} in terms of the symplectic geometry of coadjoint orbits in dual TCLAs,
generalising \S~\ref{sec:tame_nonsingular_characters}.

Consider again the element $\bm \lambda \in \mf t_r^{\dual}$,
i.e.,
the restriction of a character $\chi \cl \mf p^+_{\bm \psi} \to \mb C$.
By the vector-space splitting $\mf g_r = \mf t_r \ops \bops_\Phi (\mf g_\alpha \ots \mb C_r)$,
identify $\lambda$ with an element of $\mf g_r^{\dual}$ which vanishes on the off-Cartan part,
and take its ($r$-semisimple) coadjoint orbit $\mc O_{\bm \lambda} \ceqq G_r \cdot \bm \lambda \sse \mf g_r^{\dual}$.
(Hereafter,
this extension is tacit whenever considering coadjoint $G_r$-orbits intersecting $\mf t_r^{\dual}$.)

Introduce the coadjoint stabiliser $G_r^{\bm \lambda} \sse G_r$,
and its infinitesimal version $\mf g_r^{\bm \lambda} = \Lie \bigl( G_r^{\bm \lambda} \bigr) \sse \mf g_r$ (cf.~Rmk.~\ref{rmk:dual_version}).
These are the same as the centraliser $G_r^{\bm X}$ of $\bm X \ceqq ( \cdot | \cdot)_r^\sharp(\bm \lambda) \in \mf g_r$,
using the $G_r$-invariant pairing of Lem./Def.~\ref{lem:deeper_pairing},
and its infinitesimal centraliser $\mf g_r^{\bm X} = \Lie \bigl( G_r^{\bm X} \bigr)$---%
respectively.
In particular,
the former is connected (cf.~\S~\ref{sec:wild_orbit_strata}).
Consider then the evaluation of (the linear part of) the Poisson bracket of $\mf g_r^{\dual}$ at $\bm \lambda$:
\begin{equation}
	\label{eq:deeper_poisson_bracket}
	\mf g_r \wdg \mf g_r \lra \mb C,
	\qquad \bm Y \wdg \bm Y' \lmt \Braket{ \bm \lambda,[\bm Y,\bm Y'] } = \Set{ \bm Y,\bm Y' }(\bm \lambda),
	\qquad \bm Y,\bm Y' \in \mf g_r \simeq \mf g_r^{\dual\!\!\dual}.
\end{equation}
The KKS symplectic structure on $\mc O_{\bm \lambda} \simeq G_r \bs G_r^{\bm \lambda}$,
evaluated at $T_{\bm \lambda} \mc O_{\bm \lambda} \simeq \mf g_r \bs \mf g_r^{\bm \lambda}$,
is obtained upon modding out the radical of~\eqref{eq:deeper_poisson_bracket}.
Then note that:

\begin{lemm}
	\label{lem:reducing_poisson_bracket}

	One has $\mf l_{\bm \phi} \sse \mf g_r^{\bm \lambda}$ and $\bigl[ \mf u^\pm_{\bm \psi},\mf u^\pm_{\bm \psi} \bigr] \sse \ker(\bm \lambda)$.
\end{lemm}

\begin{proof}
	Postponed to~\ref{proof:lem_reducing_poisson_bracket}.
\end{proof}

\subsubsection{}

By Lem.~\ref{lem:reducing_poisson_bracket},
the evaluation of the Poisson bracket~\eqref{eq:deeper_poisson_bracket} yields a well-defined map $\bigwedge^2 \bigl( \mf g_r \bs \mf l_{\bm \phi} ) \to \mb C$,
i.e.,
a $\mb C$-bilinear alternating form on the vector-space complement $\mf u^+_{\bm \psi} \ops \mf u^-_{\bm \psi} \sse \mf g_r$.
But
\begin{equation}
	\bigwedge^2 \bigl( \mf u^+_{\bm \psi} \ops \mf u^-_{\bm \psi} \bigr) \simeq \bigwedge^2 \mf u^+_{\bm \psi} \ops \bigl( \mf u^+_{\bm \psi} \ots \mf u^-_{\bm \psi} \bigr) \ops \bigwedge^2 \mf u^-_{\bm \psi},
\end{equation}
and (by the same lemma) it is only the component inside $\bigl( \mf u^+_{\bm \psi} \ops \mf u^-_{\bm \psi} \bigr)^{\! \dual}$ which does \emph{not} vanish:
this now tautologically coincides with the pairing of Def.~\ref{def:nonsingular_character}.

We are now in the position to prove the:

\begin{theo}
	\label{thm:nondegenerate_character_and_symplectic_form}

	The following are equivalent:
	\begin{enumerate}
		\item the pairing~\eqref{eq:nondegenerate_pairing_deeper_character} is the restriction of the KKS symplectic form of $\mc O_{\bm \lambda}$,
		      evaluated at the marking,
		      onto the subspace $\mf u^+_{\bm \psi} \ots \mf u^-_{\bm \psi} \sse \bigwedge^2 \bigl( \mf u^+_{\bm \psi} \ops \mf u^-_{\bm \psi} \bigr)$;

		\item the formal type $\bm \lambda$ lies in the stratum $\mf t^{r,\dual}_{\bm \phi^{\dual}}$,
		      using the dual Levi filtration $\bm \phi^{\dual} \in \mc L_{\Phi^{\dual}}^{(r)}$ (cf.~\eqref{eq:dual_wild_stratum});

		\item and $\mf l_{\bm \phi} = \mf g_r^{\bm \lambda}$.
	\end{enumerate}
\end{theo}

\begin{proof}
	The equivalence between the first and the third item is now clear,
	as the latter states that $\mf l_{\bm \phi}$ coincides with the radical of~\eqref{eq:deeper_poisson_bracket}.
	Furthermore,
	by the above discussion,
	this is the same as asking that~\eqref{eq:nondegenerate_pairing_deeper_character} be nondegenerate:
	so the equivalence between the second and the third item is a rewriting of Prop.~\ref{prop:nonsingular_characters}.
\end{proof}

\begin{rema}
	We can turn the statement of Thm.~\ref{thm:nondegenerate_character_and_symplectic_form} around,
	finding an explicit formula for the symplectic structure on polarised $r$-semisimple wild orbits.

	Namely,
	choose an integer $r \geq 1$ and an element $\bm \lambda \in \mf t_r^{\dual}$.
	Let then $\bm \phi = (\phi_0,\dc,\phi_{r-1},\Phi,\dc) \sse \mc L^{(r)}_{\bm \Phi^{\dual}}$ be the Levi filtration associated with $\bm \lambda$,
	i.e.,
	\begin{equation}
		\phi_i = \phi_i(\bm \lambda) \ceqq \Set{ \alpha^{\dual} \in \Phi^{\dual} | \Braket{ \lambda_i,\alpha^{\dual} } = \dm = \Braket{ \lambda_{r-1},\alpha^{\dual} } = 0 } \in \mc L_{\Phi^{\dual}}.
	\end{equation}
	(Cf.~\eqref{eq:levi_filtration_irregular_type},
	for $s = r$,
	which is on the other side of the index-swapping duality~\eqref{eq:deeper_duality}.)

	Then $\bm \lambda$ lies in the corresponding stratum of $\mf t_r^{\dual}$.
	Choosing any opposite pair $\pm \bm \psi \in \mc P^{(r)}_\Phi$ of parabolic filtrations with Levi factor $\bm \phi^{\dual} \in \mc L_\Phi^{(r)}$ yields
	\begin{equation}
		\label{eq:kks_wild_form}
		\omega_{\bm \lambda}( \bm Y \wdg \bm Y') = B_{\bm \lambda}^{\bm \psi}( \bm Y,\bm Y' ) = \sum_{k = 0}^{r-1} \Biggl( \sum_{i+j = k} \Braket{ \lambda_k,[Y_i,Y'_j] } \Biggr),
		\qquad \bm Y \in \mf u^+_{\bm \psi},
		\quad \bm Y' \in \mf u^-_{\bm \psi},
	\end{equation}
	and this uniquely encodes the evaluation $\omega_{\bm \lambda}$ of the KKS symplectic form $\omega$.
\end{rema}

\begin{rema}
	There is a last relevant consequence of Lem.~\ref{lem:reducing_poisson_bracket}.
	Define the connected subgroup $L_{\bm \phi} \sse G_r$ as in Thm.~\ref{thm:pure_isomorphism_wild} (it only depends on the Levi filtration):
	then $L_{\bm \phi} \sse G_r^{\bm \lambda}$ whenever $\bm \lambda \in \mf Z^{\dual}_{\bm \phi} \sse \mf t_r^{\dual}$,
	and they coincide if and only if $\bm \lambda$ lies in the stratum determined by $\bm \phi$.
	In particular,
	the pairing~\eqref{eq:nondegenerate_pairing_deeper_character},
	which is a reduction of the Poisson bracket evaluation~\eqref{eq:deeper_poisson_bracket},
	is $L_{\bm \phi}$-\emph{invariant:}
	this is despite the fact that the projection $\pi_{\bm \phi} \cl \mf g_r \to \mf l_{\bm \phi}$ is \emph{not} $\mf l_{\bm \phi}$-equivariant,
	and it is implicitly used in \S~\ref{sec:shapovalov}.
\end{rema}

\section{Wild Shapovalov forms}
\label{sec:shapovalov}

\subsection{}

Here we introduce Shapovalov forms on finite generalised singularity modules,
and prove a simplicity criterion extending~\cite{wilson_2011_highest_weight_theory_for_truncated_current_lie_algebras,felder_rembado_2023_singular_modules_for_affine_lie_algebras_and_applications_to_irregular_wznw_conformal_blocks} from the generic case (cf.~\S~\ref{sec:tame_shapovalov} for the standard setup).

\subsubsection{}

Choose as usual an integer $r \geq 1$,
and a depth-bounded parabolic filtration $\bm \psi \in \mc P^{(r)}_\Phi$ with Levi factor $\bm \phi \in \mc L^{(r)}_\Phi$.
In a suitable choice of PBW bases,
the decomposition $\mf g_r = \mf u^-_{\bm \psi} \ops \mf l_{\bm \phi} \ops \mf u^+_{\bm \psi}$ yields a $\mb C$-linear splitting
\begin{equation}
	\label{eq:universal_enveloping_triangular_decomposition_wild}
	U\mf g_r
	= U\mf l_{\bm \phi} \ops \bigl( \mf u^-_{\bm \psi} \cdot U\mf g_r + U\mf g_r \cdot \mf u^+_{\bm \psi} \bigr),
\end{equation}
whence a projection $\pi_{\bm \phi} \cl U\mf g_r \thra U\mf l_{\bm \phi}$ parallel to the rightmost direct summand.

Select again a pinning,
i.e.,
nonvanishing root-vectors $E_\theta \in \mf g_\theta$ corresponding to a choice of base $\theta \in \Delta \sse \Phi$.
Then extend~\eqref{eq:transposition_involution} to an involutive antiautomorphism $\mf g_r \lxra{\simeq} \mf g_r^{\op}$,
in $\varepsilon$-graded fashion:
\begin{equation}
	\label{eq:wild_transposition_involution}
	\ps{t\!}{}{\bigl(X \varepsilon^i\bigr)} \ceqq \ps{t}{}X \cdot \varepsilon^i,
	\qquad \bm X \in \mf g_r,
	\quad i \in \set{0,\dc,r-1}.
\end{equation}
This stabilises $\mf l_{\bm \phi}$ and swaps $\mf u^\pm_{\bm \psi}$,
and it yields an involution $U\mf g_r \lxra{\simeq} (U\mf g_r)^{\op}$.

Then the (generalised) \emph{wild} universal Shapovalov pairing is
\begin{equation}
	\mc S^{\bm \psi} \cl U\mf g_r \ots U\mf g_r \lra U\mf l_{\bm \phi},
	\qquad \bm X \ots \bm X' \lmt \pi_{\bm \phi}\bigl( \ps{t}{}{\bm X} \cdot \bm X' \bigr).
\end{equation}
By construction,
it is transpose-contragredient.

\subsubsection{}

As in \S~\ref{sec:nonsingular_characters},
choose a character $\chi \in \Hom_{\Lie}(\mf l_{\bm \phi},\mb C)$ which extends to the finite singularity subalgebra $\mf S^{\bm \psi} = \mf p^+_{\bm \psi}$,
i.e.,
equivalently,
an element $\bm \lambda = (\lambda_0,\dc,\lambda_{r-1}) \in \mf Z_{\bm \phi}^{\dual} \sse \mf t_r^{\dual}$.
Extend the character to a ring morphism $U\mf l_{\bm \phi} \to \mb C$,
and define the (generalised) \emph{wild} Shapovalov form as the composition
\begin{equation}
	\label{eq:wild_shapovalov_form}
	\mc S^{\bm \psi}_{\bm \lambda} \ceqq \bm \chi \circ \mc S^{\bm \psi} \cl U\mf g_r \ots U\mf g_r \lra \mb C.
\end{equation}

One can show that~\eqref{eq:wild_shapovalov_form} is \emph{symmetric},
as in the proof~\ref{proof:lem_shapovalov_is_symmetric}:
if $\bm Y' = \sum_i P_i \in U\mf l_{\bm \phi}$ then all monomials appearing in $\bm Y' - \ps{t}{}{\bm Y'}$ must contain one factor from the subspace $\bops_{i = 0}^{r-1} \bigl( \bops_{\phi_i} \mf g_\alpha \bigr) \cdot \varepsilon^i \sse \mf l_{\bm \phi}$,
on which the character vanishes.

Consider now the finite generalised singularity module $M \ceqq M^{\bm \psi}_{\bm \lambda}$.
If $w \ceqq w^{\bm \psi}_{\bm \lambda}$ is the canonical generator,
there is an identification
\begin{equation}
	U\mf g_r \bs \Ann_{U\mf g_r} (w) \lxra{\simeq} M,
	\qquad \bm X + \Ann_{U\mf g_r} (w) \lmt \bm X w,
\end{equation}
of $U\mf g_r$-modules;
and truncating~\eqref{eq:annihilator_cyclic_vector_wild_affine_generalised_verma} yields $\Ann_{U\mf g_r} (w) = U\mf g_r \cdot K^{\bm \psi}_{\bm \lambda}$,
with
\begin{equation}
	K^{\bm \psi}_{\bm \lambda} \ceqq \bops_{i = 0}^{r-1} \Bigl( \bigl[ \mf p^+_{\psi_i},\mf p^+_{\psi_i} \bigr] \cdot \varepsilon^i \Bigr) \ops \spann_{\mb C} \Set{ X \varepsilon^i - \Braket{ \lambda_i,X } | i \in \set{0,\dc,r-1},
		\quad X \in \mf Z_{\phi_i} }.
\end{equation}

\begin{prop}
	\label{prop:annihilator_cyclic_vector_in_wild_shapovalov_radical}
	One has
	\begin{equation}
		\mc S^{\bm \psi}_{\bm \lambda}(\bm X,\bm X') = 0,
		\qquad \bm X \in U\mf g_r,
		\quad \bm X' \in K^{\bm \psi}_{\bm \lambda}.
	\end{equation}
\end{prop}

\begin{proof}
	Reasoning as in the proof~\ref{proof:lem_annihilator_cyclic_vector_in_shapovalov_radical},
	it remains to show that
	\begin{equation}
		\mc S^{\bm \psi}_{\bm \lambda} \bigl( \bm X,[\bm Y,\bm X'] \bigr) = 0,
		\qquad \bm X \in U\mf g_r,
		\quad \bm Y \in \mf u^+_{\bm \psi},
		\quad \bm X' \in K^{\bm \phi}_{\bm \lambda},
	\end{equation}
	where $K^{\bm \phi}_{\bm \lambda} \sse K^{\bm \psi}_{\bm \lambda}$ is defined by
	\begin{equation}
		K^{\bm \phi}_{\bm \lambda} \ceqq \bops_{i = 0}^{r-1} \Bigl( \bigl[ \mf l_{\phi_i},\mf l_{\phi_i} \bigr] \cdot \varepsilon^i \Bigr) \ops \spann_{\mb C} \Set{ X \varepsilon^i - \Braket{ \lambda_i,X } | i \in \set{0,\dc,r-1},
			\quad X \in \mf Z_{\phi_i} }.
	\end{equation}
	As explained above,
	in the wild nongeneric case one might have $\bigl[ \mf l_{\bm \phi},\mf u^\pm_{\bm \psi} \bigr] \nsse \mf u^\pm_{\bm \psi}$,
	so we conclude by a direct inspection of the generators of $K^{\bm \phi}_{\bm \lambda}$.

	Suppose first that $\bm X' = X \varepsilon^i - \Braket{ \lambda_i,X \varepsilon^i } \in U\mf l_{\bm \phi}$,
	for some $i \in \set{0,\dc,r-1}$ and $X \in \mf Z_{\phi_i}$.
	Then $[\bm Y,\bm X'] \in \mf u^+_{\bm \psi}$,
	since $[\mf t_r \ops \mb C,\mf u^\pm_{\psi}] \sse \mf u^\pm_{\psi}$.
	The same happens if one chooses $\bm X' \in \bops_{i = 0}^{r-1} \bigl( \bops_{\phi_i} \mf t_\alpha \bigr) \cdot \varepsilon^i \sse \mf t_r$.

	Suppose instead that $\bm X' \in \bops_{i = 0}^{r-1} \bigl( \bops_{\phi_i} \mf g_\alpha \bigr) \cdot \varepsilon^i \sse \mf g_r$.
	(This is the troublesome case.)
	If in particular $\bm X = 1$,
	then one must prove that $\braket{ \bm \chi,\pi_{\bm \phi} \bigl( [\bm Y,\bm X'] \bigr) } = 0$,
	and by (bi)linearity it is enough to do so for
	\begin{equation}
		\bm Y = E_\alpha \varepsilon^i,
		\quad \bm X' = E_\beta \varepsilon^j,
		\qquad i,j \in \set{0,\dc,r-1},
		\quad \alpha \in \nu_i,
		\quad \beta \in \phi_j,
	\end{equation}
	with $i+j \eqqc k \leq r-1$.
	Now $\alpha + \beta \neq 0$ yields $[E_\alpha,E_\beta] \varepsilon^k \in \bops_\Phi \bigl( \mf g_\gamma \cdot \varepsilon^k \bigr)$,
	which is annihilated by $\bm \chi$---%
	after projection;
	else $\alpha + \beta = 0$,
	so that $\alpha = -\beta \in \phi_j \sse \phi_k$ and
	\begin{equation}
		[E_\alpha,E_\beta] \varepsilon^k \in \bops_{\phi_k} \bigl( \mf t_\gamma \cdot \varepsilon^k \bigr) \sse \ker(\bm \chi).
	\end{equation}
	Finally,
	choose an integer $m \geq 0$ and suppose recursively that
	\begin{equation}
		\mc S^{\bm \psi}_{\bm \lambda}(\bm X,\bm X') = 0,
		\qquad \bm X \in U^{\leq m} \mf g_r,
		\quad \bm X' \in K^{\bm \phi}_{\bm \lambda},
	\end{equation}
	using the standard (transpose-invariant) filtration $U^{\leq \bullet} \mf g_r$ of the UEA.
	The missing recursive step arises from
	\begin{equation}
		\bm X \in U^{\leq m+1} \mf g_r,
		\quad \bm X' = E_\beta \varepsilon^j,
		\qquad j \in \set{0,\dc,r-1},
		\quad \beta \in \phi_j.
	\end{equation}
	Repeating the splitting argument of~\ref{proof:lem_annihilator_cyclic_vector_in_shapovalov_radical},
	it is enough to show that
	\begin{equation}
		\mc S^{\bm \psi}_{\bm \lambda} \bigl( \wt{\bm X},
		[E_\alpha,E_\beta] \varepsilon^k \bigr) = 0,
		\qquad \wt{\bm X} \in U^{\leq m} \mf g_r,
		\quad \alpha \in \nu_i,
		\quad i \in \set{0,\dc,r-1},
	\end{equation}
	where again $i+j = k \leq r-1$.
	By hypothesis,
	one has $\alpha,\beta \in \psi_k$,
	whence $\alpha + \beta \in \psi_k$ if $\alpha + \beta$ is a root,
	since $\psi_k \sse \Phi$ is a \emph{closed} subset of roots.
	In particular,
	if $\alpha + \beta \in \Phi$ one either has $\alpha + \beta \in \phi_k$,
	in which case we conclude by the recursive hypothesis;
	or one has $\alpha + \beta \in \nu_k$,
	in which case $[E_\alpha,E_\beta] \varepsilon^k \in \mf u^+_{\bm \psi}$---%
	so that $\pi_{\bm \phi} \bigl( \wt X \cdot E_{\alpha+\beta} \varepsilon^k \bigr) = 0$.
\end{proof}

\subsubsection{}

Using transpose-contragrediency,
Prop.~\ref{prop:annihilator_cyclic_vector_in_wild_shapovalov_radical} yields a reduced Shapovalov form $M \ots M \to \mb C$,
abusively written the same and (well) defined by
\begin{equation}
	\bm X w \ots \bm X' w \lmt \mc S^{\bm \psi}_{\bm \lambda}(\bm X,\bm X'),
	\qquad \bm X,\bm X' \in U\mf g_r.
\end{equation}
By construction,
it is symmetric,
and it satisfies $\mc S^{\bm \psi}_{\bm \lambda} (w,w) = 1$.

Now consider the largest Levi centre $\mf Z_{\phi_0} \sse \mf l_{\phi_0}$,
lying inside $\mf g \sse \mf g_r$.
It acts on $\mf g_r$ in the adjoint representation,
in semisimple/$\mb C_r$-linear fashion.
Precisely,
the $\mb C$-linear decomposition $\mf g = \mf l_{\phi_0} \ops \bops_{\Phi \sm \phi_0} \mf g_\alpha$ extends to its deeper version:
\begin{equation}
	\label{eq:deeper_root_splitting}
	\mf g_r = (\mf l_{\phi_0})_r \ops \bops_{\Phi \sm \phi_0} (\mf g_r)_\alpha,
	\qquad (\mf g_r)_\alpha \ceqq (\mf g_\alpha)_r = \mf g_\alpha \ots \mb C_r \sse \mf g_r.
\end{equation}
It follows that $U\mf g_r$ is also a weight $\mf Z_{\phi_0}$-module,
with weight decomposition
\begin{equation}
	\label{eq:weight_decomposition_universal_enveloping_algebra_deeper}
	U\mf g_r = \bops_{Q_{\phi_0}} (U\mf g_r)_\mu,
	\qquad (U\mf g_r)_\mu \ceqq \Set{ \bm X \in U\mf g_r | [Y,\bm X] = \braket{ \mu,Y } \bm X \text{ for } Y \in \mf Z_{\phi_0} },
\end{equation}
where again $Q_{\phi_0} \sse Q = \spann_{\mb Z}(\Phi)$ is the abelian subgroup generated by $\Phi \sm \phi_0$.
(The weight spaces are larger,
as each root-vector can be put at different $\varepsilon$-degree,
cf.~\S~\ref{sec:pbw_bases_weight_spaces}.)

Finally,
as in the tame case,
the finite generalised singularity module is a weight $\mf Z_{\phi_0}$-module,
with weight spaces $M[\mu] \sse M$ parameterised by the positive submonoid $Q^+_{\phi_0} = \spann_{\mb Z_{\geq 0}} (\nu_0) \sse Q_{\phi_0}$.\fn{
	One could in principle consider the weight-space decomposition with respect to some other (smaller) Levi centre $\mf Z_{\phi_i} \sse \mf t$,
	but then in general $\mb C w \ssne M[0]$.}
More precisely,
the active weights lie in $\lambda_0 - Q^+_{\phi_0} \sse \mf Z_{\phi_0}^{\dual}$ (cf.~\eqref{eq:weight_decomposition_generalised_verma}),
and the canonical generator spans the highest-weight line---%
of \emph{highest} weight $\lambda_0 \in \mf Z_{\phi_0}^{\dual}$.
This follows from the $\mb C$-linear isomorphism $\Sym (\mf u^-_{\bm \psi}) \lxra{\simeq} M$ of \S~\ref{sec:identification_verma_symmetric_algebra},
as $\nu_i \sse \nu_0$ for all $i \in \set{0,\dc,r-1}$ (cf.~Rmk.~\ref{rmk:twisted_adjoint_action}).
(Observe that the roots $\alpha \in \nu_0$ vanish on $\bops_{\phi_0} \mf t_\alpha$,
whence indeed $Q^+_{\phi_0} \sse \mf Z_{\phi_0}^{\dual}$.)

Moreover,
the canonical generator is acted on by $U\mf l_{\bm \phi}$ in semisimple fashion,
and it is annihilated by the whole of $U\mf g_r \cdot \mf u^+_{\bm \psi}$;
but the main point here is that some standard highest-weight representation theory carries over:

\begin{enonce}{Lemma/Definition}
	\label{lem:maximal_proper_submodule_wild}

	The finite generalised singularity module $M$ is \emph{indecomposable},
	and it contains a \emph{unique} maximal proper (weight) submodule
	\begin{equation}
		\label{eq:maximal_proper_submodule_wild}
		N = N^{\bm \psi}_{\bm \lambda} \ceqq \sum_{\substack{U\mf g_r N' \sse N' \ssne M}} N'.
	\end{equation}
\end{enonce}

\begin{proof}
	Postponed to~\ref{proof:lem_maximal_proper_submodule_wild}.
\end{proof}

\begin{rema}
	\label{rmk:twisted_adjoint_action}
	The identification~\eqref{eq:direct_sum_splitting_verma} is more than a $\mb C$-linear map.

	Indeed,
	note that $\Sym (\mf u^-_{\bm \psi})$ is naturally a weight $\mf Z_{\phi_0}$-module in the adjoint action,
	and the corresponding decomposition reads
	\begin{equation}
		\label{eq:grading_symmetric_algebra_nilradical}
		\Sym (\mf u^-_{\bm \psi}) = \bops_{Q^+_{\phi_0}} \Sym (\mf u^-_{\bm \psi})_{-\mu},
	\end{equation}
	as for~\eqref{eq:weight_decomposition_universal_enveloping_algebra_deeper},
	but now only along the submonoid $Q^+_{\phi_0} \sse Q_{\phi_0}$.
	Then the embedding $\Sym(\mf u^-_{\bm \psi}) \hra U\mf g_r$,
	obtained from a choice of (ordered) basis for $\mf u^-_{\bm \psi}$,
	is $\mf Z_{\phi_0}$-linear,
	and $(U\mf g_r)_{-\mu} w \sse M$ lies in the subspace of weight $\lambda_0 - \mu$:
	in particular,
	there are $\mb C$-linear isomorphisms $\Sym (\mf u^-_{\bm \psi})_{-\mu} \lxra{\simeq} M[\mu]$,
	and a corresponding vector-space splitting $M = \bops_{Q^+_{\phi_0}} M[\mu]$---%
	which is then intrinsically determined.
\end{rema}

\subsubsection{}

By Lem.~\ref{lem:maximal_proper_submodule_wild},
the \emph{unique} simple quotient of $M$ is $L = L^{\bm \psi}_{\bm \lambda} \ceqq M \bs N$.
Finally,
one can use the Shapovalov form to identify $N$,
after establishing the orthogonality properties of the weight-space decomposition,
which are a weaker version of the tame case:

\begin{prop}
	\label{prop:orthogonality_wild_shapovalov}

	Choose $\bm X \in (U\mf g_r)_\mu$ and $\bm X' \in (U\mf g_r)_{\mu'}$,
	with $\mu \neq \mu' \in Q_{\phi_0}$.
	Then one has $\mc S^{\bm \psi}_{\bm \lambda} (\bm X,\bm X') = 0$.
\end{prop}

\begin{proof}
	Choosing elements as in the statement yields
	\begin{equation}
		\ps{t}{}{\bm X} \in (U\mf g_r)_{-\mu},
		\qquad \ps{t}{}{\bm X} \cdot \bm X' \in (U\mf g_r)_{\mu' - \mu},
	\end{equation}
	as the transpose sends $(\mf g_r)_\alpha$ to $(\mf g_r)_{-\alpha}$,
	for $\alpha \in \Phi$.
	Hence,
	it is enough to prove that
	\begin{equation}
		\Braket{ \bm \chi,\pi_{\bm \phi}(\bm Y) } = 0,
		\qquad \bm Y \in (U\mf g_r)_\mu,
		\quad \mu \in Q_{\phi_0} \sm (0).
	\end{equation}
	Decompose now $\bm Y = \bm Y' + \bm Y''$ along the direct sum~\eqref{eq:universal_enveloping_triangular_decomposition_wild}:
	then $\bm Y' \in U \mf l_{\bm \phi} \cap (U\mf g_r)_\mu$,
	since both vector subspaces are $\mf Z_{\phi_0}$-submodules,
	while $\bm Y''$ (which is also a weight vector) is annihilated by $\pi_{\bm \phi}$.
	Write then $\bm Y' = \sum_i P_i$ as a finite sum of monomials
	\begin{equation}
		P_i = X_1^{(i)} \dm X_{m_i}^{(i)},
		\qquad m_i \in \mb Z_{\geq 0},
		\quad X^{(i)}_j \in \mf l_{\bm \phi},
	\end{equation}
	and assume---%
	w.l.o.g.---%
	that these monomials are linearly independent.
	Acting with an element $X \in \mf Z_{\phi_0} \sm (0)$ then yields
	\begin{equation}
		\ad_X(\bm Y') = \braket{ \mu,X } \bm Y' = \sum_i \braket{ \mu,X } P_i,
	\end{equation}
	as well as
	\begin{equation}
		\ad_X(\bm Y') = \sum_i \ad_X (P_i) = \sum_i \braket{ \mu_i,X } P_i,
	\end{equation}
	for suitable covectors $\mu_i \in Q_{\phi_0}$,
	because all elements of $\mf l_{\bm \phi}$ are weight vectors.
	Therefore,
	one has $\mu_i = \mu \neq 0$ for all $i$,
	which implies that each monomial $P_i$ contains a factor from $\bops_{i = 0}^{r-1} \bigl( \bops_{\phi_i} \mf g_\alpha \cdot \varepsilon^i \bigr) \sse \mf l_{\bm \phi}$,
	since $\mf Z_{\phi_0}$ acts trivially on the complementary subspace $\mf t_r \sse \mf l_{\bm \phi}$.
	But then
	\begin{equation}
		\Braket{ \bm \chi,\pi_{\bm \phi}(\bm Y') } = \Braket{ \bm \chi,\bm Y' } = \sum_i \Braket{ \bm \chi,P_i } = \sum_i \prod_{j = 1}^{m_i} \Braket{ \chi,X_j^{(i)} } = 0,
	\end{equation}
	as $\bops_{i = 0}^{r-1} \bigl( \bops_{\phi_i} \mf g_\alpha \cdot \varepsilon^i \bigr) \sse \ker(\bm \chi)$.
\end{proof}

\begin{rema}
	In the tame case $r = 1$ one had the stronger statement that the weight spaces $(U\mf g)_\mu \sse U\mf g$ are orthogonal for the \emph{universal} Shapovalov pairing $\mc S^{\psi}$,
	which is false in general in our extended setting.

	The point is that if $r \geq 2$ then there may be elements of $U\mf l_{\bm \phi}$ with nonvanishing $\mf Z_{\phi_0}$-weight (suffices,
	e.g.,
	to choose $\bm X = E_\alpha \varepsilon$ with $\alpha \in \phi_1 \sm \phi_0$).
\end{rema}

\subsubsection{}

It follows from Prop.~\ref{prop:orthogonality_wild_shapovalov} that the weight-space decomposition of the finite singularity modules is by $\mc S^{\bm \psi}_{\bm \lambda}$-\emph{orthogonal} subspaces.
Then one can state the analogous characterisation of the maximal proper submodule in the generalised wild case:

\begin{theo}
	\label{thm:shapovalov_radical_equals_maximal_proper_submodule_wild}

	One has $N = \Rad \bigl( \mc S^{\bm \psi}_{\bm \lambda} \bigr)$,
	in the notation of~\eqref{eq:maximal_proper_submodule_wild}.
\end{theo}

\begin{proof}
	By contragrediency,
	the radical of the Shapovalov form is a submodule,
	and it is proper---%
	since it does \emph{not} contain $w$.
	Thus,
	$\Rad \bigl( \mc S^{\bm \psi}_{\bm \lambda} \bigr) \sse N$.

	Conversely,
	choose $\bm X \in U\mf g_r$ and $\wt w \in N$.
	Then
	\begin{equation}
		\mc S^{\bm \psi}_{\bm \lambda} \bigl( \ps{t}{}{\bm X} w,\wt w \bigr) = \mc S^{\bm \psi}_{\bm \lambda} (w,
		\bm X \wt w ) = 0,
	\end{equation}
	as $\bm X \wt w \in N \sse \bops_{Q^+_{\phi_0}} M^{\bm \psi}_{\bm \lambda}[\mu]$,
	using Prop.~\ref{prop:orthogonality_wild_shapovalov}.
\end{proof}

\subsubsection{}

Thus,
in principle,
one may now consider the restrictions of $\mc S^{\bm \psi}_{\bm \lambda}$ to the finite-dimensional subspaces
\begin{equation}
	M[\mu] \ots M[\mu] \sse M \ots M,
	\qquad \mu \in Q^+_{\phi_0},
\end{equation}
and study the corresponding countable family of radicals;
by the above,
they coincide with the weight subspaces $N[\mu] = N \cap M[\mu]$ of $N$---%
of weight $\lambda_0 - \mu \in \mf Z_{\phi_0}^{\dual}$.

Concretely,
choosing suitable $\mb C$-bases of the weight spaces,
simplicity is governed by determinant computations à la Shapovalov/Kac--Kazhdan~\cite{shapovalov_1972_a_certain_bilinear_form_on_the_universal_enveloping_algebra_of_a_complex_semisimple_lie_algebra,kac_kazhdan_1979_structure_of_representations_with_highest_weight_of_infinite_dimensional_lie_algebras}.
In any event,
note that Prop.~\ref{prop:shapovalov_matrix} implies that $M$ is \emph{simple} for generic choices of $\bm \lambda \in \mf Z_{\bm \phi}^{\dual}$,
by choosing parameters from the complement of a countable union of vanishing loci of polynomial functions.

\begin{rema}
	Again,
	certain exact equivalences established in~\cite{chaffe_2023_category_o_for_takiff_lie_algebras,chaffe_topley_2023_category_o_for_truncated_current_lie_algebras} make it possible to bypass such computations and reduce the criterion for irreducibility to the tame---%
	parabolic---%
	case (cf.~Thm.~\ref{thm:simple_finite_modules}).
	Nonetheless,
	we continue the present discussion as it is needed in \S~\ref{sec:deformation_quantisation_our_case}.
\end{rema}

\subsection{PBW bases}

As mentioned just above,
it is helpful to construct (countable) $\mb C$-bases of $\Sym (\mf u^-_{\bm \psi}) \simeq M^{\bm \psi}_{\bm \lambda}$.
To this end,
for each root $\alpha \in \Phi$ denote again by $d_\alpha \in \set{0,\dc,r}$ the unique integer such that $\alpha \in \phi_{d_\alpha} \sm \phi_{d_\alpha-1}$(cf.~Rmk.~\ref{rmk:wild_hyperplane_complement}).

\begin{rema}
	\label{rmk:nonsingular_characters_and_levels}

	With this notation,
	an element $\bm \lambda \in \mf (t^{\dual})^r$ defines a character of $\mf p^\pm_{\bm \psi} \sse \mf g_r$ if and only if
	\begin{equation}
		\Braket{ \lambda_i,\alpha^{\dual} } = 0,
		\qquad \alpha \in \Phi,
		\quad i \in \set{ d_\alpha,\dc,r-1 }.
	\end{equation}
	Furthermore,
	the character is then nonsingular if and only if
	\begin{equation}
		\Braket{ \lambda_i,\alpha^{\dual} } \neq 0,
		\qquad \alpha \in \Phi,
		\quad i \ceqq d_{\alpha} - 1.
	\end{equation}

	(This rephrases the stratum conditions in~\eqref{eq:dual_wild_stratum}.)
\end{rema}

\subsubsection{}
\label{sec:relative_root_height}

It follows that the monomials in the variables
\begin{equation}
	\label{eq:variables_pbw_basis}
	X_{\alpha,i} \ceqq E_{-\alpha} \varepsilon^i,
	\qquad \alpha \in \nu_0,
	\quad i \in \set{0,\dc,d_\alpha-1},
\end{equation}
provide a free set of generators for $\Sym (\mf u^-_{\bm \psi})$ over $\mb C$,
and a basis is obtained upon choosing a suitable ordering (cf.~\eqref{eq:order_basis_nilradical}).

To discuss weight spaces,
consider now a finite \emph{multisubset} of $\nu_0$,
i.e.,
a subset with possibly repeated elements.
This is the same as a function $f \cl \nu_0 \to \mb Z_{\geq 0}$,
assigning to each $\alpha \in \nu_0$ its multiplicity $f_\alpha \geq 0$ in the multisubset;
its \emph{cardinality} is $\abs f \ceqq \sum_{\nu_0} f_\alpha \geq 0$.

Each multisubset defines an element $\mu_f \ceqq \sum_{\nu_0} f_\alpha \alpha \in Q^+_{\phi_0}$.
Conversely,
for $\mu \in \mf t^{\dual}$ write
\begin{equation}
	\label{eq:decomposition_set}
	\on{Dec}_{\nu_0}(\mu) \ceqq \Set{ f \in \mb Z_{\geq 0}^{\nu_0} | \mu = \mu_f } \sse \mb Z_{\geq 0}^{\nu_0}.
\end{equation}
This is the set of decompositions of $\mu$ as a $\mb Z_{\geq 0}$-linear combination of the elements of $\nu_0$:
if $\mu \not\in Q^+_{\phi_0}$ then it is empty,
otherwise it coincides with the set of decompositions of $\mu$ in the (commutative) monoid $Q^+_{\phi_0}$,
relative to the choice of generators $\nu_0 \sse Q^+_{\phi_0}$.

Note that~\eqref{eq:decomposition_set} is finite.\fn{
Indeed,
choose a base $\Delta \sse \Phi$ of simple roots such that $\nu_0 \sse \Phi^+_{\Delta}$ lies within the associate subsystem of positive roots:
if $\mu \in Q^+_{\phi_0} \sse Q^+$,
then every element $f \in \on{Dec}_{\nu_0}(\mu)$ yields a decomposition of $\mu$ as a $\mb Z_{\geq 0}$-linear combination of positive roots---%
of which there are only finitely many.}~Then the \emph{height of} $\mu \in Q^+_{\phi_0}$,
\emph{relative to} $\nu_0$,
is
\begin{equation}
	\label{eq:relative_height}
	\abs \mu_{\nu_0} \ceqq \max \Set{ \abs f | f \in \on{Dec}_{\nu_0}(\mu) } \in \mb Z_{\geq 0}.
\end{equation}
(With the convention that $\max (\vn) \ceqq 0$.)
An element $\mu \in Q^+_{\phi_0}$ is \emph{indecomposable} if $\abs \mu_{\nu_0} = 1$,
which implies that $\mu \in \nu_0$.
Denote by $\Delta_{\nu_0} \sse \nu_0$ the subset of indecomposable elements.

By definition,
one has
\begin{equation}
	\abs \mu_{\nu_0} \geq \abs{\mu'}_{\nu_0} + \abs{\mu''}_{\nu_0},
	\qquad \mu,\mu',\mu'' \in Q^+_{\phi_0},
	\quad \mu = \mu' + \mu''.
\end{equation}
Then a (strong) induction based on~\eqref{eq:relative_height} shows that every nonzero element of $Q^+_{\phi_0}$ can be written as a $\mb Z_{\geq 0}$-linear combination of indecomposable elements.
Conversely,
$\abs f = \abs{\mu_f}_{\nu_0}$ implies that $f$ is a multisubset of indecomposable elements.

\begin{exem}
	Suppose that $\phi_0 = \vn$.
	Then there is a base $\Delta \sse \Phi$ of simple roots such that $\nu_0 = \Phi^+_{\Delta}$ coincides with the associated subsystem of positive roots,
	and $\Delta_{\nu_0} = \Delta$.
	(As usual,
	the above definitions do \emph{not} actually use a choice of base.)

	But in general one might even have $\Delta_{\nu_0} = \nu_0$,
	e.g.,
	taking $\mf g = \mf{gl}_3(\mb C)$,
	and $\phi_0 \ceqq \set{\alpha_{12},\alpha_{21}}$,
	and $\psi_0 \ceqq \set{\alpha_{12},\alpha_{21},\alpha_{13},\alpha_{23}}$ (whence $\nu_0 = \set{\alpha_{23},\alpha_{13}}$,
	cf.~\S~\ref{sec:rank_3_example}).
\end{exem}

\begin{exem}
	\label{ex:indecomposable_weight_space}

	Suppose that the root $\alpha \in \nu_0$ is indecomposable.
	Then,
	in the notation of~\eqref{eq:variables_pbw_basis},
	a $\mb C$-basis of the weight space $M[\alpha] \sse M$ is given by the ordered list $\mc B = (X_{\alpha,0} w,\dc,X_{\alpha,d_\alpha-1} w)$,
	in \emph{increasing} exponents of $\varepsilon$.

	Now one has $\ps{t}{}{X_{\alpha,i}} = E_\alpha \varepsilon^i$,
	so that
	\begin{equation}
		\ps{t}{}{X_{\alpha,i}} \cdot X_{\alpha,j} = H_\alpha \varepsilon^{i+j} + E_{-\alpha} \varepsilon^j E_\alpha \varepsilon^i,
		\qquad H_\alpha \ceqq [E_{-\alpha},E_\alpha] \in \mf t_\alpha.
	\end{equation}
	Hence,
	one has $\mc S^{\bm \psi}_{\bm \lambda}(X_{\alpha,i} w,X_{\alpha,j} w) = \Braket{ \lambda_{i+j},H_\alpha }$,
	which vanishes if $i+j \geq d_\alpha$ (cf.~Rmk.~\ref{rmk:nonsingular_characters_and_levels}).
	This implies that the $d_\alpha$-by-$d_\alpha$ (symmetric) matrix of the restricted bilinear form $\eval[1]{\mc S^{\bm \psi}_{\bm \lambda}}_\alpha \cl M[\alpha] \ots M[\alpha] \to \mb C$,
	in the basis $\mc B$,
	is upper antitriangular,
	with coefficients equal to $\braket{ \lambda_{d_\alpha-1}, H_\alpha } \in \mb C$ on the main antidiagonal.
	Thus,
	the bilinear form $\eval[1]{\mc S^{\bm \psi}_{\bm \lambda}}_\alpha$ is nondegenerate if and only if $\lambda_{d_\alpha-1}$ does \emph{not} vanish on the coroot $\alpha^{\dual} \in \phi^{\dual}_{d_\alpha} \sm \phi^{\dual}_{d_\alpha-1}$.
	(In view of Thm.~\ref{thm:nondegenerate_character_and_symplectic_form},
	the latter relates with the computation of dual bases of $\mf u^\pm_{\bm \psi}$,
	with respect to the pairing~\eqref{eq:nondegenerate_pairing_deeper_character}.)
\end{exem}

\subsubsection{}
\label{sec:pbw_bases_weight_spaces}

Suppose in particular that all elements of $\nu_0$ are indecomposable.
Then,
by Exmp.~\ref{ex:indecomposable_weight_space} +~\ref{prop:nonsingular_characters},
we conclude the following:
if the Shapovalov form is nondegenerate,
the corresponding character $\bm \chi$ is nonsingular (cf.~Rmk.~\ref{rmk:nonsingular_characters_and_levels}).
In general,
however,
the latter \emph{cannot} be a sufficient condition for irreducibility,
contrary to the generic wild case.

Precisely,
choose $\mu \in Q^+_{\phi_0} \sm (0)$:
for any $f \in \on{Dec}_{\nu}(\mu)$ one can construct a monomial lying in the corresponding weight space---%
in many ways.
Namely,
for all $\alpha \in \nu_0$ one must put exactly $f_\alpha \geq 0$ occurrences of the variables $X_{\alpha,i}$,
with \emph{arbitrary} exponents $i \in \set{0,\dc,d_\alpha-1}$;
so for each root there are $d_\alpha^{f_\alpha}$ further independent choices.
More concretely,
given $\alpha \in \nu_0$ and $f \in \mb Z_{\geq 0}^{\nu_0}$ introduce the following set of nondecreasing functions
\begin{equation}
	\on{Exp}_{\alpha,f} \ceqq \Set{ \bm i \cl \set{1,\dc,f_\alpha} \lra \set{1,\dc,d_\alpha-1} | \bm i_k \leq \bm i_l \text{ for } k \leq l }.
\end{equation}
Choose also a total order $\nu_0 = (\alpha_1,\dc,\alpha_m)$,
for some integer $m \geq 0$,
such that $1 = \abs{ \alpha_1 }_{\nu_0} \leq \dm \leq \abs{ \alpha_m }_{\nu_0}$,
and write $f_j = f_{\alpha_j} \in \mb Z_{\geq 0}$.
Then there is a natural (lexicographic) total order on the set of generators~\eqref{eq:variables_pbw_basis},
viz.,
\begin{equation}
	\label{eq:order_basis_nilradical}
	X_{\alpha_k,i} \leq X_{\alpha_l,j} \quad \text{if} \quad
	\begin{cases}
		k < l,
		 & \quad \text{or} \\
		k = l,
		 & \quad i \leq j.
	\end{cases}
\end{equation}

Overall,
it follows that:

\begin{prop}
	\label{prop:pbw_basis}

	A free set of generators of $M[\mu]$,
	over $\mb C$,
	is obtained by acting on the canonical generator with the elements of the finite set
	\begin{equation}
		\label{eq:generators_weight_space}
		\Set { \bm X_{f,\bm i} \ceqq \prod_{j = 1}^m \Biggl( \,
			\prod_{k = 1}^{f_j} X_{\alpha_j,\bm i^{(j)}_k} \Biggr)
			| f \in \on{Dec}_{\nu_0}(\mu),
			\quad \bm i = \bigl( \bm i^{(1)},\dc,\bm i^{(m)} \bigr) \in \prod_{j = 0}^m \on{Exp}_{\alpha_j,f} }.
	\end{equation}
\end{prop}

(And a $\mb C$-basis of $M[\mu]$ is obtained upon giving a total order to~\eqref{eq:generators_weight_space}.)

\begin{proof}[Proof omitted]
\end{proof}

\begin{exem}
	\label{ex:nongeneric_simplicity}

	Let us look at the inclusion $\phi_0 \sse \phi_1$ of the smallest Levi subsystems.
	By definition,
	one has $d_\alpha = 1$ if and only if $\alpha \in \phi_1 \sm \phi_0$,
	which may be nonempty in the generalised case.
	Furthermore,
	up to a sign,
	this is equivalent to having $\alpha \in \nu_0 \sm \nu_1$.

	Choose now $\mu \in Q^+_{\phi_0}$ such that the following holds:
	if $f \in \on{Dec}_{\nu_0}(\mu)$,
	and if $f_\alpha \neq 0$,
	then $d_\alpha = 1$.
	In particular,
	this implies that $\mu \in \spann_{\mb Z_{\geq 0}}(\nu_0 \sm \nu_1) \sse Q^+_{\phi_0}$;
	conversely,
	if the Levi filtration corresponds to nested subsets $\Sigma_0 \sse \dm \sse \Sigma_{r-1} \sse \Delta$ of a base $\Delta \sse \Phi$,
	then this happens,
	e.g.,
	if $\mu$ is a positive-integer multiple of a simple root $\theta \in \Sigma_1 \sm \Sigma_0$.

	Then,
	as a particular case of Prop.~\ref{prop:pbw_basis},
	one finds
	\begin{equation}
		M[\mu] = \spann_{\mb C} \Set{ \Biggl( \prod_{j = 1}^m X^{f_j}_{\alpha_j,0} \Biggr) w | f \in \on{Dec}_{\nu_0}(\mu) },
	\end{equation}
	as the set $\on{Exp}_{\alpha_j,f}$ is a singleton for all $j \in \set{1,\dc,m}$.
	Hence,
	this weight space lies in the vector subspace $U\mf g \cdot w \sse M$,
	using the embedding $U\mf g \sse U\mf g_r$ obtained from the composition $\mf g \hra \mf g_r \hra U\mf g_r$,
	and all the computations involving $\eval[1]{S^{\bm \psi}_{\bm \lambda}}_\mu$ happen inside $U(\mf u^\pm_{\psi_0}) \sse U\mf g$.
	Thus,
	the usual conditions for the nondegeneracy of a standard Shapovalov form apply,
	and in general these are \emph{stronger} than the requirement that $\chi_0 \cl \mf l_{\phi_0} \to \mb C$ be nonsingular.
	(Cf.~Exmp.~\ref{ex:nonresonance}.)
\end{exem}

\subsubsection{}

The previous examples leads to a natural criterion for the simplicity of the finite generalised singularity modules.
Namely,
provided that its defining weight lies in the correct stratum,
one must still impose that an ordinary tame (parabolic) Verma module be simple,
but working within a suitable reductive Lie subalgebra of $\mf g$.
To state it precisely,
observe first the following fact:

\begin{lemm}
	\label{lem:parabolic_induction}

	Let $\mf p \sse \wt{\mf p} \sse \mf g$ be nested parabolic subalgebras,
	with nested Levi factors $\mf l \sse \wt{\mf l} \sse \mf g$---%
	all containing $\mf t$.
	Then:
	\begin{enumerate}
		\item
		      the intersection $\mf p' \ceqq \mf p \cap \wt{\mf l}$ is a parabolic subalgebra of $\wt{\mf l}$;

		\item
		      and $\mf l \sse \mf p'$ is a Levi factor thereof.
	\end{enumerate}
\end{lemm}

\begin{proof}[Postponed to~\ref{proof:lem_parabolic_induction}]
\end{proof}

\begin{theo}
	\label{thm:simple_finite_modules}

	Choose a depth-$r$ parabolic filtration $\bm \psi \in \mc P^{(r)}_\Phi$,
	with Levi factor $\bm \phi \in \mc L^{(r)}_\Phi$,
	and let $\bm \lambda \in \mf Z^{\dual}_{\bm \phi}$ be a formal type.
	Set $\mf l_1 \ceqq \mf l_{\phi_1} \sse \mf g$,
	and regard $\wt{\mf p}_0 \ceqq \mf p_{\psi_0} \cap \mf l_1$ as a parabolic subalgebra of $\mf l_1$ (cf.~Lem.~\ref{lem:parabolic_induction}).
	Then the finite generalised singularity module $M^{\bm \psi}_{\bm \lambda}$ is \emph{simple} if and only if:
	\begin{enumerate}
		\item the corresponding character $\bm \chi \cl \mf S^{\bm \psi} \to \mb C$ is nonsingular;

		\item and the standard Verma module $\wt M_0 \ceqq \Ind_{U\wt{\mf p_0}}^{U\mf l_1} \mb C_{\wt \chi_0}$ is \emph{simple},
		      considering the restricted character $\wt \chi_0 \ceqq \eval[1]{\chi_0}_{\wt{\mf p_0}}$.
	\end{enumerate}
\end{theo}

\begin{proof}
	We have already established that the first conditions is necessary.
	Assuming that it holds,
	the conclusion now follows from a recursive usage of the (exact) functors of~\cite[Lem.~3.7 +~Thm.~4.1]{chaffe_topley_2023_category_o_for_truncated_current_lie_algebras}:
	cf.~\cite[\S~4]{chaffe_rembado_yamakawa_genus_zero_wild_quantum_de_rham_spaces} for some more details.
\end{proof}

\subsubsection{}

In terms of the choice of $\bm \lambda$,
the first condition of Thm.~\ref{thm:simple_finite_modules} was made explicit in Rmk.~\ref{rmk:nonsingular_characters_and_levels},
using the root-valuation strata of $\mf t_r^{\dual}$.
Suppose therefore that $\bm\chi$ is \emph{nonsingular} to explore the second condition:
in particular $\lambda_0$ vanishes on $\phi_0^{\dual} \sse \mf t$,
but \emph{not} on the complement $\phi_1^{\dual} \sm \phi_0^{\dual}$.
In the notation of Thm.~\ref{thm:simple_finite_modules},
now Lem.~\ref{lem:parabolic_induction} implies that $\mf l_0 \ceqq \mf l_{\phi_0}$ is the unique Levi factor of $\wt{\mf p}_0$ containing $\mf t$.
Therefore,
compared to $\chi_0$,
the restricted character $\wt\chi_0 \cl \wt{\mf p}_0 \to \mb C$ is a smaller extension of $\lambda_0 \in \mf t^{\dual}$,
and one has $\wt M_0 = M_{\wt{\mf p}_0,\lambda_0}$ (cf.~\eqref{eq:tame_parabolic_verma};
again,
the point is that one now works within $\mf l_1 \sse \mf g$).

Now suppose that fission is \emph{complete/toral} (cf.~Rmk.~\ref{rmk:complete_incomplete_fission}).
This implies that $\mf l_1 = \mf t$ is as small as possible:
on the side of meromorphic gauge theory,
it means that the centralizer of the irregular type is the maximal torus $T \sse G$,
which happens in particular in the generic case (where the centralizer of the leading coefficient is already only $T$).
Then $\mf l_0 = \wt{\mf p}_0 = \mf t$ as well,
and $\wt M_0 = M_{\mf t,\lambda_0} \simeq \mb C_{\lambda_0}$ is simple:
this is consistent with~\cite{wilson_2011_highest_weight_theory_for_truncated_current_lie_algebras},
and it extends the irreducibility criterion of op.~cit.~in the nongeneric case.
However,
in the general nongeneric case one might have $\mf t \ssne \mf l_1$,
i.e.,
fission might be \emph{incomplete} (cf.~again~Rmk.~\ref{rmk:complete_incomplete_fission}).
In this case,
one can appeal to the classical results on the BGG category $\mc O$.
In particular,
the precise necessary/sufficient condition for $\wt M_0$ to be simple is Jantzen's criterion~\cite[Satz~4]{janzten_1977_kontravariante_formen_auf_induzierten_darstellungen_halbeinfacher_lie_algebren} (cf.~\cite[\S~9.13]{humphreys_2008_representations_of_semisimple_lie_algebras_in_the_bgg_category_o}),
and so indeed it is \emph{not} enough to select $\bm\lambda$ in the correct stratum (cf.~Exmp.~\ref{ex:nongeneric_simplicity}).

\begin{exem}
	\label{ex:nonresonance}

	In a particular case,
	one can finally relate the condition for the simplicity of $M_{\bm\lambda}^{\bm\psi}$ to the \emph{nonresonance} condition for untwisted meromorphic $G$-connection germs.
	Precisely,
	suppose that $\mf b_0 \ceqq \wt{\mf p_0} \sse \mf l_1$ is a \emph{Borel} subalgebra,
	i.e.,
	that $\phi_0 = \vn$ and $\mf l_0 = \mf t$.
	Then $\wt M_0 = M_{\mf b_0,\lambda_0}$ is simple if and only the following happens,
	noting that $\lambda_0$ is now a \emph{regular} weight (as it does not vanish on $\phi_1$;
	cf.~\cite[\S~9.12]{humphreys_2008_representations_of_semisimple_lie_algebras_in_the_bgg_category_o}).
	Let $\phi_1 = \phi_1^+ \cup \phi_1^-$ be the decomposition of the root system of $(\mf l_1,\mf t)$ into positive/negative roots,
	as determined by $\mf b_0$.
	If $\rho \ceqq \frac 1 2 \sum_{\phi_1^+} \alpha \in \mf t^{\dual}$ is the Weyl covector,
	then the affine wall-avoiding condition reads
	\begin{equation}
		\label{eq:affine_wall_avoiding}
		\braket{ \lambda_0 + \rho,\alpha^{\dual} } \not\in \mb Z_{> 0},
		\qquad \alpha \in \phi_1^+.
	\end{equation}
	Importantly,
	$\rho$ has integer evaluation on all coroots,
	and so~\eqref{eq:affine_wall_avoiding} follows from the stronger requirement that
	\begin{equation}
		\braket{\lambda_0,\alpha_0} \not\in \mb Z \sm \set{0},
		\qquad \alpha \in \phi_1.
	\end{equation}
	The latter does \emph{not} depend on the choice of a Borel subalgebra containing $\mf t$,
	and it expresses the nonresonance of a (formal) residue,
	on the other side of the duality~\eqref{eq:deeper_duality}:
	i.e.,
	the fact that its adjoint action has no nonzero integer eigenvalues,
	upon restriction to the centraliser of the irregular type (cf.~\cite[\S~4]{chaffe_rembado_yamakawa_genus_zero_wild_quantum_de_rham_spaces}).
\end{exem}

\subsection{A nonsymmetric variant}
\label{sec:nonsymmetric_shapovalov}

The wild Shapovalov form was defined by using the (pinned) root system of $(\mf g,\mf t)$,
and there is a more abstract version which we will also need.

Let $\bm{\mf g}$ be a (possibly nonreductive) finite-dimensional Lie algebra with a vector-space splitting $\bm{\mf g} = \bm{\mf u}^- \ops \bm{\mf l} \ops \bm{\mf u}^+$,
where $\bm{\mf l} \sse \bm{\mf g}$ is a Lie subalgebra and $\bm{\mf u}^\pm \sse \bm{\mf g}$ are vector subspaces.
Then there is a `negation' antiautomorphic involution $X \mt -X$ of $\bm{\mf g}$,
which extends to the Hopf-algebra antipode $\iota \cl U\bm{\mf g} \lxra{\simeq} (U\bm{\mf g})^{\op}$---%
a.k.a.~the `principal' antiautomorphism~\cite[\S~2.2.18]{dixmier_1996_algebres_enveloppantes}.
The $\iota$-contragredient universal Shapovalov pairing is then
\begin{equation}
	\mc S^{\iota} \cl U\bm{\mf g} \ots U\bm{\mf g} \lra U\bm{\mf l},
	\qquad X \ots Y \lmt \pi_{\bm{\mf l}} \bigl( \iota(X)Y \bigr),
	\qquad X,Y \in U\bm{\mf g},
\end{equation}
via the usual projection $\pi_{\bm{\mf l}} \cl U\bm{\mf g} \thra U\bm{\mf l}$.
But in this setup the vector-space splitting of $U\bm{\mf g}$ need \emph{not} be preserved by the involution.
By the same token,
if $\bm \chi \cl \bm{\mf l} \to \mb C$ is a character,
the corresponding bilinear form $\mc S^{\iota}_{\bm \chi} \ceqq \bm \chi \circ \mc S^{\iota}$ need \emph{not} be symmetric,
contrary to~\eqref{eq:wild_shapovalov_form}---%
which used the fixed locus of~\eqref{eq:wild_transposition_involution} in an essential way.

\subsubsection{}

Let us get back to the situation where $\bm{\mf g} = \mf g_r$,
$\bm{\mf l} = \mf l_{\bm \phi}$,
and $\bm{\mf u}^\pm = \mf u^\pm_{\bm \psi}$.
Composing transposition/negation yields the involutive \emph{wild Chevalley automorphism} $\omega \cl \bm X \mt -\ps{t}{}{\bm X} \cl \mf g_r \lxra{\simeq} \mf g_r$ (cf.~\cite[\S~1.3]{kac_1990_infinite_dimensional_lie_algebras}).
Extending it to a ring automorphism of $U\mf g_r$ tautologically yields the identity
\begin{equation}
	\label{eq:relation_shapovalov_forms}
	\mc S^{\bm \psi} (\bm X,\bm Y) + \mc S^{\iota} \bigl( \omega(\bm X),\bm Y \bigr) = 0 \in U\mf l_{\bm \phi},
	\qquad \bm X,\bm Y \in U\mf g_r.
\end{equation}
This make it possible to transfer all the above results to the nonsymmetric version of the Shapovalov form.

Namely,
choose $\bm \lambda \in \mf Z^{\dual}_{\bm \phi} \sse \mf t_r^{\dual}$ as above,
corresponding to a character $\bm \chi \cl \mf l_{\bm \phi} \to \mb C$,
and consider the opposite pair of finite singularity modules
\begin{equation}
	M^\pm \ceqq M^{\pm \bm \psi}_{\pm \bm \lambda} = U\mf g_r \ots_{U\mf p^\pm_{\bm \psi}} \mb C_{\pm \bm \chi}.
\end{equation}
Then:

\begin{enonce}{Corollary/Definition}
	\label{cor:nonsymmetric_shapovalov}

	There is a well-induced (nonsymmetric) Shapovalov form $\mc S^{\iota}_{\bm \chi} \cl M^- \ots M^+ \to \mb C$ such that:
	\begin{enumerate}
		\item $M^-$ and $M^+$ are \emph{simple} if and only if both the left and right radical of $\mc S^{\iota}_{\bm \chi}$ vanish;

		\item and the $\mf Z_{\phi_0}$-weight space decompositions of $M^-$ and $M^+$ are $\mc S^{\iota}_{\bm \chi}$-\emph{orthogonal}.
	\end{enumerate}
\end{enonce}

(As usual,
the notation $\mc S^{\iota}_{\bm \lambda}$ will be justified in our setup,
as $\bm \lambda$ uniquely encodes $\bm \chi$.)

\begin{proof}
	Consider again the subspace $K^{\bm \psi}_{\bm \lambda} \sse U\mf g_r$,
	generating the annihilator of the canonical generator of $M^+$:
	the subspace $K^{\bm \psi}_{\bm \lambda} \ots U\mf g_r + U\mf g_r \ots K^{\bm \psi}_{\bm \lambda} \sse U\mf g_r \ots U\mf g_r$ is annihilated by $\mc S^{\bm \psi}_{\bm \lambda}$ (cf.~Prop.~\ref{prop:annihilator_cyclic_vector_in_wild_shapovalov_radical}).
	Then~\eqref{eq:relation_shapovalov_forms} implies that $\omega \bigl( K^{\bm \psi}_{\bm \lambda} \bigr) \ots U\mf g_r + U\mf g_r \ots K^{\bm \psi}_{\bm \lambda}$ is annihilated by $\mc S^{\iota}_{\bm \lambda}$.
	Therefore,
	by $\iota$-contragrediency,
	there is indeed a reduced Shapovalov form,
	noting that $\omega \bigl( K^{\bm \psi}_{\bm \lambda} \bigr) = K^{-\bm \psi}_{-\bm \lambda}$ and
	\begin{equation}
		M^- \simeq U \mf g_r \bs ( U\mf g_r \cdot K^{-\bm \psi}_{-\bm \lambda} ).
	\end{equation}
	(This uses the identities $\omega(\mf u^\pm_{\bm \psi}) = \mf u^\mp_{\bm \psi}$,
	and $\omega(\mf l_{\bm \phi}) = \mf l_{\bm \phi}$,
	and $\eval[1] \omega_{\mf t_r} = -\Id_{\mf t_r}$.)

	By the same token,
	$\Rad \bigl( \mc S^{\bm \psi}_{\bm \lambda} \bigr) \sse M^+$ coincides with the right radical of $\mc S^{\iota}_{\bm \lambda}$,
	i.e.,
	with the maximal proper submodule $N^+ \ceqq N^{\bm \psi}_{\bm \lambda} \sse M^+$ (cf.~Thm.~\ref{thm:shapovalov_radical_equals_maximal_proper_submodule_wild}).
	About the left slot,
	note that by construction
	\begin{equation}
		\omega \bigl( \Ann_{U\mf g_r}(w^+) \bigr) = \omega \bigl( U\mf g_r \cdot K^{\bm \psi}_{\bm \lambda} \bigr) = U\mf g_r \cdot K^{-\psi}_{-\bm \lambda} = \Ann_{U\mf g_r}(w^-),
		\qquad w^\pm \ceqq w^{\bm \psi}_{\pm \bm \lambda} \in M^\pm.
	\end{equation}
	Hence,
	there is a well-defined $\omega$-linear isomorphism
	\begin{equation}
		\label{eq:theta_linear_isomorphism}
		\bm X w^+ \lmt \wt \omega(\bm X w^+) \ceqq \omega(\bm X) w^- \cl M^+ \lxra{\simeq} M^-,
		\qquad \bm X \in U\mf g_r.
	\end{equation}
	It follows that $\wt \omega \bigl( \Rad( \mc S^{\bm \psi}_{\bm \lambda}) \bigr) = \wt \omega(N^+) = N^- \ceqq N^{-\bm \psi}_{-\bm \lambda} \sse M^-$ is the maximal proper submodule,
	which by~\eqref{eq:relation_shapovalov_forms} coincides with the left radical of $\mc S^{\iota}_{\bm \chi}$.

	Finally,
	note that $M^-$ is a \emph{lowest}-weight $\mf Z_{\phi_0}$-module,
	of lowest weight $-\lambda_0 \in \mf Z^{\dual}_{\phi_0}$,
	with lowest-weight line spanned by $w^-$.
	Hence
	\begin{equation}
		M^- = \bops_{Q^+_{\phi_0}} M^-[\mu],
		\qquad M^-[\mu] \ceqq \Set{ \wt v \in M^- | X \wt v = \Braket{ \mu - \lambda_0,X } \wt v \text{ for } X \in \mf Z_{\phi_0} } \sse M^-.
	\end{equation}
	Furthermore,
	by construction~\eqref{eq:theta_linear_isomorphism} is a \emph{graded} isomorphism,
	i.e. (in our notation),
	$\wt \omega \bigl( M^+[\mu] \bigr) = M^-[\mu]$ for all $\mu \in Q^+_{\phi_0}$.
	Thus,
	a last application of~\eqref{eq:relation_shapovalov_forms},
	together with Prop.~\ref{prop:orthogonality_wild_shapovalov},
	yields
	\begin{equation}
		\mc S^{\iota}_{\bm \lambda} \bigl( M^-[\mu],M^+[\mu'] \bigr) = (0),
		\qquad \mu \neq \mu' \in Q^+_{\phi_0}.
		\qedhere
	\end{equation}
\end{proof}

\subsubsection{}

The $\mb C$-linear embedding $\Sym(\mf u^\pm_{\bm \psi}) \hra U\mf g_r$ of \S~\ref{sec:identification_verma_symmetric_algebra} restricts on $\mf u^\pm_{\bm \psi}$ to the composition $\mf u^\pm_{\bm \psi} \sse \mf g_r \sse U\mf g_r$,
which is basis-independent.
In particular,
it makes \emph{intrinsic} sense to restrict the nonsymmetric Shapovalov form to the subspace $\mf u^+_{\bm \psi} \ots \mf u^-_{\bm \psi} \hra M^- \ots M^+$.
One then computes
\begin{equation}
	\label{eq:relation_shapovalov_nonsingular_character}
	\mc S^{\iota}_{\bm \lambda} (\bm Y,\bm Y') + B^{\bm \psi}_{\bm \lambda} \bigl( \bm Y,\bm Y' \bigr) = 0,
	\qquad \bm Y \in \mf u^+_{\bm \psi},
	\quad \bm Y' \in \mf u^-_{\bm \psi},
\end{equation}
invoking the pairing of~\eqref{eq:nondegenerate_pairing_deeper_character}.
Hence,
if $\bm \chi$ is nonsingular,
the Shapovalov form of Cor./Def.~\ref{cor:nonsymmetric_shapovalov} extends the polarised KKS symplectic structure on the wild coadjoint orbit $\mc O_{\bm \lambda} \sse \mf g_r^{\dual}$ (cf.~\S~\ref{sec:deeper_character_and_symplectic_form}).
This is the key to the quantisation of the next section.

\section{Deformation quantisation of nongeneric wild orbits}
\label{sec:deformation_quantisation}

\subsection{}

In this section we deformation-quantise \emph{all} the $r$-semisimple wild orbits admitting a balanced polarisation,
building on \S~\ref{sec:nonsingular_characters}--\ref{sec:shapovalov} and adapting~\cite{alekseev_lachowska_2005_invariant_star_product_on_coadjoint_orbits_and_the_shapovalov_pairing}.

\subsection{Generalities on deformation quantisation}

We review the basic terminology,
which originated in~\cite{gerstenhaber_1964_on_the_deformtation_of_rings_and_algebras,bayen_flato_fronsdal_lichnerowicz_sternheimer_1978_deformation_theory_and_quantisation_i_deformations_of_symplectic_structures,bayen_flato_fronsdal_lichnerowicz_sternheimer_1978_deformation_theory_and_quantisation_ii_physical_applications}:
please refer,
e.g.,
to~\cite{kassel_1995_quantum_groups,etingof_schiffmann_1998_lectures_on_quantum_groups,etingof_2007_calogero_moser_systems_and_representation_theory,schedler_2012_deformations_of_algebras_in_noncommutative_geometry},
for more details.
(Experts might want to skip to \S~\ref{sec:deformation_quantisation_our_case}.)

\subsubsection{}
\label{sec:deformation_quantisation_basics}

Let $M$ be a finite-dimensional $C^\infty$ manifold,
equipped with a Poisson bracket $\set {\cdot,\cdot }_M \cl A_0 \wdg A_0 \to A_0$ on the commutative $\mb C$-algebra $A_0 \ceqq C^{\infty}(M,\mb C)$ of semiclassical observables on $M$---%
thinking of the latter as a phase-space in Hamiltonian mechanics.
Equivalently,
$M$ carries a Poisson bivector field $\Pi$,
i.e.,
a smooth section of $\bigwedge^2 \on T\!M \to M$ which Schouten-commutes with itself.

Let then $\hs$ be a formal variable,
and $\mb C_{\hs} \ceqq \mb C\llb \hs \rrb$ the corresponding ring of formal power series.
A (1-parameter) \emph{deformation algebra} is a topologically-free $\mb C_{\hs}$-algebra with $\hs$-adically-continuous associative product.
A \emph{formal deformation of} $A_0$ is a deformation algebra $\wh A$ such that there exists a $\mb C$-algebra isomorphism $\wh A \bs \hs \wh A \simeq A_0$.
This implies that $\wh A \simeq A_0 \llb \hs \rrb$ as (separated,
torsion-free,
$\hs$-adically-complete) $\mb C_{\hs}$-modules:
we choose such an identification.
Then the associative product $\wh A \ots_{\mb C_{\hs}} \wh A \to \wh A$ is uniquely determined by the restriction
\begin{equation}
	\label{eq:star_product}
	\ast \cl A_0 \ots A_0 \lra A_0 \llb \hs \rrb,
	\qquad a \ots b \lmt a \ast b = \sum_{i \geq 0} C_i(a \ots b) \hs^i,
\end{equation}
encoded by suitable $\mb C$-linear maps $C_i \cl A_0 \ots A_0 \to A_0$.
It follows that $C_0$ is the product of $A_0$,
and that $a \ast (b \ast c) = (a \ast b) \ast c \in \wh A$ for all triples $a,b,c \in A_0$;
conversely,
if these hold,
then~\eqref{eq:star_product} is (by definition) a $\ast$-\emph{product on} $M$.

\begin{rema}
	The associativity of~\eqref{eq:star_product} is equivalent to the following countable set of identities,
	for all triples $a,b,c \in A_0$:
	\begin{equation}
		\label{eq:associativity_star_product}
		\sum_{i+j = k} C_i \bigl( C_j(a,b),c \bigr)
		= \sum_{i+j = k} C_i \bigl( a,C_j(b,c) \bigr) \in A_0,
		\qquad k \in \mb Z_{\geq 0}.
	\end{equation}

	Moreover,
	changing the choice of identification $\wh A \lxra{\simeq} A_0 \llb \hs \rrb$ is the same as composing by continuous $\mb C_{\hs}$-linear automorphisms of the target,
	which reduce to the identity modulo $\hs A_0 \llb \hs \rrb \sse A_0\llb \hs \rrb$.
	(What matters intrinsically is the equivalence class of $\ast$ under this formal based gauge action.)
\end{rema}

\subsubsection{}

Conversely,
let $\wh A$ be a deformation algebra such that $\wh A \bs \hs \wh A$ is a commutative $\mb C$-algebra.
Then the quotient inherits a Poisson bracket $\set{ \cdot,\cdot }_{\wh A}$,
(well) defined by
\begin{equation}
	\label{eq:quantisation_poisson_bracket}
	\Big\{ \wh a + \hs \wh A,
	\wh b + \hs \wh A \Big\}_{\! \wh A} \ceqq \hs^{-1} \bigl( \wh a \cdot \wh b - \wh b \cdot \wh a \bigr) + \hs \wh A,
	\qquad \wh a,\wh b \in \wh A.
\end{equation}

Therefore,
one says that a formal deformation $\wh A$ of $A_0$ is a (formal) \emph{deformation quantisation} if further $\set{ \cdot,\cdot }_{\wh A} = \set{ \cdot,\cdot }_M$.
The canonical projection $\wh A \thra A_0$ modulo $\hs \wh A \sse \wh A$ is the `semiclassical limit'---%
mapping Heisenberg's time-evolution of quantum observables to Hamilton's.
In terms of the $\ast$-product~\eqref{eq:star_product} this means that
\begin{equation}
	\label{eq:poisson_bracket_from_dq}
	C_1(a,b) - C_1(b,a) = \set{ a,b }_M,
	\qquad a,b \in A_0.
\end{equation}

\subsubsection{}

Furthermore,
let $H$ be a finite-dimensional $C^\infty$ connected Lie group acting on $M$,
whence on $A_0$---%
by pullback.
The action extends $\mb C_{\hs}$-linearly to the $\mb C_{\hs}$-module $A_0\llb \hs \rrb$,
and by definition a $\ast$-product~\eqref{eq:star_product} on a formal deformation $\wh A \simeq A_0 \llb \hs \rrb \thra A_0$ is $H$-\emph{invariant} if all its coefficients $C_i$ are $H$-invariant.
This means that $H$ acts by automorphisms of the $\mb C_{\hs}$-algebra $\wh A$.
Furthermore,
if $\wh A$ is a deformation quantisation,
then necessarily $H$ acts on $M$ by Poisson automorphisms (cf.~\eqref{eq:poisson_bracket_from_dq}).

\subsubsection{}

One might ask that $C_i$ be given by a (complex) bidifferential operator on $M$,
in which case~\eqref{eq:star_product} is said to be a \emph{local} $\ast$-product.
This is the only case that we shall consider.
(One should then also restrict gauge transformations.)

Moreover,
we will work in the setting of \emph{homogeneous} manifolds:
assume therefore that the $H$-action on $M$ (as above) is transitive.
Marking a point $x \in M$ yields a diffeomorphism $M \simeq H \bs H'$,
invoking the stabiliser (Lie) subgroup $H' = \Stab_H(x) \sse H$,
which we suppose connected.
The corresponding inclusion of Lie algebras,
i.e.,
\begin{equation}
	\Lie(H') \eqqc \mf h' \sse \mf h \ceqq \Lie(H),
\end{equation}
now yields a vector-bundle isomorphism $TM \simeq H \ts^{H'} (\mf h \bs \mf h') \to H \bs H'$.
For any integer $k \geq 1$,
it is then convenient to identify the space $\Diff_k(M)^H$,
of $H$-invariant $k$-differential operators on $M$,
with the invariant subspace
\begin{equation}
	\label{eq:invariant_k_differential_operators}
	\bigl( \bigl( U\mf h \bs (U\mf h \cdot \mf h') \bigr)^{\! \ots k} \bigr)^{\! \mf h'} \sse \bigl( U\mf h \bs (U \mf h \cdot \mf h') \bigr)^{\! \ots k},
\end{equation}
viewing the quotient $U\mf h \bs (U\mf h \cdot \mf h')$ as an $\mf h'$-module in the adjoint action.

Now note that a local $H$-invariant $\ast$-product on $M$ can be encoded in a \emph{formal} bidifferential operator
\begin{equation}
	\label{eq:formal_bidiff_operator}
	\bm B = \sum_{i \geq 0} B_i \cdot \hs^i \in \Diff_2(M)^H \llb \hs \rrb,
\end{equation}
as follows.
Denote by $\pi \cl H \thra M$ the orbit map through the marked point,
and choose $a,b \in A_0$.
Then let also $\wt B_i \in (U\mf h)^{\ots 2}$ be a representative for $B_i$,
which acts as a bidifferential operator on functions $H \to \mb C$.
(We assume that $B_0$ is the class of $\wt B_0 \ceqq 1 \ots 1$.)
Consider the formal function
\begin{equation}
	\label{eq:star_product_from_formal_bidifferential_operator}
	m \circ \bigl( \wt{\bm B} (\pi^*a \ots \pi^*b) \bigr) \ceqq \sum_{i \geq 0} m \circ \bigl( \wt B_i (\pi^*a \ots \pi^*b) \bigr) \cdot \hs^i \in C^{\infty}(H,\mb C) \llb \hs \rrb,
\end{equation}
where $m$ is the Hopf-algebra product of $C^{\infty}(H,\mb C)$.
One can prove that the result is $H'$-invariant,
and that its class $(a \ast b) \in A_0 \llb \hs \rrb$ is independent of the representative.

To express the associativity constraints~\eqref{eq:associativity_star_product} in this formalism,
let $\Delta \cl U\mf h \to (U\mf h)^{\ots 2}$ be the standard coproduct.
Extending it (tacitly) in $\mb C_{\hs}$-linear fashion,
one has the following identity of formal \emph{tridifferential} operators:
\begin{equation}
	\label{eq:compact_associativity}
	\bm B^{(12,3)} \ceqq \bigl( \Delta \ots 1 (\bm B) \bigr) (\bm B \ots 1) = \bigl( 1 \ots \Delta (\bm B) \bigr) (1 \ots \bm B) \eqqc \bm B^{(1,23)} \in \Diff_3(M)^H \llb \hs \rrb.
\end{equation}
(Again,
both sides are computed for a lift $\wt{\bm B} \in (U\mf h)^{\ots 2} \llb \hs \rrb$,
etc.)

Finally,
the element~\eqref{eq:formal_bidiff_operator} yields a deformation quantisation provided that the antisymmetrisation $B_1 - B_1^{\op}$ coincides with the first-order bidifferential action of the Poisson bivector field $\Pi$ of $\bigl( M,\set{\cdot,\cdot} \bigr)$,
i.e.:
\begin{equation}
	B_1(a,b) - B_1(b,a) = \set{a,b} = \Pi(\dif a \wdg \dif b),
	\qquad a,b \in A_0.
\end{equation}

\subsubsection{}
\label{sec:sheaf_quantization}

All this material can be recast in the complex-analytic category,
with the caveat that it is in principle \emph{not} enough to look at global functions (as there are no partitions of unity,
etc.;
as noted in Rmk.~\ref{rmk:affine_setting},
we could also work with Poisson rings of regular functions on smooth complex affine varieties,
cf.~Rmk.~\ref{rmk:stein}.)

Therefore,
one rather considers a complex manifold $M$ with structure sheaf $\ms O_M$ and holomorphic Poisson bracket $\set{\cdot,\cdot}_M \cl \ms O_M \wdg \ms O_M \to \ms O_M$.
The previous notions are sheafified,
locally on $M$:
notably,
the assignment $U \mt \ms O_M(U) \llb \hs \rrb$ (with $U \sse M$ open) defines a sheaf $\ms O_M\llb \hs \rrb$ of commutative $\mb C_{\hs}$-algebras;
then $\hs \ms O_M\llb \hs \rrb$ is a subsheaf,
and $\ms O_M$ the corresponding quotient.
A \emph{formal deformation of} $M$ consists of a suitable sheaf $\wh{\ms A}$ of $\mb C_{\hs}$-algebras,
coming with identifications $\wh{\ms A} \bs \hs {\wh{\ms A}} \simeq \ms O_M$ and $\wh{\ms A} \simeq \ms O_M\llb \hs \rrb$ as sheaves of $\mb C$-algebras and topologically-free $\mb C_{\hs}$-modules---%
respectively.
Now the commutator of $\wh{\ms A}$ yields a sheaf morphism $[\cdot,\cdot] \cl \wh{\ms A} \ots_{\mb C_{\hs}} \wh{\ms A} \to \wh{\ms A}$ which vanishes modulo $\hs \wh{\ms A}$,
and which divides by $\hs$ (cf.~\eqref{eq:quantisation_poisson_bracket});
one finally reduces $\hs^{-1}[\cdot,\cdot]$,
asking that the result coincides with $\set{\cdot,\cdot}_M$ in order to have a \emph{deformation quantisation of} $\bigl( M,\set{\cdot,\cdot}_M \bigr)$.

In particular,
consider a (connected) complex Lie group $H$ acting transitively/holomorphically on $M$.
Suppose that one can construct a \emph{global} formal $H$-invariant bidifferential operator on $M$,
as in~\eqref{eq:formal_bidiff_operator},
which:
(i) is associative,
as in~\eqref{eq:compact_associativity};
and (ii) has coefficients $B_i$ which preserve holomorphicity.
Choosing $U \sse M$ open,
the restriction $\eval[1]{\bm B}_U \ceqq \sum_{i \geq 0} \eval[1]{B_i}_U \cdot \hs^i$ yields a $\ast$-product $\ms O_M(U) \ots \ms O_M(U) \to \ms O_M(U) \llb \hs \rrb$ as in~\eqref{eq:star_product_from_formal_bidifferential_operator},
which extends uniquely to a formal deformation $\wh{\ms A}(U)$ of $\ms O_M(U)$.
Moreover,
the assignment $U \mt \wh{\ms A}(U)$ yields a sheaf as required (since we start from a global object,
and everything glues uniquely).
To get a deformation quantisation,
one is thus left to verify that the quantum commutator induces the semiclassical Poisson bracket.

\begin{rema}
	While the existence of a deformation quantisation of symplectic manifolds was settled in~\cite{dewilde_lecomte_1983_existence_of_star_products_and_of_formal_deformations_of_the_poisson_lie_algebra_of_arbitrary_symplectic_manifolds,fedosov_1994_a_simple_geometrical_construction_of_deformation_quantization} (cf.~\cite{deligne_1995_deformations_de_l_algebre_des_fonctions_d_une_variete_symplectique_comparaison_entre_fedosov_et_de_wilde_lecomte}),
	and in~\cite{kontsevich_2001_deformation_quantization_of_algebraic_varieties,cattaneo_felder_tomassini_2002_from_local_to_global_deformation_quantization_of_poisson_manifolds,kontsevich_2003_deformation_quantisation_of_poisson_manifolds} for Poisson manifolds (cf.~\cite{cattaneo_felder_2000_path_integral_approach_to_kontsevich_quantisation_formula}),
	the problem of constructing \emph{invariant} $\ast$-products does not have a universal answer.
	This relates with the usual no-go theorems about the quantisation of semiclassical symmetries,
	which historically trace back to~\cite{groenewold_1946_on_the_principles_of_elementary_quantum_mechanics,vanhove_1951_sur_certaines_representations_unitaires_d_un_groupe_infini_de_transformations} (cf.~\cite[\S~1.3.4]{etingof_schiffmann_1998_lectures_on_quantum_groups}).
	Incidentally,
	recall that general Poisson algebras may fail to admit a quantisation,
	even when finite-dimensional~\cite{mathieu_1997_homologies_associated_with_poisson_structures} (cf.~\cite[Rmk.~2.3.14]{schedler_2012_deformations_of_algebras_in_noncommutative_geometry}).
\end{rema}

\subsection{Deformation quantisation:
	monolevel case}
\label{sec:deformation_quantisation_our_case}

Let us get back to the relevant setting.

\subsubsection{}

Choose an integer $r \geq 1$,
and let $(\mf g,\mf t)$ be a split reductive Lie algebra.
Denote by $\Phi \sse \mf t^{\dual}$ the root system,
and let $\bm \psi \in \mc P^{(r)}_\Phi$ be a depth-$r$ parabolic filtration of $\Phi$ (cf.~Def.~\ref{def:parabolic_filtration});
its Levi factor is $\bm \phi \ceqq \on{Lf}_r(\bm \psi) \in \mc L^{(r)}_\Phi$,
and this is a Levi filtration of $\Phi$ (cf.~Def.~\ref{def:levi_filtration}).

Then select an element $\bm X \in \mf t^r_{\bm \phi} \sse \mf t_r \simeq \mf t^r$ in the stratum~\eqref{eq:wild_truncated_strata} determined by $\bm \phi$,
and let $\bm \lambda \in \mf t^{\dual,r}_{\bm \phi^{\dual}} \sse \mf t_r^{\dual} \simeq (\mf t^{\dual})^r$ be the formal type corresponding to it under the duality~\eqref{eq:deeper_duality}---%
lying in the dual stratum~\eqref{eq:dual_wild_stratum}.
There is an associated (marked) coadjoint orbit $\mc O_{\bm \lambda} = G_r \cdot \bm \lambda \sse \mf g_r^{\dual}$,
which is a holomorphic-symplectic manifold with KKS structure $\omega$.
The complex Lie group which acts is $H \ceqq G_r = G(\mb C_r)$,
integrating $\mf g_r$,
while $H' \ceqq G_r^{\bm \lambda} = G_r^{\bm X}$ is the centraliser of the marking.
In the first part of this text we have shown that $H'$ is connected,
and we have identified its Lie algebra $\mf h'$ (cf.~\S~\ref{sec:wild_orbit_strata}):
we now regard the marked orbit as a homogeneous complex manifold,
and consider the $H$-invariant Lie--Poisson bivector field $\Pi \ceqq \omega^{-1}$.
It is encoded in an $H'$-invariant element of $\bigwedge^2 (\mf g_r \bs \mf g_r^{\bm \lambda}) \simeq \bigwedge^2 \bigl( \mf u^+_{\bm \psi} \ops \mf u^-_{\bm \psi} \bigr)$,
in the triangular vector-space splitting $\mf g_r = \mf u^-_{\bm \psi} \ops \mf l_{\bm \phi} \ops \mf u^+_{\bm \psi}$.

We shall adapt the main result of~\cite{alekseev_lachowska_2005_invariant_star_product_on_coadjoint_orbits_and_the_shapovalov_pairing} to construct an $H$-invariant local $\ast$-product for $(\mc O_{\bm\lambda},\Pi)$ (cf.~\cite{etingof_schiffmann_lectures_on_the_dynamical_yang_baxter_equations,etingof_varchenko_1999_exchange_dynamical_quantum_groups,etingof_schedler_schiffmann_2000_explicit_quantization_of_dynamical_r_matrices_for_finite_dimensional_semisimple_lie_algebras,enriquez_etingof_2003_quantization_of_alekseev_meinrenken_dynamical_r_matrices,enriquez_etingof_2005_quantization_of_classical_dynamical_r_matrices_with_nonabelian_base} for related work/techniques).
First,
note that a particular case follows immediately:

\begin{prop}
	\label{prop:grading_constant_case}

	If the parabolic filtration is \emph{constant} (to a fixed parabolic subset $\psi \sse \Phi$),
	then the assumptions of~\cite[\S~2.2]{alekseev_lachowska_2005_invariant_star_product_on_coadjoint_orbits_and_the_shapovalov_pairing} are verified.
\end{prop}

\begin{proof}
	By Prop.~\ref{prop:nonsingular_characters},
	the character is nonsingular.
	Consider then the triangular splitting $\mf g = \mf u^-_{\psi} \ops \mf l_{\phi} \ops \mf u^+_{\psi}$,
	involving the Levi factor and nilradicals of the opposite parabolic subalgebras $\mf p^\pm_{\psi} = \mf l_{\phi} \ops \mf u^\pm_{\psi}$.
	Then there is a natural $\mf l_{\phi}$-invariant $\mb Z$-grading $\mc G_{\bullet} = \mc G^{\psi}_{\bullet}$ on $\mf g$,
	such that $\mc G_0(\mf g) = \mf l_{\phi}$ (cf.~\cite[Exmp.~2.8]{alekseev_lachowska_2005_invariant_star_product_on_coadjoint_orbits_and_the_shapovalov_pairing}).
	In our notation,
	for an integer $i > 0$ set
	\begin{equation}
		\mc G_i(\mf g) \ceqq \bops_{\abs \alpha_{\nu} = i} \mf g_\alpha \sse \mf u^+_{\psi},
		\qquad \mc G_{-i}(\mf g) \ceqq \ps{t}{}{\mc G_i(\mf g)} \sse \mf u^-_{\psi},
	\end{equation}
	letting $\nu \ceqq \psi \sm \phi \sse \Phi$,
	defining the height $\abs \alpha_{\nu} \in \mb Z_{\geq 0}$ of a root $\alpha \in \nu$ (relative to $\nu$) as in \S~\ref{sec:relative_root_height},
	and using the involution~\eqref{eq:transposition_involution}.
	Extend the grading to $\mf g_r$ via
	\begin{equation}
		\label{eq:z_grading_constant_filtrations}
		\bm{\mc G}^{\bm \psi}_i(\mf g_r) = \bm{\mc G}_i(\mf g_r) \ceqq \mc G_i(\mf g) \ots \mb C_r \sse \mf g_r,
		\qquad i \in \mb Z,
	\end{equation}
	so that in particular $\bm{\mc G}_0(\mf g_r) = \mf l_{\phi} \ots \mb C_r \simeq \bm l_{\bm \phi}$ is the deeper Levi factor for the constant Levi filtration $\bm \phi = (\phi,\phi,\dc)$.
	One now checks that $\bm{\mc G}_{\bullet}$ is $\mf l_{\bm \phi}$-invariant,
	hence also $L_{\bm \phi}$-invariant by connectedness.
\end{proof}

\subsubsection{}

Thus,
e.g.,
in the generic one,
the main construction of~\cite{alekseev_lachowska_2005_invariant_star_product_on_coadjoint_orbits_and_the_shapovalov_pairing} deformation-quantises $( \mc O_{\bm \lambda},\Pi )$:
the goal is to extend this result to \emph{nonconstant} parabolic filtrations.

The crux of the matter is that the subspaces $\mf u^\pm_{\bm \psi} \sse \mf g_r$ are \emph{not} necessarily $\mf l_{\bm \phi}$-invariant.
More precisely,
the $\mb Z$-grading of~\eqref{eq:z_grading_constant_filtrations} breaks down precisely when $\bm \psi$ is nonconstant.
For the sake of clarifying the exposition,
we will also immediately restrict to the case where $\bm \psi$ is balanced (cf.~\S~\ref{rmk:balanced_filtrations}):

\begin{center}
	\emph{hereafter,
		and until \S~\ref{sec:irregular_blocks},
		assume that} $\mf u^\pm_{\bm \psi} \sse \mf g_r$ \emph{are} Lie subalgebras.
\end{center}

In particular there are UEAs $U(\mf u^\pm_{\bm \psi}) \sse U\mf g_r$,
as well as $U(\mf u^\pm_{\bm \psi})$-linear isomorphisms $U(\mf u^\pm_{\bm \psi}) \lxra{\simeq} M^{\mp \bm \psi}_{\mp \bm \lambda}$ with the corresponding pair of finite singularity modules.\fn{
	Strictly speaking,
	this is not used until \S~\ref{sec:associativity};
	and perhaps it is not necessary (under a more exhaustive usage of the PBW bases of~\S~\ref{sec:identification_verma_symmetric_algebra}).}

Finally,
until \S~\ref{sec:irregular_blocks} let us also drop the parabolic/Levi filtrations from all notation.

\subsection{Nongeneric extension:
	main steps}

The logical flow is as follows.

\subsubsection{}

Choose a number $c \in \mb C$,
and let $\bm \lambda_c \ceqq c \bm \lambda$ be the corresponding dilated formal type.
It corresponds as usual to a character $\bm \chi_c \in \Hom_{\Lie}(\mf l,\mb C)$ which extends to the finite singularity subalgebra (cf.~Def.~\ref{def:finite_singularity_algebra}).
Moreover,
there is an antipode-contragredient Shapovalov form $\mc S_c \cl M^-_c \ots M^+_c \to \mb C$ (cf.~Cor./Def.~\ref{cor:nonsymmetric_shapovalov}),
where $M^\pm_c \ceqq M_{\bm \lambda_{\pm c}}$ are the corresponding---%
opposite---%
finite singularity modules.

If $M^\pm_c$ are simple,
the canonical quadratic tensor $F_c$ corresponding to the inverse (nondegenerate) Shapovalov form is well-defined.
We regard the latter as a quantum version of $\Pi$,
since $\mc S_c$ restricts to the KKS symplectic form on $\mf u^+ \ots \mf u^- \sse U(\mf u^+) \ots U(\mf u^-)$ (cf.~\S~\ref{sec:deeper_character_and_symplectic_form}).
With this guiding principle,
one can consider the map $c \mt F_c$ as a function on the 1-point compactification of the complex $c$-plane,
proving that it admits a formal Taylor expansion at $c = \infty$ (cf.~\S~\ref{sec:holomorphicity_at_infinity}).
Setting $\hs \ceqq c^{-1}$,
the corresponding formal power series in $\hs$ takes coefficients in $U(\mf u^-) \ots U(\mf u^+)$,
and it can be projected down to a formal ($G_r$-invariant) bidifferential operator $\bm B \in \Diff_2(\mc O_{\bm \lambda})^{G_r^{\bm \lambda}} \llb \hs \rrb$:
then we prove the identity~\eqref{eq:compact_associativity} by using $U\mf g_r$-linear intertwiners (cf.~\S~\ref{sec:associativity}--\ref{sec:deformation_quantisation_conclusion}).
By then,
we will have already determined the first-order term of the Taylor expansion,
showing that its alternating version coincides with $\Pi$ (cf.~Prop./Def.~\ref{prop:inverse_shapovalov_expansion_infinity}).

\subsubsection{}

Before running the above steps,
let us set up a few objects here.

Recall that the finite singularity modules $M^\pm_c = \bops_{Q^+_{\phi_0}} M^\pm_c[\mu]$ are graded by the action of the largest Levi centre $\mf Z_{\phi_0} \sse \mf t$.
Using the height function~\eqref{eq:relative_height},
they also carry a (positive,
exhaustive) filtration with finite-dimensional components
\begin{equation}
	\label{eq:filtration_finite_singularity_module}
	\mc F^{\leq k} (M^\pm_c) \ceqq \bops_{\abs \mu_{\nu_0} \leq k} M^\pm_c [\mu] \sse M^\pm_c,
	\qquad k \in \mb Z_{\geq 0},
\end{equation}
so that,
e.g.,
$\mc F^{\leq 0}(M^\pm_c) = \mb C w^\pm_c$,
where $w^\pm_c \ceqq w_{\pm \bm \lambda_c}$ is the canonical generator.

We also need to consider elements with \emph{infinitely many} nonvanishing component.
Let therefore $\wt M^\pm_c \ceqq \prod_{Q^+_{\phi_0}} M^\pm_c[\mu]$,
and
\begin{equation}
	\wt M_c^- \wh \ots \wt M_c^+ \ceqq \prod_{(Q^+_{\phi_0})^2} \bigl( M_c^-[\mu] \ots M_c^+[\mu'] \bigr).
\end{equation}
These vector spaces still carry natural structures of $U\mf g_r$-modules.
Analogously,
the algebraic duals $\bigl( M^\pm_c \bigr)^{\! \dual}$ are $\mb C$-linearly isomorphic to $\prod_{Q^+_{\phi_0}} \bigl( M^\pm_c[\mu] \bigr)^{\! \dual}$;
they contain the \emph{graded duals} $\bigl( M^\pm_c\bigr)^{\! \ast}\ceqq \bops_{Q^+_{\phi_0}} \bigl( M^\pm_c[\mu] \bigr)^{\! \dual}$.
These are all \emph{left} $U\mf g_r$-modules via the antipode-pullback action,
defined by
\begin{equation}
	\Braket{ X \varphi,v } = \Braket{ \varphi,\iota(X) v },
	\qquad \varphi \in \bigl( M^\pm_c \bigr)^{\! \dual},
	\quad X \in U\mf g_r,
	\quad v \in M^\pm_c.
\end{equation}
When $M^\pm_c$ are simple,
there are left/right $\mb C$-linear musical embeddings
\begin{equation}
	\begin{aligned}
		\label{eq:musical_maps_shapovalov}
		\mc S_c^{\flat,L} \cl M^-_c \lhra & \bigl( M^+_c \bigr)^{\! \dual},
		\qquad v \lmt \mc S_c(v,\bullet),
		\\
		\mc S_c^{\flat,R} \cl M^+_c \lhra & \bigl( M^-_c \bigr)^{\! \dual},
		\qquad w \lmt \mc S_c(\bullet,w),
	\end{aligned}
\end{equation}
mapping isomorphically onto the graded duals.
By orthogonality,
the nonvanishing components of the Shapovalov form are the diagonal terms
\begin{equation}
	\mc S_{c,\mu} \ceqq \eval[1]{\mc S_c}_{M^-_c[\mu] \ots M^+_c[\mu]},
	\qquad \mu \in Q^+_{\phi_0},
\end{equation}
so that overall
\begin{equation}
	\label{eq:shapovalov_components}
	\mc S_c = (\mc S_{c,\mu})_\mu \in \prod_{Q^+_{\phi_0}} \bigl( M^-_c[\mu] \bigr)^{\! \dual} \ots \bigl( M^+_c[\mu] \bigr)^{\! \dual} \sse \bigl( M^-_c\bigr)^{\! \dual}\wh \ots \bigl( M^+_c \bigr)^{\! \dual}.
\end{equation}
The arrows~\eqref{eq:musical_maps_shapovalov} can be uniquely extended to isomorphisms $\wt M^\pm_c \lxra{\simeq} \bigl( M^\mp_c \bigr)^{\! \dual}$,
abusively written the same.
Furthermore,
the maps~\eqref{eq:musical_maps_shapovalov} are $U\mf g_r$-linear,
as $\mc S_c$ is $\iota$-contragredient,
and so the same holds for their extensions and (respective) inverses $\mc S_c^{\sharp,L}$ and $\mc S_c^{\sharp,R}$.
Finally,
the inverse Shapovalov form is the quadratic tensor
\begin{equation}
	\label{eq:canonical_quadratic_tensor}
	F_c \ceqq \bigl( \mc S_c^{\sharp,L} \ots \mc S_c^{\sharp,R} \bigr) \mc S_c \in \wt M^+_c \wh \ots \wt M^-_c.
\end{equation}
It consists of a sequence of components $F_{c,\mu} \in M^+_c[\mu] \ots M^-_c[\mu]$,
obtained by dualising each term of~\eqref{eq:shapovalov_components}.
The tensor factors are now swapped,
but more importantly:

\begin{lemm}
	\label{lem:invariance_canonical_quadratic_tensor}

	One has $F_c \in \bigl( \,
		\wt M^+_c \wh \ots \wt M^-_c \bigr)^{\! \mf g_r}$.
\end{lemm}

\begin{proof}
	Tautologically,
	one has $\mc S_c \circ \Delta(X) = 0 \cl M^-_c \ots M^+_c \to \mb C$ for all $X \in \mf g_r$,
	so this is a consequence of the equivariance of (the inverse of the extension of)~\eqref{eq:musical_maps_shapovalov}.
\end{proof}

\begin{rema}
	There are analogous embeddings
	\begin{equation}
		U (\mf u^\pm) \lhra \wt U (\mf u^\pm) \ceqq \prod_{Q^+_{\phi_0}} U(\mf u^\pm)_{\pm \mu},
	\end{equation}
	with respect to the weight-gradings $U(\mf u^\pm) = \bops_{Q^+_{\phi_0}} U(\mf u^{\pm})_{\pm \mu}$,
	constructed as in~\eqref{eq:grading_symmetric_algebra_nilradical}.
	Whenever needed,
	one can consider the Shapovalov form as an arrow $U(\mf u^+) \ots U(\mf u^-) \to \mb C$,
	and in turn $F_c \in \wt U(\mf u^-) \wh \ots \wt U(\mf u^+)$;
	the corresponding $\mf g_r$-invariant element can be written $F_c(w^+_c \ots w^-_c) \in \wt M^+_c \wh \ots \wt M^-_c$,
	acting on the left.
\end{rema}

\subsection{Nongeneric extension:
	expansions at infinity}
\label{sec:holomorphicity_at_infinity}

Here we prove that $F_c$ admits a formal Taylor expansion at infinity,
and that (the alternating version of) the first-order term of this expansion is the Lie--Poisson bivector field of $\mc O_{\bm \lambda}$.

\subsubsection{}

Prop.~\ref{prop:pbw_basis} yields a free set of generators of $M^+_c[\mu] \simeq U(\mf u^-)_{-\mu}$.
We give it a total order such that $\abs f$ is nonincreasing along $f \in \on{Dec}_{\nu_0}(\mu)$,
i.e.,
the number of factors in each monomial is nonincreasing.

Since $\bm \chi$ is nonsingular,
consider the dual of the basis $\bigl( X_{\alpha_k,i} \bigr)_{k,i}$ of $\mf u^-$ determined by~\eqref{eq:variables_pbw_basis} +~\eqref{eq:order_basis_nilradical},
with respect to \emph{the opposite of} the nondegenerate pairing $B_{\bm \lambda}^{\bm \psi}$ of~\eqref{eq:nondegenerate_pairing_deeper_character}.
By~\eqref{eq:relation_shapovalov_nonsingular_character},
this is a $\mb C$-basis $\bigl( Y_{\alpha_l,j} \bigr)_{l,j}$ of $\mf u^+$ determined by
\begin{equation}
	\label{eq:dual_basis_positive_nilradical}
	\mc S_c( Y_{\alpha_k,i},X_{\alpha_l,j}) = c\delta_{kl}\delta_{ij},
	\qquad k,l \in \set{1,\dc,m},
\end{equation}
with $i \in \set{0,\dc,d_{\alpha_k}-1}$ and $j \in \set{0,\dc,d_{\alpha_l-1}}$.
Moreover,
note that $B^{\bm \psi}_{\bm \lambda} \bigl( E_\alpha \varepsilon^i,E_\beta \varepsilon^j \bigr) = (0)$ unless if $\alpha = -\beta \in \Phi$,
so that $Y_{\alpha,i} \in (\mf g_r)_\alpha$ (cf.~\eqref{eq:deeper_root_splitting}).\fn{
Exmp.~\ref{ex:indecomposable_weight_space} shows more precisely that $Y_{\alpha,i} = \sum_{j = d_\alpha - 1 - i}^{d_\alpha-1} c^{(\alpha,i)}_j E_\alpha \varepsilon^j$,
where $c_i^{(\alpha,i)} \in \mb C$ is a polynomial in $\Braket{ \lambda_0,H_\alpha },\dc,\Braket{ \lambda_{d_\alpha-2},H_\alpha }$ and $\Braket{ \lambda_{d_\alpha-1},H_\alpha }^{\pm 1}$;
e.g.,
$\Braket{ \lambda_{d_\alpha-1},H_\alpha } Y_{\alpha,0} = E_\alpha \varepsilon^{d_\alpha-1}$.}
Finally,
one obtains a $\mb C$-basis of $M^-_c[\mu]$ by acting on the canonical generator via
\begin{equation}
	\label{eq:dual_basis_weight_space}
	\bm Y_{f,\bm i} \ceqq \prod_{j = 1}^m \Bigl( \prod_{k = 1}^{f_j} Y_{\alpha_j,\bm i^{(j)}_k} \Bigr),
	\qquad f \in \on{Dec}_{\nu_0}(\mu),
	\quad \bm i \in \prod_{j = 1}^m \on{Exp}_{f,\alpha_j},
\end{equation}
and keeping the ordering (cf.~\eqref{eq:generators_weight_space}).

Now in general $\mc S_c(\bm Y_{f,\bm i},\bm X_{g,\bm j}) \in \mb C[c]$ is a polynomial function of the dilation parameter.
Overall,
this restriction of the Shapovalov form is encoded in the corresponding square matrix
\begin{equation}
	\label{eq:shapovalov_matrix}
	A [\mu] \in \End_{\mb C}(\mb C^N) \ots \mb C[c],
	\qquad N = N_\mu \ceqq \dim_{\mb C} \bigl( M^\pm_c[\mu] \bigr),
\end{equation}
and the corresponding component $F_{c,\mu} \in M^+_c[\mu] \ots M^-_c[\mu]$ of $F_c$ by the inverse matrix---%
if it exists.
To gather some facts,
denote by $\mb C^{\leq \kappa}[c] \sse \mb C[c]$ the subspace of polynomials of degree bounded above by an integer $\kappa \geq 0$;
then:

\begin{prop}[Cf.~\cite{alekseev_lachowska_2005_invariant_star_product_on_coadjoint_orbits_and_the_shapovalov_pairing},
		Prop.~3.1]
	\label{prop:shapovalov_matrix}

	Extract elements $\bm Y_{f,\bm i} w^-_c$ and $\bm X_{g,\bm j} w^+_c$ from the PBW bases of $M^-_c[\mu]$ and $M^+_c[\mu]$,
	respectively;
	and let $k \ceqq \abs f$ and $l \ceqq \abs g$.
	Then:
	\begin{enumerate}
		\item one has $\mc S_c (\bm Y_{f,\bm i},\bm X_{g,\bm j}) \in \mb C^{\leq \kappa}[c]$,
		      with $\kappa \ceqq \min \set{k,l}$;

		\item if $k = l$,
		      then $\mc S_c (\bm Y_{f,\bm i},\bm X_{g,\bm j}) \in \mb C^{\leq k-1}[c]$ if and only if $(f,\bm i) \neq (g,\bm j)$;

		\item and one has
		      \begin{equation}
			      \mc S_c (\bm Y_{f,\bm i},\bm X_{f,\bm i}) = d c^k + P,
			      \qquad d \in \mb Z_{> 0},
			      \quad P \in \mb C^{\leq k-1}[c].
		      \end{equation}
	\end{enumerate}
\end{prop}

\begin{proof}
	Simplify the above notation to $\bm Y_{f,\bm i} = Y_1 \dm Y_k$ and $\bm X_{g,\bm j} = X_1 \dm X_l$.

	For the first statement,
	apply Lem.~\ref{lem:abstract_commutator} iteratively (for $\mc R = U\mf g_r$).
	This yields
	\begin{equation}
		Y_k \dm Y_1 \cdot X_1 \dm X_l = \sum_{I_k \discup J_k = I_{k-1}} \dm \sum_{I_1 \discup J_1 = I_0} \bm X_{I_k} \cdot [Y_k,\bm X]_{J_k} \dm [Y_1,\bm X]_{J_1} \in U\mf g_r.
	\end{equation}
	By construction,
	the projection $\pi_{\bm \phi}$ annihilates all the summands where $I_k \neq \vn$,
	whence
	\begin{equation}
		\label{eq:explicit_shapovalov}
		\mc S_c( \bm Y_{f,\bm i},\bm X_{g,\bm j}) = (-1)^k \sum_{J_1 \discup \dm \discup J_k = I_0} \Braket{ \bm \chi_c,\pi_{\bm \phi} \bigl( [Y_k,\bm X]_{J_k} \dm [Y_1,\bm X]_{J_1} \bigr) } \in \mb C[c].
	\end{equation}
	Now observe that $[Y_i,\bm X]_{J_i} \in \mf g_r$ for $i \in \set{1,\dc,k}$ (and it is a nested Lie bracket there whenever $J_i \neq \vn$).
	Hence,
	$[Y_k,\bm X]_{J_k} \dm [Y_1,\bm X]_{J_1} \in U^{\leq k} \mf g_r$,
	using the standard filtration of the UEA,
	which implies that $\pi_{\bm \phi} \bigl( [Y_k,\bm X]_{J_k} \dm [Y_1,\bm X]_{J_1} \bigr) \in U^{\leq k} \mf l$:
	after evaluation at $\bm \chi_c = c \bm \chi$,
	every summand of~\eqref{eq:explicit_shapovalov} lies in $\mb C^{\leq k}[c]$.
	The same argument can be run by exchanging the roles of $\bm Y_{f,\bm i}$ and $\bm X_{g,\bm j}$---%
	i.e.,
	making the variables $X_1,\dc,X_l$ commute to the left.

	For the second statement,
	suppose that $I_0 = \set{1,\dc,k}$,
	and (looking at~\eqref{eq:explicit_shapovalov}) choose a partition $I_0 = J_1 \discup \dm \discup J_k$ such that $J_i = \vn$ for some $i \in I_0$.
	Using again Lem.~\ref{lem:abstract_commutator},
	and writing $Z_j \ceqq [Y_j,\bm X]_{J_j}$ for $j \in I_0$ (so that $Z_i = Y_i$),
	one has
	\begin{align}
		 & Z_k \dm Z_{i+1} \cdot Y_i \cdot Z_{i-1} \dm Z_1                                                                                                                                                    \\
		 & = \sum_{\wt I \discup \wt J = \set{ i-1,\dc,1 }} Z_k \dm Z_{i+1} \cdot \Bigl( \prod_{\wt I} [Y_{\wt i},\bm X]_{J_{\wt i}} \Bigr) \cdot \bigl[ \dm \bigl[ [Y_i,Z_{\wt j_1}],Z_{\wt j_2} \bigr] \dm,
		Z_{\wt j_l} \bigr] \in U\mf g_r,
	\end{align}
	where $\wt J = \set{\wt j_1,\dc,\wt j_l} \sse \set{i-1,\dc,1}$.
	But $\pi_{\bm \phi}$ annihilates the summand with $\wt J = \vn$,
	whence $\pi_{\bm \phi} (Z_k \dm Z_{i+1} \cdot Y_i \cdot Z_{i-1} \dm Z_1) \in U^{\leq k-1}\mf l$.

	Thus,
	the top-degree contribution of~\eqref{eq:explicit_shapovalov} comes from partitions with nonempty parts,
	i.e.,
	where $J_i$ is a singleton for $i \in \set{1,\dc,k}$:
	\begin{equation}
		\label{eq:top_degree_shapovalov}
		\mc S_c( \bm Y_{f,\bm i},\bm X_{g,\bm j}) + (-1)^{k+1} \sum_{\mf S_k} \Braket{ \bm \chi_c,\pi_{\bm \phi} \bigl( [Y_k,X_{\sigma_k}] \dm [Y_1,X_{\sigma_1}] \bigr) } \in \mb C^{\leq k-1}[c],
	\end{equation}
	invoking the symmetric group $\mf S_k$---%
	of $I_0$.
	Now decompose uniquely
	\begin{equation}
		[Y_i,X_{\sigma_i}] = W_i^- + W_i + W^+_i \in \mf g_r,
		\qquad i \in I_0,
	\end{equation}
	where $W_i^\pm \in \mf u^\pm$ and $W_i \in \mf l$.
	Yet another application of Lem.~\ref{lem:abstract_commutator} shows that
	\begin{equation}
		\pi_{\bm \phi} \bigl( [Y_k,X_{\sigma_k}] \dm [Y_{i+1},X_{\sigma_{i+1}}] \cdot (W_i^- + W^+_i) \cdot [Y_{i-1},X_{\sigma_{i-1}}] \dm [Y_1,X_{\sigma_1}] \bigr) \in U^{\leq k-1} \mf l,
	\end{equation}
	so that~\eqref{eq:top_degree_shapovalov} simplifies to
	\begin{equation}
		\mc S_c (\bm Y_{f,\bm i},\bm X_{g,\bm j}) + (-1)^{k+1} \sum_{\mf S_k} \Braket{ \bm \chi_c,W_k \dm W_1 } \in \mb C^{\leq k-1}[c].
	\end{equation}
	But by construction $W_i = \pi_{\bm \phi}\bigl( [Y_i,X_{\sigma_i} ] \bigr) \in \mf l$,
	whence
	\begin{equation}
		\label{eq:top_degree_shapovalov_2}
		\Braket{ \bm \chi_c,W_k \dm W_1 } = \Braket{ \bm \chi_c,W_k} \dm \Braket{ \bm \chi_c,W_1 } = c^k B^{\bm \psi}_{\bm \lambda} (Y_k,X_{\sigma_k}) \dm B^{\bm \psi}_{\bm \lambda} (Y_1,X_{\sigma_1}).
	\end{equation}

	Finally,
	use the choice of $\mb C$-bases.
	Rewrite
	\begin{equation}
		\bm Y_{f,\bm i} = \wt Y_1^{k_1} \dm \wt Y_s^{k_s},
		\qquad k_1,\dc,k_s \in \mb Z_{\geq 0},
	\end{equation}
	where $( \wt Y_1,\dc,\wt Y_s )$ is the (ordered) dual basis of $\mf u^+$ constructed in~\eqref{eq:dual_basis_positive_nilradical}---%
	so that $s = \dim_{\mb C}(\mf u^\pm) = \sum_{\nu_0} d_\alpha$ and $k = \sum_{i = 1}^s k_i$.
	Denoting by $( \wt X_1,\dc,\wt X_s )$ the aforementioned $\mb C$-basis of $\mf u^-$,
	note that if $k_i$ and~\eqref{eq:top_degree_shapovalov_2} are nonzero then there are $k_i$ occurrences of $\wt X_i$ in the element $\bm X_{g,\bm j}$:
	on the whole,
	$\bm X_{g,\bm j}$ is a monomial in $\wt X_1^{k_1},\dc,\wt X_s^{k_s}$,
	and the choice of ordering yields $\bm X_{g,\bm j} = \wt X_1^{k_1} \dm \wt X_s^{k_s} = \bm X_{f,\bm i}$.
	Conversely,
	if $(f,\bm i) = (g,\bm j)$,
	compute
	\begin{equation}
		\sum_{\mf S_k} \Biggl( \prod_{j = k}^1 B^{\bm \psi}_{\bm \lambda} (Y_j,X_{\sigma_j}) \Biggr) = \abs{ \mf S_{k_1,\dc,k_s} } = \prod_{i = 1}^s (k_i !) \eqqc d \in \mb Z_{> 0},
	\end{equation}
	where $\mf S_{k_1,\dc,k_s} \sse \mf S_k$ is the subgroup preserving the partition $k = \sum_i k_i$---%
	i.e.,
	matching any element $\wt Y_i$ with its dual $\wt X_i$.
	This concludes the proof of the second statement,
	and incidentally proves the third one.
\end{proof}

\begin{lemm}
	\label{lem:abstract_commutator}

	Let $\mc R$ be an associative (unital) $\mb C$-algebra,
	and $n \geq 1$ an integer.
	Then,
	given elements $X,Y_1,\dc,Y_n \in \mc R$:
	\begin{enumerate}
		\item
		      one has
		      \begin{equation}
			      X \cdot Y_1 \dm Y_n = \sum_{I \discup J = \set{1,\dc,n}} \bm Y_I \cdot [X,\bm Y]_J \in \mc R,
		      \end{equation}
		      writing
		      \begin{equation}
			      \bm Y_I \ceqq Y_{i_1} \dm Y_{i_k},
			      \qquad I = \set{i_1,\dc,i_k} \sse \set{1,\dc,n},
		      \end{equation}
		      for a suitable $k \in \set{0,\dc,n}$;

		\item and one has
		      \begin{equation}
			      [X,\bm Y]_J \ceqq \bigl[ \dm \bigl[ [ X,Y_{j_1}],Y_{j_2} \bigr] \dm,
			      Y_{j_{n-k}} \bigr],
			      \qquad J = \set{j_1,\dc,j_{n-k}} \sse \set{1,\dc,n}.
		      \end{equation}
	\end{enumerate}
	(With the conventions $\bm Y_{\vn} \ceqq 1$ and $[X,\bm Y]_{\vn} \ceqq X$.)
\end{lemm}

\begin{proof}
	Postponed to~\ref{proof:lem_abstract_commutator}.
\end{proof}

\subsubsection{}

By Prop.~\ref{prop:shapovalov_matrix},
denoting by $(\bm Y_1 w^-_c,\dc,\bm Y_N w^-_c)$ and $(\bm X_1 w^+_c,\dc,\bm X_N w^+_c)$ the mutually-dual bases of $M^\pm_c[\mu]$,
one can write
\begin{equation}
	\label{eq:shapovalov_matrix_coefficient}
	A[\mu]_{ij} = d_{ij} c^{l_i} + P_{ij},
	\qquad i,j \in \set{1,\dc,N},
\end{equation}
in the notation of~\eqref{eq:shapovalov_matrix}.
Here $l_1 \geq \dm \geq l_N \in \mb Z_{\geq 1}$ are the nonincreasing lengths of the monomials $\bm Y_i,\bm X_i$,
while $d_{ij} \in \mb C$ and $P_{ij} \in \mb C^{\leq l_i - i}[c]$;
hence,
moreover,
one has $d_i \ceqq d_{ii} \neq 0$,
while $d_{ij} = 0$ for $i < j$.

Introduce then the $N$-by-$N$ diagonal matrix $D = D[\mu]$ with coefficients $(d_1 c^{l_1},\dc,d_N c^{l_N})$,
and the unipotent lower-triangular matrix $C = C[\mu]$ with (constant) coefficients $C_{ij} \ceqq \frac{d_{ij}}{d_i}$.
Then,
invoking the $N$-by-$N$ identity matrix $\on I_N$:
\begin{coro}[Cf.~\cite{alekseev_lachowska_2005_invariant_star_product_on_coadjoint_orbits_and_the_shapovalov_pairing},
		Prop.~3.1]
	\label{cor:factorisation_shapovalov_matrix}

	There is a matrix factorisation
	\begin{equation}
		A[\mu] = D C \wt Q \in \End_{\mb C}(\mb C^N) \ots \mb C [c^{\pm 1}],
	\end{equation}
	where $\wt Q - \on I_N$ has coefficients in $c^{-1}\mb C[c^{-1}] \sse \mb C[c^{\pm 1}]$.
\end{coro}

\begin{proof}
	Postponed to~\ref{proof:cor_factorisation_shapovalov_matrix}.
\end{proof}

\subsubsection{}

Cor.~\ref{cor:factorisation_shapovalov_matrix} implies that
\begin{equation}
	\det \bigl( A[\mu] \bigr) - \Bigl( \prod_{i = 1}^N d_i \Bigr) c^{\bm l} \in \mb C^{\leq \bm l-1}[c],
	\qquad \bm l \ceqq \sum_{i = 1}^N l_i \in \mb Z_{\geq 0}.
\end{equation}
Hence,
if $\mu \neq 0$ (i.e.,
if $\bm l > 0$) one sees that $A[\mu]$ is \emph{invertible} for all but finitely many values of $c \in \mb C^{\ts}$.
In turn,
the Shapovalov form is nondegenerate for all but countably many values of the dilation parameter.
(Recall that $M^-_c[0] \ots M^+_c[0] = \mb C w^-_c \ots \mb C w^+_c$,
and that $\mc S_c (w^-_c,w^+_c) = 1$.)

Moreover,
by Cor.~\ref{cor:factorisation_shapovalov_matrix},
the quadratic tensor $F_{c,\mu} \in U(\mf u^-)_{-\mu} \ots U(\mf u^+)_\mu$ has coefficients
\begin{equation}
	\label{eq:inverse_shapovalov_component}
	f_{ij} = f_{ij}[\mu] \ceqq \bigl( \wt Q^{-1} C^{-1} D^{-1} \bigr)_{ij} = \frac 1{d_j} c^{-l_j} \sum_{k = j}^N \bigl( \wt Q^{-1} \bigr)_{ik} (C^{-1})_{kj} \in \mb C \llb c^{-1} \rrb,
\end{equation}
for $i,j \in \set{1,\dc,N}$.
Here the inverse of $\wt Q$ is taken formally,
and it is thus a matrix with coefficients in $\mb C \llb c^{-1} \rrb$.\fn{
	Precisely,
	we invert $\det \bigl( \wt Q \bigr) \in \mb C[c^{-1}]$ in the completion,
	noting that $\det \bigl( \wt Q \bigr) - 1 \in c^{-1} \mb C[c^{-1}]$.}~It follows that $\wt Q^{-1} - \on I_N \in c^{-1} \mb C \llb c^{-1} \rrb$,
and indeed~\eqref{eq:inverse_shapovalov_component} defines a formal power series in $\hs \ceqq c^{-1}$.

Finally,
we gather all the formal inverses of Shapovalov components:

\begin{enonce}{Proposition/Definition}[Cf.~\cite{alekseev_lachowska_2005_invariant_star_product_on_coadjoint_orbits_and_the_shapovalov_pairing},
		Prop.~3.2 +~4.12]
	\label{prop:inverse_shapovalov_expansion_infinity}

	Consider the element
	\begin{equation}
		\label{eq:inverse_shapovalov_whole}
		F_{\hs^{-1}} \ceqq \bigl( F_{\hs^{-1},\mu} \bigr)_{\! \mu} \in \prod_{Q^+_{\phi_0}} \Bigl( \bigl( U (\mf u^-)_{-\mu} \ots U (\mf u^+)_\mu \bigr) \llb \hs \rrb \Bigr) = \bigl( \wt U(\mf u^-) \wh \ots \wt U(\mf u^+) \bigr) \llb \hs \rrb,
	\end{equation}
	and denote by $P \cl (U\mf g_r)^{\ots 2} \to (U\mf g_r)^{\ots 2}$ the factor-swapping map (tacitly extended in $\mb C_{\hs}$-bilinear fashion).
	Then:
	\begin{enumerate}
		\item the product~\eqref{eq:inverse_shapovalov_whole} actually lies in $\bigl( U (\mf u^-) \ots U (\mf u^+) \bigr) \llb \hs \rrb$;

		\item one has $(F_{\hs^{-1}} - 1) \in O(\hs)$;

		\item and one has $F_{\hs^{-1}} - P(F_{\hs^{-1}}) - \hs \Pi \in O(\hs^2)$.
	\end{enumerate}
\end{enonce}

\begin{proof}
	The fact that $\bigl( \wt Q^{-1} \bigr)_{ik} - \delta_{ik} \in c^{-1}\mb C\llb c^{-1} \rrb = \hs \mb C \llb \hs \rrb$ further yields
	\begin{equation}
		\label{eq:residue_inverse_shapovalov}
		\begin{cases}
			f_{ij} \in \hs^{l_j+1} \mb C \llb \hs \rrb,
			 & \quad i < j,
			\\
			f_{ij} - \frac 1{d_j} \hs^{l_j} (C^{-1})_{ij} \in \hs^{l_j + 1} \mb C \llb \hs \rrb,
			 & \quad i \geq j,
		\end{cases}
	\end{equation}
	in the notation of~\eqref{eq:inverse_shapovalov_component}.
	In particular,
	for a fixed integer $k \geq 0$,
	the only terms contributing to the coefficient of $\hs^k$ (for~\eqref{eq:inverse_shapovalov_whole}) come from basis elements $(\bm Y_i w^-_c,\bm X_j w^+_c)$ with $l_j \leq k$.
	But if $\abs \mu_{\nu_0} \gg 0$ there are no such vectors,
	since there are finitely many elements of $Q^+_{\phi_0}$ which can be written as the sum of $k$ (possibly nondistinct) roots:
	then the first statement follows from the fact that there are finitely many height-bounded elements of $Q^+_{\phi_0}$.

	For the second statement,
	by construction $l_j = 0$ is only possible (in~\eqref{eq:inverse_shapovalov_component}) when $\mu = 0$,
	in which case simply $f[\mu] = 1$.

	Finally,
	to study the first-order term of $F_{\hs^{-1}}$,
	look again at~\eqref{eq:residue_inverse_shapovalov}.
	If $l_N = 2$,
	this implies that the corresponding weight spaces do \emph{not} contribute to the first-order term.
	Else,
	suppose that $l_{m-1} > 1 = l_m = \dm = l_N$ for some (unique) $m \in \set{1,\dc,N}$:
	then~\eqref{eq:residue_inverse_shapovalov} further implies that it is only the lower-triangular part of $F_{\hs^{-1},\mu}$ which contributes,
	and it does so only if $l_j = 1$,
	i.e.,
	if $m \leq j \leq i$.
	This yields $l_i = 1$ as well,
	and the corresponding coefficient of $\hs$ then equals $\frac 1 {d_j} (C^{-1})_{ij} \in \mb C$.
	Now $(l_i,l_j) = (1,1)$ imposes $(\bm Y_i,\bm X_j) \in \mf u^+ \ops \mf u^-$,
	in which case $A_{ij} = \delta_{ij} c \in \mb C$---%
	by the choice of bases.
	Hence,
	$d_{ij} = \delta_{ij}$ (and $P_{ij} = 0$),
	in the notation of~\eqref{eq:shapovalov_matrix_coefficient},
	so that $C_{ij} = 0$ for $m \leq j < i$.
	Then the same holds for the inverse matrix (upon decomposing $C$ into blocks),
	so at infinity one has
	\begin{equation}
		F_{\hs^{-1},\mu} - \hs \sum_{i = m}^N (\bm X_i \ots \bm Y_i) \in \hs^2 \bigl( U (\mf u^-)_\mu \ots U (\mf u^+)_{-\mu} \bigr) \llb \hs \rrb.
	\end{equation}
	The statement follows,
	because every basis vector of $\mf u^\pm$ appears in one---%
	and only one---%
	$\mf Z_{\phi_0}$-weight space,
	and using the identity
	\begin{equation}
		\Pi = \sum_{\nu_0} \Biggl( \,
		\sum_{i = 0}^{d_\alpha-1} \bm X_{\alpha,i} \wdg \bm Y_{\alpha,i} \Biggr) \in \bigwedge^2 \bigl( \mf g_r \bs \mf g_r^{\bm \lambda} \bigr),
	\end{equation}
	which is a consequence of the choice of bases.
\end{proof}

\begin{rema}
	Incidentally,
	each component $c \mt F_{c,\mu}$ is then a genuine rational/meromorphic function on $\mb CP^1$,
	targeting a finite-dimensional complex vector space:
	it is \emph{holomorphic} at infinity,
	and it has poles precisely at the finitely-many values where the corresponding component of the dilated Shapovalov form becomes degenerate.
	(Viewing the whole of $c \mt F_c$ as a holomorphic function at infinity,
	in some suitable Fréchet sense,
	is \emph{not} required for formal deformation quantisation.)
\end{rema}

\subsection{Nongeneric extension:
	associativity}
\label{sec:associativity}

Here we prove that a reduction of~\eqref{eq:inverse_shapovalov_whole} satisfies the associativity constraints~\eqref{eq:compact_associativity}.

\subsubsection{}

Consider the $U\mf g_r$-module $\mc V_0 \ceqq \Ind_{U\mf l}^{U\mf g_r} \mb C_0$,
where $\mb C_0$ is the trivial $U\mf l$-module---%
induced from the zero character.
Identify it with the quotient $U\mf g_r \bs (U\mf g_r \cdot \mf l)$,
and denote by $p \cl U\mf g_r \thra \mc V_0$ the canonical $U\mf g_r$-linear projection.

Then take the $U(\mf u^\pm)$-linear restrictions $p^{\pm} \cl U(\mf u^\pm) \to \mc V_0$ of $p$,
which are equivalent to arrows $M^\mp_c \to \mc V_0$,
abusively written the same;\fn{
	Importantly,
	the maps $p^\pm \cl M^\mp_c \to \mc V_0$ are \emph{not} $U\mf g_r$-linear,
	as $p^\pm \bigl( \Ann_{U\mf g_r}( w^\mp_c ) \bigr) \ssne U\mf g_r \cdot \mf l$.}~explicitly
\begin{equation}
	\label{eq:explicit_projection_vacumm}
	p^\pm(\bm Y w^\mp_c) = \bm Y w_0,
	\qquad \bm Y \in U (\mf u^\pm),
	\qquad w_0 \ceqq 1 + U\mf g_r \cdot \mf l \in \mc V_0.
\end{equation}
As expected,
this breaks part of the $\mf g_r$-invariance:

\begin{lemm}
	\label{lem:projected_inverse_shapovalov}

	Suppose that the element $v \in \bops_{Q^+_{\phi_0}} \bigl( M^+_c[\mu] \ots M^-_c[\mu] \bigr) \sse M^+_c \ots M^-_c$ is $\mf g_r$-\emph{invariant}.
	Then
	\begin{equation}
		(1 \ots p^+ )v \in \bigl( M^+_c \ots \mc V_0 \bigr)^{\! \mf p^+},
		\qquad (p^- \ots 1) v \in \bigl( \mc V_0 \ots M^-_c \bigr)^{\! \mf p^-}.
	\end{equation}
\end{lemm}

\begin{proof}
	Postponed to~\ref{proof:lem_projected_inverse_shapovalov}.
\end{proof}

\subsubsection{}

We will actually use the completed version of Lem.~\ref{lem:projected_inverse_shapovalov}.
Precisely,
for any $U\mf g_r$-module $\mc V$ write
\begin{equation}
	\label{eq:completed_tensor_product_verma}
	\wt M^\pm_c \wh \ots \mc V \ceqq \prod_{Q^+_{\phi_0}} \bigl( M^\pm_c[\mu] \ots \mc V \bigr).
\end{equation}
The map $1 \ots p^+ \cl M^+_c \ots M^-_c \to M^+_c \ots \mc V_0$ extends to an arrow $\wt M^+_c \wh \ots \wt M^-_c \to \wt M^+_c \wh \ots \mc V_0$,
written the same and defined by
\begin{equation}
	\bigl( v^{(+)}_{\mu} w^+_c \ots v^{(-)}_{\mu} w^-_c \bigr)_{\!\mu} \lmt \bigl( v^{(+)}_{\mu} w^+_c \ots v^{(-)}_{\mu} w_0 \bigr)_{\!\mu},
	\qquad v^{(+)}_{\mu} \ots v^{(-)}_{\mu} \in U(\mf u^-)_{-\mu} \ots U(\mf u^+)_{\mu},
\end{equation}
in componentwise Sweedler notation.
Analogously,
there is a $\mb C$-linear map
\begin{equation}
	p^- \ots 1 \cl \wt M^+_c \wh \ots \wt M^-_c \lra \mc V_0 \wh \ots \wt M^-_c \eqqc \prod_{Q^+_{\phi_0}} \bigl( \mc V_0 \ots M^-_c[\mu] \bigr).
\end{equation}

In particular,
there are elements
\begin{equation}
	F^{(+,0)} \ceqq (1 \ots p^+) F_c \in \wt M^+_c \wh \ots \mc V_0,
	\qquad F^{(0,-)} \ceqq (p^- \ots 1) F_c \in \mc V_0 \wh \ots \wt M^-_c,
\end{equation}
and the argument in the proof~\ref{proof:lem_projected_inverse_shapovalov} extends to show that they are invariant under the actions of $\mf p^+$ and $\mf p^-$---%
respectively.
One thus finds $U\mf g_r$-linear maps
\begin{equation}
	\mc F^{(+,0)} \cl M^+_c \lra \wt M^+_c \wh \ots \mc V_0,
	\qquad \mc F^{(0,-)} \cl M^-_c \lra \mc V_0 \wh \ots \wt M^-_c,
\end{equation}
(well) defined by $w^+_c \mt F^{(+,0)}$ and $w^-_c \mt F^{(0,-)}$---%
respectively.\fn{
	In view of the identifications $\Hom_{\mf g_r} \bigl( M^\pm_c,\mc V \bigr) \simeq \Hom_{\mf p^\pm} (\mb C_{\chi_{\pm c}},\mc V) \simeq \mc V^{\mf p^{\pm}}$,
	for any $U\mf g_r$-module $\mc V$.}~They can be uniquely extended to the completions $\wt M^\pm_c$,
and finally we consider the $U\mf g_r$-linear compositions
\begin{equation}
	\label{eq:compositions_dual_shapovalov}
	\begin{tikzcd}
		\mb C \ar{r} & \wt M^+_c \wh \ots \wt M^-_c \ar[bend right=15,swap]{rr}[yshift=-3pt]{1 \ots \mc F^{(0,-)}} \ar[bend left=15]{rr}[xshift=20pt,yshift=3pt]{\mc F^{(+,0)} \ots 1} & & \wt M^+_c \wh \ots \mc V_0 \wh \ots \wt M^-_c.
	\end{tikzcd}
\end{equation}
Here:
(i) the complex line is regarded as the trivial $U\mf g_r$-module;
(ii) the leftmost arrow corresponds to the invariant element $F_c \in \bigl( \,
	\wt M^+_c \wh \ots \wt M^-_c \bigr)^{\! \mf g_r} \simeq \Hom_{\mf g_r} \bigl( \mb C,\wt M^+_c \wh \ots \wt M^-_c \bigr)$ (cf.~Lem.~\ref{lem:invariance_canonical_quadratic_tensor});
and (iii) the completed tensor product on the right is defined as a particular case of
\begin{equation}
	\label{eq:bilateral_completed_tensor_product}
	\wt M^+_c \wh \ots \mc V \wh \ots \wt M^-_c \ceqq \prod_{(Q^+_{\phi_0})^2} \bigl( M^+_c[\mu] \ots \mc V \ots M^-_c[\mu'] \bigr),
\end{equation}
where $\mc V$ is any $U\mf g_r$-module.

Then the upper composition of~\eqref{eq:compositions_dual_shapovalov} corresponds to the invariant element
\begin{equation}
	\label{eq:upper_composition}
	F_c^{(12,3)} \ceqq ( 1 \ots p \ots 1 ) \Bigl( \bigl( \Delta \ots 1 (F_c) \bigr) \cdot (F_c \ots 1) \Bigr) \in \bigl( \,
	\wt M^+_c \wh \ots \mc V_0 \wh \ots \wt M^-_c \bigr)^{\! \mf g_r},
\end{equation}
now regarding $F_c$ as an element of $\wt U(\mf u^-) \wh \ots \wh U(\mf u^+)$ to make sense of coproducts.
The lower one instead maps $1 \in \mb C$ to the invariant element
\begin{equation}
	\label{eq:lower_composition}
	F^{(1,23)}_c \ceqq ( 1 \ots p \ots 1 ) \Bigl( \bigl( 1 \ots \Delta (F_c) \bigr) \cdot (1 \ots F_c) \Bigr).
\end{equation}
(Cf.~\eqref{eq:compact_associativity} to see where this is going.)

\begin{theo}[Cf.~\cite{alekseev_lachowska_2005_invariant_star_product_on_coadjoint_orbits_and_the_shapovalov_pairing},
		Prop.~4.5]
	\label{thm:equality_compositions}

	Suppose that $M^\pm_c$ are \emph{simple:}
	then $F^{(12,3)}_c = F^{(1,23)}_c$.
\end{theo}

\begin{proof}
	Set $w^{(+,0,-)}_c \ceqq w^+_c \ots w_0 \ots w^-_c \in M^+_c[0] \ots \mc V_0 \ots M^-_c[0]$.
	Then
	\begin{equation}
		F^{(12,3)} - w^{(+,0,-)}_c,
		F^{(1,23)} - w^{(+,0,-)}_c \in \prod_{(Q^+_{\phi_0})^2 \sm \set{(0,0)}} \bigl( M^+_c[\mu] \ots \mc V_0 \ots M^-_c[\mu] \bigr).
	\end{equation}
	The statement follows from Lem.~\ref{lem:iterated_invariant_elements}---%
	which in turn relies on Prop.~\ref{prop:invariant_vectors_completed_tensor_product}.
\end{proof}

\begin{lemm}
	\label{lem:iterated_invariant_elements}

	Suppose that $M^\pm_c$ are \emph{simple:}
	then the $\mf g_r$-invariant elements of~\eqref{eq:bilateral_completed_tensor_product} are determined by their zeroth component---%
	inside $M^+_c[0] \ots \mc V \ots M^-_c[0]$.
\end{lemm}

\begin{proof}
	Choose an invariant element $\wh v$ of~\eqref{eq:bilateral_completed_tensor_product}.
	Using the mutually-Shapovalov-dual PBW bases $\mc B^\pm$ of $U(\mf u^\pm) \simeq M^\mp_c$ (cf.~the proof~\ref{proof:prop_invariant_vectors_completed_tensor_product}),
	decompose
	\begin{equation}
		\wh v = \sum_{\mc B^- \ts \mc B^+} (\bm X w^+_c) \ots v_{\bm X,\bm Y} \ots \bm Y w^-_c = \sum_{\mc B^-} (\bm X w^+_c) \ots \Biggl( \,
		\sum_{\mc B^+} v_{\bm X,\bm Y} \ots \bm Y w^-_c \Biggr),
		\qquad v_{\bm X,\bm Y} \in \mc V.
	\end{equation}
	Then let $\mc V^- \ceqq \mc V \wh \ots \wt M^-_c$,
	and note that by hypothesis $\wh v \in \bigl( \wt M^+_c \wh \ots \mc V^- \bigr)^{\! \mf u^+}$.
	By Prop.~\ref{prop:invariant_vectors_completed_tensor_product},
	the vector $\wh v$ is determined by the element
	\begin{equation}
		\label{eq:left_zeroth_component}
		v_1 \ceqq \sum_{\mc B^+} (v_{1,\bm Y} \ots \bm Y w^-_c) \in \mc V^- = \prod_{Q^+_{\phi_0}} \bigl( \mc V \ots M^-_c[\mu ]\bigr).
	\end{equation}

	The important observation now is that $v_1 \in (\mc V^-)^{\mf u^-}$,
	because $\wh v$ is $\mf g_r$-invariant,
	and since $\mf u^-$ acts freely on $\wt M^+_c$.\fn{
		Actually~\eqref{eq:left_zeroth_component} is even $\mf p^-$-invariant,
		in a completed version of Frobenius reciprocity (cf.~\cite[Rmk.~4.2]{alekseev_lachowska_2005_invariant_star_product_on_coadjoint_orbits_and_the_shapovalov_pairing}).}~Hence,
	a recursive usage of Prop.~\ref{prop:invariant_vectors_completed_tensor_product} (with the tensor factors swapped) shows that $v_1$ is determined by $v_{1,1} \in \mc V$.
\end{proof}

\begin{prop}
	\label{prop:invariant_vectors_completed_tensor_product}

	Suppose that $M^\pm_c$ are \emph{simple}:
	then the $\mf u^\pm$-invariant elements of~\eqref{eq:completed_tensor_product_verma} are determined by their zeroth component---%
	inside $M^\pm_c[0] \ots \mc V$.
\end{prop}

\begin{proof}
	Postponed to~\ref{proof:prop_invariant_vectors_completed_tensor_product}.
\end{proof}

\begin{rema}
	Prop.~\ref{prop:invariant_vectors_completed_tensor_product} strengthens~\cite[Prop.~4.1]{alekseev_lachowska_2005_invariant_star_product_on_coadjoint_orbits_and_the_shapovalov_pairing},
	even in the tame/generic case.
	(Therefore,
	strictly speaking,
	our proof of Thm.~\ref{thm:equality_compositions} takes different final steps from the proof of~\cite[Prop.~4.5]{alekseev_lachowska_2005_invariant_star_product_on_coadjoint_orbits_and_the_shapovalov_pairing}.)
\end{rema}

\subsubsection{}

The equality of~\eqref{eq:upper_composition} and \eqref{eq:lower_composition},
for all but countably many values of the dilation parameter $c$,
implies that the same holds for the formal versions---%
by componentwise rationality in $c$,
upon taking Taylor expansions at infinity.
Therefore,
one has
\begin{equation}
	\label{eq:identity_formal_compositions}
	F^{(12,3)}_{\hs^{-1}} = F^{(1,23)}_{\hs^{-1}} \in \bigl( U(\mf u^-) \ots \mc V_0 \ots U(\mf u^+) \bigr) \llb \hs \rrb,
\end{equation}
where we define,
e.g.,
\begin{equation}
	F^{(12,3)}_{\hs^{-1}} \ceqq ( 1 \ots p \ots 1 ) \Bigl( \bigl( \Delta \ots 1 (F_{\hs^{-1}} ) \bigr) \cdot (F_{\hs^{-1}} \ots 1) \Bigr),
\end{equation}
tacitly extending $p$ to an arrow $U\mf g_r \llb \hs \rrb \to \mc V_0 \llb \hs \rrb$,
and using Prop./Def.~\ref{prop:inverse_shapovalov_expansion_infinity} to see that completions are no longer necessary after Taylor expansions.

\subsection{Nongeneric extension:
	conclusion}
\label{sec:deformation_quantisation_conclusion}

Now combine the maps
\begin{equation}
	p^\pm \cl U(\mf u^\pm) \lra \mc V_0 \simeq U\mf g_r \bs (U\mf g_r \cdot \mf l)
\end{equation}
into a single arrow $\bm p \ceqq p^- \ots p^+ \cl U(\mf u^-) \ots U(\mf u^+) \to \mc V_0^{\ots 2}$.
Once more,
extend it tacitly by $\mb C_{\hs}$-linearity to take values into $\mc V_0^{\ots 2} \llb \hs \rrb$.

The main result of this section is that:

\begin{enonce}{Theorem/Definition}[Cf.~\cite{alekseev_lachowska_2005_invariant_star_product_on_coadjoint_orbits_and_the_shapovalov_pairing},
		Thm.~4.9]
	\label{thm:deformation_quantisation}

	The element $\bm B \ceqq \bm p \bigl( F_{\hs^{-1}} \bigr) \in \mc V_0^{\ots 2} \llb \hs \rrb$ yields a \emph{local} $G_r$-\emph{invariant deformation quantisation} of $\bigl( \mc O_{\bm \lambda},\Pi \bigr)$.
\end{enonce}

\begin{proof}
	First,
	note that $\bm B \in \bigl( \mc V_0^{\ots 2} \bigr)^{\! \mf l} \llb \hs \rrb$:
	this is proven analogously to~\ref{proof:lem_projected_inverse_shapovalov},
	observing that $F_{\hs^{-1}} \in \bigl( U(\mf u^-) \ots U(\mf u^+) \bigr)^{\! \mf g_r} \llb \hs \rrb$.
	Hence,
	$\bm B$ defines a formal bidifferential operator on the orbit,
	whose coefficients are finite-order algebraic bidifferential operators---%
	preserving holomorphicity.
	Now projecting the equality~\eqref{eq:identity_formal_compositions} yields
	\begin{equation}
		\label{eq:equality_projected_compositions}
		\bm B^{(12,3)} = \bm B^{(1,23)} \in \mc V_0^{\ots 3} \llb \hs \rrb.
	\end{equation}
	Thus,
	indeed,
	using~\eqref{eq:star_product_from_formal_bidifferential_operator} yields an \emph{associative} operation $\ast$.

	Second,
	note that
	\begin{equation}
		(\mb C \ops \mf u^\pm) \cap \ker(p^\pm) = (\mb C \ops \mf u^\pm) \cap (U\mf g_r \cdot \mf l) = (0) \sse U\mf g_r.
	\end{equation}
	Therefore,
	there are natural $\mb C$-linear embeddings $(\mb C \ops \mf u^\pm) \hra \mc V_0$,
	and in this identification one has $\bm p(1 \ots 1) = 1 \ots 1$ and $\bm p(\Pi) = \Pi$.
	Now Prop./Def.~\ref{prop:inverse_shapovalov_expansion_infinity} implies that the zeroth-order term of $\bm B$ (resp.,
	the antisymmetrisation of its first-order term) acts as the commutative product of holomorphic functions on $\mc O_{\bm\lambda}$ (resp.,
	as the KKS Poisson bracket).
	The conclusion follows from the discussion in \S~\ref{sec:sheaf_quantization}.
\end{proof}

\begin{rema}
	\label{rmk:stein}
	Once more,
	these results can be recast in the complex-algebraic setting,
	deforming the ring of (global) regular algebraic functions on $\mc O_{\bm \lambda} \sse \mf g_r^{\dual}$,
	viewed instead as a smooth complex affine Poisson \emph{variety} (cf.~Rmk.~\ref{rmk:affine_setting}).
	Incidentally,
	this has consequences in the complex-analytic setting:
	the fact that the orbit is Zariski-closed implies that it is also closed for the strong topology,
	whence it is a \emph{Stein manifold}~\cite{stein_1951_analytische_funktionen_mehrerer_komplexer_veraenderlichen_zu_vorgegebenen_periodizitaetsmoduln_und_das_zweite_cousinsche_problem},
	and Cartan's theorems A/B~\cite{cartan_1960_sur_les_fonctions_de_plusieurs_variables_complexes_les_espaces_analytiques} can in principle be used to localise some global constructions on $\mc O_{\bm\lambda}$.
\end{rema}

\begin{rema}
	Using~\cite[Prop.~3.6]{calaque_naef_2015_a_trace_formula_for_the_quantization_coadjoint_orbits} it is possible to show that the $\ast$-product does \emph{not} depend on the choice of marking on the orbit (cf.~\cite[\S~4]{chaffe_rembado_yamakawa_genus_zero_wild_quantum_de_rham_spaces}).
\end{rema}

\section{Generalised irregular (co)vacua and flat connections}
\label{sec:irregular_blocks}

\subsection{}

Here we extend the construction of flat vector bundles of~\cite{felder_rembado_2023_singular_modules_for_affine_lie_algebras_and_applications_to_irregular_wznw_conformal_blocks},
still with a view towards genus-zero irregular conformal blocks in 2d CFT.

\subsection{Generalised irregular (co)vacua}

Let $\Sigma \ceqq \mb CP^1$ be the Riemann sphere,
$n \geq 1$ an integer,
and $\bm a = (a_1,\dc,a_n) \in \Sigma^n$ an $n$-tuple of distinct marked points.
Given integers $r_1,\dc,r_n \geq 1$,
denote by $D \ceqq \sum_i r_i[a_i] \in \mb Z[\Sigma]$ the associated Weil divisor.
If $z \cl U \lxra{\simeq} \mb C$ is a local (affine) chart whose domain $U \sse \Sigma$ contains all marked points,
write $t_i = z(a_i) \in \mb C$ for the position of the $i$-th one,
and let $z_i \ceqq z - t_i$ be the corresponding local coordinate vanishing at $a_i$.

Consider the vector space $\ms O(\ast D) = \ms O_{\Sigma}(\ast D)$ of meromorphic functions on $\Sigma$ with poles at the marked points.
Denote by $\mf g(\ast D) = \mf g_{\Sigma}(\ast D) \ceqq \mf g \ots \ms O(\ast D)$ the Lie algebra of $\mf g$-valued such functions,
equipped with the Lie bracket coming from $\mf g$.
Let also $\ms O(\ast D) \to \mb C(\!(z_i)\!)$ be the map taking Laurent expansions at $a_i \in \Sigma$,
and extend it in $\mb C$-linear fashion to a function $\mc L_{a_i} \cl \mf g(\ast D) \to \mf g(\!(z_i)\!) = \mf g \ots \mb C(\!(z_i)\!)$.

\begin{rema}
	\label{rmk:intrinsic_description}

	Intrinsically,
	denote by $\ms O_i = \ms O_{\Sigma,a_i}$ the local ring of $\Sigma$ at $a_i$,
	let $\wh{\ms O}_i$ be its completion with respect to the maximal ideal $\mf M_i \sse \ms O_i$ (of germs of holomorphic functions vanishing at $a_i$),
	and consider the fraction field $\wh{\ms O}_i \hra \wh{\ms K}_i$.
	Then one has the disc $\wh{\mc D}_i \ceqq \Spec \wh{\ms O}_i$,
	and Lie algebras $\mf g \ots \wh{\ms O}_i \hra \mf g \ots \wh{\ms K}_i$,
	without appealing to local coordinates.
	Conversely,
	the choice of local coordinates yields identifications
	\begin{equation}
		\mb C \set{z_i} \simeq \ms O_i \lhra \wh{\ms O}_i \simeq \mb C \llb z_i \rrb \lhra \mb C (\!(z_i)\!) \simeq \wh{\ms K}_i,
	\end{equation}
	whence $\mf g \llb z_i \rrb \simeq \mf g \ots \wh{\ms O}_i \hra \mf g \ots \wh{\ms K}_i \simeq \mf g(\!(z_i)\!)$.
	What is actually used is the choice of a \emph{uniformiser} $\varpi_i \in \wh{\mf M_i}$:
	cf.~\cite{tsuchimoto_1993_on_the_coordinate_free_description_of_the_conformal_blocks} for coordinate-independent constructions of conformal blocks,
	as well as~\cite[\S~2.3]{biswas_mukhopadhyay_wentworth_2024_geometrization_of_the_tuy_wzw_kz_connection} (and Rmk.~\ref{rmk:changing_trivialization_and_uniformizer}).
\end{rema}

\subsubsection{}

Laurent expansions define a $\mf g(\ast D)$-action on a tensor product of generalised singularity modules,
after fixing \emph{discrete} parameters (parabolic filtrations) and \emph{continuous} ones (characters/levels).

Namely,
choose:
(i) a tuple $\bm \Psi = (\bm \psi_1,\dc,\bm \psi_n) \in \bigl( \mc P_\Phi^{(\infty)} \bigr)^{\! n}$ of parabolic filtrations of depths bounded by $r_1,\dc,r_n \geq 1$,
respectively;
(ii) a tuple $\bm{\Lambda} = (\bm \lambda_1,\dc,\bm \lambda_n) \in \prod_{i = 1}^n \mf Z^{\dual}_{\bm \phi_i}$ of formal types,
where $\bm \phi_i \ceqq \on{Lf}_{r_i}(\bm \psi_i) \in \mc L^{(r_i)}_\Phi$ are the Levi factors;
and (iii) let $\kappa \in \mb C$ be a level.
Then,
writing $M_i \ceqq M^{\bm \psi_i}_{\bm \lambda_i} \sse \wh M^{\bm \psi_i}_{\bm \lambda_i,\kappa} \ceqq \wh M_i$ for $i \in \set{1,\dc,n}$,
consider the (uncompleted) tensor products
\begin{equation}
	\mb M = \mb M^{\bm \Psi}_{\bm{\Lambda}} \ceqq \bots_{i = 1}^n M_i \lhra \bots_{i = 1}^n \wh M_i \eqqc \wh{\mb M}^{\bm \Psi}_{\bm{\Lambda},\kappa} = \wh{\mb M}.
\end{equation}

Summing the slotwise action $f \mt \sum_{i = 1}^n \mc L_{a_i}(f)^{(i)}$ of the Laurent expansions of $f \in \mf g(\ast D)$ yields a Lie-algebra morphism $\mf g(\ast D) \to \mf{gl}_{\mb C} \bigl( \wh{\mb M} \bigr)$,
as the sum of residues of a meromorphic 1-form on $\Sigma$ vanishes (cf.~\cite[Lem.~1.16]{kohno_2002_conformal_field_theory_and_topology}).

\begin{defi}
	\label{def:covacua_vacua}

	The (generalised) space of \emph{Verma irregular covacua} is
	\begin{equation}
		\label{eq:covacua}
		\mc W = \mc W(\Sigma,\bm a,\bm{\Lambda},\mf g,\bm \Psi,\kappa) \ceqq \wh{\mb M} \bs \mf g(\ast D) \wh{\mb M},
	\end{equation}
	and the (generalised) space of \emph{Verma irregular vacua} is
	\begin{equation}
		\label{eq:vacua}
		\mc W^\dagger
		= \mc W^\dagger(\Sigma,\bm a,\bm{\Lambda},\mf g,\bm \Psi,\kappa) \ceqq \Hom_{\mf g(\ast D)} \bigl(\wh{\mb M},\mb C \bigr).
	\end{equation}
\end{defi}

\subsection{Finite description}

One can now state/prove the following identification:
\begin{theo}
	\label{thm:finite_description}

	The natural composition $\mb M \hra \wh{\mb M} \thra \mc W$ yields a $\mb C$-linear isomorphism
	\begin{equation}
		\mb M \bs \mf g \mb M \lxra{\simeq} \mc W.
	\end{equation}
\end{theo}

(Dually,
the composition $\mc W^\dagger \hra \wh{\mb M}^{\dual} \thra \mb M^{\dual}$ yields $\mc W^\dagger \simeq \Hom_{\mf g}(\mb M,\mb C)$.)

\begin{proof}[Reference to a proof]
	This can be regarded as the multi-point case of Cor.~\ref{cor:from_finite_to_loop_algebra_wild},
	and the proof extends from the generic case~\cite{felder_rembado_2023_singular_modules_for_affine_lie_algebras_and_applications_to_irregular_wznw_conformal_blocks}.
	(First show surjectivity inductively,
	along a filtration of $\mc W$;
	then establish a nontrivial inclusion to determine the kernel.)
\end{proof}

\begin{rema}
	Thm.~\ref{thm:finite_description} is specific to the usage of Verma modules,
	and it typically does \emph{not} hold for other category-$\mc O$ representations.
	Incidentally,
	note that it seems necessary to consider even more general setups in order to find $G$-integrable quotients which are not equivalent to the simple $\mf g$-modules featuring in the standard WZNW conformal blocks~\cite{tsuchiya_kanie_1987_vertex_operators_in_conformal_field_theory_on_cp1_and_monodromy_representations_of_braid_groups,tsuchiya_ueno_yamada_1989_conformal_field_theory_on_universal_family_of_stable_curves_with_gauge_symmetries} (cf.~\cite{sorger_1996_la_formule_de_verlinde,kohno_2002_conformal_field_theory_and_topology};
	e.g.,
	the affinisation of $\mf g_r$ and/or the truncated-current version of $\wh{\mf g}$ might feature in an \emph{affine} BGG category-$\mc O$ for TCLAs).
\end{rema}

\begin{rema}
	\label{rmk:quantum_de_rham}

	The triple $\bm \Sigma \ceqq (\Sigma,\bm a,\bm \Lambda)$ underlies a (labelled) wild Riemann sphere~\cite{boalch_2014_geometry_and_braiding_of_stokes_data_fission_and_wild_character_varieties},
	upon identifying each formal type $\bm \lambda_i \in \mf t_r^{\dual}$ with the principal part of a meromorphic $G$-connection germ.
	In particular,
	the moduli of~\eqref{eq:covacua}--\eqref{eq:vacua} identify a (naive) complex symplectic genus-zero wild de Rham subspace $\mc M^*_{\dR} \sse \mc M_{\dR}$,
	corresponding to untwisted irregular-singular (stable) meromorphic connections on the \emph{holomorphically-trivial} principal $G$-bundle $P \ceqq \Sigma \ts G \to \Sigma$:
	with polar divisor bounded by $D$,
	and prescribed (formal) normal forms at each pole.
	However,
	here one must take \emph{nonresonant} normal forms (cf.~Exmp.~\ref{ex:nonresonance}).
	See also~\cite{boalch_2001_symplectic_manifolds_and_isomonodromic_deformations} in the generic case when $G = \GL_m(\mb C)$,\fn{
		Quiver-theoretic descriptions of the general-linear examples,
		beyond the generic case,
		were also given in~\cite{boalch_2012_simply_laced_isomonodromy_systems},
		and generalised in~\cite{hiroe_yamakawa_2014_moduli_spaces_of_meromorphic_connections_and_quiver_varieties}.}~and~\cite[\S~5]{boalch_2007_quasi_hamiltonian_geometry_of_meromorphic_connections} in the generic case for arbitrary reductive structure groups,
	and finally~\cite{yamakawa_2019_fundamental_two_forms_for_isomonodromic_deformations} in the general nonresonant case.

	This viewpoint was recently taken up in~\cite{chaffe_rembado_yamakawa_genus_zero_wild_quantum_de_rham_spaces},
	working in the complex-algebraic category,
	and using the additional data of the polarisations (determined by $\bm \Psi$) to obtain a deformation quantisation of $\mc M^*_{\dR}$ in the general nonresonant case.
	Therefore,
	the present spaces of vacua/covacua may be viewed as quantum versions of wild de Rham spaces.
\end{rema}

\subsection{Flat connections}

Letting the noncoalescing marked points move---%
in the finite part---%
yields a \emph{sheaf} of generalised irregular (co)vacua,
on the configuration space
\begin{equation}
	\bm B_n \ceqq \Set{(t_1,\dc,t_n) \in \mb C^n | t_i \neq t_j \, \text{ if } \, i \neq j}.
\end{equation}
Then Thm.~\ref{thm:finite_description} implies that this actually underlies a trivializable \emph{vector bundle}.
Finally,
the Virasoro uniformisation~\cite{benzvi_frenkel_2004_geometric_realization_of_the_segal_sugawara_construction} suggests that the variation of the coordinate $t_i$ should be controlled by the action of $L_{-1} \in \mf{Witt} \sse \mf{Vir}$.
Indeed,
if $L_{-1}^{(i)}$ denotes the action on the $i$-th slot of $\wh{\mb M}$,
then one can prove that:

\begin{theo}
	\label{thm:flat_connection}

	The (strongly) flat connection
	\begin{equation}
		\label{eq:generalised_irregular_kz}
		\wh \nabla = \dif - \wh \varpi,
		\qquad \wh \varpi \ceqq \sum_{i = 1}^n L_{-1}^{(i)} \dif t_i,
	\end{equation}
	defined on the trivial vector bundle $\wh{\mb M} \ts \bm B_n \to \bm B_n$,
	is compatible with the $\mf g(\ast D)$-action,
	i.e.,
	\begin{equation}
		\bigl[ f,\wh \nabla_{\partial_{t_i}} \bigr]\wh \psi = (\partial_{t_i}f) \wh \psi,
	\end{equation}
	for all local sections $\wh \psi$ of the vector bundle,
	and all local sections $f$ of---%
	the sheaf-theoretic version of---%
	$\mf g(\ast D)$.
\end{theo}

\begin{proof}[Reference to a proof]
	First establish the smoothness of the affine generalised singularity modules,
	as in the generic case~\cite{felder_rembado_2023_singular_modules_for_affine_lie_algebras_and_applications_to_irregular_wznw_conformal_blocks}.
	Then use~\cite[Lem.~12.8]{kac_1990_infinite_dimensional_lie_algebras} to relate the adjoint action of $L_{-1}$ on $\mf g(\!(z_i)\!)$ with the `derivative' operator
	\begin{equation}
		X \ots z_i^m \lmt m X \ots z_i^{m-1},
		\qquad m \in \mb Z,
		\quad X \in \mf g.
	\end{equation}
	(The latter is \emph{universal},
	i.e.,
	independent of the choice of loop-algebra module.)
\end{proof}

\subsubsection{}

Therefore,
the connection~\eqref{eq:generalised_irregular_kz} descends to a (strongly) flat $\mf g$-invariant connection on $\mb M \ts \bm B_n \to \bm B_n$,
i.e.,
a connection on the vector bundle of covacua.
An analogous construction works in the dual/vacua picture.

In the tame case,
when the divisor $D$ is reduced,
one thus finds the Verma version of the (generalised) KZ connection~\cite{knizhnik_zamolodchikov_1984_current_algebra_and_wess_zumino_model_in_two_dimensions}.
Taking generic polarisations at each marked point,
and arbitrary pole orders,
one finds instead the \emph{irregular} KZ connection of ~\cite{felder_rembado_2023_singular_modules_for_affine_lie_algebras_and_applications_to_irregular_wznw_conformal_blocks};
and finally there is a new \emph{nongeneric} generalisation thereof,
from the most general version of the singularity modules.
In particular,
it is a new representation of the \emph{universal} irregular KZ connection~\cite{reshetikhin_1992_the_knizhnik_zamolodchikov_system_as_a_deformation_of_the_isomonodromy_problem},
defined on $(U \mf g_r)^{\ots n} \ts \bm B_n \to \bm B_n$.
(It is the same as a nonautonomous Hamiltonian system,
whose integrability is equivalent to the Yang--Baxter equation for an $r$-matrix with dynamical parameter~\cite{belavin_drinfeld_1982_solutions_of_the_classical_yang_baxter_equation_for_simple_lie_algebras};
cf.~\cite{felder_rembado_2023_singular_modules_for_affine_lie_algebras_and_applications_to_irregular_wznw_conformal_blocks}.)

\begin{rema}
	\label{rmk:quantum_monodromy_1}

	The fundamental group $\on{PBr}_n \ceqq \pi_1 \bigl( \bm B_n,(1,\dc,n) \bigr)$,
	i.e.,
	the pure Artin braid group on $n$ strands,
	acts on the space of (co)vacua in the monodromy representation of the generalised irregular KZ connection.
	In particular,
	restricting to weight spaces yields finite-dimensional modules.

	More generally,
	let $\bm \psi \in \mc P_\Phi^{(\infty)}$ be \emph{any} parabolic filtration,
	and choose a formal type $\bm \lambda \in \mf Z_{\bm \phi}^{\dual}$ for its Levi factor $\bm \phi$.
	Then consider a partition of the set $\set{1,\dc,n}$,
	with parts
	\begin{equation}
		\on{Iso}^{\bm \psi}_{\bm \lambda} \ceqq \Set{ i \in \set{1,\dc,n} | \wh M_i \simeq \wh M^{\bm \psi}_{\bm \lambda,\kappa} } = \Set{ i \in \set{1,\dc,n} | M_i \simeq M^{\bm \psi}_{\bm \lambda} },
	\end{equation}
	and fix such identifications of modules.
	Let $\mf S_n$ be the symmetric group of $\set{1,\dc,n}$,
	and consider the subgroup preserving this partition into isomorphism classes:
	\begin{equation}
		\mf S^{\bm \Psi}_{\bm \Lambda} \ceqq \Set{ \sigma \in \mf S_n | \sigma(\on{Iso}^{\bm \psi}_{\bm \lambda}) \sse \on{Iso}^{\bm \psi}_{\bm \lambda} \text{ for all } (\bm \psi,\bm \lambda) }.
	\end{equation}
	By construction,
	there are (compatible) actions of $\mf S^{\bm \Psi}_{\bm \Lambda}$ on both $\mb M$ and $\wh{\mb M}$,
	permuting the factors of the tensor products,
	which one can combine with the natural action on $\bm B_n \sse \mb C^n$ to construct vector bundles over the quotient space $\bm B^{\bm \Psi}_{\bm \Lambda} \ceqq \bm B_n \bs \mf S^{\bm \Psi}_{\bm \Lambda}$ of \emph{semiordered configurations} (two points are indistinguishable if they lie in one and the same part).
	E.g.,
	in the affine case the total space of the vector bundle is
	\begin{equation}
		\wh{\mb M} \ts^{\mf S^{\bm \Psi}_{\bm \Lambda}} \bm B_n \ceqq \bigl(\wh{\mb M} \ts \bm B_n \bigr) \bs \mf S^{\bm \Psi}_{\bm \Lambda},
	\end{equation}
	acting antidiagonally on the direct product.

	The connection~\eqref{eq:generalised_irregular_kz} is also well-defined on $\wh{\mb M} \ts^{\mf S^{\bm \Psi}_{\bm \Lambda}} \bm B_n \to \bm B^{\bm \Psi}_{\bm \Lambda}$,
	and the $\mf S^{\bm \Psi}_{\bm \Lambda}$-action on $\wh{\mb M}$ commutes with that of $\mf g(\ast D)$---%
	since one sums Laurent expansions over \emph{all} slots.
	Thus,
	there is an induced $\mf g$-invariant connection on $\mb M \ts^{\mf S^{\bm \Psi}_{\bm \Lambda}} \bm B_n \to \bm B^{\bm \Psi}_{\bm \Lambda}$,
	viz.,
	the reduction of the generalised irregular KZ connection with respect to the natural $\mf S^{\bm \Psi}_{\bm \Lambda}$-action on $\mc W$.
	Everything remains flat,
	so there is also a monodromy representation of the \emph{semipure} braid group $\on{Br}_{\bm \Psi,\bm{\bm \lambda}} \ceqq \pi_1 \bigl( \bm B^{\bm \Psi}_{\bm \Lambda},[1,\dc,n] \bigr)$,
	where $[1,\dc,n] \in \bm B_{\bm \Lambda}^{\bm \Psi}$ denotes the $\on S_{\bm \Lambda}^{\bm \Psi}$-orbit of the tuple $(1,\dc,n)$.
	This is a subgroup of the \emph{full/nonpure} braid group on $n$ strands,
	viz.,
	of $\on{Br}_n \ceqq \pi_1 \bigl( \bm B_n \bs \mf S_n,
		\set{1,\dc,n} \bigr)$,
	consisting of braids whose underlying permutation lies in $\mf S^{\bm \Psi}_{\bm \Lambda}$;
	it sits at the middle of the top exact group sequence in the following commutative diagram (and it is a.k.a.~the `mixed' braid group,
	cf.~\cite{manfredini_1997_some_subgroups_of_artin_s_braid_group, doucot_rembado_2025_topology_of_irregular_isomonodromy_times_on_a_fixed_pointed_curve}):
	\begin{equation}
		\begin{tikzcd}[row sep=5pt]
			& & \on{Br}_{\bm \Psi,\bm{\bm \lambda}} \ar[hookrightarrow,dd] \ar[r] & \mf S^{\bm \Psi}_{\bm \Lambda} \ar[hookrightarrow,dd] \ar[rd] & \\
			1 \ar{r} & \on{PBr}_n \ar{ur} \ar{dr} & & & 1 \,
			.
			\\
			& & \on{Br}_n \ar[r] & \mf S_n \ar[ru] &
		\end{tikzcd}
	\end{equation}

	This symmetry-breaking is only seen when choosing representations.
	Otherwise,
	one can act on $\smash{\bigl( U \wh{\mf g} \,
			\bigr)}^{\! \ots n}$ with the whole of $\mf S_n$,
	compatibly with the universal flat connection.
	(On the semiclassical side,
	this is related with the passage from a Poisson de Rham space to a symplectic leaf therein.)
\end{rema}

\begin{rema}
	\label{rmk:quantum_monodromy_2}

	Moreover,
	identify $\bm B_n$ with the regular part of the standard Cartan subalgebra of $\mf g = \mf{gl}_n(\mb C)$;
	the passage to $[\mf g,\mf g] = \mf{sl}_n(\mb C)$ corresponds to subtracting the barycentre of any given configuration,
	which yields a deformation-retraction (of $\bm B_n$) with the standard Cartan subalgebra $\mf t$ in type $A_{n-1}$.
	Here all root subsystems are Levi,
	and they are determined by partitions of $\set{1,\dc,n}$ as above.
	Thus,
	the parameters of~\eqref{eq:covacua}--\eqref{eq:vacua} also determine one stratum of $\mf t$,
	and in turn $\mf S^{\bm \Psi}_{\bm \Lambda}$ is the Weyl group of the associated infinitesimal centraliser $\mf l \sse \mf g$---%
	identifying $\mf S_n \simeq W$.

	The present example is about the motion of marked points,
	while the general situation also involves deformations of irregular types/classes,
	i.e.,
	ultimately,
	of wild Riemann surfaces:
	this yields actions of $G$-braid groups à la Brieskorn--Deligne~\cite{brieskorn_1971_die_fundamentalgruppe_des_raumes_der_regulaeren_orbits_einer_endlichen_komplexen_spiegelungsgruppe,deligne_1972_les_immeubles_des_groupes_de_tresses_generalises},
	generalising from type $A$;\fn{
		In type $A$,
		the topology of the new deformation parameters can be formalised via operadic \emph{cabling},
		both in the pure~\cite{doucot_rembado_tamiozzo_local_wild_mapping_class_groups_and_cabled_braids} and full/nonpure case~\cite{doucot_rembado_2025_topology_of_irregular_isomonodromy_times_on_a_fixed_pointed_curve};
		this is also referred to as a~`braiding of braids'~\cite{ramis_2012_iso_irregular_deformations_of_linear_ode_and_dynamics_of_painleve_equations}.}~more generally,
	it yields actions of \emph{global} wild mapping class groups (cf.~\cite{doucot_rembado_tamiozzo_moduli_spaces_of_untwisted_wild_riemann_surfaces}).
	As mentioned in \S~\ref{sec:some_references},
	the prototype in this sense is the DMT connection~\cite{millson_toledanolaredo_2005_casimir_operators_and_monodromy_representations_of_generalised_braid_groups},
	linked to meromorphic 2d gauge theory in~\cite{boalch_2002_g_bundles_isomonodromy_and_quantum_weyl_groups}:
	its monodromy representation is (isomorphic to) Lusztig's action~\cite{lusztig_1990_quantum_groups_at_roots_of_1} of the $q$-Weyl group (cf.~\cite{soibelman_1990_algebra_of_functions_on_a_compact_quantum_group_and_its_representations,kirillov_reshetikhin_1990_q_weyl_group_and_a_multiplicative_formula_for_universal_r_matrices}),
	on the Jimbo--Drinfel'd quantum group $U_q\mf g$~\cite{jimbo_1985_a_q_difference_analogue_of_u_g_and_the_yang_baxter_equation,drinfeld_1987_quantum_groups}.
	In turn,
	the latter can be viewed as a quantisation of the corresponding Betti space of Stokes data---%
	viz.,
	the dual Poisson--Lie group $G^*$.
	(Cf.~\cite{toledanolaredo_2002_a_kohno_drinfeld_theorem_for_quantum_weyl_groups} for this Kohno--Drinfel'd theorem,
	and~\cite{toledanolaredo_xu_2023_stokes_phenomena_poisson_lie_groups_and_quantum_groups} for recent related work.)
	Nonetheless,
	this is just for generic poles of order 2,
	missing the present extension.
	Cf.~instead~\cite{rembado_2019_simply_laced_quantum_connections_generalising_kz,yamakawa_2022_quantization_of_simply_laced_isomonodromy_systems_by_the_quantum_spectral_curve_method} for the quantisation of the `simply-laced' isomonodromy systems~\cite{boalch_2012_simply_laced_isomonodromy_systems},
	which allows for one-step braid-cabling in the context of a nongeneric pole of order 3,
	generalising KZ/DMT;
	but also their combination,
	i.e.,
	the connection of FMTV~\cite{felder_markov_tarasov_varchenko_2000_differential_equations_compatible_with_kz_equations}.
\end{rema}

\section*{Acknowledgements}

Various exchanges with the mathematicians listed here (in alphabetical order) improved this text:
A.~Alekseev,
P.~Boalch,
M.~Chaffe,
G.~Cotti,
J.~Douçot,
L.~Jaburi,
N.~Nikolaev,
C.~Sabbah,
O.~Schiffmann,
D.~Schwein,
M.~Tamiozzo,
V.~Toledano Laredo,
L.~Topley,
N.~Williams,
and D.~Youmans.
We thank them all.

\appendix

\section{Topological stratifications}
\label{sec:stratifications}

Throughout this section,
let $T$ be a topological space and $(\mc P,\leq)$ a poset.

\subsection{Basic terminology}

\begin{defi}
	\label{def:stratification}

	A $\mc P$-\emph{stratification of} $T$ is a set-theoretic surjection $\varphi \cl T \thra \mc P$,
	with fibres $T_i \ceqq \varphi^{-1}(i) \sse T$ for $i \in \mc P$,
	such that:
	\begin{enumerate}
		\item
		      the subspace $T_i \sse T$ is locally-closed,
		      for all $i \in \mc P$;

		\item
		      the family of subspaces $\set{T_i}_{\mc P}$ is locally-finite;

		\item
		      and the following conditions are equivalent,
		      for any pair $i,j \in \mc P$:
		      \begin{enumerate}
			      \item
			            one has $T_i \cap \ol{T_j} \neq \vn$;

			      \item
			            one has $T_i \sse \ol{T_j}$;

			      \item
			            and one has $i \leq j$.
		      \end{enumerate}
	\end{enumerate}
	The topological subspaces $T_i \sse T$ are the \emph{strata of} $\varphi$.
\end{defi}

\begin{rema}
	\label{rmk:stratifications_and_closures}

	In particular,
	one has a $\mc P$-\emph{partition of} $T$,
	i.e.,
	a disjoint union $T = \bigcup_{\mc P} T_i$ indexed by $\mc P$---%
	with nonempty parts.

	Moreover,
	it follows that
	\begin{equation}
		\label{eq:stratified_closure}
		\ol{T_i} = \bigcup_{j \leq i} T_j,
		\qquad \ol{T_i} \sm T_i = \bigcup_{j < i} T_j = \bigcup_{j < i} \ol{T_j},
		\qquad i \in \mc P.
	\end{equation}
	When $\mc P$ is \emph{finite},
	the latter identity implies that all the strata are automatically locally-closed,
	whence the first two axioms of Def.~\ref{def:stratification} are redundant.
	Conversely,
	the former identity,
	which is a.k.a.~the \emph{strong frontier condition},
	is equivalent to the last axiom.
	(A weaker variant only involves the inclusion of $\ol{T_i}$.)
\end{rema}

\begin{rema}
	In particular,
	a stratum $T_i \sse T$ satisfies $\ol{T_i} = T$ if and only if $i \in \mc P$ is the greatest element,
	and it is then called the \emph{dense stratum}.
	Conversely,
	one has $T_i = \bigcap_{\mc P} \ol{T_j}$ if and only if $i \in \mc P$ is the least element,
	and it is then called the \emph{minimal stratum}---%
	automatically closed.
\end{rema}

\subsection{Substratifications}
\label{sec:substratifications}

\begin{defi}
	\label{def:substratification}

	Let $\varphi \cl T \thra \mc P$ be a $\mc P$-stratification of $T$.
	A \emph{substratification of} $\varphi$ consists of a topological subspace $T' \sse T$,
	such that the restriction $\varphi' \ceqq \eval[1]\varphi_{T'}$ defines a stratification of $T'$ indexed by the subposet $\mc P' \ceqq \varphi(T') \sse \mc P$.
\end{defi}

\begin{rema}
	It follows that the strata $T'_i = (\varphi')^{-1}(i) \sse T'$ of $\varphi'$ are the \emph{nonempty} intersections of $T'$ with the strata of $\varphi$,
	for all $i \in \mc P'$.
	Moreover,
	if the least/greatest element $i \in \mc P$ lies in $\mc P'$,
	then it is also a least/greatest element therein,
	and the dense/minimal stratum of $\varphi'$ is $T'_i$.
\end{rema}

\subsubsection{}

The condition of Def.~\ref{def:substratification} is not vacuous,
as the surjection $\varphi' \cl T' \thra \mc P'$ need \emph{not} satisfy~\eqref{eq:stratified_closure} in general.
E.g.,
take $T \ceqq \mb R^2$ with Euclidean topology,
and stratify by $\mc P \ceqq \set{0 < 1}$ and $T_0 \ceqq \set{0} \ts \mb R$---%
whence $T_1 = \mb R^2 \sm T_0$;
then consider the (closed) subspace $T' \ceqq \set{(0,1)} \cup \bigl( [0,1] \ts \set{0} \bigr)$.
The issue is that there are points of $\ol{T_1} \cap T' = T'$ which cannot be reached as limit points in $T'_1 = (0,1] \ts \set{0}$ (the `ambient' stratum $T_1$ approaches $(0,1) \in T'$ through points in $T \sm T'$.)

This example suggests the following sufficient condition:

\begin{lemm}
	\label{lem:closed_substratification}

	Let $T' \sse T$ be a topological subspace such that $T'_i$ is \emph{dense} in $\ol{T_i} \cap T'$,
	for all $i \in \mc P'$.
	Then $\varphi' \cl T' \thra \mc P'$ is a substratification of $\varphi$.
\end{lemm}

\begin{proof}
	Local finiteness/closedness are automatic upon restriction.
	Rather,
	one must prove that the $\mc P'$-partition $T' = \bigcup_{\mc P'} T'_i$ satisfies axiom (3.)~of Def.~\ref{def:stratification}.
	To this end,
	for any $i \in \mc P'$ the given hypothesis yields the strong frontier condition (first taking closure in $T'$):
	\begin{equation}
		\ol{T'_i}^{T'}
		= T' \cap \ol{T_i}
		= T' \cap \bigcup_{j \leq i \in \mc P} T_j
		= \bigcup_{j \leq i \in \mc P'} (T \cap T_j)
		= \bigcup_{j \leq i \in \mc P'} T'_j. \qedhere
	\end{equation}
\end{proof}

\begin{coro}
	\label{cor:union_of_strata}

	For any subposet $\mc P' \sse \mc P$,
	let $T' = T(\mc P') \ceqq \bigcup_{\mc P'} T_i \sse T$.
	Then the subspace $T'$ inherits a substratification of $\varphi$.
\end{coro}

\begin{proof}
	By construction,
	one has $T'_i = T_i$ for $i \in \mc P'$,
	which is dense in $\ol{T_i} \cap T'$.
\end{proof}

\subsection{Quotient stratifications}
\label{sec:quotient_stratifications}

Hereafter,
let $H$ be a topological group,
$T$ a topological $H$-space,
and $\mc P$ an $H$-poset;
i.e.,
$T$ carries a continuous $H$-action by homeomorphisms and $\mc P$ is equipped with an $H$-action by order-preserving bijections.
(We denote the latter by $i \mt g.i$,
for $i \in \mc P$ and $g \in H$.)

\begin{defi}
	A $\mc P$-stratification of $T$ is $H$-\emph{compatible} if it intertwines the $H$-actions on $T$ and $\mc P$,
	viz.,
	if one has
	\begin{equation}
		\label{eq:H_stratification}
		g(T_i) = T_{g.i} \sse T,
		\qquad i \in \mc P.
	\end{equation}
\end{defi}

\begin{rema}
	In view of~\eqref{eq:stratified_closure},
	one also has
	\begin{equation}
		g(\ol{T_i}) = \ol{T_{g.i}},
		\qquad g(\ol{T_i} \sm T_i) = \ol{T_{g.i}} \sm T_{g.i},
		\qquad i \in \mc P,
		\quad g \in H.
	\end{equation}
	Moreover,
	the extremal strata of $\varphi$ are $H$-stable---%
	if they exists.
\end{rema}

\subsubsection{}

If $S \sse T$ is a topological subspace,
consider the setwise stabiliser
\begin{equation}
	H' = N_H(S) \ceqq \Set{ g \in H | g(S) \sse S } \sse H.
\end{equation}
The composition $S \hra T \thra T \bs H$ yields a continuous map $S^H \ceqq S \bs H' \to T \bs H$,
which in general is \emph{not} injective.
Nonetheless,
injectivity holds provided that
\begin{equation}
	Hx \cap S = H' x \sse S,
	\qquad x \in S,
\end{equation}
where $Hx \sse T$ (resp.,
$H'x \sse S$) is the $H$-orbit through $x$ (resp.,
the $H'$-orbit).
In this case,
the inclusion $S^H \hra T \bs H$ corresponds to the subspace of $H$-orbits intersecting $S \sse T$:
this automatically holds for each stratum $T_i \sse T$ of an $H$-compatible stratification,
so that one has topological subspaces $T^H_i \sse T \bs H$.
Moreover,
the equality~\eqref{eq:H_stratification} yields
\begin{equation}
	N_H(T_i) = H^i \ceqq \Set{ g \in H | g.i = i } \sse H.
\end{equation}

Now consider the $H$-action on $\mc P$.
The $H$-orbit of $i \in \mc P$ is denoted by $\ol i \in \mc P \bs H$,
and if $H$ is \emph{finite} then the quotient set $\mc P \bs H$ has a natural partial order such that $\ol i \leq \ol j$ if $i \leq j$~\cite[Prop.~6.12]{williams_a_survey_of_congruence_and_quotients_of_partially_ordered_sets},\fn{
	The tricky part is proving that the induced relation on $\mc P \bs H$ is \emph{antisymmetric},
	for which one uses that any element of $H$ has finite order.}~making it into the \emph{quotient poset}.

Hereafter,
let $\varphi \cl T \thra \mc P$ be an $H$-compatible $\mc P$-stratification of $T$,
where $H$ is a \emph{finite} (discrete) group.

\begin{lemm}
	\label{lem:quotient_stratification_1}

	For all $i,j \in \mc P$,
	the following are equivalent:
	\begin{enumerate}
		\item
		      one has $T^H_i \cap T^H_j \neq \vn$;

		\item
		      one has $T^H_i = T^H_j \sse T \bs H$;

		\item
		      and one has $\ol i = \ol j \in \mc P \bs H$.
	\end{enumerate}
\end{lemm}

\begin{proof}
	Postponed to~\ref{proof:lem_quotient_stratification_1}.
\end{proof}

\subsubsection{}

In view of Lem.~\ref{lem:quotient_stratification_1},
the subspaces $(T \bs H)_{\ol i} \ceqq T_i^H \sse T \bs H$ are well-defined,
for any choice of representative.

\begin{lemm}
	\label{lem:quotient_stratification_2}

	The closure $\ol{T^H_i} \sse T \bs H$ coincides with the subspace of $H$-orbits intersecting $\ol{T_i} \sse T$,
	for all $i \in \mc P$.
\end{lemm}

\begin{proof}
	Postponed to~\ref{proof:lem_quotient_stratification_2}.
\end{proof}

\begin{enonce}{Proposition/Definition}
	\label{prop:quotient_stratification}

	Suppose in addition that $\mc P$ is \emph{finite}.
	Then the surjection $\varphi \bs H \cl T \bs H \thra \mc P \bs H$,
	with fibres $\varphi \bs H^{-1}(\ol i) = (T \bs H)_{\ol i}$,
	is a $\mc P \bs H$-stratification of $T \bs H$:
	it is called the \emph{quotient stratification of} $\varphi$---%
	\emph{modulo} $H$.
\end{enonce}

\begin{proof}
	Lem.~\ref{lem:quotient_stratification_1} yields a suitable partition,
	and we must verify the axioms of Def.~\ref{def:stratification}:
	the first two however are automatic (cf.~Rmk.~\ref{rmk:stratifications_and_closures}.)

	Suppose therefore that $T^H_i \cap \ol{T^H_j} \neq \vn$.
	By Lem.~\ref{lem:quotient_stratification_2},
	there is an $H$-orbit intersecting $T_i$ and $\ol{T_j} \sse T$.
	So there exist a point $x \in T_i$,
	and an element $g \in H$,
	such that $g.x \in \ol{T_j} \cap T_{g.i}$.
	It follows that $g(T_i) \sse \ol{T_j}$,
	whence \emph{all} the orbits through $T_i$ intersect $\ol{T_j}$,
	i.e.,
	one has $T_i^H \sse \ol{T_j^H}$.
	Furthermore,
	one has $g.i \leq j$,
	and it follows that $\ol i \leq \ol j$ in $\mc P \bs H$.
	The remaining implications are clear.
\end{proof}

\begin{rema}
	The dense/minimal stratum of $T \bs H$ is then the topological quotient $T_i \bs H$,
	where $T_i \sse T$ is the dense/minimal stratum of $\varphi$,
	if it exists---%
	noting that $N_H(T_i) = H^i = H$ here.
\end{rema}

\begin{rema}
	\label{rmk:about_stratifications}

	What we call `stratification' here is referred to as a `decomposition' in~\cite[Part~I, \S~1.1]{goresky_macpherson_1988_stratified_morse_theory}.
	Op.~cit.~reserves the former terminology for \emph{Whitney stratifications} of closed subspaces of smooth manifolds (cf.~\cite{thom_1969_ensembles_et_morphismes_stratifies,mather_1970_notes_on_topological_stability,mather_2012_notes_on_topological_stability}).
	But all the finite arrangements $\mc A = \set{A_1,\dc,A_m}$ of affine subspaces $A_i \sse \mb C^n$ yield Whitney stratifications,
	cf.~(the complexified version of)~\cite[Part~III, \S~3.1]{goresky_macpherson_1988_stratified_morse_theory}.
	The corresponding poset $\mc P$ consists of the \emph{flats of} $\mc A$,
	i.e.,
	the intersections of elements of $\mc A$,
	typically ordered by anti-inclusions.
	(The flats are then the closures of the strata.)

	Therefore,
	a particular case of this matches up with the (Whitney) Levi/root-valuation stratifications considered in the body of this text,
	where $\mc A$ consists of linear hyperplanes.
	The main points are that:
	(i) there is an order-inverting bijection between the Levi subsystems $\phi$ of a root system $\Phi$,
	and the flats of its root-hyperplane arrangement,
	via $\phi \mt \ker(\phi)$;
	and (ii) the maximal proper subflats of $\ker(\phi)$ are the traces---%
	on $\ker(\phi)$---%
	of the hyperplanes corresponding to roots in $\Phi \sm \phi$.
\end{rema}

\section{Deferred proofs}
\label{sec:missing_proofs}

\subsection{Proof of Lem.~\ref{lem:levi_subsystems}}
\label{proof:lem_levi_subsystems}

Let $\phi \sse \Phi$ be a root subsystem.
Then~\eqref{eq:tame_stratum} is empty if and only if $\ker(\phi) \sse \ker(\alpha)$ for some $\alpha \in \Phi \sm \phi$,
which is equivalent to $\alpha \in \spann_{\mb C}(\phi)$.
Moreover,
if $\mf t_{\phi}$ is nonempty then tautologically $\phi = \Set{ \alpha \in \Phi | \Braket{ \alpha,X } = 0 }$,
for any $X \in \mf t_{\phi}$.
Finally,
suppose that $\phi$ is the set of roots vanishing on $X \in \mf t$:
then the same is true of any linear combination of its elements.

\subsection{Proof of Lem.~\ref{lem:graded_poset}}
\label{proof:lem_graded_poset}

By definition,
one has $\wt \phi < \phi \in \mc L_\Phi$ if and only if $\ker(\wt \phi) \ssne \ker(\phi)$.
(This is false in general for nonlevi subsystems.)

Suppose further that $\phi$ \emph{covers} $\wt \phi$,
i.e.,
that there is no Levi subsystem in between.
Up to replacing the triple $(\mf g,\mf t,\Phi)$ with $(\mf l_{\wt \phi},\mf t,\wt \phi)$,
it is enough to show that $\rho_\Phi(\phi) = \dim_{\mb C}(\mf Z_{\mf g}) + 1$ if $\phi$ covers $\Phi$.
Now choose a base $\Delta \sse \Phi$ of simple roots such that $\phi = \Phi_{\Sigma} \ceqq \spann_{\mb C}(\Sigma) \cap \Phi$,
for a suitable subset $\Sigma \ssne \Delta$.
It follows that $\abs \Sigma = \abs \Delta - 1$,
since $\Phi < \Phi_{\Sigma'} < \phi$ for $\Sigma \ssne \Sigma'$;
in turn
\begin{equation}
	\rho_\Phi(\phi) = \dim_{\mb C} \bigl( \ker(\Sigma) \bigr) = \dim_{\mb C}(\mf t) - \abs \Sigma = \dim_{\mb C}(\mf Z_{\mf g}) + 1,
\end{equation}
because $\dim_{\mb C}(\mf Z_{\mf g}) = \rk(\mf g) - \rk(\Phi) = \rk(\mf g) - \abs \Delta$,
and noting that $\Sigma \sse \mf t^{\dual}$ is a linearly independent set.

\subsection{Proof of Lem.~\ref{lem:tame_stratification}}
\label{proof:lem_tame_stratification}

Local closedness/finitess are clear,
cf.~Rmk.~\ref{rmk:stratifications_and_closures}.
For the former,
explicitly,
note that
one has $\ol{\mf t_{\phi}} = \ker(\phi) \sse \mf t$,
whence
\begin{equation}
	\label{eq:local_closedness}
	\ol{\mf t_{\phi}} \sm \mf t_{\phi} = \bigcup_{\Phi \sm \phi} \ker(\alpha) \cap \ker(\phi) \sse \mf t,
\end{equation}
which is a (closed) union of hyperplanes in $\ker(\phi)$.

Now,
given $X \in \mf t$,
one tautologically has $X \in \mf t_{\phi_X}$.
Conversely,
if $X \in \mf t_{\phi} \cap \mf t_{\phi'}$ then $\Braket{ \phi \cup \phi',X } = (0)$,
while $\Braket{ \alpha,X } \neq 0$ for $\alpha \in (\Phi \sm \phi) \cup (\Phi \sm \phi')$:
then both $\phi \nsse \phi'$ and $\phi' \nsse \phi$ are impossible.
So indeed one has a partition.

Moreover,
suppose that $X \in \mf t_{\phi} \cap \ker(\phi')$.
Then $\Braket{ \phi \cup \phi',X } = 0$,
while $\Braket{ \alpha,X } \neq 0$ for $\alpha \in \Phi \sm \phi$.
This forces $\phi' \sse \phi$ (i.e.,
$\phi \leq \phi'$ inside $\mc L_\Phi$),
and in turn $\mf t_{\phi} \sse \ker(\phi) \sse \ker(\phi')$.
The remaining implications are clear,
and the statement about extremal elements follows from the fact that $\mf Z_{\mf g} = \ker(\Phi)$.

\subsection{Proof of Lem.~\ref{lem:base_levi_filtrations}}
\label{proof:lem_base_levi_filtrations}

Let $\Delta_s$ be any input base,
and consider the subspace
\begin{equation}
	\mf D \ceqq \Set{ \lambda \in \mb C | \on{Re}(\lambda) \geq 0;
		\text{ if } \on{Re}(\lambda) = 0,
		\text{ then } \on{Im}(\lambda) \geq 0 } \sse \mb C.
\end{equation}
Then a fundamental domain for the action of the Weyl group on $\mf t$ is
\begin{equation}
	\mf C_{\Delta_s} \ceqq \Set{ X \in \mf t | \Braket{ \theta,X } \in \mf D \text{ for } \theta \in \Delta_s } \sse \mf t,
\end{equation}
cf.~\cite[\S~5.2]{crooks_2019_complex_adjoint_orbits_in_lie_theory_and_geometry} and~\cite[\S~2.2]{collingwood_mcgovern_1993_nilpotent_orbits_in_semisimple_lie_algebras}.
(Note that $\mf Z_{\mf g} = \ker(\Phi)$ always lies in $\mf C_{\Delta_s}$,
and the Weyl group acts trivially there.)

Thus,
there exists $w_s \in W$ such that $\mf t_{\phi_{s-1}} \cap \mf C_{\Delta_{s-1}} \neq \vn$ (i.e.,
$\mf t_{\phi_{s-1}}$ contains a $\Delta_{s-1}$-dominant element),
where $\Delta_{s-1} \ceqq w_s(\Delta_s)$.
It follows that $\Delta_{s-1} \cap \phi_{s-1}$ is a base of $\phi_{s-1}$.
Now iterate inside the pair $(\mf l_{\phi_{s-1}},\mf t)$:
there exists $w_{s-1} \in W_{\phi_{s-1}} \sse W$ such that $\mf t_{\phi_{s-2}} \cap \mf C_{\Delta_{s-2}} \neq \vn$,
with $\Delta_{s-2} \ceqq w_{s-1}(\Delta_{s-1})$.
But $W_{\phi_{s-1}}$ acts trivially on $\mf t_{\phi_{s-1}}$,
so that
\begin{equation}
	\mf t_{\phi_{s-1}} \cap \mf C_{\Delta_{s-1}} = \ps{t}{}w^{-1}_{s-1} \bigl( \mf t_{\phi_{s-1}} \cap \mf C_{\Delta_{s-1}} \bigr) \neq \vn.
\end{equation}
Etc.

\subsection{Proof of Lem.~\ref{lem:deeper_pairing}}
\label{proof:lem_deeper_pairing}

Expanding the Lie bracket $[ \cdot,\cdot ] \cl \bigwedge^2 \mf g_r \to \mf g_r$ yields
\begin{equation}
	\label{eq:deeper_bracket}
	[ \bm X,\bm Y ] = \sum_{l = 0}^{r-1} \Biggl(\sum_{i+j = l} [X_i,Y_j] \Biggr) \varepsilon^l,
\end{equation}
in the notation of~\eqref{eq:element_deeper_lie_algebra}.
Then one directly checks that~\eqref{eq:deeper_pairing} is $\ad_{\mf g_r}$-invariant if and only if $c = r$.
Furthermore,
since the bilinear form on $\mf g$ is $\Ad_G$-invariant,
the diagonal action of the subgroup $G \sse G_r$ preserves~\eqref{eq:deeper_pairing}.
As per the $\Ad_{\on{Bir}_r}$-action,
note that the adjoint action of $\mf g_r$ is now skew-symmetric for $(\cdot \mid \cdot)_r$:
the conclusion follows from the surjectivity of the exponential map $\mf{bir}_r \to \on{Bir}_r$.

\subsection{Proof of Lem.~\ref{lem:generalised_irregular_type}}
\label{proof:lem_generalised_irregular_type}

Suppose that $\bm X' = \exp(\ad_{\bm Y}) (\bm X) \in \mf g_r$,
for some $\bm X,
	\bm X' \in \mf g_r$ and $\bm Y \in \mf{bir}_r$,
with coefficients
\begin{equation}
	X_i,X'_i,Y_j \in \mf g,
	\qquad i \in \Set{0,\dc,r},
	\quad j \in \Set{1,\dc,r}.
\end{equation}
Clearly $X_0 = X'_0$,
since the $\on{Bir}_r$-action fixes the leading coefficient.
Then we inductively prove that:
i) $X_i = X'_i$,
for $i \in \Set{0,\dc,k-1}$;
and ii) $[Y_i,X_j] = 0$ if $i+j \leq k-1$.

Assume the statement for some integer $l \in \Set{0,\dc,k-2}$.
Then
\begin{equation}
	\label{eq:recursive_step}
	X'_{l+1} = X_{l+1} + [Y_1,X_l] + \dm + [Y_{l+1},X_0] \in \mf g,
\end{equation}
as all the remaining terms (involving nested Lie brackets of the $X_i$'s and the $Y_i$'s) vanish by recursive hypothesis.
Now rewrite the above as
\begin{equation}
	\ad_{X_0}(Y_{l+1}) = X_{l+1} - X'_{l+1} + [Y_1,X_l] + \dm + [Y_k,X_1],
\end{equation}
and note that the right-hand side lies in $\mf g^{X_0}$,
so that both sides vanish by Rmk.~\ref{rmk:semisimple_trick}.
The same argument can be repeated for $\ad_{X_1},\dc,\ad_{X_l} \in \mf{der}(\mf g)$,
which yields
\begin{equation}
	X_{l+1} - X'_{l+1} = 0 = [Y_{l+1},X_0] = \dm = [Y_1,X_l] \in \mf g.
\end{equation}

\subsection{Proof of Lem.~\ref{lem:commutation_lower_coefficients}}
\label{proof:lem_commutation_lower_coefficients}

Choose an element $\bm Y_k = Y_k \varepsilon^k \in \mf{bir}^{\mf g}_r$,
for some $Y_k \in \mf g$.
Then
\begin{equation}
	\tau_k \bigl( e^{\ad_{\bm Y_k}}(\bm X) \bigr) - \sum_{i = 0}^{k-1} X_i \varepsilon^i = \bigl( X_k + [Y_k,X_0] \bigr) \varepsilon^k \in \mf g_k.
\end{equation}
By Rmk.~\ref{rmk:semisimple_trick},
there is a (unique) element $Y_k \in \ad_{X_0}(\mf g)$ such that $X_k - \ad_{X_0}(Y_k) \in \mf g^{X_0} \sse \mf g$.
Then repeat for (unique) elements
\begin{equation}
	\bm Y_{k+1} = Y_{k+1} \varepsilon^{k+1},\dc,\bm Y_{r-1} = Y_{r-1} \varepsilon^{r-1} \in \mf{bir}^{\ad_{X_0}(\mf g)}_r.
\end{equation}
In the end,
one can assume that the lowest coefficients commute with $X_0$.
If $k = 0$ the algorithm terminates.
Otherwise,
iterate inside $\mf{bir}^{\mf g^{X_0}}_r$.
Namely,
since $[X_0,X_1] = 0$ one has $\ad_{X_1} \bigl( \mf g^{X_0} \bigr) \sse \mf g^{X_0}$,
and repeating the same construction for suitable elements of $\mf{bir}_r^{\ad_{X_1}(\mf g)}$ one can assume that $X_k,\dc,X_{r-1} \in \smash{\bigl( \mf g^{X_0} \bigr)}^{\! X_1} = \mf g^{X_0} \cap \mf g^{X_1} \sse \mf g$.
Then,
if necessary,
iterate inside $\mf{bir}_r^{\mf g^{X_0} \cap \mf g^{X_1}}$;
etc.

\subsection{Proof of Lem.~\ref{lem:wild_infinitesimal_centraliser}}
\label{proof:lem_wild_infinitesimal_centraliser}

Expanding the identity $[\bm Y,\bm X] = 0$,
by~\eqref{eq:deeper_bracket},
yields the following system of equations:
\begin{equation}
	\label{eq:equation_infinitesimal_wild_centraliser}
	\sum_{i+j = l} [Y_i,X_j] = 0 \in \mf g,
	\qquad l \in \Set{0,\dc,r-1}.
\end{equation}
In particular $Y_0 \in \ker(\ad_{X_0})$.

Now fix $l \in \Set{0,\dc,s-1}$,
and suppose inductively that $[Y_i,X_j] = 0$ if $i+j \leq l$.
Then the $(l+1)$-th equation~\eqref{eq:equation_infinitesimal_wild_centraliser} yields
\begin{equation}
	[Y_0,X_{l+1}] + \dm + [Y_l,X_1] = \ad_{X_0}(Y_{l+1}) \in \mf g.
\end{equation}
By hypothesis the left-hand side lies in $\mf g^{X_0}$,
so by Rmk.~\ref{rmk:semisimple_trick} one has
\begin{equation}
	[Y_0,X_{l+1}] + \dm + [Y_l,X_1] = 0 = [Y_{l+1},X_0].
\end{equation}
Repeating the same argument for the semisimple elements $X_1,\dc,X_l \in \mf t$ then yields
\begin{equation}
	[Y_0,X_{l+1}] = \dm = [Y_{l+1},X_0] = 0  \in \mf g.
\end{equation}

(We do \emph{not} use that $X_{l+1}$ be semisimple,
so this also works at the last step.)

\subsection{Proof of Lem.~\ref{lem:down_to_weyl}}
\label{proof:lem_down_to_weyl}

We generalise the argument of~\cite[Thm.~2.2.4]{collingwood_mcgovern_1993_nilpotent_orbits_in_semisimple_lie_algebras}.

Namely,
let $g \in G$ be such that $\Ad_g(X_i) = X'_i \in \mf t$ for $i \in \set{0,\dc,s-1}$.
Then $X'_0,\dc,X'_{s-1} \in \mf t \cap \Ad_g(\mf t) \sse \mf g$,
and both $\mf t$ and $\Ad_g(\mf t)$ are abelian subalgebras of $\mf g$:
it follows that $\mf t,\Ad_g(\mf t) \sse \mf l_{\phi'_0}$,
with $\phi'_0 = \bigcap_{i = 0}^{s-1} \phi_{X'_i} \sse \Phi$,
and both are Cartan subalgebras of the common infinitesimal centraliser.
Since the latter is reductive,
there exists an element $g' \in L_{\phi'_0} \sse G$ such that $\Ad_{g'}\Ad_g(\mf t) = \mf t$,
in the notation of~\eqref{eq:common_centraliser}.
Hence,
$g'' \ceqq g'g \in N_G(\mf t)$ and $\Ad_{g''}(X_i) = X'_i$ for $i \in \set{0,\dc,s-1}$.

\subsection{Proof of Lem.~\ref{lem:parabolic_to_parabolic}}
\label{proof:lem_parabolic_to_parabolic}

Clearly the $U\mf g$-submodule generated by the subspace $\mf u^-_{\psi \mid \wt \psi} \cdot v_{\mf p,\lambda} \sse M_{\mf p,\lambda}$ lies in the kernel,
as $\mf u^-_{\psi \mid \wt \psi} \sse \bigl [\,
		\wt{\mf p},\wt{\mf p} \,
		\bigr] \sse \ker(\wt \chi)$.

Conversely,
evaluating on the cyclic vector yields a $U(\mf u^-_{\psi})$-linear identification $U(\mf u^-_{\psi}) \lxra{\simeq} M_{\mf p,\lambda}$,
invoking the nilradical $\mf u^-_{\psi} = \bops_{\nu} \mf g_{-\alpha}$ of the opposite parabolic subalgebra $\mf p^- = \mf p_{-\psi}$---%
and analogously for $M_{\wt{\mf p},\wt \lambda}$.
In these identification,
the induced map $U (\mf u^-_{\psi}) \thra U( \mf u^-_{\wt \psi})$ is the canonical projection modulo the left ideal $U (\mf u^-_{\psi}) \cdot \mf u^-_{\wt \psi} \sse U (\mf u^-_{\psi})$:\fn{
Note that $U(\mf u^-_{\psi}) \cdot \mf u^-_{\wt \psi}$ is actually \emph{bilateral} in $U (\mf u^-_{\psi})$,
e.g.,
applying Lem.~\ref{lem:nested_nilradicals} to the inclusion $\mf t \sse \mf p^-_{\psi} \sse \mf p^-_{\wt \psi}$;
so the projection is a ring morphism.}~the result follows by choosing suitable PBW bases,
cf.~the proof of Thm.~\ref{thm:wild_parabolic_to_wild_parabolic} more generally.

\subsection{Proof of Lem.~\ref{lem:induction_from_submodule}}
\label{proof:lem_induction_from_submodule}

The action map $U\wt{\mf h} \ts \wt V \to \wt V$ restricts to an $U\mf h$-balanced ($\mb C$-bilinear) function $U\wt{\mf h} \ts V \to \wt V$,
whence a universal $U\wt{\mf h}$-linear arrow $\Ind_{U\mf h}^{U\wt{\mf h}} V \to \wt V$ defined by
\begin{equation}
	\label{eq:restricted_action_map}
	X \ots_{U\mf h} v \lmt Xv \in \wt V,
	\qquad X \in U\wt{\mf h},
	\quad v \in V.
\end{equation}
Its image is the $U\wt{\mf h}$-submodule (of $\wt V$) generated by $V$.

Now choose a vector $v \in V$ as in the statement,
whence~\eqref{eq:restricted_action_map} is surjective---%
by the first item.
By the second item,
the (unique) $U\wt{\mf h}$-linear function $\wt V \to \Ind_{U\mf h}^{U\wt{\mf h}} V$ mapping $v \mt 1 \ots_{U\mf h} v$ is well-defined.
Indeed,
by hypothesis,
for all $Z \in \Ann_{U\wt{\mf h}} ( v )$ there are finitely many elements $X_i \in U\wt{\mf h}$ and $Y_i \in \Ann_{U\mf h} ( v )$ such that $Z = \sum_i X_i Y_i$,
whence
\begin{equation}
	Z (1 \ots_{U\mf h} v ) = Z \ots_{U\mf h} v = \sum_i (X_iY_i) \ots_{U\mf h} v = \sum_i X_i\ots_{U\mf h} (Y_i v) = 0 \in \Ind_{U\mf h}^{U\wt{\mf h}} V.
\end{equation}

It follows that this is an inverse to~\eqref{eq:restricted_action_map},
in view of the defining identity
\begin{equation}
	X (Y \ots_{U\mf h} v) = (XY) \ots_{U\mf h} v,
	\qquad X,Y \in U\wt{\mf h},
	\quad v \in V.
\end{equation}

\subsection{Proof of Lem.~\ref{lem:polarisation}}
\label{proof:lem_polarisation}

Choose a subset $\Sigma$ of a base $\Delta \sse \Phi$ of simple roots such that $\psi = \Phi^+_{\Delta} \cup \Phi^-_{\Sigma}$.
It follows that
\begin{equation}
	\mf u_{\psi} = \bops_{\Phi^+_{\Delta} \sm \Phi^+_{\Sigma}} \mf g_\alpha \sse \mf g,
\end{equation}
so that $\mf g = \mf u^-_{\psi} \ops \mf l_{\phi} \ops \mf u^+_{\psi}$,
with $\mf u^\pm_{\psi} \ceqq \mf u_{\pm \psi}$.
Then $P_{\mf u} = \mf l_{\phi} + \mf u_{\psi}$ is the corresponding parabolic subalgebra $\mf p_{\psi} \sse \mf g$,
and involutivity follows.
Furthermore,
if $\alpha,\beta \in \Phi^+_{\Delta} \sm \Phi^+_{\Sigma} \sse \Phi^+_{\Delta}$ then $\alpha + \beta \neq 0$,
whence
\begin{equation}
	\bigl( X \mid [Y,Z] \bigr) = \bigl( [X,Y] \mid Z \bigr) \sse (\mf g_\alpha \mid \mf g_\beta) = (0),
	\qquad Y \in \mf g_\alpha,
	\quad Z \in \mf g_\beta,
\end{equation}
which proves $\omega$-isotropicity;
coisotropicity follows for dimensional reasons.

\subsection{Proof of Lem.~\ref{lem:nested_nilradicals}}
\label{proof:lem_nested_radicals}

Choosing a base $\Delta \sse \Phi$ of simple roots,
write
\begin{equation}
	\mf{nil}(\wt{\mf p}) = \bops_{\Phi^+_{\Delta} \sm \Phi_{\wt \Sigma}} \mf g_\alpha \sse \bops_{\Phi^+_{\Delta} \sm \Phi_{\Sigma}} \mf g_\alpha = \mf{nil}(\mf p),
\end{equation}
for suitable nested subsets $\Sigma \sse \wt \Sigma$.
In particular $\mf{nil}(\wt{\mf p}) \sse \mf{nil}(\mf p)$.

Hence,
it is enough to prove the following:
if $\mf p$ is a standard parabolic subalgebra of $(\mf g,\mf b)$,
then $\bigl[\mf{nil}(\mf b),\mf{nil}(\mf p) \bigr] \sse \mf{nil}(\mf p)$.
Choose then as above $\Sigma \sse \Delta \sse \Phi$,
as well as $\alpha,\beta \in \Phi^+_{\Delta}$.
We prove that if $\gamma \ceqq \alpha + \beta$ lies in $\Phi^+_{\Sigma}$ then the same holds for $\alpha$ and $\beta$.

Indeed,
we can now decompose $\alpha,
	\beta \in \mf t^{\dual}$ as a linear combination of the elements of $\Delta = \Sigma \cup (\Delta \sm \Sigma)$ with nonnegative integer coefficients;
and by hypothesis $\gamma$ as a linear combination of the elements of $\Sigma$ only.
Now both $\Sigma$ and $\Delta \sm \Sigma$ are linearly independent sets (as they are contained in a $\mb C$-basis of $\mf t \cap [\mf g,\mf g]$),
so the uniqueness of the decomposition and the sign of the coefficients imply $\alpha,\beta \in \spann_{\mb Z_{\geq 0}}(\Sigma) \sse \Phi^+_{\Sigma}$.

\subsection{Proof of Lem.~\ref{lem:bracket_nested_parabolic}}
\label{proof:lem_bracket_nested_parabolic}

Using the splitting~\eqref{eq:parabolic_subalgebra_splitting},
one must establish the inclusion
\begin{equation}
	\bigl[\,\wt{\mf l},\wt{\mf l}\,\bigr] \ops \mf{nil}(\wt{\mf p}) \sse \bigl[ \mf p,\wt{\mf p}\bigr],
	\qquad \wt{\mf p} = \wt{\mf l} \ops \mf{nil}(\wt{\mf p}).
\end{equation}
To this end,
choose as usual (sub)sets $\Sigma \sse \wt \Sigma \sse \Delta$ such that $\mf p = \mf p_{\psi} = \mf l_{\phi} \ops \mf u_{\psi}$,
with $\psi = \Phi^+_{\Delta} \cup \Phi^-_{\Sigma} \in \mc P_\Phi$ and $\phi = \on{Lf}(\psi) \in \mc L_\Phi$;
and analogously for $\wt{\mf p} = \mf p_{\wt \psi}$.

First,
Lem.~\ref{lem:nested_nilradicals} in particular yields
\begin{equation}
	\mf u_{\wt \psi} \sse \mf u_{\psi} = [\mf t,\mf u_{\psi}] \sse \bigl[\wt{\mf p},\mf p \bigr].
\end{equation}

Second,
using that $\phi = \Phi_{\Sigma} \sse \Phi$ is a closed subset of roots one computes
\begin{equation}
	\label{eq:decomposition_derived_levi}
	[\mf l,\mf l] = \bops_{\Phi^+_{\Sigma}} \mf t_\alpha \ops \bops_{\Phi_{\Sigma}} \mf g_\alpha,
	\qquad \mf t_\alpha \ceqq \bigl[ \mf g_\alpha,\mf g_{-\alpha} \bigr] \sse \mf t,
\end{equation}
and analogously for $\wt{\mf l}$.
In turn,
if $\alpha \in \Phi^+_{\wt \Sigma}$ then
\begin{equation}
	\mf g_{\pm \alpha} = [\mf t,\mf g_{\pm \alpha}] \sse \bigl[ \mf p,\wt{\mf p}\bigr],
	\qquad \mf t_\alpha = \bigl[ \mf g_\alpha,\mf g_{-\alpha} \bigr] \sse \bigl[ \mf p,\wt{\mf p}\bigr],
\end{equation}
as $\mf p$ contains all the positive root lines.

\subsection{Proof of Prop.~\ref{prop:nonsingular_characters}}
\label{proof:prop_nonsingular_characters}

Choose an element $\bm Y \in \mf u^+_{\bm \psi}$ such that $B_{\bm \lambda}^{\bm \psi}(\bm Y,\bm Y') = 0$ for all $\bm Y' \in \mf u^-_{\bm \psi}$,
and suppose that $\bm \lambda \in \mf t^{r,\dual}_{\bm \phi^{\dual}}$ as in the statement:
we prove that $\bm Y = 0$,
i.e.,
the left radical is trivial.
To this end,
denote by $\pi^\pm_{\psi} \cl \mf g \thra \mf u^\pm_{\psi}$ the projection parallel to $\mf p_{\mp \psi}$ (for any pair of opposite parabolic subsets $\pm \psi \in \mc P_\Phi$).

Write as above $\bm Y = \sum_i Y_i \varepsilon^i$ and $\bm Y' = \sum_j Y'_j \varepsilon^j$.
If we choose $\bm Y' = Y'_{r-1} \varepsilon^{r-1}$,
then the vanishing of~\eqref{eq:nondegenerate_pairing_deeper_character} simplifies to
\begin{equation}
	\Braket{ \chi_{r-1},\pi_{\phi_{r-1}} \bigl( [Y_0,Y'_{r-1}] \bigr) } = 0,
\end{equation}
and the assumptions imply that $\chi_{r-1} \cl \mf l_{\phi_{r-1}} \to \mb C$ is a nonsingular character---%
as $\lambda_{r-1} \in \bigl( \phi_{r-1}^{\dual} \bigr)^{\! \perp} \bigsm \bigcup_{\Phi^{\dual} \sm \phi^{\dual}_{r-1}} \set{\alpha^{\dual}}^\perp$.
Hence,
we must have $\pi^+_{\psi_{r-1}}(Y_0) = 0$.
(Recall that $\mf u^+_{\psi_{r-1}} \sse \mf u^+_{\psi_0}$,
and it is actually a Lie ideal therein by Lem.~\ref{lem:nested_nilradicals}.)

Now suppose recursively that $\pi^+_{\psi_{r-1}}(Y_0) = \dm = \pi^+_{\psi_{r-1}}(Y_k) = 0$ for an integer $k \in \set{0,\dc,r-2}$,
i.e.,
that $Y_0,\dc,Y_k \in \mf p^-_{\psi_{r-1}}$.
If we choose $\bm Y' = Y'_{r-k-2} \varepsilon^{r-k-2}$,
with $Y'_{r-k-2} \in \mf u^-_{\psi_{r-1}} \sse \mf u^-_{\psi_{r-k-2}}$,
then the vanishing of~\eqref{eq:nondegenerate_pairing_deeper_character} simplifies to
\begin{equation}
	\Braket{ \chi_{r-1},\pi_{\phi_{r-1}} \bigl( [Y_{k+1},Y'_{r-k-2}] \bigr) } = 0.
\end{equation}
Indeed,
by (recursive) hypothesis the other nontruncated Lie brackets lie in $[\mf p^-_{\psi_{r-1}},\mf u^-_{\psi_{r-1}}] \sse \mf u^-_{\psi_{r-1}}$,
and all characters $\chi_i$ vanish on $\pi_{\phi_i}(\mf u^-_{\psi_{r-1}}) \sse \bops_\Phi \mf g_\alpha$---%
as $\chi_i \cl \mf l_{\phi_i} \to \mb C$ is the extension \emph{by zero} of a linear map $\lambda_i \cl \mf t \to \mb C$.
Hence,
$Y_{k+1} \in \ker\bigl( \pi^+_{\psi_{r-1}} \bigr)$ as well.

Thus,
by induction,
one has $\pi^+_{\psi_{r-1}}(Y_0) = \dm = \pi^+_{\psi_{r-1}}(Y_{r-1}) = 0$,
so that in particular $Y_{r-1} = 0$.
Now use this as the base of a new (descending) recursion:
suppose thus that $\pi^+_{\psi_k}(Y_0) = \dm = \psi^+_{\psi_k}(Y_{r-1}) = 0$ for some integer $k \in \set{r-1,\dc,1}$---%
so that in particular $Y_k = \dm = Y_{r-1} = 0$.
Choosing $\bm Y' = Y_{k-1}\varepsilon^{k-1}$,
with $Y_{k-1} \in \mf u^-_{\psi_{k-1} \mid \psi_k}$ (in the notation of Lem.~\ref{lem:parabolic_to_parabolic}),
the vanishing of~\eqref{eq:nondegenerate_pairing_deeper_character} simplifies to
\begin{equation}
	\Braket{ \chi_{k-1},\pi_{\phi_{k-1}} \bigl( [Y_0,Y'_{k-1}] \bigr) } = 0,
\end{equation}
using Lem.~\ref{lem:nested_left_radical_condition}---%
as $\mf u^-_{\psi_{k-1} \mid \psi_k} \sse \bops_{\phi_k} \mf g_\alpha$,
in view of the inclusion $\nu_{k-1} \sm \nu_k \sse \phi_k$.
Therefore,
$\pi^+_{\psi_{k-1}}(Y_0) = 0$,
since by hypothesis $\lambda_{k-1} \in \mf t^{\dual}$ does \emph{not} vanish on $\mf t_\alpha \sse \mf t$ for all $\alpha \in \phi_k \sm \phi_{k-1}$.

Now run the nested recursive step,
and suppose that $\pi^+_{k-1}(Y_0) = \dm = \pi^+_{k-1}(Y_l) = 0$ for some integer $l \in \set{0,\dc,k-2}$.
Choosing $\bm Y' = Y'_{k-l-2} \varepsilon^{k-l-2}$,
with $Y'_{k-l-2} \in \mf u^-_{\psi_{k-1} \mid \psi_k}$,
the vanishing of~\eqref{eq:nondegenerate_pairing_deeper_character} simplifies to
\begin{equation}
	\Braket{ \chi_{k-1},\pi_{\phi_{k-1}} \bigl( [Y_{l+1},Y'_{k-l-2}] \bigr) } = 0.
\end{equation}
Here we have used both the recursive hypotheses and Lem.~\ref{lem:nested_left_radical_condition},
and one concludes that $\pi^+_{k-1}(Y_{l+1}) = 0$.
(This is the master recursive step.)

By induction,
one has $\pi^+_{k-1}(Y_0) = \dm = \pi^+_{k-1}(Y_{k-1}) = 0$---%
so that in particular $Y_{k-1} = 0$;
on the whole one has $\pi^+_{k-1}(Y_0) = \dm = \pi^+_{k-1}(Y_{r-1}) = 0$,
and finally by nested induction one has $\pi^+_0(Y_0) = \dm = \pi^+_0(Y_{r-1}) = 0$,
i.e.,
$\bm Y = 0$.

The right radical is then also trivial,
as $\dim_{\mb C} \bigl( \mf u^+_{\bm \psi} \bigr) = \dim_{\mb C} \bigl( \mf u^-_{\bm \psi} \bigr) < \infty$.

Conversely,
start from a character $\bm \chi \cl \mf l_{\bm \chi} \to \mb C$ which extends to $\mf S^{\bm \psi}$,
corresponding to a tuple $\bm \lambda = (\lambda_0,\dc,\lambda_{r-1})$ with $\lambda_i \in \mf Z_{\phi_i}^{\dual}$.
Assume that $\bm \lambda$ does \emph{not} lie in the stratum $\mf t^{\dual,r}_{\bm \phi^{\dual}}$:
then there exists a minimal integer $k \in \set{0,\dc,r-1}$ such that $\braket{ \lambda_{r-k-1},\alpha^{\dual} } = 0$ for some $\alpha \in \phi_{r-k} \sm \phi_{r-k-1}$,
and the above recursion breaks at that stage.
Namely,
up to changing sign we can assume that $\alpha \in \nu_{r-k-1} \sm \nu_{r-k} \sse \nu_0$,
and consider the element $\bm Y = E_\alpha$.
If $Y'_i \in \mf u^-_{\psi_i}$ one has
\begin{equation}
	\pi_{\phi_i} \bigl( [\bm Y,Y'_i] \bigr) \in
	\begin{cases}
		\mf l_{\phi_i} \cap \bops_\Phi \mf g_\alpha,
		    & \quad i \in \set{r-k,\dc,r-1},
		\\
		(0) & \quad i \in \set{0,\dc,r-k-2},
	\end{cases}
\end{equation}
as $E_\alpha \in \mf u^+_{\psi_l}$ if and only if $l \leq r-k-1$,
and conversely $E_\alpha \in \bops_{\phi_l} \mf g_\alpha$ if and only if $l \geq r-k-1$.
Hence,
finally,
one has
\begin{equation}
	B_{\bm \lambda}^{\bm \psi}(E_\alpha,\bm Y') = \Braket{ \chi_{r-k-1},\pi_{\phi_{r-k-1}} \bigl( [E_\alpha,Y'_{r-k-1}] \bigr) },
	\qquad \bm Y' \in \mf u^-_{\bm \psi},
\end{equation}
which vanishes by the hypothesis on $\bm \lambda$---%
so that $\bm Y \neq 0$ lies in the left radical.

\subsection{Proof of Lem.~\ref{lem:reducing_poisson_bracket}}
\label{proof:lem_reducing_poisson_bracket}

For the first inclusion,
one must prove that $[\mf l_{\bm \phi},\mf g_r] \sse \ker(\bm \lambda)$.
To this end,
choose elements $\bm X,\bm X' \in \mf t_r$,
and moreover
\begin{equation}
	\bm Y \in \bops_{i = 0}^{r-1} \Bigl( \bops_{\phi_i} \mf g_\alpha \cdot \varepsilon \Bigr),
	\qquad \bm Y' \in \bops_{j = 0}^{r-1} \Bigl( \bops_\Phi \mf g_\beta \cdot \varepsilon^j \Bigr) \sse \mf g_r.
\end{equation}
Then $[\bm X,\bm X'] = 0$,
and $\pi_{\mf t_r} \bigl( [\bm X,\bm Y'] + [\bm X',\bm Y] \bigr) = 0$,
denoting by $\pi_{\mf t_r} \cl \mf g_r \thra \mf t_r$ the projection along $\bops_{i = 0}^{r-1} \bigl( \bops_\Phi \mf g_\alpha \cdot \varepsilon^i \bigr) \sse \mf g_r$.
Finally,
for the troublesome addend of the expansion $[\bm X + \bm Y,\bm X' + \bm Y'] \in \mf g_r$,
one computes
\begin{equation}
	[\bm Y,\bm Y'] = \bops_{k = 0}^{r-1} \Biggl( \sum_{i+j = k} \sum_{\phi_i \ts \Phi} [\mf g_\alpha,\mf g_\beta] \Biggr) \cdot \varepsilon^k \sse \mf g_r,
\end{equation}
so that $\pi_{\mf t_r} \bigl( [\bm Y,\bm Y'] \bigr) \in \mf t_r$ is spanned by Lie brackets
\begin{equation}
	[E_\alpha,E_\beta] \varepsilon^k,
	\qquad \alpha = -\beta \in \phi_i \sse \Phi,
	\quad i \leq k \in \set{0,\dc,r-1}.
\end{equation}
But the latter lie in the subspace $\bigl( \bops_{\phi_i} \mf t_\alpha \bigr) \cdot \varepsilon^k \sse \ker(\lambda_k)$,
as by hypothesis $\lambda_k$ vanishes on $\mf t_\alpha$ for $\alpha \in \phi_k$---%
and one has $\phi_i \sse \phi_k$.

The proof of the second inclusion is analogous to proving that the subspaces $\mf u^\pm_{\bm \psi} \sse \mf g_r$ are isotropic for the (dual) KKS structure (cf.~Thm.~\ref{thm:wild_polarisations}).
Namely,
choose a Borel subalgebra contained in $\mf p^+_{\psi_0} \sse \mf g$,
i.e.,
a subsystem of positive roots $\Phi^+ \sse \Phi$ such that $\mf u^+_{\psi_i} \sse \bops_{\Phi^+} \mf g_\alpha$ for $i \in \set{0,\dc,r-1}$.
Then one finds
\begin{equation}
	\bigl[ \mf u^+_{\bm \psi},\mf u^+_{\bm \psi} \bigr] \sse \bops_{i = 0}^{r-1} \Biggl( \bops_{\Phi^+} \mf g_\alpha \cdot \varepsilon^i \Biggr) \sse \ker(\pi_{\mf t_r}).
\end{equation}
(The same argument works verbatim for the opposite subspace.)

\subsection{Proof of Lem.~\ref{lem:shapovalov_is_symmetric}}
\label{proof:lem_shapovalov_is_symmetric}

If $X,X' \in U\mf g$,
one has
\begin{equation}
	\ps{t}{} X \cdot X' - \ps{t}{}{X'} \cdot X = \ps{t}{} X \cdot X' - \ps{t\!}{}{\bigl( \ps{t}{}{ X} \cdot X' \bigr) } \in U\mf g.
\end{equation}
So it is enough to show that
\begin{equation}
	\Braket{ \chi,\pi_{\phi}(Y - \ps{t}{}{ Y}) } = 0,
	\qquad Y \in U\mf g.
\end{equation}
To this end,
decompose (uniquely) $Y = Y' + Y''$ along the direct sum~\eqref{eq:universal_enveloping_triangular_decomposition},
and note that $Y'' - \ps{t}{}{Y''} \in \mf u^-_{\psi} \cdot U\mf g + U\mf g \cdot \mf u^+_{\psi}$ is annihilated by $\pi_{\phi}$.
Furthermore,
write $Y' = \sum_i P_i$ as a finite sum of monomials
\begin{equation}
	P_i = X_1^{(i)} \dm X_{m_i}^{(i)},
	\qquad X^{(i)}_j \in \mf l_{\phi},
	\quad m_i \in \mb Z_{\geq 0}.
\end{equation}
By construction,
the transposition fixes $\mf t$ pointwise,
hence all monomials of $Y' - \ps{t}{}{Y'} = \sum_i P_i - \ps{t}{}{P_i}$ contain at least one factor from the subspace $\bops_{\phi} \mf g_\alpha \sse \mf l_{\phi}$.
But then
\begin{equation}
	\Braket{ \chi,\pi_{\phi} (Y' - \ps{t}{}{Y'}) } = \sum_i \Braket{ \chi,P_i - \ps{t}{}{P_i} } = 0,
\end{equation}
as $\bops_{\phi} \mf g_\alpha \sse \ker(\chi)$.

\subsection{Proof of Lem.~\ref{lem:annihilator_cyclic_vector_in_shapovalov_radical}}
\label{proof:lem_annihilator_cyclic_vector_in_shapovalov_radical}

If $X' \in \mf u^+_{\psi}$,
then $\ps{t}{} X \cdot X' \in U\mf g \cdot \mf u^+_{\psi} \sse \ker(\pi_{\phi})$.

Choose otherwise an element $X' \in \ker(\chi) \sse U\mf l_{\phi}$,
and decompose
\begin{equation}
	\ps{t}{} X = Y + Y',
	\qquad Y \in U\mf l_{\phi},
	\quad Y' \in \mf u^-_{\psi} \cdot U\mf g + U\mf g \cdot \mf u^+_{\psi}.
\end{equation}
Then $YX' = \pi_{\phi}(YX') \in \ker(\chi)$,
while for the other summand write further (albeit not uniquely)
\begin{equation}
	Y' = Z Z' + \wt Z \wt Z',
	\qquad Z \in \mf u^-_{\psi},
	\quad Z',\wt Z \in U\mf g,
	\quad \wt Z' \in \mf u^+_{\psi}.
\end{equation}
Thus,
\begin{equation}
	Y'X' = ZZ'X' + \wt Z\wt Z' X' = ZZ'X' + \wt Z \bigl( X'\wt Z' + [\wt Z',X'] \bigr),
\end{equation}
and one has
\begin{equation}
	ZZ'X' \in \mf u^-_{\psi} \cdot U\mf g,
	\qquad \wt ZX'\wt Z',
	\wt Z [\wt Z',X'] \in U\mf g \cdot \mf u^+_{\psi},
\end{equation}
using $[U\mf l_{\phi},\mf u^\pm_{\psi}] \sse U\mf g \cdot u^\pm_{\psi}$ for the latter inclusion;
in turn,
this can be proven by (strong) induction along the standard filtration of $U\mf l_{\phi}$,
noting that $[\mf l_{\phi},\mf u^\pm_{\psi}] \sse \mf u^\pm_{\psi}$.
The statement now follows from~\eqref{eq:annihilator_cyclic_vector_finite_generalised_verma}.

\subsection{Proof of Lem.~\ref{lem:maximal_proper_submodule_wild}}
\label{proof:lem_maximal_proper_submodule_wild}

Let $M^{\bm \psi}_{\bm \lambda} = N_1 \ops N_2$ be a $U\mf g_r$-linear splitting.
Then $v^{\bm \psi}_{\bm \lambda} = v_1 + v_2$,
for unique vectors $v_i \in N_i$ ($i \in \set{1,2}$):
acting with $X \in \mf Z_{\phi_0} \sm (0)$ shows that $v_1$ is a weight vector of weight $\lambda_0$,
so that it is collinear with $v^{\bm \psi}_{\bm \lambda}$,
and $(0) \ssne N_1 \ssne M^{\bm \psi}_{\bm \lambda}$ is impossible.
(Note that $\mf Z_{\phi_0} = (0)$ implies that $\phi_0 = \Phi$,
in which case $\mf l_{\bm \phi} = \mf p_{\bm \psi} = \mf g_r$ and one has $M^{\bm \psi}_{\bm \lambda} \simeq \mb C_{\bm \chi}$ as $U\mf g_r$-module.)

The second statement follows from the fact that all proper submodules of $M^{\bm \psi}_{\bm \lambda}$ lie in $\bops_{Q^+_{\phi_0} \sm (0)} M^{\bm \psi}_{\bm \lambda}[\mu]$,
whence~\eqref{eq:maximal_proper_submodule_wild} is proper.

\subsection{Proof of Lem.~\ref{lem:parabolic_induction}}
\label{proof:lem_parabolic_induction}

Let $\psi \sse \wt\psi \sse \Phi$ (resp.,
$\phi \sse \wt\phi \sse \Phi$) be the parabolic subsets of roots of $\Phi = \Phi(\mf g,\mf t)$ corresponding to the parabolic subalgebras (resp.,
the Levi subsystems corresponding to the Levi factors).

The first statement follows from the fact that $\psi' \ceqq \psi \cap \wt\phi \sse \wt\phi$ is a parabolic subset of roots---%
viewing $\wt\phi$ as a root system in its own right.
The second statement then follows from the equality $\psi' \cap (-\psi') = \phi \sse \wt\phi$.
(Note that $-\psi' = (-\psi) \cap \wt \phi$;
etc.)

\subsection{Proof of Lem.~\ref{lem:abstract_commutator}}
\label{proof:lem_abstract_commutator}

The case $n = 1$ follows from the definition of the commutator of $\mc R$.

Suppose instead that $n \geq 1$:
then recursively
\begin{align}
	X \cdot Y_0 \dm Y_n & = [X,Y_0] \cdot Y_1 \dm Y_n + Y_0 X \cdot Y_1 \dm Y_n                                                                          \\
	                    & = \sum_{I \discup J = \set{1,\dc,n}} \bigl( \bm Y_I \cdot \bigl[ [X,Y_0],\bm Y \bigr]_J + Y_0 \bm Y_I \cdot [X,\bm Y]_J \bigr) \\
	                    & = \sum_{\substack{ I \discup J = \set{0,\dc,n},
	\\ 0 \in I} } \bm Y_I \cdot [X,\bm Y]_J + \sum_{\substack{ I \discup J = \set{0,\dc,n},
	\\ 0 \in J} } \bm Y_I \cdot [X,\bm Y]_J \\
	                    & = \sum_{I \discup J = \set{0,\dc,n}} \bm Y_I \cdot [X,\bm Y]_J.
\end{align}

\subsection{Proof of Cor.~\ref{cor:factorisation_shapovalov_matrix}}
\label{proof:cor_factorisation_shapovalov_matrix}

Let $Q \ceqq \wt Q - \on I_N$.
By construction,
one has $(DC)_{11} = d_1 c^{l_1}$,
and $(DC)_{1j} = 0$ for $j \in \set{2,\dc,N}$.
Thus,
in particular
\begin{equation}
	Q_{11} = \wt Q_{11} - 1 = \frac 1{d_1} c^{-l_1} P_{11},
	\quad Q_{1j} = \wt Q_{1j} = \frac 1{d_1} c^{-l_1} P_{1j},
	\qquad j \in \set{2,\dc,N}.
\end{equation}

Now fix $j \in \set{1,\dc,N}$,
and suppose recursively that:
(i) the coefficients $\wt Q_{1j},\dc,\wt Q_{i-1,j}$ have been determined;
and (ii) they satisfy the stated condition---%
for some $i \in \set{1,\dc,N-1}$.
The equality $\sum_{k = 1}^N (DC)_{ik} \wt Q_{kj} = A[\mu]_{ij}$ imposes
\begin{equation}
	\wt Q_{ij} = \frac 1{d_i} \Bigl( d_{ij} + c^{-l_i} P_{ij} - \sum_{k = 1}^{i-1} d_{ik} \wt Q_{kj} \Bigr) \in \mb C[c^{-1}],
\end{equation}
and by hypothesis $Q_{kj} = \wt Q_{kj} - \delta_{kj} \in c^{-1}\mb C[c^{-1}]$:
hence,
the degree-zero term is nonzero if and only $i = j$,
in which case it equals $1$.

\subsection{Proof of Lem.~\ref{lem:projected_inverse_shapovalov}}
\label{proof:lem_projected_inverse_shapovalov}

Let $v^{(+,0)} \ceqq (1 \ots p^+) v$ and $v^{(0,-)} \ceqq (p^- \ots 1) v$.
The $U(\mf u^\pm)$-linearity of $p^{\pm}$ yields $v^{(+,0)} \in \bigl( M^+_c \ots \mc V_0 \bigr)^{\! \mf u^+}$ and $v^{(0,-)} \in \bigl( \mc V_0 \ots M^-_c \bigr)^{\! \mf u^-}$.

We conclude the proof in the lowest-weight case---%
the highest-weight one being analogous.
The point is proving that
\begin{equation}
	\Delta(\bm X) v^{(0,-)} = \Braket{ \bm \chi_{-c},\bm X} v^{(0,-)},
	\qquad  \bm X \in \mf l.
\end{equation}

Suppose first that $\bm X \in \mf Z_{\bm \phi} \sse t_r$.
Using~\eqref{eq:explicit_projection_vacumm},
and writing (in Sweedler notation) $v = v^{(+)} w^+_c \ots v^{(-)} w^-_c \in M^+_c \ots M^-_c$,
one finds:
\begin{align}
	\label{eq:cartan_action_reduced_inverse_shapovalov}
	\Delta(\bm X) v^{(-,0)} & = \bigl( (\bm Xp^-) \ots 1 + p^- \ots \bm X \bigr) v                                                                                                  \\
	                        & = (\bm X \cdot v^{(+)} w_0) \ots v^{(-)} w^-_c + ( v^{(+)} w_0) \ots \bm X \cdot v^{(-)} w^-_c                                                        \\
	                        & = \bigl( \ad_{\bm X} \ots 1 + 1 \ots \ad_{\bm X} \bigr) \bigl( (v^{(+)} w_0) \ots v^{(-)} w^-_c \bigr) + \Braket{ \bm \lambda_{-c},\bm X } v^{(-,0)},
\end{align}
as by construction $\bm X$ acts by zero on $w_0$ and by $\bm \lambda_{-c}$ on $w_c^-$.
Moreover,
observe that
\begin{equation}
	\label{eq:equivariance_projection}
	\bigl[ \bm X,v^{(+)} \bigr] w_0 = p^- \bigl( [\bm X,v^{(+)}] w^+_c \bigr) \in \mc V_0,
\end{equation}
using again~\eqref{eq:explicit_projection_vacumm},
since $\bigl[ \mf t_r,\mf u^\pm \bigr] \sse \mf u^\pm$.
Hence,
~\eqref{eq:cartan_action_reduced_inverse_shapovalov} simplifies to
\begin{equation}
	\bigl( \Delta(\bm X) - \Braket{ \bm \chi_c,\bm X } \bigr) v^{(0,-)}  = \bigl( p^- \ots 1 \bigr) \Delta( \ad_{\bm X} ) v \in \mc V_0 \ots M^-_c.
\end{equation}
Now the $\mf g_r$-invariance of $v$ implies that the right-hand side vanishes,
noting that $\Delta(\bm X) v = \Delta(\ad_{\bm X}) v$.

Choose instead $\bm X \in \ker(\bm \chi_c) \sse \mf l$:
one must prove that $\Delta(\bm X) v^{(0,-)} = 0$.
Computing as in~\eqref{eq:cartan_action_reduced_inverse_shapovalov},
the point is establishing the identity
\begin{equation}
	p^- \bigl( \bm X \cdot v^{(+)} w^+_c \bigr) = \bm X \cdot v^{(+)} w_0 \in \mc V_0,
\end{equation}
as one can then conclude by the $\mf g_r$-invariance of $v$.
Clearly the above holds if $v^{(+)} = 1$,
but (contrary to~\eqref{eq:equivariance_projection}) we \emph{cannot} immediately use~\eqref{eq:explicit_projection_vacumm} in general.
(As usual,
this is only visible in the nongeneric wild case.)
Choose therefore an integer $k \geq 0$,
and suppose recursively that
\begin{equation}
	p^- (\bm X \cdot \bm Y w^+_c) = \bm X \cdot \bm Y w_0 \in \mc V_0,
	\qquad \bm Y w^+_c \in \mc F^{\leq k} (M^+_c),
\end{equation}
using the filtration of~\eqref{eq:filtration_finite_singularity_module}.
Then let $\wt{\bm Y} w^+_c \in \mc F^{\leq k+1}(M^+_c)$ for a suitable vector $\wt{\bm Y}$ in the PBW basis of $U(\mf u^-)$ constructed above.
One can factor
\begin{equation}
	\wt{\bm Y} = \bm X_{\alpha,i} \cdot \bm Y' \in U(\mf u^-),
	\qquad \alpha \in \nu_0,
	\quad i \in \set{0,\dc,d_\alpha-1},
\end{equation}
in the notation of~\eqref{eq:variables_pbw_basis},
so that $\bm Y' w^+_c \in \mc F^{\leq k}(M^+_c)$;
compute
\begin{equation}
	p^- \bigl( \bm X \cdot \wt{\bm Y} w^+_c \bigr) = \bm X_{\alpha,i} \cdot p^- \bigl( \bm X \bm Y' w^+_c \bigr) + p^- \bigl( [\bm X,\bm X_{\alpha,i}] \cdot \bm Y' w^+_c \bigr),
\end{equation}
by the $U(\mf u^-)$-linearity of $p^-$.
Now $\bm X \in \mf t_r$ yields $[\bm X,\bm X_{\alpha,i}] \sse \mf u^-$.
Otherwise,
one has $\bm X \in \bops_{i = 0}^{r-1} \bigl( \bops_{\phi_i} \mf g_\alpha \bigr) \cdot \varepsilon^i \sse \mf l$,
in which case a priori $[\bm X,\bm X_{\alpha,i}] \in [\mf l,\mf u^-] \sse \mf p^-$.
Finally,
one can (uniquely) decompose
\begin{equation}
	[\bm X,\bm X_{\alpha,i}] = Z^- + Z_0,
	\qquad Z^- \in \mf u^-,
	\quad Z_0 \in \mf l,
\end{equation}
and (reasoning as in Lem.~\ref{lem:nested_left_radical_condition}) one finds $Z_0 \in \ker(\bm \chi_c)$.
Thus,
the $U(\mf u^-)$-linearity of $p^-$ and the recursive hypothesis together yield
\begin{equation}
	p^- \bigl( \bm X \cdot \wt{\bm Y} w^+_c \bigr) = X_{\alpha,i} \cdot (\bm X \bm Y' w_0) + [\bm X,\bm X_{\alpha,i}] \cdot (\bm Y' w_0) = \bm X \cdot \wt{\bm Y} w_0.
\end{equation}

\subsection{Proof of Prop.~\ref{prop:invariant_vectors_completed_tensor_product}}
\label{proof:prop_invariant_vectors_completed_tensor_product}

We prove the statement in the highest-weight case:
the argument is the same in the lowest-weight one,
but the Shapovalov form must be modified---%
taking projections parallel to $\mf u^+ \cdot U\mf g_r + U\mf g_r \cdot \mf u^- \sse U\mf g_r$.
Importantly,
the argument we give also works for a completed tensor product $\mc V \wh \ots \wt M^\pm_c$,
with the factors swapped.

Introduce again the (ordered) basis $( X_{\alpha_k,i} )_{k,i}$ of $\mf u^-$,
and let $\mc B^- = \mc B^-_{\bm \psi}$ be the associated basis of $U (\mf u^-) \simeq M^+_c$ (cf.~\S~\ref{sec:identification_verma_symmetric_algebra}),
with tacit embedding $\mc B^- \sse U\mf g_r$.
Every element of~\eqref{eq:completed_tensor_product_verma} can be written $\wh v = \sum_{\mc B^-} (\bm X w^+_c) \ots v_{\bm X}$,
for suitable vectors $v_{\bm X} \in \mc V$---%
infinitely many of which may be nonzero.
The zeroth component is $w^+_c \ots v_1 \in M^+_c[0] \ots \mc V$,
and one can further decompose
\begin{equation}
	\mc B^- = \discup_{Q^+_{\phi_0}} \mc B^-[\mu],
	\qquad \mc B^-[\mu] \ceqq \mc B^- \cap U (\mf u^-)_{-\mu} \sse \mc B^-.
\end{equation}
Finally,
denote by $\mc B^+ = \coprod_{Q^+_{\phi_0}} \mc B^+[\mu]$ the dual basis of $U (\mf u^+)$ with respect to $\mc S_c$,
i.e.,
the matrix of $\mc S_{c,\mu}$ is the identity when expressed in the corresponding bases of $M^\pm_c[\mu]$,
for all $\mu \in Q^+_{\phi_0}$.
(Beware that this is \emph{not} the same as the PBW basis corresponding to~\eqref{eq:dual_basis_positive_nilradical},
which only relies on a nonsingular character.)

Now suppose that $\wh v \in \bigl( \wt M^+_c \wh \ots \mc V \bigr)^{\! \mf u^+}$,
and that there is an integer $k \geq 0$ such that the corresponding vectors $v_{\bm X} \in \mc V$ have been determined whenever $\bm X w^+_c \in \mc F^{\leq k} (M^+_c)$,
using the filtration~\eqref{eq:filtration_finite_singularity_module}.
Choose an element $\wt{\bm X} \in \mc B^- \bigl[ \wt \mu \bigr]$ with $\abs{\wt \mu}_{\nu_0} = k+1$,
and denote by $\wt{\bm Y} \in \mc B^+ \bigl[ \wt \mu \bigr]$ its corresponding dual.
The main point is that the Shapovalov form satisfies
\begin{equation}
	\label{eq:shapovalov_and_cyclic_vector}
	\bigl( \iota(\bm Y) \bm X - \mc S_c( \bm Y,\bm X) \bigr) w^+_c \in \mf u^- U\mf g_r \cdot w^+_c = \bops_{Q^+_{\phi_0} \sm (0)} M^+_c[\mu],
	\qquad \bm Y ,\bm X \in U\mf g_r,
\end{equation}
since $\mf u^+ w^+_c = (0)$,
and $U\mf l$ acts on the canonical generator by the given character.
Consider then the element $\wt{\bm Z} \ceqq \iota \bigl( \wt{\bm Y} \bigr) \in U(\mf u^+)$.
If $\wt{\bm Z} = Z_1 \dm Z_l$ for some integer $l \geq 1$,
one has
\begin{equation}
	\Delta \bigl( \wt{\bm Z} \bigr) = \sum_{I \discup J = \set{1,\dc,l}} \bigl( \wt{\bm Z}_I \ots \wt{\bm Z}_J \bigr) \in U(\mf u^+) \ots U(\mf u^+),
\end{equation}
in the notation of Lem.~\ref{lem:abstract_commutator}.
Now for $I \ssne \set{1,\dc,l}$ one has $\wt{\bm Z}_I w^+_c \in \mc F^{\leq k} (M^+_c)$,
and the Shapovalov-orthogonality of the decomposition yields
\begin{equation}
	\mc S_c \bigl( \iota (\wt{\bm Z}_I),\bm X \bigr) = 0,
	\qquad \bm X \in \bops_{\abs \mu_{\nu_0} \geq k_1} U(\mf u^-)_{-\mu},
\end{equation}
as $\iota$ preserves the $\mf Z_{\phi_0}^{\dual}$-grading.
Together with~\eqref{eq:shapovalov_and_cyclic_vector},
this implies that
\begin{equation}
	\label{eq:dual_element_action_1}
	\bigl( \wt{\bm Z}_I \ots \wt{\bm Z}_J \bigr) \wh v - \sum_{\abs \mu_{\nu_0} \leq k} \Biggl( \,
	\sum_{\mc B^-[\mu]} \mc S_c \bigl( \iota(\wt{\bm Z}_I),\bm X \bigr) \bigr( w^+_c \ots \wt{\bm Z}_J v_{\bm X} \bigr) \Biggr) \in \prod_{Q^+_{\phi_0} \sm (0)} \bigl( M^+_c[\mu] \ots \mc V \bigr),
\end{equation}
if 	$J \neq \vn$,
and analogously
\begin{equation}
	\label{eq:dual_element_action_2}
	\bigl( \wt{\bm Z} \ots 1 \bigr) \wh v - w^+_c \ots v_{\wt{\bm X}} \in \prod_{Q^+_{\phi_0} \sm (0)} \bigl( M^+_c[\mu] \ots \mc V \bigr) \sse \wt M^+_c \wh \ots \mc V.
\end{equation}
Finally,
sum~\eqref{eq:dual_element_action_1}--\eqref{eq:dual_element_action_2} using the identity $\Delta \bigl( \wt{\bm Z} \bigr) \wh v = 0$.
This yields $w^+_c \ots v_{\wt{\bm Z}} \in \prod_{Q^+_{\phi_0} \sm (0)} \bigl( M^+_c[\mu] \ots \mc V \bigr)$,
where
\begin{equation}
	v_{\wt{\bm Z}} \ceqq v_{\wt{\bm X}} + \sum_{\substack{ I \discup J = \set{1,\dc,l},
			\\ J \neq \vn }} \Biggl( \,
	\sum_{\abs \mu_{\nu_0} \leq k} \sum_{\mc B^-[\mu]} \mc S_c \bigl( \iota(\wt{\bm Z}_I),\bm X \bigr) \wt{\bm Z}_J v_{\bm X} \Biggr) \in \mc V.
\end{equation}
(The nested sum is finite.)
In view of the given decomposition~\eqref{eq:completed_tensor_product_verma},
one has $v_{\wt{\bm Z}} = 0$,
and so indeed $v_{\wt{\bm X}} \in \mc V$ is a linear combination of lower-height terms.

\subsection{Proof of Lem.~\ref{lem:quotient_stratification_1}}
\label{proof:lem_quotient_stratification_1}

Suppose that there is an $H$-orbit intersecting $T_i \cap T_j$.
Then $H$ maps some element of $T_i$ into $T_j$,
and by~\eqref{eq:H_stratification} there exists $g \in H$ such that $j = g.i$ and $T_j = g(T_i)$:
therefore,
one has $\ol i = \ol j$,
and in conclusion an $H$-orbit intersects $T_i$ if and only if it intersects $T_j$.
The remaining implications are clear.

\subsection{Proof of Lem.~\ref{lem:quotient_stratification_2}}
\label{proof:lem_quotient_stratification_2}

The subspace in the statement is closed:
its complement corresponds to the subspace of $H$-orbits contained in $T \sm \ol{T_i} \sse T$,
whose lift in $T$ is the $H$-invariant open subspace
\begin{equation}
	T \sm H \bigl( \ol{T_i} \bigr) = T \sm \Biggl(\bigcup_{g \in H} \ol{T_{g.i}} \Biggr) \sse T.
\end{equation}
(Here we use that $H$ is finite.)
Thus,
the subspace $\ol{T^H_i}$ is contained therein.

Conversely,
choose $x \in \ol{T_i}$.
Any open neighbourhood of $Hx \in T \bs H$ lifts to an $H$-invariant open subspace $U \sse T$ containing the orbit.
In particular,
one has $x \in U$,
so that $U \cap T_i \neq \vn$;
by $H$-invariance,
$U$ contains an $H$-orbit intersecting $T_i$,
and its projection in the quotient thus intersects $T^H_i$,
whence $Hx \in \ol{T^H_i}$.

\section{Acronyms (alphabetical order)}
\label{sec:acronyms}

\begin{description}
	\item[BGG] the category $\mc O$ of Bernstein--Gelfand--Gelfand;

	\item[BMT/GT] the Bonelli--Maruyoshi--Tanzini/Gaiotto--Teschner Virasoro pair,
	      and corresponding Whittaker vector;

	\item[CFT] conformal field theory;

	\item[DMT] the flat connection of De Concini/Millson--Toledano Laredo;

	\item[FMTV] the flat connection of Felder--Markov--Tarasov--Varchenko;

	\item[KKB] the hyperkähler structure of Kronheimer--Kovalev--Biquard;

	\item[KKS] the symplectic structure of Kirillov--Kostant--Souriau;

	\item[KZ] the flat connection of Knizhnik--Zamolodchikov;

	\item[KZB] Bernard's higher-genus version of KZ;

	\item[PBW] the $\mb C$-bases à la Poincaré--Birkhoff--Witt;

	\item[TCLA] a truncated-current Lie algebra;

	\item[TCLG] a truncated-current Lie group;

	\item[TUY] the projectively-flat connection of Tsuchiya--Ueno--Yamada;

	\item[UEA] the universal enveloping algebra of a Lie algebra;

	\item[WZNW] the model of Wess--Zumino--Novikov--Witten.
\end{description}

\backmatter

\bibliographystyle{amsplain}
\bibliography{bibliography}
\end{document}